\documentclass[reqno,10pt]{amsart}

\usepackage{amsmath,amssymb,mathrsfs,graphicx,enumitem}

\usepackage[pdftex=true,
	bookmarks=true,
	bookmarksopen=true,
	bookmarksnumbered=true,
	pdftoolbar=true,
	pdfmenubar=true,
	pdffitwindow=false,
	pdfstartview={FitH},
	colorlinks=true,
	linkcolor=red,
	citecolor=blue,
	breaklinks=true]{hyperref}

\allowdisplaybreaks

\newlength{\enummargin}
\newcounter{CounterEnumi}

\numberwithin{equation}{section}

\newtheorem{thm}{Theorem}[section]
\newtheorem{prop}[thm]{Proposition}
\newtheorem{lemma}[thm]{Lemma}
\newtheorem{cor}[thm]{Corollary}
\newtheorem*{thm*}{Theorem}
\newtheorem*{prop*}{Proposition}
\newtheorem*{cor*}{Corollary}
\newtheorem*{conj*}{Conjecture}

\theoremstyle{definition}
\newtheorem{definition}[thm]{Definition}

\theoremstyle{remark}
\newtheorem{rmk}[thm]{Remark}

\newcommand{\itref}[1]{(\ref{#1})}

\newcommand{\Hom}{\operatorname{Hom}}
\newcommand{\End}{\operatorname{End}}

\newcommand{\ms}{\mathscr}
\newcommand{\sF}{\ms F}
\newcommand{\sM}{\ms M}
\newcommand{\sE}{\ms E}

\newcommand{\mc}{\mathcal}
\newcommand{\cI}{\mc I}

\newcommand{\cH}{\mc H}
\newcommand{\cL}{\mc L}
\newcommand{\cO}{\mc O}

\newcommand{\cR}{\mc R}
\newcommand{\cS}{\mc S}
\newcommand{\cV}{\mc V}
\newcommand{\cX}{\mc X}
\newcommand{\cY}{\mc Y}

\newcommand{\la}{\langle}
\newcommand{\ra}{\rangle}
\newcommand{\pa}{\partial}
\newcommand{\bpa}{{}^\bl\!\pa}

\newcommand{\tn}{\textnormal}

\newcommand{\R}{\mathbb R}

\newcommand{\HH}{\mathbb H}
\newcommand{\C}{\mathbb C}

\newcommand{\N}{\mathbb N}
\newcommand{\Z}{\mathbb Z}

\newcommand{\Sphere}{\mathbb S}
\newcommand{\eps}{\epsilon}

\newcommand{\supp}{\operatorname{supp}}

\newcommand{\F}{\digamma}
\newcommand{\xra}{\xrightarrow}

\renewcommand{\Im}{\operatorname{Im}}
\renewcommand{\Re}{\operatorname{Re}}

\newcommand{\CI}{\mc C^\infty}
\newcommand{\CnI}{\mc C^{-\infty}}
\newcommand{\CIc}{\mc C^\infty_\cl}
\newcommand{\CIdot}{\dot{\mc C}^\infty}
\newcommand{\CIdotc}{\CIdot_\cl}

\newcommand{\Diff}{\mathrm{Diff}}
\newcommand{\Diffb}{\Diff_\bl}
\newcommand{\bl}{{\mathrm{b}}}
\newcommand{\cl}{{\mathrm{c}}}
\newcommand{\Hb}{H_{\bl}}
\newcommand{\Hbloc}{H_{\bl,\loc}}
\newcommand{\Tb}{{}^{\bl}T}
\newcommand{\rcTb}{{}^{\bl}\overline{T}}
\newcommand{\CTb}{{}^{\C\bl}T}
\newcommand{\Sb}{{}^{\bl}S}
\newcommand{\loc}{{\mathrm{loc}}}
\newcommand{\WF}{\mathrm{WF}}
\newcommand{\WFb}{\WF_{\bl}}
\newcommand{\Psib}{\Psi_{\bl}}
\newcommand{\Psibc}{\Psi_{\bl,\cl}}
\newcommand{\fpsi}{\Psi}
\newcommand{\fpsib}{\Psi_\bl}
\newcommand{\Vb}{\cV_\bl}
\newcommand{\Db}{{}^{\bl}\!D}
\newcommand{\lb}{\mathrm{lb}}
\newcommand{\rb}{\mathrm{rb}}
\newcommand{\Rhalf}{{\overline{\R_+}}}
\newcommand{\Rnhalf}{{\overline{\R^n_+}}}
\newcommand{\ftrans}{\;\!\widehat{\ }\;\!}
\newcommand{\sym}{\mathrm{sym}}
\newcommand{\dv}{\mathrm{div}}
\newcommand{\dS}{{\mathrm{dS}}}

\newcommand{\sfa}{\mathsf{a}}
\newcommand{\sfb}{\mathsf{b}}
\newcommand{\sfe}{\mathsf{e}}
\newcommand{\sff}{\mathsf{f}}
\newcommand{\sfp}{\mathsf{p}}
\newcommand{\sfH}{\mathsf{H}}
\newcommand{\wt}{\widetilde}
\newcommand{\wh}{\widehat}

\newcommand{\ft}{\mathfrak{t}}
\newcommand{\bdiff}{{}^\bl\!d}
\newcommand{\Sym}{S}
\newcommand{\Sph}{\mathbb S}
\newcommand{\psdo}{ps.d.o.}

\newcommand{\wozero}{\setminus o}

\newcommand{\weakto}{\rightharpoonup}

\begin{document}
\title[Quasilinear wave equations]{Global analysis of quasilinear wave equations on asymptotically de Sitter spaces}

\author{Peter Hintz}

\address{Department of Mathematics, Stanford University, CA 94305-2125, USA}
\email{phintz@math.stanford.edu}

\address{Department of Mathematics, University of California, Berkeley, CA 94720-3840, USA}
\email{phintz@berkeley.edu}

\date{November 26, 2013. Final revision: October 6, 2015.}

%nonlinear second order PDE of hyperbolic type: 35L70
%asymptotic behavior of solutions: 35B40
%General Theory of PsDO: 35S05
%propagation of singularities; initial value problems: 58J47

\begin{abstract}
  We establish the small data solvability of suitable quasilinear wave and Klein-Gordon equations in high regularity spaces on a geometric class of spacetimes including asymptotically de Sitter spaces. We obtain our results by proving the \emph{global} invertibility of linear operators with coefficients in high regularity $L^2$-based function spaces and using iterative arguments for the nonlinear problems. The linear analysis is accomplished in two parts: Firstly, a regularity theory is developed by means of a calculus for pseudodifferential operators with non-smooth coefficients, similar to the one developed by Beals and Reed, on manifolds with boundary. Secondly, the asymptotic behavior of solutions to linear equations is studied using resonance expansions, introduced in this context by Vasy using the framework of b-analysis.
\end{abstract}

\maketitle

%%%%%%%%%%%%%%%%%%%%%%%%%%%%%%%%%%%%%%%%%%%%%%%%%%%%%%%%%%%%%%%%%%%%%
\section{Introduction}
\label{SecIntroduction}

We study quasilinear wave equations on spacetimes for which infinity has a structure generalizing that of static de Sitter space. We prove global existence and exponential decay to constants of scalar quasilinear waves. For concreteness, we state our results in the special case of scalar waves on static de Sitter space, but it is important to keep in mind that the geometric and analytic settings are more general, see also the discussion further below.

The region of de Sitter space we are working on is a (non-compact) $n$-dimensional manifold
\[
  M^\circ = \R_{t_*} \times \{|Z|<1+2\delta\},\quad\delta>0, Z\in\R^{n-1},
\]
which extends past the cosmological horizon $|Z|=1$ of static de Sitter space. The manifold $M^\circ$ is equipped with a stationary Lorentzian metric $g_\dS$ (i.e.\ $\pa_{t_*}$ is a Killing vector field) which depends on the cosmological constant $\Lambda>0$, though we drop this in the notation. Concretely, in static coordinates $(t,r,\omega)$ on $\R_t\times(0,1)_r\times\Sphere^{n-2}_\omega$, the static de Sitter metric, see e.g.\ \cite[\S4.2]{VasyMicroKerrdS},\footnote{Our $t,t_*$ are denoted $\tilde t,t$ in \cite{VasyMicroKerrdS}.} does not extend to the cosmological horizon $r=1$ due to a coordinate singularity, but introducing $t_*=t+h(r)$ with a suitable function $h$ as in \cite[\S4.3]{VasyMicroKerrdS}, the metric does extend smoothly to $r=1$ and beyond.

In order to set up our problem, see Figure~\ref{FigDS} for an illustration, we consider the domain
\[
  \Omega^\circ = [0,\infty)_{t_*}\times\{|Z|\leq 1+\delta\} \subset M^\circ,
\]
which is a submanifold with corners with two boundary hypersurfaces, which are the intersections of
\[
  H_1 = \{t_*=0\}, \quad H_2 = \{|Z|=1+\delta\}
\]
with $\Omega^\circ$. Thus, $H_1$ is a Cauchy hypersurface, and $H_2$ has two connected components, both spacelike hypersurfaces. We are interested in solving the forward problem for wave-like equations in $\Omega^\circ$, i.e.\ imposing vanishing Cauchy data at $H_1$; initial value problems with general Cauchy data can always be converted into an equation of this type.

\begin{figure}[!ht]
  \centering 
  \includegraphics{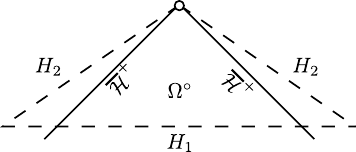}
  \caption{Penrose diagram of the domain $\Omega^\circ$, bounded by the dashed lines, on which we solve quasilinear wave equations. $\overline{\cH}^+$ is the cosmological horizon. Moreover, $H_1$ is a Cauchy hypersurface, where we impose vanishing Cauchy data, and $H_2$ has two spacelike connected components.}
\label{FigDS}
\end{figure}

The simplest (even though not the most natural from a geometric perspective, see below) wave equations we consider are of the form
\[
  \Box_{g(u)}u = f+q(u,du),
\]
$u$ real-valued, where $g(0)=g_\dS$ is the static de Sitter metric, and at each $p=(t_*,Z)\in M^\circ$, the metric $g(u)$ is $g_p(u(p))$, where
\[
  g_p\colon\R \to S^2 T^*_p M
\]
depends smoothly on $p$, and in fact only depends on $Z$, \emph{but not on $t_*$}; further
\[
  q(u,du) = \sum_{j=1}^{N'} a_j u^{e_j}\prod_{k=1}^{N_j} X_{jk}u, \quad e_j,N_j\in\N_0,\ N_j+e_j\geq 2,
\]
with $a_j\in\CI(M^\circ)$ and $X_{jk}\in\cV(M^\circ)$ real and independent of $t_*$. (Here $a_j$ is only relevant if $N_j=0$, since it can otherwise be absorbed into $X_{j1}$.)

Using the stationary nature of the spacetime further, we will use the Sobolev spaces $H^s$ on $M^\circ=\R_{t_*}\times X$, with $X=\{|Z|<1+2\delta\}$ considered as an open subset of $\R^{n-1}_Z$; thus, the norm of $u\in H^s$ is defined by
\begin{equation}
\label{EqKdSSobolev}
  \|u\|_{H^s}^2 = \sum_{j+|\beta|\leq s} \|\pa_{t_*}^j \pa_y^\beta u\|_{L^2(M^\circ,|dg_\dS|)}^2
\end{equation}

Our central result in the form which is easiest to state is:

\begin{thm}
\label{ThmIntroBaby}
  For $g(0)=g_\dS$ a static de Sitter metric as above, and for $0<\alpha>1$ and $f\in\CIc(M^\circ)$ real-valued with sufficiently small $H^{n/2+8}$-norm, the wave equation
  \begin{equation}
  \label{EqIntroPDE}
    \Box_{g(u)} u=f+q(u,du)
  \end{equation}
  in $\Omega^\circ$, with vanishing Cauchy data, and with $q$ as above with $N_j\geq 1$ for all $j$, has a unique real smooth (in $\Omega^\circ$) global forward solution of the form $u=u_0+\wt u$, $e^{\alpha t_*}\wt u$ bounded, $u_0=c\chi$, $c\in\R$, where $\chi=\chi(t_*)$ is identically $1$ for large $t_*$. More precisely, we have $e^{\alpha t_*}\wt u\in H^\infty$.
\end{thm}

Theorem~\ref{ThmIntroBaby} follows from Theorem~\ref{ThmIntroFull} below which is in a much more general geometric setting and also allows for a larger class of nonlinearities. See Theorem~\ref{ThmAppWavePoly} for the full statement of Theorem~\ref{ThmIntroBaby} in the more general setting, in particular for statements regarding stability and higher regularity, and the subsequent Remark~\ref{RmkPreciseExpansion} for more precise asymptotics. One can also consider equations on natural vector bundles; see the discussion later in the introduction. In a different direction, we can also solve \emph{backward} problems in spaces with fast exponential decay as $t_*\to\infty$, see Theorem~\ref{ThmAppWaveBack}, where we can in fact replace $\Box_{g(u)}$ by $\Box_{g(u)}+L$ for any first order operator $L$.

It is important to note that the presence of a cosmological horizon and the asymptotically hyperbolic nature of the (cosmological) end of the spacetime lead to dispersive properties of waves which are much stronger than in the asymptotically flat case, and they lead to very good low frequency behavior: In particular, there are \emph{no} conditions on the nonlinear term $q(u,du)$ other than the (essentially necessary) requirement that it vanish on constants, see \cite[Remark~2.31]{HintzVasySemilinear}, and the dimension of the spacetime is arbitrary; this is in stark contrast to the case of quasilinear wave equations on asymptotically flat spacetimes, where a `null condition' needs to be imposed in low spacetime dimensions, including the physical case of $3+1$ dimensions, to guarantee global existence even for semilinear equations; see \cite{KlainermanNullCondition} and further references below. In addition, the dispersive behavior of de Sitter space leads to \emph{exponential} decay to constants and a (partial) mode expansion as in the above theorem, as opposed to merely \emph{polynomial} decay for (quasi)linear waves on Minkowski space and other asymptotically flat spacetimes.

The novelty of our analysis of quasilinear wave and Klein-Gordon equations lies in combining the methods used by Vasy and the author \cite{HintzVasySemilinear} to treat semilinear equations on static asymptotically de Sitter (and more general) spaces with the technology of pseudodifferential operators with non-smooth coefficients in the spirit of Beals and Reed \cite{BealsReedMicroNonsmooth}, which is used to understand the regularity properties of operators like $\Box_{g(u)}$ in the above theorem. Our approach, appropriately adapted, also works in a variety of other settings, in particular on asymptotically Kerr-de Sitter spaces, where however a much more delicate analysis is necessary in view of issues coming from trapping.\footnote{In fact, making heavy use of the machinery developed in the present paper, Vasy and the author \cite{HintzVasyQuasilinearKdS} recently completed the required analysis for nonlinear problems on spaces with normally hyperbolic trapping, thereby in particular obtaining global well-posedness results for quasilinear wave equations on asymptotically Kerr-de Sitter spaces; the class of equations considered there is in fact even more general than \eqref{EqIntroPDE} in that the metric is also allowed to depend on derivatives of $u$.} In a different direction, asymptotically Minkowski spaces in the sense of Baskin, Vasy and Wunsch \cite{BaskinVasyWunschRadMink} should be analyzable as well using similar methods.

Concretely, as in \cite{HintzVasySemilinear}, rather than solving an evolution equation for a short amount of time, controlling the solution using (almost) conservation laws and iterating, we use a different iterative procedure, where at each step we solve a linear equation, with non-smooth coefficients, of the form
\begin{equation}
\label{EqIntroPDEIter}
  P_{u_k}u_{k+1}\equiv \Box_{g(u_k)}u_{k+1}=f+q(u_k,du_k)
\end{equation}
\emph{globally} on $L^2$-based spacetime Sobolev spaces or analogous spaces that encode partial expansions. Since the non-linearity $q$ (as well as $g$) must be well-behaved relative to these, we work on high regularity spaces; recall here that $H^s(\R^n)$ is an algebra for $s>n/2$. Moreover, we need to prove decay (or at least non-growth) for solutions of \eqref{EqIntroPDEIter} so that $q$ can be considered a perturbation. Since linear scalar waves on static de Sitter space decay exponentially fast to constants, the metric $g(u_k)$ in \eqref{EqIntroPDEIter} will similarly be equal to a stationary metric $g_{k,0}$ plus an exponentially decaying remainder $\wt g_k$; the asymptotics of solutions to the linear equation \eqref{EqIntroPDEIter} are then dictated by the \emph{stationary} part $N(P_{u_k}):=\Box_{g_{k,0}}$, also called the \emph{normal operator}, of the operator $P_{u_k}$. Since $g_{k,0}$ is a smooth stationary metric, we can appeal to the results of Vasy \cite{VasyMicroKerrdS,VasyWaveOndS} to understand linear waves globally on the (approximately) static de Sitter space described by the metric $g_{k,0}$: In particular, \emph{resonances} (also known as quasinormal modes) and the associated resonant states describe the asymptotics of linear waves. Just as in the semilinear setting, we need to require the resonances to lie in the `unphysical half-plane' $\Im\sigma<0$ (a simple resonance at $0$ is fine as well), since resonances in the `physical half-plane' $\Im\sigma>0$ would allow growing solutions to the equation, making the non-linearity non-perturbative and thus causing our method to fail. The linear analysis of equations like \eqref{EqIntroPDEIter} is carried out in \S\ref{SecGlobalForSecond} in two steps: the invertibility on high regularity spaces which however contain functions that are growing as $t_*\to\infty$ (Theorem~\ref{ThmWaveGlobalSolveCont}) and the proof of decay corresponding to the location of resonances (Theorem~\ref{ThmExistenceAndExpansion}).

In the iteration scheme \eqref{EqIntroPDEIter}, notice that if $u_k$ is in a (weighted) spacetime $H^s$ space, then the right hand side is in $H^{s-1}$. Now $P_{u_k}$ has leading order coefficients in $H^s$ and subprincipal terms with regularity $H^{s-1}$; therefore, in order to have a well-defined iteration scheme, we need the solution operator for $P_{u_k}$ to map $H^{s-1}$ into $H^s$, the loss of one derivative being standard for hyperbolic problems. In other words, there is a delicate balance of the regularities involved; at the heart of this paper thus lies a robust regularity theory for operators like $P_{u_k}$. We will achieve this by adapting a number of methods of microlocal analysis to the non-smooth setting we are interested in here.

In order to conveniently encode the asymptotic structure of static de Sitter space and its perturbations, we compactify $M^\circ$ and $\Omega^\circ$ at future infinity by introducing the function $\tau:=e^{-t_*}$ (whose real powers are thus naturally used to measure growth/decay at infinity) and adding $\tau=0$ to the spacetime; thus, we let
\[
  M = [0,\infty)_\tau \times \{|Z|<1+2\delta\},\quad \Omega=[0,1]_\tau\times\{|Z|\leq 1+\delta\}\subset M,
\]
which introduces an `ideal boundary' $Y=\pa M=\{\tau=0\}$.

Compactifying the spacetime at infinity puts equation \eqref{EqIntroPDE} into the framework of \emph{b-analysis}. Here `b' refers to analysis based on vector fields tangent to the boundary of the (compactified) space, so b-vector fields are spanned by $\tau\pa_\tau$ and $\pa_Z$, and b-differential operators are linear combinations of products of these. Note that $\tau\pa_\tau=-\pa_{t_*}$, and smoothness (resp.\ the better, invariant, notion of conormality) in $\tau$ near $\tau=0$ corresponds to smoothness (resp.\ conormality) in $e^{-t_*}$ as $t_*\to\infty$; thus, the use of the language of b-geometry and b-analysis is a very economic way for deal with asymptotically stationary problems. (The b-analysis originates in Melrose's work on the propagation of singularities on manifolds with smooth boundary; Melrose described a systematic framework for elliptic b-equations in \cite{MelroseAPS}. We will give more details later in the introduction.) A first indication for this is that the normal operator of $P_{u_k}$ is now obtained in a natural way by freezing the coefficients of $P_{u_k}$ at the boundary $\tau=0$ as a b-operator. More importantly though, in this b-framework, the PDE \eqref{EqIntroPDE} reveals a rich \emph{microlocal structure}: For instance, one of the central features is that the null-geodesic flow for the unperturbed wave operator $\Box_{g_\dS}$ extends smoothly to a flow within $\tau=0$, more precisely within the b-cotangent bundle (see below), which can be thought of as a uniform version of the standard cotangent bundle all the way up to future infinity; and this flow has saddle points where the cosmological horizon intersects future infinity. Microlocally speaking, the Hamilton vector field of $\Box_{g_\dS}$ is \emph{radial} there, i.e.\ a multiple of the generator of dilations in the fibers of the b-cotangent bundle; hence the standard propagation of singularities result by Duistermaat and H\"ormander \cite{DuistermaatHormanderFIO2} does not yield any regularity information there, and we must appeal to more refined results on the propagation of singularities near radial points, dating back to Melrose \cite{MelroseEuclideanSpectralTheory}; see \S\ref{SecPropagation} for further references. See Figure~\ref{FigRadial}.

\begin{figure}[!ht]
  \centering
  \includegraphics{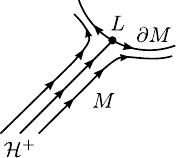}
  \caption{Illustration of the null-geodesic flow near the cosmological horizon, and its extension to the boundary $\pa M=\{\tau=0\}$ at future infinity. No null-geodesic starting in $M^\circ=\{\tau>0\}$ can reach $\tau=0$ in finite time, but the special structure of the flow near the radial set $L$ at the (conormal bundle of the) horizon allows for refined microlocal regularity results at $L$.}
  \label{FigRadial}
\end{figure}

In order to emphasize the generality of the method, let us point out that \emph{given an appropriate structure of the null-geodesic flow at $\infty$, for example radial points as above, the only obstruction to the solvability of quasilinear equations are growing modes and bounded but oscillatory modes}; in particular, in the presence of resonances in the upper half plane, our methods cannot be applied.

The main ingredient of the framework in which will analyze b-operators with non-smooth coefficients on manifolds with boundary is a partial calculus for what we call \emph{b-Sobolev b-pseudodifferential operators}; for brevity, we will refer to these as `non-smooth operators' to distinguish them from `smooth operators,' by which we mean standard b-pseudodifferential operators, recalled below. b-Sobolev b-\psdo{}s are (generalizations of) b-\psdo{}s with coefficients in b-Sobolev spaces, which partly extends a corresponding partial calculus on manifolds without boundary in the form developed by Beals and Reed \cite{BealsReedMicroNonsmooth}. (Beals and Reed consider coefficients in microlocal Sobolev spaces; this generality is not needed for our purposes, even though including it in our general calculus would only require more care in bookkeeping.) This calculus allows us to prove microlocal regularity results -- which are standard in the smooth setting -- for b-Sobolev b-\psdo{}s, namely elliptic regularity, real principal type propagation of singularities, including with (microlocal) complex absorbing potentials, and propagation near radial points; see \S\ref{SecPropagation}. We only develop a local theory since this is all we need for the purposes of our application.\footnote{We refer the reader to the paper \cite[Remark~4.6]{HintzVasyQuasilinearKdS}, which appeared after the first version of the present manuscript, for a partition of unity argument showing that even in more complicated geometries, the local non-smooth theory suffices.} The exposition of the calculus and its consequences in \S\S\ref{SecBFunctionAndSymbol}--\ref{SecPropagation} comprises the bulk of the paper.

To set up the main theorem, recall from \cite{VasyWaveOndS} that an asymptotically de Sitter space $\wt M$ is an appropriate generalization of the Riemannian conformally compact spaces of Mazzeo and Melrose \cite{MazzeoMelroseHyp} to the Lorentzian setting, namely a smooth manifold with boundary, with the interior of $\wt M$ equipped with a Lorentzian signature (taken to be $(1,n-1)$) metric $\wt g$, and with a boundary defining function $\tau$ (i.e.\ $\tau=0$ defines the boundary, and $d\tau\neq 0$ there) such that $\wh g=\tau^2\wt g$ is a smooth symmetric 2-tensor of signature $(1,n-1)$ up to the boundary of $\wt M$, and $\wh g(d\tau,d\tau)=1$ so that the boundary defining function is timelike and the boundary itself is spacelike. In addition, $\pa\wt M$ has two components $X_\pm$, each of which may be a union of connected components, with all null-geodesics $\gamma(s)$ tending to $X_\pm$ as $s\to\pm\infty$.

We now blow up a point $p\in X_+$, which amounts to introducing polar coordinates around $p$, and obtain a manifold with corners $[\wt M;p]$, with a blow-down map $[\wt M;p]\to\wt M$. The backward light cone from $p$ lifts to a smooth manifold transversal to the front face of $[\wt M;p]$ and intersects the front face in a sphere. The interior of this backward light cone, at least near the front face, is a generalization of the static model of de Sitter space; we will refer to a neighborhood $M$ of the closure of the interior of the backward light cone from $p$ in $[\wt M;p]$ that only intersects the boundary of $[\wt M;p]$ in the interior of the front face as the \emph{static asymptotically de Sitter model}, with boundary $Y$ (which is non-compact) and a boundary defining function $\tau$; see also \cite{VasyWaveOndS,VasyMicroKerrdS} for more on the relation of the `global' and `static' problems. Since we are interested in \emph{forward} problems for wave and Klein-Gordon equations and therefore work with energy estimates, we consider a compact region $\Omega\subset M$, bounded by (a part of) $Y$ and two `artificial' spacelike hypersurfaces $H_1$ and $H_2$, see Figures~\ref{FigDS} and \ref{FigStaticDS}. For definiteness, let us assume $H_1=\{\tau=1\}$. We will demonstrate this construction explicitly for exact (static) de Sitter space in \S\ref{SubsecStaticdS}. 

\begin{figure}[!ht]
  \centering
  \includegraphics{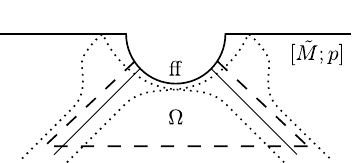}
  \caption{Geometric setup of the static (asymptotically) de Sitter problem. Indicated are the blow-up of $\wt M$ at $p$ and the front face of the blow-up, further the lift of the backward light cone to $[\wt M;p]$ (solid), and lifts of backward light cones from points near $p$ (dotted); moreover, $\Omega$ is bounded by the front face and the dashed spacelike boundaries.}
  \label{FigStaticDS}
\end{figure}

On $M$, we naturally have the b-tangent bundle $\Tb M$, whose sections are the b-vector fields $\Vb(M)$, i.e.\ vector fields tangent to the boundary; in local coordinates $\tau,y$ near the boundary, $\Tb M$ is spanned by $\tau\pa_\tau$ and $\pa_y$. The enveloping algebra of $\Vb$ of b-differential operators is denoted $\Diffb^*(M)$. The b-cotangent bundle, the dual of $\Tb M$, is denoted $\Tb^* M$ and spanned by $\frac{d\tau}{\tau}$ and $dy$, and we have the b-differential $\bdiff\colon\CI(M;\C)\xra{d}\CI(M;T^*M)\to\CI(M;\Tb^* M)$, where the last map comes from the natural map $\Tb M\to TM$. Now, the metric $g$ on $M$ is a smooth, symmetric, Lorentzian signature section of the second tensor power of $\Tb M$. The associated d'Alembertian (or wave operator) $\Box_g$ thus is an element of $\Diffb^2(M)$ and therefore naturally acts on weighted b-Sobolev spaces $\Hb^{s,\alpha}(M)=\tau^\alpha\Hb^s(M)$, where we define
\begin{align*}
  \Hb^s(M)=\{u\in L^2(M,vol_{g_0})\colon X_1\cdots X_s &u\in L^2(M,vol_{g_0}), \\
    & X_1,\ldots,X_s\in\Vb(M)\}
\end{align*}
for $s\in\N_0$, and for general $s\in\R$ using duality and interpolation. Denote by $\Hb^{s,\alpha}(\Omega)^{\bullet,-}$ the space of restrictions of $\Hb^{s,\alpha}(M)$-functions with support in $\{\tau\leq 1\}$ to $\Omega$; that is, elements of $\Hb^{s,\alpha}(\Omega)^{\bullet,-}$ are supported at $H_1$ and extendible at $H_2$ in the sense of H\"ormander \cite[Appendix~B]{HormanderAnalysisPDE}. Finally, let $\cX^{s,\alpha}$ be the space of all $u$ which near $\tau=0$ asymptotically look like a constant plus an $\Hb^{s,\alpha}$-function, i.e.\ for some $c\in\C$, $u'=u-c\chi(\tau)\in\Hb^{s,\alpha}(\Omega)^{\bullet,-}$, where $\chi\in\CI_\cl(\R)$, $\chi\equiv 1$ near $0$, is a cutoff near $Y$; for such a function $u$, define its squared norm by
\[
  \|u\|_{\cX^{s,\alpha}}^2=|c|^2+\|u'\|_{\Hb^{s,\alpha}(\Omega)^{\bullet,-}}^2.
\]
Our main theorem then is:

\begin{thm}
\label{ThmIntroFull}
  Let $s>n/2+7$ and $0<\alpha<1$. Assume that for $j=0,1$,
  \begin{gather*}
    g\colon\cX^{s-j,\alpha}\to(\CI+\Hb^{s-j,\alpha})(M;\Sym^2\Tb^*M), \\
	q\colon\cX^{s-j,\alpha}\times\Hb^{s-1-j,\alpha}(\Omega;\Tb^*_\Omega M)^{\bullet,-}\to\Hb^{s-1-j,\alpha}(\Omega)^{\bullet,-}
  \end{gather*}
  are continuous, $g$ is Lipschitz near $0$, and
  \[
    \|q(u,\bdiff u)-q(v,\bdiff v)\|_{\Hb^{s-1-j,\alpha}(\Omega)^{\bullet,-}}\leq L_q(R)\|u-v\|_{\cX^{s-j,\alpha}}
  \]
  for $u,v\in\cX^{s-j,\alpha}$ with norm $\leq R$, where $L_q\colon\R_{\geq 0}\to\R$ is continuous and non-decreasing. Then there is a constant $C_L>0$ so that the following holds: If $L_q(0)<C_L$, then for small $R>0$, there is $C_f>0$ such that for all $f\in\Hb^{s-1,\alpha}(\Omega)^{\bullet,-}$ with norm $\leq C_f$, there exists a unique solution $u\in\cX^{s,\alpha}$ of the equation
  \[
    \Box_{g(u)}u=f+q(u,\bdiff u)
  \]
  with norm $\leq R$, and in the topology of $\cX^{s-1,\alpha}$, $u$ depends continuously of $f$.
\end{thm}

See Theorem~\ref{ThmAppWaveExistence}. Another case we study is $g(u)=\mu(u)g$, i.e.\ we only allow conformal changes of the metric; here, one can partly improve the above theorem, in particular allow non-linearities of the form $q(u,\bdiff u,\Box_{g(u)}u)$; see \S\ref{SubsecAppConformal}. The point of the Lipschitz assumptions on $q$ in all these cases is to ensure that $q(u,\bdiff u)$ has a sufficient order of vanishing at $u=0$ so that $q(u,\bdiff u)$ can be considered a perturbation of $\Box_{g(u)}$; quadratic vanishing is enough, but slightly less (simple vanishing will small Lipschitz constant near or at $0$) also suffices.

Similar results hold for quasilinear Klein-Gordon equations with positive mass, where the asymptotics of solutions, hence the function spaces used, are different, namely the leading order term is now decaying; see \S\ref{SubsecQlKleinGordon} for details.

In \S\ref{SubsecBackward} finally, we will discuss backward problems; it is expected that the results there extend to the setting of Einstein's equations (after fixing a gauge) on static de Sitter and even on Kerr-de Sitter spacetimes, thus enabling constructions of dynamical black hole spacetimes in the spirit of recent work by Dafermos, Holzegel and Rodnianski \cite{DafermosHolzegelRodnianskiKerrBw}, however the issue of constructing appropriate initial data sets is rather subtle.

While all results were stated for scalar equations, corresponding results hold for operators acting on natural vector bundles, provided that all resonances lie in the unphysical half-plane $\Im\sigma<0$ (with a simple resonance at $0$ being fine as well): Indeed, the linear arguments go through in general for operators with scalar principal symbols; only the numerology of the needed regularities depends on estimates of the subprincipal symbol at (approximate) radial points.

Lastly, let us mention that paradifferential methods would give sharper results with respect to the regularity of the spaces in which we solve equation~\eqref{EqIntroPDE}, and correspondingly we have not made any efforts here to push the regularity down. However, our entirely $L^2$-based method is both conceptually and technically relatively straightforward, powerful enough for our purposes, and lends itself very easily to generalizations in other contexts.

Non-linear wave and Klein-Gordon equations on asymptotically de Sitter spacetimes (or static patches thereof) have been studied in various contexts: Friedrich \cite{FriedrichStability,FriedrichDeSitterPastSimple} proved the global nonlinear stability of 4-dimensional asymptotically de Sitter spaces using a conformal method, see also \cite{FriedrichConformalStructure} for a discussion of more recent developments; also in four dimensions, Rodnianski and Speck \cite{RodnianskiSpeckEulerEinsteinDS} proved the stability of the Euler-Einstein system. Anderson \cite{AndersonStabilityEvenDS} proved the nonlinear stability of all even-dimensional asymptotically de Sitter spaces by generalizing Friedrich's argument. On the semilinear level, Baskin \cite{BaskinStrichartzDeSitterShort,BaskinStrichartzDeSitterLong} established Strichartz estimates for the linear Klein-Gordon equation using his parametrix construction \cite{BaskinParamKGondS} and used them to prove global well-posedness results for classes of semilinear equations with no derivatives; Yagdjian and Galstian \cite{YagdjianGalstianFundamentalSolKG} derived explicit formulas for the fundamental solution of the Klein-Gordon equation on exact de Sitter spaces, which were subsequently used by Yagdjian \cite{YagdjianSemilinearKG,YagdjianGlobalScalarField} to solve semilinear equations with no derivatives. Vasy and the author \cite{HintzVasySemilinear} proved global well-posedness results for a large class of semilinear wave and Klein-Gordon equations on (static) asymptotically de Sitter spaces, where the non-linearity can also involve derivatives; however, just as in the present paper, the (b-)microlocal, high regularity approach used does not apply to low-regularity non-linearities covered by the results of Baskin and Yagdjian. There is more work on the \emph{linear} problem in de Sitter spaces; see e.g.\ the bibliography of \cite{VasyWaveOndS}.

The study of \psdo{}s with non-smooth coefficients has a longer history: Beals and Reed \cite{BealsReedMicroNonsmooth} developed a partial calculus with coefficients in $L^2$-based Sobolev spaces on Euclidean space, which is the basis for our extension to manifolds with boundary. Marschall \cite{MarschallPseudo} gave an extension of the calculus to $L^p$-based Sobolev spaces (and even more general spaces) and in addition proved the invariance of certain classes of non-smooth operators under changes of coordinates. Witt \cite{WittCalculusForNonsmooth} extended the $L^2$-based calculus to contain elliptic parametrices. Pseudodifferential calculi for coefficients in $C^k$ spaces have been studied by Kumano-go and Nagase \cite{KumanogoNagasePsdo}. In a slightly different direction, paradifferential operators, pioneered by Bony \cite{BonyPropagation} and Meyer \cite{MeyerRegularity}, are a widely used tool in nonlinear PDE; see e.g.\ H\"ormander \cite{HormanderNonlinearLectures} and Taylor \cite{TaylorPDE,TaylorNonlinearPDE} and the references therein.

%%%%%%%%%%%%%%%%%%%%%%%%%%%%%%%%%%%%%%%%%%%%%%%%%
\subsection{b-preliminaries and outline of the paper}

We will now give some background on b-pseudodifferential operators and microlocal regularity results along with indications as to how to generalize them to the non-smooth setting, thereby giving a brief, mostly chronological, outline of some of the technical aspects of the paper.

We recall from Melrose \cite{MelroseAPS} that the small calculus of b-\psdo{}s on a compact manifold $M$ with boundary is the microlocalization of the algebra of b-differential operators on $M$, and the kernels of b-\psdo{}s are conceptually best described as conormal distributions on a certain blow-up $M^2_\bl$ of $M\times M$, smooth up to the front face, and vanishing to infinite order at the left/right boundary faces. More prosaically, using local coordinates $(x,y)\in\Rnhalf:=[0,\infty)_x\times\R^{n-1}_y$ near the boundary of $M$, i.e.\ $x$ is a local boundary defining function, and using the corresponding coordinates $\lambda,\eta$ in the fibers of $\Tb^*M$, i.e.\ writing b-covectors as
\[
  \lambda\,\frac{dx}{x}+\eta\,dy,
\]
the action of a b-ps.d.o $A\in\Psib^m$ of order $m$ on $u\in\CIdotc(\Rnhalf)$, the dot referring to infinite order of vanishing at the boundary, is computed by
\begin{equation}
\label{EqIntroBpsdoAction}
  Au(x,y)=\int_{\R\times\R^{n-1}\times(0,\infty)\times\R^{n-1}} e^{i(y-y')\eta}s^{i\lambda}a(x,y,\lambda,\eta)u(x/s,y')\,d\lambda\,d\eta\,\frac{ds}{s}\,dy',
\end{equation}
where $a(x,y,\lambda,\eta)$, the full symbol of $A$ in the local coordinate chart, lies in the symbol class $S^m(\Tb^*\Rnhalf)$, i.e.\ satisfies the symbolic estimates
\[
  |\pa_{x,y}^\alpha\pa_{\lambda,\eta}^\beta a(x,y,\lambda,\eta)|\leq C_{\alpha\beta}\la\lambda,\eta\ra^{m-|\beta|}\tn{ for all multiindices }\alpha,\beta.
\]
We say that $A$ is a left quantization of $a$. Using the formula for the behavior of the full symbol under a coordinate change, one finds that one can invariantly define a \emph{principal symbol}
\[
  \sigma_\bl^m(A)\in S^m(\Tb^* M)/S^{m-1}(\Tb^* M)
\]
of $A$, which is locally just given by (the equivalence class of) $a$. If the principal symbol admits a homogeneous representative $a_m$, meaning $a_m(z,\lambda\zeta)=\lambda^m a_m(z,\zeta)$ for $\lambda\geq 1$, then we say that $A$ has a \emph{homogeneous principal symbol} and, by a slight abuse of notation, set $\sigma_\bl^m(A)=a_m$. We will sometimes identify homogeneous functions on $\Tb^*M\wozero$ with functions on the unit cosphere bundle $\Sb^*M$, viewed as the boundary of the fiber-radial compactification $\rcTb^*M$ of $\Tb^*M$.\footnote{Strictly speaking, this identification is only well-defined for functions which are homogeneous of order $0$; in the general case, one should identify homogeneous functions with sections of a natural line bundle on $\Sb^* M$ which encodes the differential of a boundary defining function of fiber infinity.} The first key point now is that there is a symbolic calculus for b-\psdo{}s, with the most important features being that for $A\in\Psib^m(M),B\in\Psib^{m'}(M)$,
\[
  \sigma_\bl^0(I)=1,\quad\sigma_\bl^m(A^*)=\overline{\sigma_\bl^m(A)},\quad \sigma_\bl^{m+m'}(A\circ B)=\sigma_\bl^m(A)\sigma_\bl^{m'}(B),
\]
where we fixed a b-density on $M$, which in local coordinates is of the form $a\left|\frac{dx}{x}\,dy\right|$ with $a>0$ smooth down to $x=0$, to define the adjoint. For local computations, it is very useful to have the asymptotic expansion
\begin{equation}
\label{EqIntroAsymp}
  \sigma_{\tn{full}}(A\circ B)(z,\zeta)\sim\sum_{\beta\geq 0} \frac{1}{\beta!} (\pa_\zeta^\beta a \Db_z^\beta b)(z,\zeta)
\end{equation}
for the full symbol of a composition of b-\psdo{}s, where $a$ and $b$ are the full symbols of $A$ and $B$, and $\Db_z=(xD_x,D_y)$, where $D=-i\pa$; the notation `$\sim$' means that the difference of the left hand side and the sum on the right hand side, restricted to $|\beta|<N$, lies in $S^{m+m'-N}(\Tb^*\Rnhalf)$, for any $N$. In particular, this gives that for $A\in\Psib^m,B\in\Psib^{m'}$ with principal symbols $a,b$, the principal symbol of the commutator is $\sigma_\bl^{m+m'-1}([A,B])=\frac{1}{i}H_a b$, where
\[
  H_a=(\pa_\lambda a)x\pa_x+(\pa_\eta a)\pa_y - (x\pa_x a)\pa_\lambda-(\pa_y a)\pa_\eta.
\]
This follows from the expansion \eqref{EqIntroAsymp} if we keep track of terms up to first order. The vector field $H_a$ is in fact the smooth extension to the boundary of the standard Hamilton vector field $H_a\in\CI(T^*M^\circ,TT^*M^\circ)$ of $a\in\CI(T^*M^\circ)$.

The second key point for us is that b-\psdo{}s naturally act on weighted b-Sobolev spaces $\Hb^{s,\alpha}(M)$, defined above:
\[
  \Psib^m(M)\ni A\colon\Hb^{s,\alpha}(M)\to\Hb^{s-m,\alpha}(M),\quad s,\alpha\in\R.
\]
We will collect some more information on b-Sobolev spaces and b-\psdo{}s in \S\ref{SecBFunctionAndSymbol}.

The analogous `non-smooth' operators that play the starring role in this paper, b-Sobolev b-\psdo{}s, are locally defined by \eqref{EqIntroBpsdoAction}, but we now allow the symbol $a$ to be less regular. As an example, for many remainder terms in our computations, it will suffice to merely have
\begin{equation}
\label{EqIntroNonsmoothEx}
  \left\|\frac{a(z,\zeta)}{\la\zeta\ra^m}\right\|_{\Hb^s((\Rnhalf)_z)}\leq C,\tn{ uniformly in }\zeta\in\R^n,
\end{equation}
which already implies that $A=a(z,\Db_z)$ defines a continuous map
\begin{equation}
\label{EqIntroNonsmoothMap}
  A\colon\Hb^{s'}\to\Hb^{s'-m}, \quad s\geq s'-m,s>n/2+\max(0,m-s');
\end{equation}
see Proposition~\ref{PropOpCont}. Assuming more regularity of the symbols in $\zeta$, we can study compositions of such non-smooth operators; the main tool here is the asymptotic expansion \eqref{EqIntroAsymp}, which must be cut off after finitely many terms in view of the limited regularity of the symbols, and the remainder term will be estimated carefully. In \S\ref{SecCalculus}, we will develop the (partial) calculus of b-Sobolev b-\psdo{}s as far as needed for the remainder of the paper, in particular for the proofs of microlocal regularity results, which will be essential for the linear analysis of equation~\eqref{EqIntroPDEIter}.

Let us briefly recall a few such regularity results in the smooth setting, working on unweighted b-Sobolev spaces for brevity. First, we define the \emph{b-wavefront set} $\WFb^s(u)\subset\Tb^*M\wozero$ of $u\in\Hb^{-\infty}(M)$ as the complement of the set of all $\omega\in\Tb^*M\wozero$ such that $Au\in L^2_\bl(M)=\Hb^0(M)$ for some $A\in\Psib^s(M)$ elliptic at $\omega$; recall that a b-ps.d.o $A\in\Psib^s(M)$ with homogeneous principal symbol $a$ is \emph{elliptic} at $\omega\in\Tb^*M\wozero$ iff $|a(\lambda\omega)|\geq c|\lambda|^m$ for $\lambda\geq 1$, where we let $\R_+$ act on $\Tb^*M\wozero$ by dilations in the fiber. We informally say that $u$ is in $\Hb^s$ microlocally at $\omega$ iff $\omega\notin\WFb^s(u)$. By definition, the wavefront set is closed and conic, thus we can view it as a subset of $\Sb^*M$; moreover, it can capture global $\Hb^s$-regularity in the sense that $\WFb^s(u)=\emptyset$ implies $u\in\Hb^s(M)$ (and vice versa). \emph{Elliptic regularity} then states that if $u\in\Hb^{-\infty}$ satisfies $Pu\in\Hb^{\sigma-m}$ for $P\in\Psib^m$ which is elliptic at $\omega$, then $u$ is in $\Hb^\sigma$ microlocally at $\omega$. The proof is an easy application of the symbolic calculus -- one essentially takes the reciprocal of the symbol of $P$ near $\omega$ to obtain an approximate inverse of $P$ there -- and readily generalizes to the non-smooth setting as shown in \S\ref{SecEllipticReg}; the main technical task is to understand reciprocals of non-smooth symbols, which we will deal with in \S\ref{SecHsRecAndComp}.

Next, given an operator $P\in\Psib^m$ with real homogeneous principal symbol $p$, we need to study the singularities for solutions $u\in\Hb^{-\infty}$ of $Pu=f\in\Hb^{\sigma-m+1}$ within the characteristic set $\Sigma=p^{-1}(0)$ of $u$; note that elliptic regularity gives $u\in\Hb^{\sigma+1}$ microlocally off $\Sigma$. Let us assume that $dp\neq 0$ at $\Sigma$ so that $\Sigma$ is a smooth conic codimension $1$ submanifold of $\Tb^*M\wozero$. The \emph{real principal type propagation of singularities}, in the setting of closed manifolds originally due to Duistermaat and H\"ormander \cite{DuistermaatHormanderFIO2}, then states that $\WFb^\sigma(u)$ is invariant under the flow of the Hamilton vector field $H_p$ of $p$. In other words, $\WFb^\sigma(u)$ is the union of maximally extended \emph{null-bicharacteristics} of $P$, which are by definition flow lines of $H_p$. We recall the key idea of the proof using a positive commutator argument in \S\ref{SubsecRealPrType}. In this section, we will then generalize this statement and the commutator proof to the case of non-smooth $P$. Since $P$ now only acts on a certain range of b-Sobolev spaces, the allowed degrees $\sigma$ of regularity that we can propagate have bounds both from above and from below in terms of the regularity $s$ of the coefficients of $P$; also, since non-smooth operators like the ones given by symbols as in \eqref{EqIntroNonsmoothEx} have very restricted mapping properties on low or negative order spaces, see~\eqref{EqIntroNonsmoothMap}, we need to assume higher regularity $\Hb^s$ of the coefficients of $P$ when we want to propagate low regularity $\Hb^\sigma$ of solutions $u$. The main bookkeeping overhead of the proof of the propagation of singularities thus comes from the need to make sense of all compositions, dual pairings, adjoints and actions of non-smooth operators that appear in the course of the positive commutator argument.

In order to complete the microlocal picture, we also need to consider the propagation of singularities near \emph{radial points}, which are points in the b-cotangent bundle where the Hamilton vector field $H_p$ is radial, i.e.\ a multiple of the generator of dilations in the fiber. The above propagation of singularities statement does not give any information at radial points. Now, in many geometrically interesting cases, the Hamilton flow near the set of radial points has a lot of structure, e.g.\ if the radial set is a set of sources/sinks/saddle points for the flow. The proof of a microlocal estimate near a class of radial points in \S\ref{SubsecRadialPoints} (see the introduction to \S\ref{SecPropagation} for references in the smooth setting) again proceeds via positive commutators, thus similar comments about the interplay of regularities as in the real principal type setting apply.

In \S\ref{SubsecHypotheses}, we will combine the microlocal regularity results with standard energy estimates for second order hyperbolic equations from \S\ref{SubsecEnergy}, see e.g.\ H\"ormander \cite[Chapter~XXIII]{HormanderAnalysisPDE} or \cite[\S{2}]{HintzVasySemilinear}, and prove the existence and higher regularity of global forward solutions to linear wave equations with non-smooth coefficients under certain geometric and dynamical assumptions, in particular non-trapping. The idea is to start off with forward solutions in a space $\Hb^{0,r}$, $r\ll 0$, obtain higher regularity at elliptic points, propagate higher regularity (from the `past,' where the solution vanishes) using the real principal type propagation of singularities, propagate this regularity into radial points, which lie over the boundary, and propagate from there within the boundary; the non-trapping assumption guarantees that by piecing together all such microlocal regularity statements, we get a global membership in a high regularity b-Sobolev space, however still with weight $r\ll 0$. To improve the decay of the solution, we use a contour deformation argument using the normal operator family as in \cite[\S{3}]{VasyMicroKerrdS}.

Finally, to apply the machinery developed thus far to quasilinear wave and Klein-Gordon equations on `static' asymptotically de Sitter spaces, we check in \S\ref{SecApp} that they fit into the framework of \S\ref{SubsecHypotheses}, thereby proving Theorems~\ref{ThmIntroBaby} and \ref{ThmIntroFull}. To keep the discussion in \S\ref{SecApp} simple, we will in fact only consider quasilinear equations on static patches of de Sitter space explicitly, but the reader should keep in mind that the arguments apply in more general settings; see \S\ref{SecApp}, in particular Remark~\ref{RmkBoxOnBundles2}, for further details.

%%%%%%%%%%%%%%%%%%%%%%%%%%%%%%%%%%%%%%%%%%%%%%%%%
\section*{Acknowledgments}

I am very grateful to my Ph.D.\ advisor Andr\'as Vasy for suggesting the problem, for countless invaluable discussions, for providing some of the key ideas and for constant encouragement throughout this project. I would also like to thank Kiril Datchev for several very helpful discussions and for carefully reading parts of the manuscript. Thanks also to Dean Baskin and Xinliang An for their interest and support. I am very grateful to an anonymous referee whose detailed suggestions significantly improved the manuscript.

I gratefully acknowledge partial support from a Gerhard Casper Stanford Graduate Fellowship, the German National Academic Foundation and Andr\'as Vasy's National Science Foundation grants DMS-0801226 and DMS-1068742.

%%%%%%%%%%%%%%%%%%%%%%%%%%%%%%%%%%%%%%%%%%%%%%%%%%%%%%%%%%%%%%%%%%%%%
\section{Function and symbol spaces for local b-analysis}
\label{SecBFunctionAndSymbol}

We work on an $n$-dimensional manifold $M$ with boundary $\partial M$. Since almost all results we will describe are local, we consider a product decomposition $\Rnhalf=(\overline{\R_+})_x\times\R_y^{n-1}$ near a point on $\partial M$. Whenever convenient, we will assume that all distributions and kernels of all operators we consider have compact support. Whenever the distinction between $x$ and $y$ (or their dual variables, $\lambda$ and $\eta$) is unimportant, we also write $z=(x,y)$ (or $\zeta=(\lambda,\eta)$).

On $\cS(\R_y^{n-1})$, we have the Fourier transform $(\sF v)(\eta)=\int e^{-iy\eta}v(y)\,dy$ with inverse $(\sF^{-1}v)(y)=\int e^{iy\eta}v(\eta)\,d\eta$, where we normalize the measure $d\eta$ to absorb the factor $(2\pi)^{-(n-1)}$. Likewise, on $\CIdotc(\Rhalf)$, i.e.\ functions vanishing to infinite order at $0$ with compact support, we have the Mellin transform $(\sM u)(\lambda)=\int_0^\infty x^{-i\lambda}u(x)\,\frac{dx}{x}$ with inverse $(\sM_\alpha^{-1}u)(x)=\int_{\Im\lambda=-\alpha} x^{i\lambda}u(\lambda)\,d\lambda$, where $\alpha\in\R$ is arbitrary; here, we also normalize $d\lambda$ to absorb the factor $(2\pi)^{-1}$. For any function $u=u(x,y)\in\CIdotc(\Rnhalf)$, we shall write 
\[
  \wh u(\lambda,\eta)=(\sM_{x\to\lambda}\sF_{y\to\eta} u)(\lambda,\eta).
\]
Weighted b-Sobolev spaces on $\Rnhalf$ can then be defined by
\[
  u\in\Hb^{s,\alpha}(\Rnhalf)\equiv\tau^\alpha\Hb^s(\Rnhalf) \iff \la\zeta\ra^s(\sM|_{\Im\lambda=-\alpha}\sF u)(\zeta)\in L^2(\R^n_\zeta),
\]
where the restriction to $\Im\lambda=-\alpha$ effectively removes the weight $x^\alpha$. We will also write $L_\bl^2(\Rnhalf)=\Hb^0(\Rnhalf)$, which agrees with the usual definition of $L_\bl^2(\Rnhalf)=L^2(\Rnhalf,\frac{dx}{x}\,dy)$, since $\ms M\ms F\colon L^2(\Rnhalf,\frac{dx}{x}\,dy)\to L^2(\R^n_\zeta)$ is an isometric isomorphism by Plancherel's theorem.

As in the introduction, we define the b-wavefront set of $u\in\Hb^{-\infty}(\Rnhalf)$ by
\[
  (z_0,\zeta_0)\notin\WFb^s(u) \iff \exists A\in\Psibc^s(\Rnhalf),\sigma_\bl^s(A)(z_0,\zeta_0)\neq 0\tn{ s.t. }Au\in L^2_\bl(\Rnhalf).
\]
Here, $\Psibc^*$ consists of operators with compactly supported Schwartz kernel, and we write $A=A(z,\Db_z)\equiv A(x,y,xD_x,D_y)$. The b-wavefront set in a weighted b-Sobolev sense is defined by
\[
  \WFb^{s,\alpha}(u):=\WFb^s(x^{-\alpha}u),\quad u\in\Hb^{-\infty,\alpha}(\Rnhalf).
\]

There is the following simple characterization of $\WFb^s(u)$.

\begin{lemma}
\label{LemmaWFbCharacterization}
  Let $u\in\Hb^{-\infty}(\Rnhalf)$. Then $(z_0,\zeta_0)\notin\WFb^s(u)$ if and only if there exists $\phi\in\CI_\cl(\Rnhalf),\phi(z_0)\neq 0$, and a conic neighborhood $K$ of $\zeta_0$ in $\R^n$ such that
  \begin{equation}
  \label{EqBSobolevCondition}
    \chi_K(\zeta)\la\zeta\ra^s\widehat{\phi u}\in L^2(\R^n),
  \end{equation}
  where $\chi_K$ is the characteristic function of $K$.
\end{lemma}
\begin{proof}
  It suffices to prove the lemma when $\chi_K$ is replaced by $\wt\chi_K\in\CI(\R^n)$, where $\wt\chi_K\equiv 1$ on the half line $\R_{\geq 1}\zeta_0$, and $\wt\chi_K$ is homogeneous of degree $0$ in $|\zeta|\geq 1$. Given such a $\wt\chi_K$ and $\phi\in\CI_\cl(\Rnhalf)$ so that \eqref{EqBSobolevCondition} holds (with $\chi_K$ replaced by $\wt\chi_K$), the map
  \[
    A\colon v\mapsto (\wt\chi_K(\Db)\la\Db\ra^s+r(\Db))(\phi v)
  \]
  is an element of $\Psibc^s(\Rnhalf)$ for an appropriate choice of $r(\zeta)\in S^{-\infty}$ (see Lemma~\ref{LemmaBQuantization}). Since $r(\Db)\colon\Hb^{-\infty}(\Rnhalf)\to\Hb^\infty(\Rnhalf)$, we conclude that $\widehat{Au}\in L^2(\R^n)$, which by Plancherel's theorem gives $(z_0,\zeta_0)\notin\WFb^s(u)$, as desired.

 For the converse direction, given $A\in\Psibc^s(\Rnhalf)$, $\sigma_{\bl}^s(A)(z_0,\zeta_0)\neq 0$, take $\phi\in\CI_\cl(\Rnhalf)$ and $\wt\chi_K\in\CI(\R^n)$ with $\phi(z_0)\neq 0,\wt\chi_K(\zeta_0)\neq 0$ such that $A$ is elliptic on $\WFb'(B)$, where $B=(\wt\chi_K(\Db)\la\Db\ra^s+r(\Db))\phi\in\Psibc^s(\Rnhalf)$, again with an appropriately chosen $r\in S^{-\infty}$. A straightforward application of the symbol calculus gives the existence of $C\in\Psibc^0(\Rnhalf),R'\in\Psibc^{-\infty}(\Rnhalf)$ such that $B=CA-R'$; thus $Bu=C(Au)-R'u\in L^2_\bl(\Rnhalf)$. Since $r(\Db)\colon\Hb^{-\infty}\to\Hb^\infty$, we conclude that $\chi_K(\zeta)\la\zeta\ra^s\widehat{\phi u}\in L^2(\R^n)$, and the proof is complete.
\end{proof}

It is convenient to build up the calculus of smooth b-\psdo{}s on $M$ using the kernels of b-\psdo{}s explicitly, as done by Melrose \cite{MelroseAPS}: On the one hand, they are conormal distributions, namely the partial Fourier transform of a symbol\footnote{For clarity, the semicolon `;' will often be used to separate base and fiber variables (resp.\ differential operators) in symbols (resp.\ operators).} $a(x,y;\lambda,\eta)$ near the diagonal of the b-stretched product $M^2_\bl$, smoothly up to the front face, and on the other hand, they vanish to infinite order at the two boundaries $\lb(M^2_\bl)$ and $\rb(M^2_\bl)$, which in particular ensures that b-\psdo{}s act on weighted spaces. However, we will refrain from describing the kernels of the non-smooth b-operators to be considered later and rather keep track of more information on the symbol $a$, wherever this is necessary. The idea is the following: Given a conormal distribution
\[
  \wt I_a(s):=\int e^{i\lambda\log s}a(\lambda)\,d\lambda,\quad a\in S^m(\R),
\]
The function $I_a(t):=\wt I_a(e^t)$ is rapidly decaying as $|t|\to\infty$. If we require however that $\wt I_a(s)$ be rapidly decaying as $s\to 0$ and $s\to\infty$, i.e. $I_a$ is super-exponentially decaying as $|t|\to\infty$, it turns out that the symbol $a(\lambda)$ can be extended to an entire function of $\lambda$ with symbol bounds in $\Re\lambda$ which are locally uniform in $\Im\lambda$; see Lemma~\ref{LemmaBSymbolFT} below.

\begin{definition}
  Let $m\in\R$. Then $S^m_\bl((\overline{\R_+})_x\times\R^{n-1}_y;\R_\lambda\times\R^{n-1}_\eta)$ is the space of all symbols $a\in S^m((\Rnhalf)_z;\R_\zeta^n)$ such that the partial inverse Fourier transform $\sF_{\lambda\to t}^{-1}a$ is super-exponentially decaying as $|t|\to\infty$, i.e. for all $\mu\in\R$, there is $C_\mu<\infty$ such that $|\sF_{\lambda\to t}^{-1}a(x,y;t,\eta)|\leq C_\mu e^{-\mu|t|}$ for $|t|>1$.
\end{definition}

\begin{lemma}
\label{LemmaBSymbolFT}
 Let $m\in\R$. Then $a(\lambda)\in S^m_\bl(\R_\lambda)$ if and only if $a$ extends to an entire function, also denoted $a(\lambda)$, which for all $N,K\in\N$ satisfies an estimate
 \begin{equation}
 \label{EqBSymbolFTEstimate}
   |D_\lambda^k a(\lambda)| \leq C_{N,K}\la\lambda\ra^{m-k},\quad |\Im\lambda|\leq N, k\leq K.
 \end{equation}
 for a constant $C_{N,K}<\infty$.
\end{lemma}
\begin{proof}
  Given $a\in S^m_\bl$, we write $a=a_0+a_1$, where for $\phi\in\CI_\cl(\R)$, $\phi\equiv 1$ near $0$,
  \[
    a_0=\sF(\phi\sF^{-1}a),\quad a_1=\sF((1-\phi)\sF^{-1} a).
  \]
  Since $\sF^{-1} a_1\in\CI(\R)$ is super-exponentially decaying, we easily get the estimate \eqref{EqBSymbolFTEstimate} for $a_1$ (in fact, the estimate holds for arbitrary $m$); see e.g. \cite[Theorem~5.1]{MelroseAPS}. Next, $\phi\sF^{-1} a\in\sE'(\R)$, thus $a_0$ is entire, and we write for $\lambda,\mu\in\R$:
  \begin{align*}
    a_0(\lambda+i\mu)&=\iint_{\R^2} e^{i(\sigma-\lambda-i\mu)x}\phi(x) a(\sigma)\,d\sigma\,dx \\
	  &=\int_\R a(\sigma+\lambda)\left(\int_\R e^{i\sigma x}e^{\mu x}\phi(x)\,dx\right) d\sigma
  \end{align*}
  Since $e^{\mu x}\phi(x)$ is a locally bounded (in $\mu$) family of Schwartz functions, we have for $|\mu|\leq N$, $N\in\N$ arbitrary,
  \begin{align}
\label{EqBSymbolBoundPf1}&|a_0(\lambda+i\mu)|\leq C_N\int\la\sigma+\lambda\ra^m\la\sigma\ra^{-N}\,d\sigma \\
\label{EqBSymbolBoundPf2}&\quad = C_N\left(\int_{|\sigma|\leq |\lambda|}\la\sigma+\lambda\ra^m\la\sigma\ra^{-N}\,d\sigma+\int_{|\sigma|>|\lambda|}\la\sigma+\lambda\ra^m\la\sigma\ra^{-N}\,d\sigma\right).
  \end{align}
  First, we consider the case $m\geq 0$. Then the first integral in \eqref{EqBSymbolBoundPf2} is bounded by
  \[
    \int_\R\la\lambda\ra^m\la\sigma\ra^{-N}\,d\sigma\leq C_N\la\lambda\ra^m
  \]
  for $N>1$, and the second integral is bounded by
  \[
    \int_\R \la\sigma\ra^{-N+m}\,d\sigma\leq C_{N,m}\leq C_N\la\lambda\ra^m
  \]
  for $-N+m<-1$ in view of $m\geq 0$; thus we obtain \eqref{EqBSymbolFTEstimate} for $k=0$.

  Next, we consider the case $m<0$. The integral in \eqref{EqBSymbolBoundPf1} is dominated by
  \begin{align*}
    \la\lambda\ra^m \int\frac{\la\lambda\ra^{-m}}{\la\sigma+\lambda\ra^{-m}\la\sigma\ra^{-m}}\la\sigma\ra^{-N-m}\,d\sigma & \leq \la\lambda\ra^m \int_\R C_m\la\sigma\ra^{-N-m}\,d\sigma \\
	  &\leq C_{N,m}\la\lambda\ra^m.
  \end{align*}
  This proves \eqref{EqBSymbolFTEstimate} for $k=0$. To get the estimate for the derivatives of $a_0$, we compute
  \begin{align*}
    D_\lambda^k a_0(\lambda)&=\iint (-x)^k e^{i(\sigma-\lambda-i\mu)x}\phi(x)a(\sigma)\,d\sigma\,dx \\
	 &=\iint (-D_\sigma)^k e^{i(\sigma-\lambda-i\mu)x}\phi(x)a(\sigma)\,d\sigma\,dx \\
	 &=\iint e^{i(\sigma-\lambda-i\mu)x}\phi(x)D_\sigma^k a(\sigma)\,d\sigma\,dx,
  \end{align*}
  and since $|D_\sigma^k a(\sigma)|\leq C_k\la\sigma\ra^{m-k}$, the above estimates yield \eqref{EqBSymbolFTEstimate} for arbitrary $K$.

  For the converse direction, it suffices to prove the super-exponential decay of $\sF^{-1} a$. Fix $\mu\in\R$. Then for $|x|>1$, $k\in\N$, we compute
  \[
    e^{x\mu}\sF^{-1} a(x)=e^{x\mu}\int_\R e^{ix\lambda}a(\lambda)\,d\lambda=\frac{e^{x\mu}}{x^k}\int_\R e^{ix\lambda}(-D_\lambda)^k a(\lambda)\,d\lambda.
  \]
  Choose $k$ such that $m-k<-1$, then we can shift the contour of integration to $\Im\lambda=\mu$, thus
  \[
    |e^{x\mu}\sF^{-1} a(x)|\leq |x|^{-k}\int_\R |D_\lambda^k a(\lambda+i\mu)|\,d\lambda\in L^\infty_x.
  \]
  Since this holds for any $\mu\in\R$, this gives the super-exponential decay of $\sF^{-1}a$ for $|x|\to\infty$, and the proof is complete.
\end{proof}

In particular, the operator with full symbol $\la\zeta\ra^s$ is not a b-\psdo{}\ unless $s\in 2\N$. However, we can fix this by changing $\la\zeta\ra^s$ by a symbol of order $-\infty$; more generally:

\begin{lemma}
\label{LemmaBQuantization}
  For any symbol $a\in S^m((\overline{\R_+})_x\times\R^{n-1}_y;\R_\lambda\times\R^{n-1}_\eta)$, there is a symbol $\wt a\in S_\bl^m$ with $a-\wt a\in S^{-\infty}$.
\end{lemma}
\begin{proof}
  Fix $\phi\in\CI_\cl(\R)$ identically $1$ near $0$ and put
  \[
    \wt a(x,y;\lambda,\eta)=\sF_{\lambda\to t}\bigl((\sF_{\lambda\to t}^{-1}a)(x,y;t,\eta)\phi(t)\bigr).
  \]
  Then $\wt a\in S_\bl^m$ by the proof of Lemma~\ref{LemmaBSymbolFT}. Moreover, $\sF_{\lambda\to t}^{-1}(a-\wt a)$ is smooth and rapidly decaying, thus the lemma follows.
\end{proof}

\begin{cor}
\label{CorBStdOp}
  For each $s\in\R$, there is $\Lambda_s\in\Psib^s(\Rnhalf)$ with full symbol $\lambda_s\in S_\bl^s$, $\lambda_s(\zeta)\neq 0$ for all $\zeta\in\R^n$, such that $\lambda_s-\la\zeta\ra^s\in S^{-\infty}$.
\end{cor}
\begin{proof}
  The only statement left to be proved is that $\lambda_s$ can be arranged to be non-vanishing. Let $\wt\lambda_s\in S_\bl^s$ be the symbol constructed in Lemma~\ref{LemmaBQuantization}. Since $\wt\lambda_s$ differs from the positive function $\la\zeta\ra^s\in S^s\setminus S^{s-1}$ by a symbol of order $S^{-\infty}$, it is automatically positive for large $|\zeta|$; thus we can choose $C=C(s)$ large such that $\lambda_s(\zeta)=\wt\lambda_s(\zeta)+C(s)e^{-\zeta^2}$ is positive for all $\zeta\in\R^n$. Since $e^{-\zeta^2}\in S_\bl^{-\infty}$, the proof is complete.
\end{proof}

%%%%%%%%%%%%%%%%%%%%%%%%%%%%%%%%%%%%%%%%%%%%%%%%%%%%%%%%%%%%%%%%%%%%%
\section{A calculus for operators with b-Sobolev coefficients}
\label{SecCalculus}

We continue to work in local coordinates on $M$. To analyze the action of operators with non-smooth coefficients on b-Sobolev functions, we need a convenient formula. Given $A\in\Psib^m(M)$ with full symbol $a(x,y;\lambda,\eta)\in S^m(\Tb^*M)$, compactly supported in $x,y$, we have for $u\in\CIdot(M)$
\begin{align*}
  Au(x,y)&=\iiiint e^{i\lambda\log(x/x')}e^{i\eta(y-y')}a(x,y;\lambda,\eta)u(x',y')\,\frac{dx'}{x'}\,dy'\,d\lambda\,d\eta \\
    &=\iint x^{i\lambda}e^{i\eta y}a(x,y;\lambda,\eta)\wh u(\lambda,\eta)\,d\lambda\,d\eta.
\end{align*}
Writing $\wh a$ for the Mellin transform in $x$ and the Fourier transform in $y$, we obtain
\begin{align}
\label{EqOperatorStep1}
  \widehat{Au}(\sigma,\gamma)&=\iiiint x^{-i(\sigma-\lambda)}e^{-i(\gamma-\eta)y}a(x,y;\lambda,\eta)\wh u(\lambda,\eta)\,d\lambda\,d\eta\,\frac{dx}{x}\,dy \\
    &=\iint\wh a(\sigma-\lambda,\gamma-\eta;\lambda,\eta)\wh u(\lambda,\eta)\,d\lambda\,d\eta.\nonumber
\end{align}
Even though this makes sense as a distributional pairing, it is technically inconvenient to use directly: The problem is that if $a$ does not vanish at $x=0$, then $\wh a(\sigma,\gamma;\lambda,\eta)$ has a pole at $\sigma=0$ (cf. \cite[Proposition~5.27]{MelroseAPS}). This is easily dealt with by decomposing
\begin{equation}
\label{EqDecomposition}
  a=a^{(0)}(y;\lambda,\eta)+a^{(1)}(x,y;\lambda,\eta),
\end{equation}
where $a^{(0)}(y;\lambda,\eta)=a(0,y;\lambda,\eta)$ and $a^{(1)}(x,y;\lambda,\eta)=x\wt a^{(1)}(x,y;\lambda,\eta)$ with $\wt a^{(1)}\in S^m$. (Of course, $a^{(0)}$ in general no longer has compact support; however, this will be completely irrelevant for the analysis, due to the fact that $a^{(0)}$ has `nice' behavior in $y$, \emph{independently} in $x$.) Then $\widehat{a^{(1)}}(\sigma,\gamma;\lambda,\eta)$ is smooth and rapidly decaying in $(\sigma,\gamma)$, and we write
\begin{equation}
\label{EqOperatorAction}
  (A^{(1)} u)\ftrans(\zeta)=\int\widehat{a^{(1)}}(\zeta-\xi;\xi)\wh u(\xi)\,d\xi.
\end{equation}
For $A^{(0)}=a^{(0)}(y,\Db)$, we obtain
\begin{equation}
\label{EqInvariantOpAction}
  (A^{(0)}u)\ftrans(\sigma,\gamma)=\int\sF a^{(0)}(\gamma-\eta;\sigma,\eta)\wh u(\sigma,\eta)\,d\eta,
\end{equation}
and $\sF a^{(0)}(\gamma;\sigma,\eta)$ is rapidly decaying in $\gamma$.

\begin{rmk}
\label{RmkMellinOfConst}
  Either we read off equation \eqref{EqInvariantOpAction} directly from equation \eqref{EqOperatorStep1}, where we observe that the symbol $a^{(0)}$ is independent of $x$, thus the integrals over $x$ and $\lambda$ are Mellin transform and inverse Mellin transform, respectively, and therefore cancel; or we observe that, with $a^{(0)}(x,y;\lambda,\eta):=a^{(0)}(y;\lambda,\eta)$, we have $\widehat{a^{(0)}}(\sigma-\lambda,\gamma-\eta;\lambda,\eta)=2\pi\delta_{\sigma=\lambda}\sF a^{(0)}(0,\gamma-\eta;\lambda,\eta)$. The second argument also shows that many manipulations on integrals that compute $A^{(1)}u$ (or compositions of b-operators) also apply to the computation of $A^{(0)}u$ if one reads integrals as appropriate distributional pairings.
\end{rmk}

Notice that \eqref{EqOperatorAction} is, with the change in meaning of $\widehat{a^{(1)}}$ and $\wh u$ and keeping in mind that $a^{(1)}=x\wt a^{(1)}$ is a rather special symbol, the same formula as for pseudodifferential operators on a manifold without boundary used by Beals and Reed \cite{BealsReedMicroNonsmooth}. Since also the characterization of $H^s_\bl$ functions in terms of their mixed Mellin and Fourier transform (Lemma~\ref{LemmaWFbCharacterization}) is completely analogous to the characterization of $H^s$ functions in terms of their Fourier transform, the arguments presented in \cite{BealsReedMicroNonsmooth} carry over to this restricted b-setting. In order to introduce necessary notation and construct a (partial) calculus in the full b-setting, containing weights, we will go through most arguments of \cite{BealsReedMicroNonsmooth}, extending and adapting them to the b-setting; and of course we will have to treat the term $A_0$ separately.

The class of operators we are interested in are b-differential operators whose coefficients lie in (weighted) b-Sobolev spaces of high order. Let us remark that we do not attempt to develop an invariant calculus that can be transferred to a manifold; in particular, all definitions are on $\Rnhalf$, see also the beginning of \S\ref{SecBFunctionAndSymbol}. We thus define the following classes of non-smooth symbols:

\begin{definition}
\label{DefProdS}
  For $m,s\in\R$, define the spaces of symbols
  \[
    \Hb^s S_{(\bl)}^m=\Bigl\{\sum_{\tn{finite}}a_j(z)p_j(z,\zeta)\colon a_j\in\Hb^s,p_j\in S_{(\bl)}^m\Bigr\},
  \]
  and denote by $\Hb^s\Psi_{(\bl)}^m$ the corresponding spaces of operators, i.e.
  \[
    \Hb^s\Psi_{(\bl)}^m=\{a(z,\Db)\colon a(z,\zeta)\in\Hb^s S_{(\bl)}^m\}.
  \]
  Moreover, let $\Psi^m=\{a(z,\Db)\colon a(z,\zeta)\in S^m\}$.
\end{definition}

\begin{rmk}
\label{RmkNobNotation}
  In this paper, we will only deal with operators that are quantizations of symbols on the b-cotangent bundle, and thus with $\Psi^m$ we will \emph{always} mean the space defined above.
\end{rmk}

\begin{rmk}
  In a large part of the development of the calculus for non-smooth b-\psdo{}s in this section, we will keep track of additional information on the symbols of most \psdo{}s, encoded in the space of symbols $S_\bl^*$, in order to ensure that they act on weighted b-Sobolev spaces. Although this requires a small conceptual overhead, it simplifies some computations later on.
\end{rmk}

The spaces $\Hb^*\Psi_{(\bl)}^*$ are not closed under compositions, in fact they are not even left $\Psib^*$-modules. To get around this, which will be necessary in order to develop a sufficiently powerful calculus, we will consider less regular spaces, which however are still small enough to allow for good analytic (i.e.\ mapping and composition) properties.

\begin{definition}
\label{DefFS}
  For $s,m\in\R,k\in\N_0$, define the space
  \begin{align*}
	S^{m;0}\Hb^s&=\Bigl\{p(z,\zeta)\colon p\in\la\zeta\ra^m L^\infty_\zeta\left((\Hb^s)_z\right)\Bigr\} \\
	  &=\Bigl\{p(z,\zeta)\colon \frac{\la\eta\ra^s\wh p(\eta;\zeta)}{\la\zeta\ra^m}\in L^\infty_\zeta L^2_\eta\Bigr\}.
  \end{align*}
  Let $S_\bl^{m;0}\Hb^s$ be the space of all symbols $p(x,y;\lambda,\eta)\in S^{m;0}\Hb^s$ which are entire in $\lambda$ with values in $\la\eta\ra^m L^\infty_\eta\left((\Hb^s)_z\right)$ such that for all $N$ the following estimate holds:
  \begin{equation}
  \label{EqBSymbolDefEstimate}
    \|p(z;\lambda+i\mu,\eta)\|_{\Hb^s} \leq C_N\la\lambda,\eta\ra^m,\quad |\mu|\leq N.
  \end{equation}

  Finally, define the spaces
  \[
	S_{(\bl)}^{m;k}\Hb^s=\Bigl\{p(z,\zeta)\colon \pa_\zeta^\beta p\in S_{(\bl)}^{m-|\beta|;0}\Hb^s,|\beta|\leq k\Bigr\}.
  \]
  The spaces of operators which are left quantizations of these symbols are denoted by $\fpsi^{m;0}\Hb^s$, $\fpsib^{m;0}\Hb^s$ and $\fpsi_{(\bl)}^{m;k}\Hb^s$, respectively.
  
  Weighted versions of these spaces, involving $\Hb^{s,\alpha}$ for $\alpha\in\R$, are defined analogously.
\end{definition}

We can also define similar symbol and operator classes for operators acting on bundles: Let $E,F,G$ be the trivial (complex or real) vector bundles over $\Rnhalf$ of ranks $d_E,d_F,d_G$, respectively, equipped with a smooth metric (Hermitian for complex bundles) on the fibers which is the standard metric on the fibers over the complement of a compact subset of $\Rnhalf$, then we can define
\[
  \Hb^s S^m(\Rnhalf;G):=\{(a_i)_{1\leq i\leq d_G}\colon a_i\in\Hb^s S^m\}.
\]
We then define the space $\Hb^s\Psi^m(\Rnhalf;E,F)$ to consist of left quantizations of symbols in $\Hb^s S^m(\Rnhalf;\Hom(E,F))$; likewise for all other symbol and operator classes.\footnote{Since we are only concerned with local constructions, we use the somewhat imprecise notation just introduced; the proper class that the symbol of a b-pseudodifferential operator (with smooth coefficients), mapping sections of $E$ to sections of $F$, lies in, is $S^m(\Tb^*M;\pi^*\Hom(E,F))$, where $\pi\colon\Tb^*M\to M$ is the projection; see \cite{MelroseAPS}.} We shall also write $\Hb^s\Psi^m(\Rnhalf;E):=\Hb^s\Psi^m(\Rnhalf;E,E)$.

\begin{rmk}
  If we considered, as an example, the wave operator corresponding to a non-smooth metric acting on differential forms, the natural metric on the fibers of the form bundle would be non-smooth. Even though this could be dealt with directly in this setting, we simplify our arguments by choosing an `artificial' smooth metric to avoid regularity considerations when taking adjoints, etc.
\end{rmk}

The first step is to prove mapping properties of operators in the classes just defined; compositions will be discussed in \S\ref{SubsecComp}.

%%%%%%%%%%%%%%%%%%%%%%%%%%%%%%%%%%%%%%%%%%%%%%%%%
\subsection{Mapping properties}
\label{SubsecMapping}

The mapping properties of operators in $\fpsi^{m;0}\Hb^s$ are easily proved using the following simple integral operator estimate.

\begin{lemma}
\label{LemmaIntegralEstimate}
  (Cf.\ \cite[Lemma~1.4]{BealsReedMicroNonsmooth}.) Let $g(\eta,\xi)\in L^\infty_\xi L^2_\eta$ and $G(\eta,\xi)\in L^\infty_\eta L^2_\xi$. Then the operator
  \[
    Tu(\eta)=\int G(\eta,\xi)g(\eta-\xi,\xi)u(\xi)\,d\xi
  \]
  is bounded on $L^2$ with operator norm $\leq\|G\|_{L^\infty_\eta L^2_\xi}\|g\|_{L^\infty_\xi L^2_\eta}$.
\end{lemma}
\begin{proof}
  Cauchy-Schwartz gives
  \begin{align*}
    \|Tu\|_{L^2}^2&\leq\int\left(\int|G(\eta,\xi)|^2\,d\xi\right)\left(\int|g(\eta-\xi,\xi)u(\xi)|^2\,d\xi\right)\,d\eta \\
	  &\leq\|G\|_{L^\infty_\eta L^2_\xi}^2\int\left(\int|g(\eta-\xi,\xi)|^2\,d\eta\right)|u(\xi)|^2\,d\xi \\
	  &\leq\|G\|_{L^\infty_\eta L^2_\xi}^2\|g\|_{L^\infty_\xi L^2_\eta}^2\|u\|_{L^2}^2.\qedhere
  \end{align*}
\end{proof}

The most common form of $G$ in this paper is given by and estimated in the following lemma. We use the notation
\begin{equation}
\label{EqPositivePart}
  a_+:=\max(a,0),\quad a\in\R.
\end{equation}

\begin{lemma}
\label{LemmaFractionEstimate}
  Suppose $s,r\in\R$ are such that $s\geq r,s>n/2+(-r)_+$, then
  \[
    G(\eta,\xi)=\frac{\la\eta\ra^r}{\la\eta-\xi\ra^s\la\xi\ra^r}\in L^\infty_\eta(\R^n;L^2_\xi(\R^n)).
  \]
\end{lemma}
\begin{proof}
  First, suppose $r\geq 0$. Then we obtain
  \[
    G(\eta,\xi)^2\leq\frac{1}{\la\eta-\xi\ra^{2(s-r)}\la\xi\ra^{2r}}+\frac{1}{\la\eta-\xi\ra^{2s}}.
  \]
  Since $s>n/2$, the $\xi$-integral of the second fraction is finite and $\eta$-independent. For the $\xi$-integral of the first fraction, we split the domain of integration into two parts and obtain
  \begin{align*}
    \int_{|\xi|\leq|\eta-\xi|}&\frac{1}{\la\eta-\xi\ra^{2(s-r)}\la\xi\ra^{2r}}\,d\xi+\int_{|\eta-\xi|\leq|\xi|}\frac{1}{\la\eta-\xi\ra^{2(s-r)}\la\xi\ra^{2r}}\,d\xi \\
	&\leq\int\frac{1}{\la\xi\ra^{2s}}\,d\xi+\int\frac{1}{\la\eta-\xi\ra^{2s}}\,d\xi\in L^\infty_\eta.
  \end{align*}
  Next, if $r<0$, then we estimate
  \[
    G(\eta,\xi)^2=\frac{\la\xi\ra^{-2r}}{\la\eta-\xi\ra^{2s}\la\eta\ra^{-2r}}\leq\frac{1}{\la\eta-\xi\ra^{2(s-(-r))}}+\frac{1}{\la\eta-\xi\ra^{2s}},
  \]
  where in the first fraction, we discarded the term $\la\eta\ra^{-2r}\geq 1$. Since $s-(-r)>n/2$, the integrals of both fractions are finite, and the proof is complete.
\end{proof}

\begin{prop}
\label{PropOpCont}
  Let $m\in\R$. Suppose $s\geq s'-m$ and $s>n/2+(m-s')_+$. Then every $A=a(z,\Db)\in\fpsi^{m;0}\Hb^s(\Rnhalf;E,F)$ is a bounded operator $\Hb^{s'}(\Rnhalf;E)\to \Hb^{s'-m}(\Rnhalf;F)$. If $A\in\fpsib^{m;0}\Hb^s(\Rnhalf;E,F)$, then $A$ is also a bounded operator $\Hb^{s',\alpha}(\Rnhalf;E)\to\Hb^{s'-m,\alpha}(\Rnhalf;F)$ for all $\alpha\in\R$.
\end{prop}

Note that this proposition also deals with `low' regularity in the sense that negative b-Sobolev orders are permitted in the target space. We shall have occasion to use this in arguments involving dual pairings in \S\ref{SecPropagation}.

\begin{proof}[Proof of Proposition~\ref{PropOpCont}.]
  Let us first prove the statement without bundles, i.e.\ for complex-valued symbols and functions. Let $u\in\Hb^{s'}$ be given. Then
  \[
    \la\zeta\ra^{s'-m}\widehat{Au}(\zeta)=\int\frac{\la\zeta\ra^{s'-m}\la\xi\ra^m}{\la\zeta-\xi\ra^s \la\xi\ra^{s'}} a_0(\zeta-\xi;\xi)u_0(\xi)\,d\xi
  \]
  for $a_0(\zeta;\xi)\in L^\infty_\xi L^2_\zeta$, $u_0\in L^2$. Lemma~\ref{LemmaFractionEstimate} ensures that the fraction in the integrand is an element of $L^\infty_\zeta L^2_\xi$, and then Lemma~\ref{LemmaIntegralEstimate} implies $\la\zeta\ra^{s'-m}\widehat{Au}(\zeta)\in L^2_\zeta$.

  In order to prove the second statement, we write for $u\in\CIdotc$
  \begin{align*}
    a(x,y,xD_x,D_y)u(x,y)&=\iint_{\Im\lambda=0} e^{i\lambda\log x}e^{i\eta y}a(x,y;\lambda,\eta)\wh u(\lambda,\eta)\,d\lambda\,d\eta \\
      &=\iint_{\Im\lambda=0}\wt a(\lambda)(\eta;x,y)\wh u(\lambda,\eta)\,d\lambda\,d\eta,
  \end{align*}
  where
  \[
    \wt a(\lambda)(\eta;x,y)=x^{i\lambda}e^{i\eta y}a(x,y;\lambda,\eta);
  \]
  we want to shift the contour of integration to $\Im\lambda=-\alpha$. Assuming that $\supp_{x,y}a$ is compact, we have that for any $N$,
  \[
    \|\wt a(\lambda)(\eta,\cdot,\cdot)\|_{\Hb^{s,-N}}\leq C_N\la\lambda,\eta\ra^{m+s},\quad |\Im\lambda|<N,
  \]
  and $\wt a(\lambda)$ is holomorphic in $\lambda$ with values in $\Hb^{s,-N}$ for fixed $\eta$. Since $\wh u(\lambda,\eta)$ is rapidly decaying, we infer for all sufficiently large $M>0$
  \begin{align*}
    \int \|\wt a(\lambda)(\eta,\cdot,\cdot)\|_{\Hb^{s,-N}}|\wh u(\lambda,\eta)|\,d\eta & \leq C_N\int\la\lambda,\eta\ra^{m+s-M}\,d\eta \\
	  &=C_{NM}\la\lambda\ra^{m+s-M+n-1},
  \end{align*}
  thus
  \[
    \wt a'(\lambda)(x,y):=\int\wt a(\lambda)(\eta;x,y)\wh u(\lambda,\eta)\,d\eta \in \la\lambda\ra^{-M}L^\infty_\lambda(\Hb^{s,-N})
  \]
  for all $M>0$, and $\wt a'\colon\C\to \Hb^{s,-N}$ is holomorphic. Therefore, if we choose $N>|\alpha|$, we can shift the contour of integration to the horizontal line $\R-i\alpha$:
  \begin{equation}
  \begin{split}
  \label{EqAContourShift}
    a(x,y&,xD_x,D_y)u(x,y)=\int_{\Im\lambda=-\alpha} \wt a'(\lambda)(x,y)\,d\lambda \\
	  &=x^\alpha\iint_{\Im\lambda=0}e^{i\lambda\log x}e^{i\eta y}a(x,y;\lambda-i\alpha,\eta)(x^{-\alpha}u)\ftrans(\lambda,\eta)\,d\lambda\,d\eta.
  \end{split}
  \end{equation}
  By definition, $a|_{\Im\lambda=-\alpha}$ satisfies symbolic bounds just like $a|_{\Im\lambda=0}$, thus we are done by the first half of the proof.

  Adding bundles is straightforward: Write $A\in\fpsi^{m;0}\Hb^s(\Rnhalf;E,F)$ as $A=(A_{ij})$, $A_{ij}\in\fpsi^{m;0}\Hb^s(\Rnhalf)$ and $u\in\Hb^{s'}(\Rnhalf;E)$ as $u=(u_j)$, $u_j\in\Hb^{s'}(\Rnhalf)$. Then $Au=(\sum_{j=1}^{d_E}A_{ij}u_j)$, thus $Au\in\Hb^{s'-m}(\Rnhalf;F)$ follows by component-wise application of what we just proved.
\end{proof}

\begin{cor}
\label{CorHbModule}
  Let $s>n/2$. Then $\Hb^s(\Rnhalf;\End(E))$ is an algebra. Moreover, $\Hb^{s'}(\Rnhalf;\Hom(E,F))$ is a left $\Hb^s(\Rnhalf;\End(E))$- and a right $\Hb^s(\Rnhalf;\End(F))$-module for $|s'|\leq s$.
\end{cor}
\begin{proof}
  As in the proof of Proposition~\ref{PropOpCont}, we can reduce the proof to the case of complex-valued functions. For $s'\geq 0$, the claim follows from $\Hb^{s'}\subset\fpsib^{0;0}\Hb^{s'}$ and the previous Proposition. For $s'\leq 0$, use duality.
\end{proof}

%%%%%%%%%%%%%%%%%%%%%%%%%%%%%%%%%%%%%%%%%%%%%%%%%
\subsection{Compositions}
\label{SubsecComp}

The basic idea is to mimic the formula for the asymptotic expansion of the full symbol of an operator which is the composition of $P=p(z,\Db)\in\Psib^m$ and $Q=q(z,\Db)\in\Psib^{m'}$, namely
\[
  \sigma_{\tn{full}}(P\circ Q)(z,\zeta)\sim\sum_{\beta\geq 0} \frac{1}{\beta!} (\pa_\zeta^\beta p \Db_z^\beta q)(z,\zeta).
\]
If $p$ or $q$ only have limited regularity in $\zeta$ or $z$, we only keep finitely many terms of this expansion and estimate the resulting remainder term carefully. More precisely, keeping Remark~\ref{RmkMellinOfConst} in mind, we compute for $u\in\CIdotc$
\begin{align}
  (PQ&u)\ftrans(\eta)=\iint\wh p(\eta-\xi;\xi)\wh q(\xi-\zeta;\zeta)\wh u(\zeta)\,d\zeta\,d\xi \nonumber\\
\label{EqPQu}&=\int\left(\int\wh p(\eta-\zeta-\xi;\zeta+\xi)\wh q(\xi;\zeta)\,d\xi\right)\wh u(\zeta)\,d\zeta,
\end{align}
and
\begin{align*}
  [(\pa_\zeta^\beta p&\Db_z^\beta q)(z,\Db)u]\ftrans(\eta)=\int(\pa_\zeta^\beta p \Db_z^\beta q)\ftrans(\eta-\zeta;\zeta)\wh u(\zeta)\,d\zeta \\
    &=\int\left(\int\pa_\zeta^\beta\wh p(\eta-\zeta-\xi;\zeta)\xi^\beta\wh q(\xi;\zeta)\,d\xi\right)\wh u(\zeta)\,d\zeta.
\end{align*}
We now apply Taylor's theorem to the second argument of $\wh p$ at $\xi=0$ in the inner integral in \eqref{EqPQu}, keeping track of terms up to order $k-1$ which we assume to be $\geq 0$ (the remaining case $k=0$ is handled easily), and obtain a remainder
\[
  \wh r(\eta-\zeta;\zeta)=\sum_{|\beta|=k}\frac{k}{\beta!}\int\left(\int_0^1(1-t)^{k-1}\pa_\zeta^\beta\wh p(\eta-\zeta-\xi;\zeta+t\xi)\,dt\right)\xi^\beta\wh q(\xi;\zeta)\,d\xi,
\]
corresponding to the operator
\begin{equation}
\label{EqRemainderOp}
  r(z,\Db)=P\circ Q-\sum_{|\beta|<k}\frac{1}{\beta!}(\pa_\zeta^\beta p \Db_z^\beta q)(z,\Db).
\end{equation}
We rewrite the remainder as
\begin{equation}
\label{EqRemainder}
  \wh r(\eta;\zeta)=\sum_{|\beta|=k}\frac{k}{\beta!}\int\left(\int_0^1(1-t)^{k-1}\pa_\zeta^\beta\wh p(\eta-\xi;\zeta+t\xi)\,dt\right)(\Db_z^\beta q)\ftrans(\xi;\zeta)\,d\xi.
\end{equation}

We will start by analyzing the terms in an expansion like \eqref{EqRemainderOp} when the symbols involved are not smooth. When we deal with smooth b-operators by using the decomposition \eqref{EqDecomposition} of their symbols, we will need multiple sets of dual variables of $x$ and $y$. For clarity, we will stick to the following names for them:
\begin{align*}
  & \tn{(Mellin-)dual variables of }x\colon \sigma,\lambda,\rho, \\
  & \tn{(Fourier-)dual variables of }y\colon \gamma,\eta,\theta.
\end{align*}

\begin{lemma}
\label{LemmaSymbolMult}
  Let $s,s',m,m'\in\R$ be such that $s>n/2$, $|s'|\leq s$. Then
  \begin{gather*}
    S^{m;0}\Hb^s\cdot S^{m';0}\Hb^{s'}\subset S^{m+m';0}\Hb^{s'}, \\
	S^m\cdot S^{m';0}\Hb^{s'}\subset S^{m+m';0}\Hb^{s'}.
  \end{gather*}
  The same statements are true if all symbol classes are replaced by the corresponding b-symbol classes.
\end{lemma}
\begin{proof}
  In light of the definitions of the symbol classes, we can assume $m=m'=0$. The first statement then is an immediate consequence of Corollary~\ref{CorHbModule}. In order to prove the second statement, we simply observe that, given $p\in S^0$, $p(\cdot;\zeta)$ is a uniformly bounded family of multipliers on $\Hb^{s'}$. A direct proof of the sort that we will use in the sequel goes as follows: Decompose the symbol $p$ as in \eqref{EqDecomposition}. The part $p^{(1)}\in S^{0;0}\Hb^\infty$ can then be dealt with using the first statement. Thus, we may assume $p=p^{(0)}$, i.e.\ $p=p(y;\lambda,\eta)$ is $x$-independent. Let $q\in S^{0;0}\Hb^{s'}$ be given. Choose $N$ large and put
  \begin{align*}
    p_0(\gamma;\lambda,\eta)&=\la\gamma\ra^N|\sF p(\gamma;\lambda,\eta)|,\quad q_0(\sigma,\gamma;\lambda,\eta)=\la\sigma,\gamma\ra^{s'}|\wh q(\sigma,\gamma;\lambda,\eta)|, \\
	r_0(\sigma,\gamma;\lambda,\eta)&=\la\sigma,\gamma\ra^{s'}|\widehat{pq}(\sigma,\gamma;\lambda,\eta)|.
  \end{align*}
  Then
  \begin{align*}
    \iint r_0(\sigma,&\gamma;\lambda,\eta)^2\,d\sigma\,d\gamma \\
	  &\leq\iint\left(\int\frac{\la\sigma,\gamma\ra^{s'}}{\la\gamma-\theta\ra^N\la\sigma,\theta\ra^{s'}}p_0(\gamma-\theta;\lambda,\eta)q_0(\sigma,\theta;\lambda,\eta)\,d\theta\right)^2 d\sigma\,d\gamma \\
	  &\lesssim \|p_0(\gamma;\lambda,\eta)\|_{L^\infty_{\lambda,\eta}L^2_\gamma}^2\|q_0(\sigma,\theta;\lambda,\eta)\|_{L^\infty_{\lambda,\eta}L^2_{\sigma,\theta}}^2
  \end{align*}
  by Cauchy-Schwartz.
\end{proof}

Recall Remark~\ref{RmkNobNotation} for the notation used in the following theorem on the composition properties of non-smooth operators:

\begin{thm}
\label{ThmComp}
  Let $m,m',s,s'\in\R,k,k'\in\N_0$. For two operators $P=p(z,\Db)$ and $Q=q(z,\Db)$ of orders $m$ and $m'$, respectively, let
  \[
    R=P\circ Q-\sum_{|\beta|<k}\frac{1}{\beta!} (\pa_\zeta^\beta p \Db_z^\beta q)(z,\Db).
  \]
  Denote the sum of the terms in the expansion for which $|\beta|=j$ by $E_j$.
  \begin{enumerate}[leftmargin=\enummargin]
    \item\label{ThmCompRRAndSRR} Composition of non-smooth operators, $k\geq m+k'$, $k\geq k'$.
	  \begin{enumerate}[leftmargin=\enummargin]
	    \item \label{ThmCompRR} Suppose $s>n/2$ and $s\leq s'-k$ $[s\leq s'-2k+m+k']$. If $P\in\fpsi^{m;k}\Hb^s$, $Q\in\fpsi^{m';0}\Hb^{s'}$, then $E_j\in\fpsi^{m+m'-j;0}\Hb^s$ and $R\in\fpsi^{m'-k';0}\Hb^s$ $[\fpsi^{m+m'-k;0}\Hb^s]$.
		\item \label{ThmCompSRR} If $P\in\fpsi^{m;k}\Hb^\infty$, $Q\in\fpsi^{m';0}\Hb^{s'}$, then $E_j\in\fpsi^{m+m'-j;0}\Hb^{s'-j}$ and $R\in\fpsi^{m'-k';0}\Hb^{s'-k}\cap\fpsi^{m+m'-k;0}\Hb^{s'-2k+m+k'}$.
	  \end{enumerate}
	\item\label{ThmCompBRAndRB} Composition of smooth with non-smooth operators.
	   \begin{enumerate}[leftmargin=\enummargin]
		 \item \label{ThmCompBR} Suppose $k\geq m+k'$, $k\geq k'$. If $P\in\Psi^m$, $Q\in\fpsi^{m';0}\Hb^{s'}$, then $E_j\in\fpsi^{m+m'-j;0}\Hb^{s'-j}$ and $R\in\fpsi^{m'-k';0}\Hb^{s'-k}\cap\fpsi^{m+m'-k;0}\Hb^{s'-2k+m+k'}$.
		 \item \label{ThmCompRB} Suppose $k\leq k'$ and $k'\geq m$. If $P\in\fpsi^{m;k'}\Hb^s$, $Q\in\Psi^{m'}$, then $E_j\in\fpsi^{m+m'-j;0}\Hb^s$ and $R\in\fpsi^{m+m'-k;0}\Hb^s$.
	  \end{enumerate}
	\item \label{ThmCompBRFunny} Composition of smooth with non-smooth operators, $k\leq m+k'$, $k\geq k'$. If $P\in\Psi^m$, $Q\in\fpsi^{m';0}\Hb^{s'}$, then $E_j\in\fpsi^{m+m'-j;0}\Hb^{s'-j}$ and $R=R_1\Lambda_{m+k'-k}+\Lambda_{m+k'-k}R_2$, where $R_1,R_2\in\fpsi^{m'-k';0}\Hb^{s'-k}$.
  \end{enumerate}
  Moreover, (1)-(2) hold as well if all operator spaces are replaced by the corresponding b-spaces. Also, all results hold, \emph{mutatis mutandis}, if $P$ maps sections of $F$ to sections of $G$, and $Q$ maps sections of $E$ to sections of $F$.
\end{thm}
\begin{proof}
  The statements about the $E_j$ follow from Lemma~\ref{LemmaSymbolMult}. It remains to analyze the remainder operators. We will only treat the case $k>0$; the case $k=0$ is handled in a similar way. We prove parts \itref{ThmCompRRAndSRR}, \itref{ThmCompBR} and \itref{ThmCompBRFunny} of the theorem for $k'=0$ first.

  \itref{ThmCompRR}. Consider the case $s\leq s'-k$. We use formula \eqref{EqRemainder} and define
  \begin{gather*}
    p_0(\eta,\xi;\zeta)=\sum_{|\beta|=k}\frac{k}{\beta!}\la\eta\ra^s\int_0^1|\pa_\zeta^\beta\wh p(\eta;\zeta+t\xi)|\,dt, \\
	q_0(\xi;\zeta)=\frac{\la\xi\ra^{s'-k}|(\Db_z^k q)\ftrans(\xi;\zeta)|}{\la\zeta\ra^{m'}},
  \end{gather*}
  where $\Db_z^k$ denotes the vector $(\Db_z^\beta)_{|\beta|=k}$. Since $p_0\in L^\infty_{\zeta,\xi}L^2_\eta$ in view of $k\geq m$, i.e.\ $\pa_\zeta^\beta p$ is a symbol of order $m-k\leq 0$, and $q_0\in L^\infty_\zeta L^2_\xi$, we obtain
  \[
    \frac{\la\eta\ra^s|\wh r(\eta;\zeta)|}{\la\zeta\ra^{m'}}\leq\int\frac{\la\eta\ra^s}{\la\eta-\xi\ra^s\la\xi\ra^{s'-k}}p_0(\eta-\xi,\xi;\zeta)q_0(\xi;\zeta)\,d\xi\in L^\infty_\zeta L^2_\eta
  \]
  by Lemma~\ref{LemmaIntegralEstimate}, as claimed. Next, if $s\leq s'-2k+m$, we instead define
  \begin{equation}
  \label{EqP0StrongerNorm}
    p_0(\eta,\xi;\zeta)=\sum_{|\beta|=k}\frac{k}{\beta!}\la\eta\ra^s\int_0^1\la\zeta+t\xi\ra^{k-m}|\pa_\zeta^\beta\wh p(\eta;\zeta+t\xi)|\,dt \in L^\infty_{\zeta,\xi}L^2_\eta,
  \end{equation}
  thus
  \begin{align*}
    \frac{\la\eta\ra^s|\wh r(\eta;\zeta)|}{\la\zeta\ra^{m+m'-k}}&\leq\int\frac{\la\eta\ra^s}{\la\eta-\xi\ra^s\la\xi\ra^{s'-k}}\cdot\frac{\la\zeta\ra^{k-m}}{\inf_{0\leq t\leq 1}\la\zeta+t\xi\ra^{k-m}} \\
	  &\hspace{25ex}\times p_0(\eta-\xi,\xi;\zeta)q_0(\xi;\zeta)\,d\xi
  \end{align*}
  with $q_0\in L^\infty_\zeta L^2_\xi$ as above. Now
  \begin{equation}
  \label{EqInfEstimate}
    \frac{\la\zeta\ra^{k-m}}{\inf_{0\leq t\leq 1}\la\zeta+t\xi\ra^{k-m}}\lesssim\la\xi\ra^{k-m},
  \end{equation}
  since for $|\xi|\leq|\zeta|/2$, the left hand side is uniformly bounded, and for $|\zeta|\leq 2|\xi|$, we estimate the infimum from below by $1$ and the numerator from above by $\la\xi\ra^{k-m}$. Therefore, we get $r_0\in L^\infty_\zeta L^2_\eta$ in this case as well.

  \itref{ThmCompSRR}. This is proved similarly: Define $q_0(\xi;\zeta)$ as above, and choose $N$ large and put
  \[
    p_0(\eta,\xi;\zeta)=\sum_{|\beta|=k}\frac{k}{\beta!}\la\eta\ra^N\int_0^1|\pa_\zeta^\beta\wh p(\eta;\zeta+t\xi)|\,dt.
  \]
  Then
  \[
    \frac{\la\eta\ra^{s'-k}|\wh r(\eta;\zeta)|}{\la\zeta\ra^{m'}}\leq\int\frac{\la\eta\ra^{s'-k}}{\la\eta-\xi\ra^N\la\xi\ra^{s'-k}}p_0(\eta-\xi,\xi;\zeta)q_0(\xi;\zeta)\,d\xi,
  \]
  and the fraction in the integrand is an element of $L^\infty_\eta L^2_\xi$ by Lemma~\ref{LemmaFractionEstimate}, thus an application of Lemma~\ref{LemmaIntegralEstimate} yields $R\in\fpsi^{m';0}\Hb^{s'-k}$. In a similar manner, now using \eqref{EqInfEstimate}, we obtain $R\in\fpsi^{m+m'-k;0}\Hb^{s'-2k+m}$.

  \ \\
  \itref{ThmCompBRAndRB}. Decomposing the smooth operator as in \eqref{EqDecomposition}, the $x$-dependent part has coefficients in $\Hb^\infty$, thus we can apply part \itref{ThmCompRRAndSRR}. Therefore, we may assume that the smooth operator is $x$-independent in both cases.

  \itref{ThmCompBR}. The remainder is
  \begin{align*}
    \wh r(\sigma,\gamma;\lambda,\eta)&=\sum_{|\beta|=k}\frac{k}{\beta!}\int\left(\int_0^1(1-t)^{k-1}\pa_{\lambda,\eta}^\beta\sF p(\gamma-\theta;\lambda+t\sigma,\eta+t\theta)\,dt\right) \\
	  &\hspace{15ex}\times (\Db_z^\beta q)\ftrans(\sigma,\theta;\lambda,\eta)\,d\theta;
  \end{align*}
  therefore, choosing $N$ large and defining
  \begin{align*}
	p_0(\gamma,\sigma,\theta;\lambda,\eta)&=\sum_{|\beta|=k}\frac{k}{\beta!}\la\gamma\ra^N\int_0^1|\pa_{\lambda,\eta}^\beta\sF p(\gamma;\lambda+t\sigma,\eta+t\theta)|\,dt \in L^\infty_{\sigma,\theta,\lambda,\eta}L^2_\gamma, \\
	q_0(\sigma,\theta;\lambda,\eta)&=\frac{\la\sigma,\theta\ra^{s'-k}|(\Db_z^\beta q)\ftrans(\sigma,\theta;\lambda,\eta)|}{\la\lambda,\eta\ra^{m'}}\in L^\infty_{\lambda,\eta}L^2_{\sigma,\theta},
  \end{align*}
  we get
  \begin{align*}
    &\frac{\la\sigma,\gamma\ra^{s'-k}|\wh r(\sigma,\gamma;\lambda,\eta)|}{\la\lambda,\eta\ra^{m'}} \\
	  &\qquad\leq\int\frac{\la\sigma,\gamma\ra^{s'-k}}{\la\gamma-\theta\ra^N\la\sigma,\theta\ra^{s'-k}}p_0(\gamma-\theta,\sigma,\theta;\lambda,\eta)q_0(\sigma,\theta;\lambda,\eta)\,d\theta,
  \end{align*}
  which is an element of $L^\infty_{\lambda,\eta}L^2_{\sigma,\gamma}$ by Lemmas~\ref{LemmaFractionEstimate} and~\ref{LemmaIntegralEstimate}. This proves $R\in\fpsi^{m';0}\Hb^{s'-k}$, and in a similar way we obtain $R\in\fpsi^{m+m'-k;0}\Hb^{s'-2k+m}$.

  \itref{ThmCompRB}. Here, the remainder is
  \begin{align*}
    \wh r(\sigma,\gamma;\lambda,\eta)&=\sum_{|\beta|=k}\frac{k}{\beta!}\int\left(\int_0^1(1-t)^{k-1}\pa_{\lambda,\eta}^\beta\wh p(\sigma,\gamma-\theta;\lambda,\eta+t\theta)\,dt\right) \\
	  &\hspace{15ex}\times \sF(\Db_z^\beta q)(\theta;\lambda,\eta)\,d\theta,
  \end{align*}
  and arguments similar to those used in (a) give the desired conclusion if $k=k'$. If $k<k'$, we just truncate the expansion after $E_{k-1}$ and note that the resulting remainder term, which is the sum of the remainder term after expanding to order $k'$ and the expansion terms $E_k,\ldots,E_{k'-1}$, indeed lies in $\fpsi^{m+m'-k;0}\Hb^s$.

  \ \\
  \itref{ThmCompBRFunny}. We again use formula \eqref{EqRemainder} for the remainder term and put
  \[
    \wh r_1(\eta;\zeta)=\frac{\wh r(\eta;\zeta)\chi(|\zeta|\geq|\eta+\zeta|)}{\lambda_{m-k}(\zeta)},\quad \wh r_2(\eta;\zeta)=\frac{\wh r(\eta;\zeta)\chi(|\zeta|<|\eta+\zeta|)}{\lambda_{m-k}(\eta+\zeta)},
  \]
  the point being that, by equation \eqref{EqOperatorAction}, for any $u\in\CIdotc$,
  \begin{align*}
    (r&(z,\Db)u)\ftrans(\eta)=\int\wh r(\eta-\zeta,\zeta)\wh u(\zeta)\,d\zeta \\
	  &=\int\wh r_1(\eta-\zeta,\zeta)(\Lambda_{m-k}u)\ftrans(\zeta)\,d\zeta+\lambda_{m-k}(\eta)\int\wh r_2(\eta-\zeta,\zeta)\wh u(\zeta)\,d\zeta \\
	  &=(r_1(z,\Db)\Lambda_{m-k}u)\ftrans(\eta)+(\Lambda_{m-k}r_2(z,\Db)u)\ftrans(\eta).
  \end{align*}
  It remains to prove that $r_1(z,\Db),r_2(z,\Db)\in\fpsi^{m';0}\Hb^{s'-k}$. First, we treat the case $P\in x\Psi^m$. Then for any $N\in\N$, we obtain, using
  \[
    \sup_{0\leq t\leq 1}\la\zeta+t\xi\ra^{m-k}\lesssim\la\zeta\ra^{m-k}+\la\xi\ra^{m-k},
  \]
  that
  \begin{align*}
    \frac{\la\eta\ra^{s'-k}|\wh r_1(\eta,\zeta)|}{\la\zeta\ra^{m'}}&\lesssim\int\frac{\la\eta\ra^{s'-k}(1+\la\xi\ra^{m-k}/\la\zeta\ra^{m-k})}{\la\eta-\xi\ra^N\la\xi\ra^{s'-k}}\chi(|\zeta|\geq|\eta+\zeta|) \\
	 &\hspace{12ex}\times p_0(\eta-\xi,\xi;\zeta)q_0(\xi;\zeta)\,d\xi \\
	 &\equiv\int G(\eta,\xi;\zeta)p_0(\eta-\xi,\xi;\zeta)q_0(\xi;\zeta)\,d\xi,
  \end{align*}
  where $p_0(\eta,\xi;\zeta)\in L^\infty_{\xi,\zeta}L^2_\eta$ is defined as in \eqref{EqP0StrongerNorm} (with $s$ replaced by $N$) and $q_0\in L^\infty_\zeta L^2_\xi$ as before. We have to show $G(\eta,\xi;\zeta)\in L^\infty_{\eta,\zeta}L^2_\xi$ in order to be able to apply Lemma~\ref{LemmaIntegralEstimate}. For $|\xi|\geq 2|\eta|$, we immediately get, for $N$ large enough,
  \[
    G(\eta,\xi;\zeta)\lesssim\frac{1}{\la\xi\ra^{N'}}\left(1+\frac{\la\xi\ra^{m-k}}{\la\zeta\ra^{m-k}}\right)\in L^\infty_{\eta,\zeta}L^2_\xi(|\xi|\geq 2|\eta|),
  \]
  where $N'=N-(k-s')_+$. On the other hand, if $|\xi|<2|\eta|$, we estimate
  \[
    G(\eta,\xi;\zeta)\lesssim\frac{\la\eta\ra^{s'-k}}{\la\eta-\xi\ra^N\la\xi\ra^{s'-k}}\left(1+\frac{\la\eta\ra^{m-k}}{\la\zeta\ra^{m-k}}\right)\chi(|\zeta|\geq|\eta+\zeta|)
  \]
  and use that $|\zeta|\geq|\eta+\zeta|$ implies $|\eta|\leq|\eta+\zeta|+|-\zeta|\leq 2|\zeta|$, hence the product of the last two factors is uniformly bounded, giving $G(\eta,\xi;\zeta)\in L^\infty_{\eta,\zeta}L^2_\xi(|\xi|<2|\eta|)$ by Lemma~\ref{LemmaFractionEstimate}.

  In the case $P=p(0,y;xD_x,D_y)$, we get the estimate
  \[
    \frac{\la\sigma,\gamma\ra^{s'-k}|\wh r_1(\sigma,\gamma;\lambda,\eta)|}{\la\lambda,\eta\ra^{m'}}\leq\int G(\sigma,\gamma,\theta;\lambda,\eta)p_0(\gamma-\theta,\sigma,\theta;\lambda,\eta)q_0(\sigma,\theta;\lambda,\eta)\,d\theta,
  \]
  where $p_0(\gamma,\sigma,\theta;\lambda,\eta)\in L^\infty_{\sigma,\theta,\lambda,\eta}L^2_\gamma$, $q_0(\sigma,\theta;\lambda,\eta)\in L^\infty_{\lambda,\eta}L^2_{\sigma,\theta}$, and
  \begin{align*}
    G(\sigma,\gamma,\theta;\lambda,\eta)&=\frac{\la\sigma,\gamma\ra^{s'-k}}{\la\gamma-\theta\ra^N\la\sigma,\theta\ra^{s'-k}}\left(1+\frac{\la\sigma,\theta\ra^{m-k}}{\la\lambda,\eta\ra^{m-k}}\right) \\
	 &\hspace{12ex}\times\chi\Bigl(|(\lambda,\eta)|\geq|(\sigma,\gamma)+(\lambda,\eta)|\Bigr).
  \end{align*}
  As above, separating the cases $|(\sigma,\theta)|\geq 2|(\sigma,\gamma)|$ and $|(\sigma,\theta)|<2|(\sigma,\gamma)|$, one obtains $G\in L^\infty_{\sigma,\gamma,\lambda,\eta}L^2_\theta$, and we can again apply Lemma~\ref{LemmaIntegralEstimate}.
  
  The second remainder term $r_2$ is handled in the same way.

  \ \\
  Next, we prove that \itref{ThmCompRRAndSRR}-\itref{ThmCompBRAndRB} also hold for the corresponding b-operator spaces. Using exactly the same estimates as above, one obtains the respective symbolic bounds for the remainders on each line $\Im\lambda=\alpha_0$. What remains to be shown is the holomorphicity of the remainder operator in $\lambda$. This is a consequence of the fact that the derivatives $\pa_\lambda\pa_\zeta^\beta p$, $|\beta|=k$, and $\pa_\lambda q$, satisfy the same (in the case of symbols of smooth b-\psdo{}s, even better by one order) symbol estimates as $\pa_\zeta^\beta p$ and $q$, respectively. Indeed, for \itref{ThmCompRR}, i.e.\ for non-smooth b-symbols, this follows from the Cauchy integral formula, which for $\pa_\lambda q$ gives
  \[
    \pa_\lambda q(z;\lambda,\eta)=\frac{1}{2\pi i}\oint_{\gamma(\lambda)} \frac{q(z;\sigma,\eta)}{(\sigma-\lambda)^2}\,d\sigma
  \]
  where $\gamma(\lambda)$ is the circle around $\lambda$ with radius $1$. Namely, since $|\sigma-\lambda|=1$ for $\sigma\in\gamma(\lambda)$, we get the desired estimate for $\pa_\lambda q$ from the corresponding estimate for $q$ itself. We handle $\pa_\lambda\pa_\zeta^\beta p$ similarly. \itref{ThmCompSRR} and \itref{ThmCompBRAndRB} for b-operators follow in the same way.

  \ \\
  Finally, let us prove \itref{ThmCompRRAndSRR}, \itref{ThmCompBR} and \itref{ThmCompBRFunny} for $k'>0$ following the argument of Beals and Reed in \cite[Corollary~1.6]{BealsReedMicroNonsmooth}, starting with \itref{ThmCompRR}: Choose a partition of unity on $\R^n$ consisting of smooth non-negative functions $\chi_0,\ldots,\chi_n$ with $\supp\chi_0\subset\{|\zeta|\leq 2\}$, and $|\zeta_l|\geq 1$ on $\supp\chi_l$. Then
  \[
    P\circ Q\chi_0(\Db)\in\fpsi^{m;k}\Hb^s\circ\fpsi^{-\infty;0}\Hb^{s'}
  \]
  can be treated using (1a) with $k'=0$, taking an expansion up to order $k\geq m+k'\geq m$; all terms in the expansion as well as the remainder term are elements of $\fpsi^{-\infty;0}\Hb^s$, hence $P\circ Q\chi_0(\Db)\in\fpsi^{-\infty;0}\Hb^s$ can be put into the remainder term of the claimed expansion.
  
  Let us now consider $P\circ Q\chi_l(\Db)$. For brevity, let us replace $Q$ by $Q\chi_l(\Db)$ and thus assume $|\zeta_l|\geq 1$ on $\supp q(z,\zeta)$. Then by the Leibniz rule,
  \[
    P\circ Q\Db_{z_l}^{k'}=P\Db_{z_l}^{k'}\circ Q-\sum_{j=1}^{k'} c_{jk'} P\Db_{z_l}^{k'-j}\circ(\Db_{z_l}^j q)(z,\Db)
  \]
  for some constants $c_{jk'}\in\R$. Composing on the right with\footnote{To be precise, one should take $\Db_{z_l}^{-k'}\wt\chi_l(\Db)$, where $\wt\chi_l\equiv 1$ on $\supp\chi_l$ and $|\zeta_l|\geq 1/2$ on $\supp\wt\chi_l$.} $\Db_{z_l}^{-k'}$ thus shows that $P\circ Q$ is an element of the space
  \[
    \sum_{j=0}^{k'} \fpsi^{m+k'-j;0}\Hb^s\circ\fpsi^{m'-k'}\Hb^{s'-j}.
  \]
  In view of the part of \itref{ThmCompRR} already proved, the $j$-th summand has an expansion to order $k-j\geq m+k'-j$ with error term in $\fpsib^{m'-k';0}\Hb^s\ [\fpsib^{m+m'-k;0}\Hb^s]$, where we use $k-j\geq k-k'\geq 0$ and $s\leq(s'-j)-(k-j)$ [$s\leq(s'-j)-2(k-j)+(m+k'-j)$]. Using the same idea, one can prove \itref{ThmCompSRR}, \itref{ThmCompBR} and \itref{ThmCompBRFunny}.
\end{proof}

Notice that we do not claim in \itref{ThmCompBRFunny} that $R_1$ and $R_2$ lie in b-operator spaces if $q$ does. The issue is that $1/\lambda_m(\zeta)$ in general has singularities for non-real $\zeta$. In applications later in this paper, we will only need the proposition as stated, with the additional assumption that $p$ is a b-symbol, since instead of letting the operators in the expansion and the remainder operator act on weighted spaces, we will conjugate $P$ and $Q$ by the weight before applying the theorem.

%%%%%%%%%%%%%%%%%%%%%%%%%%%%%%%%%%%%%%%%%%%%%%%%%
\section{Reciprocals of and compositions with $\Hb^s$ functions}
\label{SecHsRecAndComp}

In this section, we recall some basic results about $1/u$ and, more generally, $F(u)$, for $u$ in appropriate b-Sobolev spaces on an $n$-dimensional compact manifold with boundary $M$, and smooth/analytic functions $F$.

\begin{rmk}
\label{RmkNoMoser}
  We will give direct proofs which in particular do not give Moser-type bounds; see \cite[\S\S{13.3, 13.10}]{TaylorPDE} for examples of the latter. However, at least special cases of the results below (e.g.\ when $\CI(M)$ is replaced by $\C$ or $\R$) can easily be proved in a way as to obtain such bounds: The point is that the analysis can be localized and thus reduced to the case $M=\Rnhalf$; a logarithmic change of coordinates then gives an isometric isomorphism of $\Hb^s(\Rnhalf)$ and $H^s(\R^n)$, and on the latter space, Moser-type reciprocal/composition results are standard, see \cite{TaylorPDE}.
\end{rmk}

%%%%%%%%%%%%%%%%%%%%%%%%%%%%%%%%%%%%%%%%%%%%%%%%%
\subsection{Reciprocals}
\label{SubsecHsRec}

Let $M$ be a compact $n$-dimensional manifold with boundary.

\begin{lemma}
\label{LemmaRecHs}
  Let $s>n/2+1$. Suppose $u,w\in\Hb^s(M)$ and $a\in\CI(M)$ are such that $|a+u|\geq c_0>0$ on $\supp w$. Then $w/(a+u)\in\Hb^s(M)$, and one has an estimate
  \begin{equation}
  \label{EqHsReciprocalEstimate}
    \left\|\frac{w}{a+u}\right\|_{\Hb^s}\leq C_K\|w\|_{\Hb^s}\left(1+\|u\|_{\Hb^s}\right)^{\lceil s\rceil}\left(1+\left\|\frac{1}{a+u}\right\|_{L^\infty(K)}\right)^{\lceil s\rceil+1}
  \end{equation}
  for any neighborhood $K$ of $\supp w$.
\end{lemma}
\begin{proof}
  We can assume that $\supp w$ and $\supp u$ lie in a coordinate patch of $M$. Note that clearly $w/(a+u)\in L^2_\bl$. We will give an iterative argument that improves on the regularity of $w/(a+u)$ by (at most) $1$ at each step, until we can eventually prove $\Hb^s$-regularity.
  
  To set this up, let us assume $w/(a+u)\in\Hb^{s'-1}$ for some $1\leq s'\leq s$. Recall the operator $\Lambda_{s'}=\lambda_{s'}(\Db)$ from Corollary~\ref{CorBStdOp}, and choose $\psi_0,\psi\in\CI(M)$ such that $\psi_0\equiv 1$ on $\supp w$, $\psi\equiv 1$ on $\supp\psi_0$, and such that moreover $|a+u|\geq c_0'>0$ on $\supp\psi$, which can be arranged since $u\in\Hb^s\subset C^0$. Then for $K=\supp\psi$,
  \begin{align}
    \Bigl\|&\Lambda_{s'}\frac{w}{a+u}\Bigr\|_{L^2_\bl}\leq\Bigl\|(1-\psi)\Lambda_{s'}\frac{\psi_0 w}{a+u}\Bigr\|_{L^2_\bl}+\Bigl\|\psi \Lambda_{s'}\frac{\psi_0 w}{a+u}\Bigr\|_{L^2_\bl} \nonumber\\
	    &\lesssim \Bigl\|\frac{w}{a+u}\Bigr\|_{L^2_\bl}+\Bigl\|\frac{1}{a+u}\Bigr\|_{L^\infty(K)}\Bigl\|\psi(a+u)\Lambda_{s'}\frac{w}{a+u}\Bigr\|_{L^2_\bl} \nonumber\\
	    &\lesssim \Bigl\|\frac{w}{a+u}\Bigr\|_{L^2_\bl}+\Bigl\|\frac{1}{a+u}\Bigr\|_{L^\infty(K)}\left(\|\psi \Lambda_{s'}w\|_{L^2_\bl}+\Bigl\|\psi[\Lambda_{s'},a+u]\frac{w}{a+u}\Bigr\|_{L^2_\bl}\right) \nonumber\\
      \begin{split}
	  \label{EqHsRecEllEstimate}
	    &\lesssim \Bigl\|\frac{w}{a+u}\Bigr\|_{L^2_\bl}+\Bigl\|\frac{1}{a+u}\Bigr\|_{L^\infty(K)} \\
		&\hspace{20ex}\times\biggl(\|w\|_{\Hb^{s'}}+\Bigl\|\frac{w}{a+u}\Bigr\|_{\Hb^{s'-1}}+\Bigl\|\psi[\Lambda_{s'},u]\frac{w}{a+u}\Bigr\|_{L^2_\bl}\biggr),
      \end{split}
  \end{align}
  where we used that the support assumptions on $\psi_0$ and $\psi$ imply $(1-\psi)\Lambda_{s'}\psi_0\in\Psi^{-\infty}$, and $\psi[\Lambda_{s'},a]\in\Psi^{s'-1}$. Hence, in order to prove that $w/(a+u)\in\Hb^{s'}$, it suffices to show that $[\Lambda_{s'},u]\colon \Hb^{s'-1}\to L^2_\bl$. Let $v\in\Hb^{s'-1}$. Since
  \begin{align*}
    (\Lambda_{s'}uv)\ftrans(\zeta)&=\int \lambda_{s'}(\zeta)\wh u(\zeta-\eta)\wh v(\eta)\,d\eta \\
	(u\Lambda_{s'}v)\ftrans(\zeta)&=\int\wh u(\zeta-\eta)\lambda_{s'}(\eta)\wh v(\eta)\,d\eta,
  \end{align*}
  we have, by taking a first order Taylor expansion of $\lambda_{s'}(\zeta)=\lambda_{s'}(\eta+(\zeta-\eta))$ around $\zeta=\eta$,
  \[
    ([\Lambda_{s'},u]v)\ftrans(\zeta)=\sum_{|\beta|=1}\int\left(\int_0^1\pa_\zeta^\beta \lambda_{s'}(\eta+t(\zeta-\eta))\,dt\right)(\Db_z^\beta u)\ftrans(\zeta-\eta)\wh v(\eta)\,d\eta.
  \]
  We will to prove that this is an element of $L^2_\zeta$ using Lemma~\ref{LemmaIntegralEstimate}. Since for $|\beta|=1$,
  \[
    |\pa_\zeta^\beta \lambda_{s'}(\eta+t(\zeta-\eta))|\lesssim\la\eta+t(\zeta-\eta)\ra^{s'-1}, \\
  \]\[
    |(\Db_z^\beta u)\ftrans(\zeta-\eta)|=\frac{u_0(\zeta-\eta)}{\la\zeta-\eta\ra^{s-1}},\quad |\wh v(\eta)|=\frac{v_0(\eta)}{\la\eta\ra^{s'-1}}
  \]
  for $u_0,v_0\in L^2$, it is enough to observe that
  \[
    \frac{\la\eta+t(\zeta-\eta)\ra^{s'-1}}{\la\zeta-\eta\ra^{s-1}\la\eta\ra^{s'-1}}\lesssim\frac{1}{\la\zeta-\eta\ra^{s-1}}+\frac{1}{\la\zeta-\eta\ra^{s-s'}\la\eta\ra^{s'-1}}\in L^\infty_\zeta L^2_\eta,
  \]
  uniformly in $t\in[0,1]$, since $s-1>n/2$.

  To obtain the estimate \eqref{EqHsReciprocalEstimate}, we proceed inductively, starting with the obvious estimate
  \[
    \|w/(a+u)\|_{L^2_\bl}\leq\|w\|_{L^2_\bl}\|1/(a+u)\|_{L^\infty(K)}\leq \|w\|_{\Hb^s}\left(1+\Bigl\|\frac{1}{a+u}\Bigr\|_{L^\infty(K)}\right).
  \]
  Then, assuming that for integer $1\leq m\leq s$, one has
  \[
    \|w/(a+u)\|_{H^{m-1}_\bl}\lesssim \|w\|_{\Hb^s}\left(1+\Bigl\|\frac{1}{a+u}\Bigr\|_{L^\infty(K)}\right)^m\left(1+\|u\|_{\Hb^s}\right)^{m-1}
  \]
  we conclude, using the estimate \eqref{EqHsRecEllEstimate},
  \begin{align*}
    \Bigr\|&\frac{w}{a+u}\Bigr\|_{\Hb^m} \\
	  &\lesssim\Bigl\|\frac{w}{a+u}\Bigr\|_{L^2_\bl}+\Bigl\|\frac{1}{a+u}\Bigr\|_{L^\infty(K)}\left(\|w\|_{H^s_\bl}+(1+\|u\|_{\Hb^s})\Bigl\|\frac{w}{a+u}\Bigr\|_{\Hb^{m-1}}\right) \\
	  &\lesssim \|w\|_{\Hb^s}\left(1+\Bigl\|\frac{1}{a+u}\Bigr\|_{L^\infty(K)}\right)^{m+1}\left(1+\|u\|_{\Hb^s}\right)^m.
  \end{align*}
  Thus, one gets such an estimate for $m=\lfloor s\rfloor$; then the same type of estimate gives \eqref{EqHsReciprocalEstimate}, since one has control over the $H^{s-1}_\bl$-norm of $w/(a+u)$ in view of $s-1<\lfloor s\rfloor$ and the bound on $\|w/(a+u)\|_{\Hb^{\lfloor s\rfloor}}$.
\end{proof}

In particular:

\begin{cor}
\label{CorReciprocal}
  Let $s>n/2+1$.
  \begin{enumerate}[leftmargin=\enummargin]
    \item If $u\in\Hb^s(M)$ does not vanish on $\supp\phi$, where $\phi\in\CI_\cl(M)$, then $\phi/u\in\Hb^s(M)$.
	\item Let $\alpha\geq 0$. If $u\in\Hb^{s,\alpha}(M)$ is bounded away from $-1$, then $1/(1+u)\in 1+\Hb^{s,\alpha}(M)$.
  \end{enumerate}
\end{cor}
\begin{proof}
  The second statement follows from
  \[
    1-\frac{1}{1+u}=\frac{u}{1+u}\in\Hb^{s,\alpha}(M).\qedhere
  \]
\end{proof}

We also obtain the following result on the inversion of non-smooth elliptic symbols:

\begin{prop}
\label{PropSymbolDivision}
  Let $s>n/2+1$, $m\in\R$, $k\in\N_0$.
  \begin{enumerate}
    \item Suppose $p(z,\zeta)\in S^{m;k}\Hb^s(\Rnhalf;\Hom(E,F))$ and $a(z,\zeta)\in S^0$ are such that $\|p(z,\zeta)^{-1}\|_{\Hom(F,E)}\leq c_0\la\zeta\ra^{-m}$, $c_0<\infty$, on $\supp a$. Then
	  \[
	    ap^{-1}\in S^{-m;k}\Hb^s(\Rnhalf;\Hom(F,E)).
	  \]
	 \item Let $\alpha\geq 0$. Suppose that $p'(z,\zeta)\in S^{m;k}\Hb^{s,\alpha}(\Rnhalf;\Hom(E,F))$, $p''(z,\zeta)\in S^m(\Rnhalf;\Hom(E,F))$ and $a(z,\zeta)\in S^0$ are such that
	 \[
	   \|(p'')^{-1}\|_{\Hom(F,E)},\|(p'+p'')^{-1}\|_{\Hom(F,E)}\leq c_0\la\zeta\ra^{-m}\tn{ on }\supp a.
	 \]
	 Then
	 \[
	   a(p'+p'')^{-1}\in a(p'')^{-1}+S^{-m;k}\Hb^{s,\alpha}(\Rnhalf;\Hom(F,E)).
	 \]
  \end{enumerate}
\end{prop}
\begin{proof}
  By multiplying the symbols $p$ and $p'$ by $\la\zeta\ra^{-m}$, we may assume that $m=0$.

  \begin{enumerate}[leftmargin=\enummargin]
    \item Let us first treat the case of complex-valued symbols. By Corollary~\ref{CorReciprocal}, $a(\cdot,\zeta)/p(\cdot,\zeta)\in\Hb^s$, uniformly in $\zeta$; thus $a/p\in S^{0;0}\Hb^s$. Moreover, for $|\alpha|\leq k$,
	  \[
	    \pa_\zeta^\alpha\left(\frac{a}{p}\right)=\sum c_{\beta_1\cdots\gamma_\nu}\frac{\prod_{j=1}^\mu \pa_\zeta^{\beta_j}a \prod_{l=1}^\nu \pa_\zeta^{\gamma_l}p}{p^{\nu+1}},
	  \]
	  where the sum is over all $\beta_1+\cdots+\beta_\mu+\gamma_1+\cdots+\gamma_\nu=\alpha$ with $|\gamma_j|\geq 1$, $1\leq j\leq\nu$. Hence, using that $\Hb^s$ is an algebra and that the growth order of the numerator is $-|\alpha|$, we conclude, again by Corollary~\ref{CorReciprocal}, that $\pa_\zeta^\alpha(a/p)\in S^{-|\alpha|;0}\Hb^s$; thus $a/p\in S^{0;k}\Hb^s$.

	  If $p$ is bundle-valued, we obtain $ap^{-1}\in S^{0;0}\Hb^s(\Rnhalf;\Hom(F,E))$ using the explicit formula for the inverse of a matrix and Corollaries~\ref{CorHbModule} and \ref{CorReciprocal}; then, by virtue of
	  \[
	    \pa_\zeta(ap^{-1})=(\pa_\zeta a-ap^{-1}(\pa_\zeta p))p^{-1},
	  \]
	  similarly for higher derivatives, we get $ap^{-1}\in S^{0;k}\Hb^s(\Rnhalf;\Hom(F,E))$.
	
	\item Since $a(p'+p'')^{-1}=\bigl(a(p'')^{-1}\bigr)\bigl(I+p'(p'')^{-1}\bigr)^{-1}$, we may assume $p''=I$, $a\in S^0(\Rnhalf;\Hom(F,E))$ and $p'\in S^{0;k}\Hb^{s,\alpha}(\Rnhalf;\End(F))$, and we need to show
	\[
	  (I+p')^{-1}-I\in S^{0;k}\Hb^{s,\alpha}(\Rnhalf;\End(F)).
	\]
	But we can write
	\[
	  (I+p')^{-1}-I=-p'(I+p')^{-1},
	\]
	which is an element of $S^{0;0}\Hb^{s,\alpha}(\Rnhalf;\End(F))$ by assumption. Then, by an argument similar to the one employed in the first part, we obtain the higher symbol estimates.\qedhere
  \end{enumerate}
\end{proof}

%%%%%%%%%%%%%%%%%%%%%%%%%%%%%%%%%%%%%%%%%%%%%%%%%
\subsection{Compositions}
\label{SubsecHsComp}

Using the results of the previous subsection and the Cauchy integral formula, we can prove several results on the regularity of $F(u)$ for $F$ smooth or holomorphic and $u$ in a weighted b-Sobolev space. The main use of such results for us will be that they allow us to understand the regularity of the coefficients of wave operators associated to non-smooth metrics.

In all results in this section, we shall assume that $M$ is a compact $n$-dimensional manifold with boundary, $s>n/2+1$, and $\alpha\geq 0$.

\begin{prop}
\label{PropCompWithAnalytic}
  Let $u\in\Hb^{s,\alpha}(M)$. If $F\colon\Omega\to\C$ is holomorphic in a simply connected neighborhood $\Omega$ of $u(M)$, then $F(u)-F(0)\in\Hb^{s,\alpha}(M)$. Moreover, there exists $\eps>0$ such that $F(v)-F(0)\in\Hb^{s,\alpha}(M)$ depends continuously on $v\in\Hb^{s,\alpha}(M)$, $\|u-v\|_{\Hb^{s,\alpha}}<\eps$.
\end{prop}
\begin{proof}
  Observe that $u(M)$ is compact. Let $\gamma\subset\C$ denote a smooth contour which is disjoint from $u(M)$, has winding number $1$ around every point in $u(M)$, and lies within the region of holomorphicity of $F$. Then, writing $F(z)-F(0)=zF_1(z)$ with $F_1$ holomorphic in $\Omega$, we have
  \[
    F(u)-F(0)=\frac{u}{2\pi i}\oint_\gamma F_1(\zeta)\frac{1}{\zeta-u}\,d\zeta,
  \]
  Since $\gamma\ni\zeta\mapsto u/(\zeta-u)\in\Hb^{s,\alpha}(M)$ is continuous by Lemma~\ref{LemmaRecHs}, we obtain the desired conclusion $F(u)-F(0)\in\Hb^{s,\alpha}$.
  
  The continuous dependence of $F(v)-F(0)$ on $v$ near $u$ is a consequence of Lemma~\ref{LemmaRecHs} and Corollary~\ref{CorHbModule}.
\end{proof}

\begin{prop}
\label{PropCompWithAnalytic2}
  Let $u'\in\CI(M)$, $u''\in\Hb^{s,\alpha}(M)$; put $u=u'+u''$. If $F\colon\Omega\to\C$ is holomorphic in a simply connected neighborhood $\Omega$ of $u(M)$, then $F(u)\in\CI(M)+\Hb^{s,\alpha}(M)$; in fact, $F(v)$ depends continuously on $v$ in a neighborhood of $u$ in the topology of $\CI(M)+\Hb^{s,\alpha}(M)$.
\end{prop}
\begin{proof}
  Let $\gamma\subset\C$ denote a smooth contour which is disjoint from $u(M)$, has winding number $1$ around every point in $u(M)$, and lies within the region of holomorphicity of $F$. Since $u''=0$ at $\pa M$ and $u''$ is continuous by the Riemann-Lebesgue lemma, we can pick $\phi\in\CI(M)$, $\phi\equiv 1$ near $\pa M$, such that $\gamma$ is disjoint from $\overline{u'(\supp\phi)}$. Then
  \begin{align*}
    \phi F(u)&=\frac{1}{2\pi i}\oint_\gamma \phi\frac{F(\zeta)/(\zeta-u')}{1-u''/(\zeta-u')}\,d\zeta \\
	  &=\frac{1}{2\pi i}\oint_\gamma \phi\frac{F(\zeta)}{\zeta-u'}\,d\zeta+\frac{1}{2\pi i}\oint_\gamma \phi\frac{(F(\zeta)/(\zeta-u'))u''}{(\zeta-u')-u''}\,d\zeta;
  \end{align*}
  the first term equals $\phi F(u')$, and the second term is an element of $\Hb^{s,\alpha}$ by Corollary~\ref{CorReciprocal}. Next, let $\wt\phi\in\CI(M)$ be identically equal to $1$ on $\supp(1-\phi)$, and $\wt\phi\equiv 0$ near $\pa M$. Then $\wt\phi u\in\Hb^s$; in fact, it lies in any weighted such space. Thus,
  \[
    (1-\phi)F(u)=\frac{1}{2\pi i}\oint_\gamma\frac{(1-\phi)F(\zeta)}{\zeta-\wt\phi u}\,d\zeta \in\Hb^{s,\alpha},
  \]
  and the proof is complete.
\end{proof}

If we only consider $F(u)$ for real-valued $u$, it is in fact sufficient to assume $F\in\CI(\R;\C)$ using almost analytic extensions, see e.g.\ Dimassi and Sj\"ostrand \cite[Chapter~8]{DimassiSpectralAsymptotics}: For any such function $F$ and an integer $N\in\N$, let us define
\begin{equation}
\label{EqAlmostAnalytic}
  \wt F_N(x+iy)=\sum_{k=0}^N\frac{(iy)^k}{k!}(\pa_x^k F)(x)\chi(y),\quad x,y\in\R,
\end{equation}
where $\chi\in\CI_\cl(\R)$ is identically $1$ near $0$. Then, writing $z=x+iy$, we have for $y$ close to $0$:
\begin{equation}
\label{EqAlmostAnalyticBar}
  \pa_{\bar z}\wt F_N(z)=\frac{1}{2}(\pa_x+i\pa_y)\wt F_N(z)=\frac{(iy)^N}{2N!}(\pa_x^{N+1}F)(x)\chi(y)=\cO(|\Im z|^N).
\end{equation}
Observe that all $u\in\CI(M)+\Hb^{s,\alpha}(M)$ are bounded, hence in analyzing $F(u)$, we may assume without restriction that $F\in\CI_\cl(\R;\C)$.

\begin{prop}
\label{PropCompWithSmooth}
  Let $F\in\CI_\cl(\R;\C)$. Then for $u\in\Hb^{s,\alpha}(M;\R)$, we have $F(u)-F(0)\in\Hb^{s,\alpha}(M)$; in fact, $F(u)-F(0)$ depends continuously on $u$.
\end{prop}
\begin{proof}
  Write $F(x)-F(0)=xF_1(x)$. Then, with $(\wt F_1)_N$ defined as in~\eqref{EqAlmostAnalytic}, the Cauchy-Pompeiu formula gives the pointwise identity
  \[
    F(u)-F(0)=-\frac{u}{\pi}\int_\C\frac{\pa_{\bar\zeta}(\wt F_1)_N(\zeta)}{\zeta-u}\,dx\,dy,\quad \zeta=x+iy.
  \]
  Here, note that the integrand is compactly supported, and $1/(\zeta-u(z))$ is locally integrable for all $z$. In particular, we can rewrite
  \begin{equation}
  \label{EqFuF0Limit}
    F(u)-F(0)=-\frac{1}{\pi}\lim_{\delta\searrow 0}\int_{|\Im\zeta|>\delta}\pa_{\bar\zeta}(\wt F_1)_N(\zeta)\frac{u}{\zeta-u}\,dx\,dy.
  \end{equation}
  Now Lemma~\ref{LemmaRecHs} gives
  \[
    \Bigl\|\frac{u}{\zeta-u}\Bigr\|_{\Hb^{s,\alpha}}\lesssim C(\|u\|_{\Hb^{s,\alpha}})|\Im\zeta|^{-s-2},
  \]
  since $u$ is real-valued. Thus, if we choose $N\geq s+2$, then
  \[
    \C\setminus\R\ni\zeta\mapsto\pa_{\bar\zeta}(\wt F_1)_N(\zeta)\frac{u}{\zeta-u}\in\Hb^{s,\alpha}
  \]
  is bounded by~\eqref{EqAlmostAnalyticBar}, hence integrable, and therefore the limit in~\eqref{EqFuF0Limit} exists in $\Hb^{s,\alpha}$, proving the proposition.
\end{proof}

We also have an analogue of Proposition~\ref{PropCompWithAnalytic2}.
\begin{prop}
\label{PropCompWithSmooth2}
  Let $F\in\CI_\cl(\R;\C)$, and $u'\in\CI(M;\R),u''\in\Hb^{s,\alpha}(M;\R)$; put $u=u'+u''$. Then $F(u)\in\CI(M)+\Hb^{s,\alpha}(M)$; in fact, $F(u)$ depends continuously on $u$.
\end{prop}
\begin{proof}
  As in the proof of the previous proposition, we have the pointwise identity
  \begin{align*}
    F(u'&+u'')-F(u') \\
	  &=-\frac{1}{\pi}\lim_{\delta\searrow 0}\int_{|\Im\zeta|>\delta}\pa_{\bar\zeta}(\wt F_1)_N(\zeta)\biggl(\frac{1}{\zeta-u'-u''}-\frac{1}{\zeta-u'}\biggr)\,dx\,dy \\
	  &=-\frac{1}{\pi}\lim_{\delta\searrow 0}\int_{|\Im\zeta|>\delta}\frac{\pa_{\bar\zeta}(\wt F_1)_N(\zeta)}{\zeta-u'}\cdot\frac{u''}{(\zeta-u')-u''}\,dx\,dy
  \end{align*}
  Writing $f_N:=\pa_{\bar\zeta}(\wt F_1)_N$, we estimate the $\Hb^{s,\alpha}$-norm of the integrand for $\zeta\in\C\setminus\R$ using Lemma~\ref{LemmaRecHs} by
  \[
    \left\|\frac{f_N(\zeta)}{\zeta-u'}\right\|_{\cL(\Hb^{s,\alpha})} \left\|\frac{u''}{(\zeta-u')-u''}\right\|_{\Hb^{s,\alpha}} \lesssim \left\|\frac{f_N(\zeta)}{\zeta-u'}\right\|_{\cL(\Hb^{s,\alpha})} |\Im\zeta|^{-s-2};
  \]
  here, we denote by $\|h\|_{\cL(\Hb^{s,\alpha})}$, for a function $h$, the operator norm of multiplication by $h$ on $\Hb^{s,\alpha}$. We claim that the operator norm
  \[
    b_s:=\left\|\frac{f_N(\zeta)}{\zeta-u'}\right\|_{\cL(\Hb^{s,\alpha})}
  \]
  is bounded by $|\Im\zeta|^{N-s-1}$; then choosing $N\geq 2s+3$ finishes the proof as before. To prove this bound, we use interpolation: First, since $u'$ is real-valued, we have $b_0=\cO(|\Im\zeta|^{-1}|f_N(\zeta)|)=\cO(|\Im\zeta|^{N-1})$ by~\eqref{EqAlmostAnalyticBar}. Next, for integer $k\geq 1$, the Leibniz rule gives
  \[
    b_k\lesssim\sum_{j=0}^k |\Im\zeta|^{-1-j}|\pa_x^{k-j}f_N(\zeta)|\lesssim|\Im\zeta|^{N-k-1},
  \]
  where we use that $|\pa_x^\ell f_N(\zeta)|=\cO(|\Im\zeta|^N)$ for all $\ell$, as follows directly from the definition of $f_N$. By interpolation, we thus obtain $b_s\lesssim|\Im\zeta|^{N-s-1}$, as claimed.
\end{proof}

%%%%%%%%%%%%%%%%%%%%%%%%%%%%%%%%%%%%%%%%%%%%%%%%%%%%%%%%%%%%%%%%%%%%%
\section{Elliptic regularity}
\label{SecEllipticReg}

With the partial calculus developed in \S\ref{SecCalculus}, it is straightforward to prove elliptic regularity for b-Sobolev b-pseudodifferential operators. Notice that operators with coefficients in $\Hb^s$ for $s>n/2$ must vanish at the boundary by the Riemann-Lebesgue lemma, thus they cannot be elliptic there. A natural class of operators which can be elliptic at the boundary is obtained by adding smooth b-\psdo{}s to b-Sobolev b-\psdo{}s, and we will deal with such operators in the second part of the following theorem.

\begin{thm}
\label{ThmEllipticReg}
  Let $m,s,r\in\R$ and $\zeta_0\in\Sb^*\Rnhalf$. Suppose $\wt P=\wt P_m+\wt R$, where $\wt P_m\in\Hb^s\Psib^m(\Rnhalf;E,F)$ has principal symbol $\wt p$, and $\wt R\in\fpsib^{m-1;0}\Hb^{s-1}(\Rnhalf;E,F)$.
  \begin{enumerate}
    \item Let $P=\wt P$, and suppose $p\equiv\wt p$ is elliptic at $\zeta_0$, or
	\item let $P=P_0+\wt P$, where $P_0\in\Psib^m(\Rnhalf;E,F)$ has principal symbol $p_0$, and suppose $p=\wt p+p_0$ is elliptic at $\zeta_0$.
  \end{enumerate}
	Let $\wt s\in\R$ be such that $\wt s\leq s-1$ and $s>n/2+1+(-\wt s)_+$. Then in both cases, if $u\in\Hb^{\wt s+m-1,r}(\Rnhalf;E)$ satisfies
	\[
	  Pu=f\in\Hb^{\wt s,r}(\Rnhalf;F),
	\]
	it follows that $\zeta_0\notin\WFb^{\wt s+m,r}(u)$.
\end{thm}
\begin{proof}
  We will only prove the theorem without bundles; adding bundles only requires simple notational changes. In both cases, we can assume that $r=0$ by conjugating $P$ by $x^{-r}$; moreover, $\wt R u\in\Hb^{\wt s}$ by Proposition~\ref{PropOpCont} by the assumptions on $s$ and $\wt s$, thus we can absorb $\wt R u$ into the right hand side and hence assume $\wt R=0$. Choose $a_0\in S^0$ elliptic at $\zeta_0$ such that $p$ is elliptic on $\supp a_0$, and non-vanishing there, which only matters near the zero section.
  \begin{enumerate}[leftmargin=\enummargin]
    \item \label{EnumPfEllRegInterior} Let $\lambda_m$ be as in Corollary~\ref{CorBStdOp}. By Proposition~\ref{PropSymbolDivision},
	  \[
	    q(z,\zeta):=a_0(z,\zeta)\lambda_m(\zeta)/p(z,\zeta)\in S^{0;\infty}\Hb^s.
	  \]
	  Put $Q=q(z,\Db)$. Then by Theorem~\ref{ThmComp} \itref{ThmCompRR}, using $P=\wt P_m\in\fpsi^{m;0}\Hb^s$,
	  \[
	    Q\circ P=a_0(z,\Db)\Lambda_m+R'
	  \]
	  with $R'\in\fpsi^{m-1;0}\Hb^{s-1}$, hence by\footnote{For $Qf\in\Hb^{\wt s}$, we need $s\geq\wt s$ and $s>n/2+(-\wt s)_+$. For $R'u\in\Hb^{\wt s}$, we need $s-1\geq\wt s$ and $s-1>n/2+(-\wt s)_+$.} Proposition~\ref{PropOpCont}
	  \[
	    a_0(z,\Db)\Lambda_m u=Qf-R'u\in\Hb^{\wt s}.
	  \]
	  Then standard microlocal ellipticity implies $\zeta_0\notin\WFb^{\wt s+m}(u)$.

	\item If $\zeta_0\notin\Tb_{\pa \Rnhalf}^*\Rnhalf$, then the proof of part (\ref{EnumPfEllRegInterior}) applies, since away from $\pa \Rnhalf$, one has $\Psib^m\subset\Hb^s\Psib^m$. Thus, assuming $\zeta_0\in\Tb_{\pa\Rnhalf}^*\Rnhalf$, we note that the ellipticity of $p$ at $\zeta_0$ implies $p_0\neq 0$ near $\zeta_0$, since the function $\wt p$ vanishes at $\pa\Rnhalf$. Therefore, Proposition~\ref{PropSymbolDivision} applies if one chooses $a_0\in S^0$ as in the proof of part (\ref{EnumPfEllRegInterior}), yielding
	  \[
	    q(z,\zeta):=a_0(z,\zeta)\lambda_m(\zeta)/p(z,\zeta)=\wt q_0(z,\zeta)+q_0(z,\zeta),
	  \]
	  where $\wt q_0\in S^{0;\infty} \Hb^s$, $q_0\in S^0$. Put $\wt Q_0=\wt q_0(z,\Db),Q_0=q_0(z,\Db)$, then
	  \[
	    (\wt Q_0+Q_0)\circ(\wt P_m+P_0)=a_0(z,\Db)\Lambda_m+R'
	  \]
	  with
	  \begin{align*}
	    R'\in &\ \fpsi^{m-1;0}\Hb^{s-1} + \fpsi^{m-1;0}\Hb^s + \fpsi^{m-1;0}\Hb^{s-1} + \Psib^{m-1} \\
		  &\hspace{5ex}\subset \fpsi^{m-1;0}\Hb^{s-1} + \Psib^{m-1},
	  \end{align*}
	  where the terms are the remainders of the first order expansions of $\wt Q_0\circ\wt P_m$, $\wt Q_0\circ P_0$, $Q_0\circ\wt P_m$ and $Q_0\circ P_0$, in this order; to see this, we use Theorem~\ref{ThmComp} \itref{ThmCompRR}, \itref{ThmCompRB}, \itref{ThmCompBR} and composition properties of b-\psdo{}s, respectively. Hence
	  \[
	    a_0(z,\Db)\Lambda_m u=\wt Q_0 f+Q_0 f-R'u\in\Hb^{\wt s},
	  \]
	  which implies $\zeta_0\notin\WFb^{\wt s+m}(u)$.\qedhere
  \end{enumerate}
\end{proof}

\begin{rmk}
\label{RmkEllipticLocalAssm}
  Notice that it suffices to have only local $\Hb^{\wt s,r}$-membership of $f$ near the base point of $\zeta_0$. Under additional assumptions, even microlocal assumptions are enough, see in particular \cite[Theorem~3.1]{BealsReedMicroNonsmooth}; we will not need this generality though.
\end{rmk}

%%%%%%%%%%%%%%%%%%%%%%%%%%%%%%%%%%%%%%%%%%%%%%%%%%%%%%%%%%%%%%%%%%%%%
\section{Propagation of singularities}
\label{SecPropagation}

We next study the propagation of singularities (equivalently the propagation of regularity) for certain classes of non-smooth operators. The results cover operators that are of real principal type (\S\ref{SubsecRealPrType}) or have a specific radial point structure (\S\ref{SubsecRadialPoints}). For a microlocally more complete picture, we also include a brief discussion of complex absorption (\S\ref{SubsubsecComplexAbsorption}).

The statements of the theorems and the ideas of their proofs are (mostly) standard in the context of smooth pseudodifferential operators; see for example H\"ormander \cite{HormanderAnalysisPDE} and Vasy \cite{VasyMicroKerrdS} for statements on manifolds without boundary and Hassell, Melrose and Vasy \cite{HassellMelroseVasySymbolicOrderZero}, Baskin, Vasy and Wunsch \cite{BaskinVasyWunschRadMink} as well as \cite{HintzVasySemilinear} for the propagation of b-regularity near radial points in various settings. Beals and Reed \cite{BealsReedMicroNonsmooth} discuss the propagation of singularities on manifolds without boundary for non-smooth \psdo{}s, and parts of \S\S\ref{SubsecSharpGarding} and \ref{SubsecRealPrType} follow their exposition closely.

%%%%%%%%%%%%%%%%%%%%%%%%%%%%%%%%%%%%%%%%%%%%%%%%%
\subsection{Sharp G\aa rding inequalities}
\label{SubsecSharpGarding}

We will need various versions of the sharp G\aa rding inequality, which will be used to obtain one-sided bounds for certain terms in positive commutator arguments later. For the first result, we follow the proof of \cite[Lemma~3.1]{BealsReedMicroNonsmooth}.

\begin{prop}
\label{PropGardingNonsmooth}
  Let $s,m\in\R$ be such that $s\geq 2-m$ and $s>n/2+2+m_+$, where $m_+=\max(m,0)$. Let $p(z,\zeta)\in S^{2m+1;2}\Hb^s(\Rnhalf;\End(E))$ be a symbol with non-negative real part, i.e.
  \[
    \Re\la p(z,\zeta)e,e\ra\geq 0\quad z\in\Rnhalf,\zeta\in\R^n, e\in E,
  \]
  where $\la\cdot,\cdot\ra$ is the inner product on the fibers of $E$. Then there is $C>0$ such that $P=p(z,\Db)$ satisfies the estimate
  \[
    \Re\la Pu,u\ra \geq -C\|u\|_{\Hb^m}^2,\quad u\in\CIdotc(\Rnhalf;E).
  \]
\end{prop}
\begin{proof}
  Let $q\in\CI_\cl(\R^n)$ be a non-negative even function, supported in $|\zeta|\leq 1$, with $\int q^2(\zeta)\,d\zeta=1$, and put
  \[
    F(\zeta,\xi)=\frac{1}{\la\zeta\ra^{n/4}}q\left(\frac{\xi-\zeta}{\la\zeta\ra^{1/2}}\right).
  \]
  Define the symmetrization of $p$ to be
  \[
    p_\sym(\eta,z,\zeta)=\int F(\eta,\xi)p(z,\xi)F(\zeta,\xi)\,d\xi.
  \]
  Observe that the integrand has compact support in $\xi$ for all $\eta,z,\zeta$, therefore $p_\sym$ is well-defined. Moreover,
  \[
    (p_\sym(\Db,z,\Db)u)\ftrans(\eta)=\int\wh p_\sym(\eta,\eta-\zeta,\zeta)\wh u(\zeta)\,d\zeta,
  \]
  hence, writing $u=(u_j)$, $p=(p_{ij})$, $p_\sym=((p_\sym)_{ij})$, and summing over repeated indices,
  \begin{align*}
    \Re&\la p_\sym(\Db,z,\Db)u,u\ra=\Re\iint \wh p_\sym(\eta,\eta-\zeta,\zeta)_{ij}\wh u(\zeta)_j\overline{\wh u_i(\eta)}\,d\zeta\,d\eta \\
	  &=\Re\iint\left(\int e^{i z\zeta}F(\zeta,\xi)\wh u(\zeta)\,d\zeta\right)_j\overline{\left(\int e^{i z\eta}F(\eta,\xi)\wh u(\eta)\,d\eta\right)_i} p_{ij}(z,\xi)\,d\xi\,dz \\
	  &=\iint\Re\bigl\langle p(z,\xi)F(\Db;\xi)u(z),F(\Db;\xi)u(z)\bigr\rangle\,d\xi\,dz\geq 0.
  \end{align*}
  Thus, putting $r(z,\Db)=p_\sym(\Db,z,\Db)-p(z,\Db)$, it suffices to show that $r(z,\zeta)\in S^{2m;0}\Hb^{s-2}(\Rnhalf;\End(E))$, i.e.
  \begin{equation}
  \label{EqGDesiredEstimate}
    \frac{\la\eta\ra^{s-2}\|\wh r(\eta;\zeta)\|_{\End(E)}}{\la\zeta\ra^{2m}}\leq r_0(\eta;\zeta),\quad r_0(\eta;\zeta)\in L^\infty_\zeta L^2_\eta.
  \end{equation}
  in order to conclude the proof, since Proposition~\ref{PropOpCont} then implies the continuity of $r(z,\Db)\colon \Hb^m(\Rnhalf;E)\to \Hb^{-m}(\Rnhalf;E)$. \emph{From now on, we will suppress the bundle $E$ in our notation and simply write $|\cdot|$ for $\|\cdot\|_{\End(E)}$.} Now, $r(z,\Db)$ acts on $\CIdotc$ by
  \[
    (r(z,\Db)u)\ftrans(\eta)=\int\wh r(\eta-\zeta,\zeta)\wh u(\zeta)\,d\zeta;
  \]
  hence
  \begin{align}
    \wh r(\eta;\zeta)&=\wh p_\sym(\eta+\zeta,\eta,\zeta)-\wh p(\eta;\zeta) \nonumber\\
\label{EqGRemX} &=\int F(\eta+\zeta,\xi)\wh p(\eta;\xi)F(\zeta,\xi)\,d\xi-\wh p(\eta;\zeta) \\
    \begin{split}
	\label{EqGRemII}
	  &=\int F(\eta+\zeta,\xi)\bigl(\wh p(\eta;\xi)-\wh p(\eta;\zeta)\bigr)F(\zeta,\xi)\,d\xi \\
	  &\quad +\int \bigl(F(\eta+\zeta,\xi)-F(\zeta,\xi)\bigr)\wh p(\eta;\zeta)F(\zeta,\xi)\,d\xi,
	\end{split}
  \end{align}
  where we use $\int F(\zeta,\xi)^2\,d\xi=1$. To estimate $\wh r(\eta;\zeta)$, we use that
  \[
    |\wh p(\eta;\zeta)|=\frac{\la\zeta\ra^{2m+1}}{\la\eta\ra^s}p_0(\eta;\zeta),\quad p_0(\eta;\zeta)\in L^\infty_\zeta L^2_\eta.
  \]
  We get a first estimate from \eqref{EqGRemX}:
  \begin{align*}
    |\wh r(\eta;\zeta)|\lesssim \int_S \frac{1}{\la\eta+\zeta\ra^{n/4}\la\zeta\ra^{n/4}\la\eta\ra^s}\la\xi\ra^{2m+1}p_0(\eta;\xi)\,d\xi + \frac{\la\zeta\ra^{2m+1}}{\la\eta\ra^s}p_0(\eta;\zeta),
  \end{align*}
  where $S$ is the set
  \[
    S=\{|\xi-\zeta|\leq\la\zeta\ra^{1/2},|\xi-(\eta+\zeta)|\leq\la\eta+\zeta\ra^{1/2}\}.
  \]
  In particular, we have $\la\zeta\ra\sim\la\xi\ra\sim\la\eta+\zeta\ra$ on $S$, which yields
  \[
    |\wh r(\eta;\zeta)|\lesssim\frac{\la\zeta\ra^{2m+1-n/2}}{\la\eta\ra^s}\int_{|\xi-\zeta|\leq\la\zeta\ra^{1/2}}p_0(\eta;\xi)\,d\xi+\frac{\la\zeta\ra^{2m+1}}{\la\eta\ra^s}p_0(\eta;\zeta).
  \]
  We contend that
  \[
    p_0'(\eta;\zeta):=\la\zeta\ra^{-n/2}\int_{|\xi-\zeta|\leq\la\zeta\ra^{1/2}} p_0(\eta;\xi)\,d\xi\in L^\infty_\zeta L^2_\eta.
  \]
  Indeed, this follows from Cauchy-Schwartz:
  \begin{align*}
    \int\;\biggl| \int_{|\xi-\zeta|\leq\la\zeta\ra^{1/2}} p_0(\eta;\xi)\,d\xi\biggr|^2\,d\eta&\lesssim \int \la\zeta\ra^{n/2}\int_{|\xi-\zeta|\leq\la\zeta\ra^{1/2}}|p_0(\eta;\xi)|^2\,d\xi\,d\eta \\
	  &\lesssim \la\zeta\ra^n\|p_0(\eta;\xi)\|_{L^\infty_\xi L^2_\eta}^2.
  \end{align*}
  We deduce
  \[
    |\wh r(\eta;\zeta)|\leq\frac{\la\zeta\ra^{2m+1}}{\la\eta\ra^s}p_0''(\eta;\zeta),\quad p_0''(\eta;\zeta)\in L^\infty_\zeta L^2_\eta.
  \]
  If $|\eta|\geq |\zeta|/2$, this implies
  \begin{equation}
  \label{EqGFirstEstimate}
    \frac{\la\eta\ra^{s-1}|\wh r(\eta;\zeta)|}{\la\zeta\ra^{2m}}\leq\frac{\la\zeta\ra}{\la\eta\ra}p_0''(\eta;\zeta)\lesssim p_0''(\eta;\zeta),
  \end{equation}
  thus we obtain a forteriori the desired estimate~\eqref{EqGDesiredEstimate} in the region $|\eta|\geq|\zeta|/2$.

  From now on, let us thus assume $|\eta|\leq|\zeta|/2$. We estimate the first integral in \eqref{EqGRemII}. By Taylor's theorem,
  \begin{align*}
    \wh p(\eta;\xi)-\wh p(\eta;\zeta)&=\pa_\zeta\wh p(\eta;\zeta)\cdot(\xi-\zeta) \\
	  &\qquad+\int_0^1 (1-t)\la\xi-\zeta,\pa_\zeta^2\wh p(\eta;\zeta+t(\xi-\zeta))\cdot(\xi-\zeta)\ra\,dt,
  \end{align*}
  and since $\la\xi\ra\sim\la\zeta\ra$ on $\supp F(\zeta,\xi)$, this gives
  \[
    \wh p(\eta;\xi)-\wh p(\eta;\zeta)=\pa_\zeta\wh p(\eta;\zeta)\cdot(\xi-\zeta)+|\xi-\zeta|^2 \cO(\la\zeta\ra^{2m-1})\tn{ on }\supp F(\zeta,\xi),
  \]
  where we say $f\in\cO(g)$ if $|f|\leq |g|h$ for some $h\in L^\infty_\zeta L^2_\eta$. The first integral in \eqref{EqGRemII} can then be rewritten as
  \begin{align*}
    \pa_\zeta & \wh p(\eta;\zeta)\cdot\int(\xi-\zeta)\bigl(F(\eta+\zeta,\xi)-F(\zeta,\xi)\Bigr)F(\zeta,\xi)\,d\xi \\
	  &+\cO(\la\zeta\ra^{2m-1}) \int |\xi-\zeta|^2 F(\eta+\zeta,\xi)F(\zeta,\xi)\,d\xi,
  \end{align*}
  where we use $\int(\xi-\zeta)F(\zeta,\xi)^2\,d\xi=0$, which is a consequence of $q$ being even.

  Taking the second integral in \eqref{EqGRemII} into account, we obtain
  \begin{equation}
  \label{EqGSecondEstimate}
    |\wh r(\eta;\zeta)|\lesssim (M_1+M_2+M_3) p'''_0(\eta;\zeta), \quad p'''_0(\eta;\zeta)\in L^\infty_\zeta L^2_\eta,
  \end{equation}
  where
  \begin{align*}
    M_1(\eta,\zeta) &= \frac{\la\zeta\ra^{2m+1}}{\la\eta\ra^s}\int\frac{|\xi-\zeta|}{\la\zeta\ra} |F(\eta+\zeta,\xi)-F(\zeta,\xi)| F(\zeta,\xi)\,d\xi \\
	M_2(\eta,\zeta) &= \frac{\la\zeta\ra^{2m+1}}{\la\eta\ra^s}\int\frac{|\xi-\zeta|^2}{\la\zeta\ra^2} F(\eta+\zeta,\xi)F(\zeta,\xi)\,d\xi \\
	M_3(\eta,\zeta) &= \frac{\la\zeta\ra^{2m+1}}{\la\eta\ra^s}\left|\int\bigl(F(\eta+\zeta,\xi)-F(\zeta,\xi)\bigr) F(\zeta,\xi)\,d\xi\right|.
  \end{align*}

  $M_2$ is estimated easily: On the support of the integrand, one has $|\xi-\zeta|^2\leq\la\zeta\ra$, thus
  \[
    M_2(\eta,\zeta) \lesssim \frac{\la\zeta\ra^{2m}}{\la\eta\ra^s}\cdot\frac{\la\zeta\ra^{n/2}}{\la\eta+\zeta\ra^{n/4}\la\zeta\ra^{n/4}};
  \]
  here, the term $\la\zeta\ra^{n/2}$ in the numerator is (up to a constant) an upper bound for the volume of the domain of integration. Since we are assuming $|\eta|\leq|\zeta|/2$, we have $\la\eta+\zeta\ra\gtrsim\la\zeta\ra$, which gives $M_2(\eta,\zeta)\lesssim\la\zeta\ra^{2m}/\la\eta\ra^s$.

  In order to estimate $M_1$ and $M_3$, we will use
  \begin{align*}
    \pa_\zeta F(\zeta,\xi) &= \frac{a_0(\zeta)}{\la\zeta\ra^{n/4+1}}q_1\left(\frac{\xi-\zeta}{\la\zeta\ra^{1/2}}\right)+\frac{a_1(\zeta)}{\la\zeta\ra^{n/4+1/2}}\pa_\zeta q\left(\frac{\xi-\zeta}{\la\zeta\ra^{1/2}}\right), \\
	\pa_\zeta^2 F(\zeta,\xi) &= \frac{a_2(\zeta)}{\la\zeta\ra^{n/4+1}}q_2\left(\frac{\xi-\zeta}{\la\zeta\ra^{1/2}}\right),
  \end{align*}
  where the $a_j$ are scalar-, vector- or matrix-valued symbols of order $0$, and $q_j\in\CI_\cl(\R^n)$.

  Hence, writing $F(\eta+\zeta,\xi)-F(\zeta,\xi)=\eta\cdot\pa_\zeta F(\zeta+\bar t\eta,\xi)$ for some $0\leq\bar t\leq 1$, we get
  \[
    M_1(\eta,\zeta) \lesssim \frac{\la\zeta\ra^{2m+1}}{\la\eta\ra^s}\cdot\frac{\la\zeta\ra^{n/2} |\eta|}{\la\zeta\ra^{1/2}\la\zeta+\bar t\eta\ra^{n/4+1/2}\la\zeta\ra^{n/4}} \lesssim \frac{\la\zeta\ra^{2m}}{\la\eta\ra^{s-1}},
  \]
  where we again use $|\eta|<|\zeta|/2$ and $\la\zeta+\bar t\eta\ra\gtrsim\la\zeta\ra$.
  
  Finally, to bound $M_3$, we write
  \[
    F(\eta+\zeta,\xi)-F(\zeta,\xi)=\eta\cdot\pa_\zeta F(\zeta,\xi)+\int_0^1 (1-t)\la\eta,\pa_\zeta^2 F(\zeta+t\eta,\xi)\cdot\eta\ra\,dt
  \]
  and deduce
  \begin{align*}
    M_3(\eta,\zeta)&\lesssim\frac{\la\zeta\ra^{2m+1}}{\la\eta\ra^s}\biggl(\frac{\la\zeta\ra^{n/2} |\eta|}{\la\zeta\ra^{n/4+1}\la\zeta\ra^{n/4}} \\
	&\quad+\frac{|\eta|}{\la\zeta\ra^{n/4+1/2}}\left|\int(\pa_\zeta q)\left(\frac{\xi-\zeta}{\la\zeta\ra^{1/2}}\right)q\left(\frac{\xi-\zeta}{\la\zeta\ra^{1/2}}\right)\,d\xi\right| \\
	&\quad+\frac{\la\zeta\ra^{n/2} |\eta|^2}{\la\zeta\ra^{n/4+1}\la\zeta\ra^{n/4}}\biggr) \\
	&\lesssim\frac{\la\zeta\ra^{2m}}{\la\eta\ra^{s-2}},
  \end{align*}
  where we use
  \[
    \int(\pa_\zeta q)\left(\frac{\xi-\zeta}{\la\zeta\ra^{1/2}}\right) q\left(\frac{\xi-\zeta}{\la\zeta\ra^{1/2}}\right)\,d\xi=0,
  \]
  which holds since $q$ has compact support. Plugging the estimates for $M_j$, $j=1,2,3$, into \eqref{EqGSecondEstimate} proves that \eqref{EqGDesiredEstimate} holds. The proof is complete.
\end{proof}

The idea of the proof can also be used to prove the sharp G\aa rding inequality for smooth b-\psdo{}s:

\begin{prop}
\label{PropGardingSmooth}
  Let $m\in\R$, and let $p(z,\zeta)\in S^{2m+1}(\Rnhalf;\End(E))$ be a symbol with non-negative real part. Then there is $C>0$ such that $P=p(z,\Db)$ satisfies the estimate
  \[
    \Re\la Pu,u\ra \geq -C\|u\|_{\Hb^m}^2,\quad u\in\CIdotc(\Rnhalf;E).
  \]
\end{prop}
\begin{proof}
  Write $p(x,y;\zeta)=p^{(0)}(y;\zeta)+p^{(1)}(x,y;\zeta)$, where $p^{(0)}(y;\zeta)=p(0,y;\zeta)$ and $p^{(1)}=x\wt p\in\Hb^\infty S^{2m+1}$. The symmetrized operator $p(\Db,z,\Db)$, defined as in the proof of Proposition~\ref{PropGardingNonsmooth} is again non-negative, and the symbol of the remainder operator $r(z,\Db)=p_\sym(\Db,z,\Db)-p(z,\Db)$ is the sum of two terms $p^{(0)}_\sym-p^{(0)}$ and $p^{(1)}_\sym-p^{(1)}$. The proof of Proposition~\ref{PropGardingNonsmooth} shows that $p^{(1)}_\sym-p^{(1)}\in S^{2m;0}\Hb^\infty$. It thus suffices to assume that $p=p^{(0)}$ is independent of $x$, which implies that $p_\sym$ is independent of $x$ as well, and to prove $r(y,\Db)=(p_\sym-p)(y,\Db)\colon \Hb^m\to \Hb^{-m}$.
  
  Similarly to the proof of Proposition~\ref{PropGardingNonsmooth}, we put
  \begin{align*}
    F(\lambda,\eta;\sigma,\gamma)&=\frac{1}{\la\lambda,\eta\ra^{n/4}}q\left(\frac{(\sigma-\lambda,\gamma-\eta)}{\la\lambda,\eta\ra^{1/2}}\right) \\
    p_\sym(\rho,\theta;y;\lambda,\eta)&=\iint F(\rho,\theta;\sigma,\gamma)p(y;\sigma,\gamma)F(\lambda,\eta;\sigma,\gamma)\,d\sigma\,d\gamma
  \end{align*}
  and obtain
  \begin{align*}
	(p_\sym(\Db;y;\Db)u)\ftrans(\rho,\theta)&=\int\sF p_\sym(\rho,\theta;\theta-\eta;\rho,\eta)\wh u(\rho,\eta)\,d\eta \\
	\sF r(\theta;\lambda,\eta)&=\sF p_\sym(\lambda,\theta+\eta;\theta;\lambda,\eta)-\sF p(\theta;\lambda,\eta),
  \end{align*}
  thus
  \begin{align*}
	 \sF r&(\theta;\lambda,\eta) \\
	   &=\iint F(\lambda,\theta+\eta;\sigma,\gamma)\bigl(\sF p(\theta;\sigma,\gamma)-\sF p(\theta;\lambda,\eta)\bigr) F(\lambda,\eta;\sigma,\gamma)\,d\sigma\,d\gamma \\
	   &\ +\iint \bigl(F(\lambda,\theta+\eta;\sigma,\gamma)-F(\lambda,\eta;\sigma,\gamma)\bigr)\sF p(\theta;\lambda,\eta) F(\lambda,\eta;\sigma,\gamma)\,d\sigma\,d\gamma.
  \end{align*}
  Then, following the argument in the previous proof, we obtain
  \begin{equation}
  \label{EqGSRemainder}
    |\sF r(\theta;\lambda,\eta)|\leq\frac{\la\lambda,\eta\ra^{2m}}{\la\theta\ra^N}r_0(\theta;\lambda,\eta),\quad r_0(\theta;\lambda,\eta)\in L^\infty_{\lambda,\eta}L^2_\theta,
  \end{equation}
  where we use
  \[
    |\sF p(\theta;\lambda,\eta)|=\frac{\la\lambda,\eta\ra^{2m+1}}{\la\theta\ra^{N+2}}p_0(\theta;\lambda,\eta),\quad p_0(\theta;\lambda,\eta)\in L^\infty_{\lambda,\eta}L^2_\theta,
  \]
  which holds for every integer $N$ (with $p_0$ depending on the choice of $N$). An estimate similar to the one used in the proof of Proposition~\ref{PropOpCont} shows that \eqref{EqGSRemainder} implies $r(y,\Db)\colon\Hb^s\to\Hb^{s-2m}$ for all $s\in\R$.
\end{proof}

Finally, we merge Propositions~\ref{PropGardingNonsmooth} and \ref{PropGardingSmooth}.

\begin{cor}
\label{CorGardingSmoothPlusNon}
  Let $s,m\in\R$ be such that $s\geq 2-m$, $s>n/2+2+m_+$. Let $p'(z,\zeta)\in S^{2m+1;2}\Hb^s(\Rnhalf;\End(E))$, $p''(z,\zeta)\in S^{2m+1}(\Rnhalf;\End(E))$ be symbols such that $p=p'+p''$ has non-negative real part. Then there is $C>0$ such that $P=p(z,\Db)$ satisfies the estimate
  \[
    \Re\la Pu,u\ra \geq -C\|u\|_{\Hb^m}^2,\quad u\in\CIdotc(\Rnhalf;E).
  \]
\end{cor}
\begin{proof}
  The symmetrized operator $p_\sym(\Db,z,\Db)$ is again non-negative, and the symbol of the remainder operator $r(z,\Db)=p_\sym(\Db,z,\Db)-p(z,\Db)$ is the sum of two terms $p'_\sym-p'$ and $p''_\sym-p''$. The proofs of Propositions~\ref{PropGardingNonsmooth} and \ref{PropGardingSmooth} show that $(p'_\sym-p')(z,\Db)$ and $(p''_\sym-p'')(z,\Db)$ map $\Hb^m$ to $\Hb^{-m}$, hence $r(z,\Db)$ maps $\Hb^m$ to $\Hb^{-m}$, and the proof is complete.
\end{proof}

%%%%%%%%%%%%%%%%%%%%%%%%%%%%%%%%%%%%%%%%%%%%%%%%%
\subsection{Mollifiers}
\label{SubsecMollifiers}

In order to deal with certain kinds of non-smooth terms in \S\S\ref{SubsecRealPrType} and \ref{SubsecRadialPoints}, we will need smoothing operators in order to smooth out and approximate non-smooth functions in a precise way. We only state the results for unweighted spaces, but the corresponding statements for weighted spaces hold true by the same proofs.

\begin{lemma}
\label{LemmaLocNearBdy}
  Let $s\in\R$, $\chi\in\CI_\cl(\Rhalf)$. Then $\chi(x/\eps)\to 0$ strongly as a multiplication operator on $\Hb^s(\Rnhalf)$ as $\eps\to 0$, and in norm as a multiplication operator from $\Hb^{s,\alpha}(\Rnhalf)\to\Hb^s(\Rnhalf)$ for $\alpha>0$.
\end{lemma}
\begin{proof}
  We start with the first half of the lemma: For $s=0$, the statement follows from the dominated convergence theorem. For $s$ a positive integer, we use that
  \[
    (x\pa_x)^s\left(\chi\left(\frac{x}{\eps}\right)\right)=\sum_{j=1}^s c_{sj}\left(\frac{x}{\eps}\right)^j\chi^{(j)}\left(\frac{x}{\eps}\right),\quad c_{sj}\in\R,
  \]
  is bounded and converges to $0$ pointwise in $x>0$ as $\eps\to 0$, thus by virtue of the Leibniz rule and the dominated convergence theorem, we obtain $\chi(x/\eps)u(x,y)\to 0$ in $\Hb^s(\Rnhalf)$ for $u\in\Hb^s(\Rnhalf)$. For $s\in -\N$, the statement follows by duality.

  Finally, to treat the case of general $s$, we first show that $\chi(\cdot/\eps)$ is a uniformly bounded family (in $\eps>0$) of multiplication operators on $\Hb^s(\Rnhalf)$ for all $s\in\R$: For $s\in\N_0$, this follows from the above estimates, for $s\in\Z$ again by duality, and then for general $s\in\R$ by interpolation. Now, put $M=\sup_{0<\eps\leq 1}\|\chi(\cdot/\eps)\|_{\Hb^s\to\Hb^s}<\infty$. Let $w\in\Hb^s$ and $\delta>0$ be given, and choose $w'\in\Hb^\infty$ such that $\|w'-w\|_{\Hb^s}<\delta/2M$. By what we have already proved, we can choose $\eps_0>0$ so small that
  \[
    \|\chi(\cdot/\eps)w'\|_{\Hb^s}\leq\|\chi(\cdot/\eps)w'\|_{\Hb^{\lceil s\rceil}}<\delta/2,\quad \eps<\eps_0;
  \]
  then
  \[
    \|\chi(\cdot/\eps)w\|_{\Hb^s}\leq\|\chi(\cdot/\eps)(w-w')\|_{\Hb^s}+\|\chi(\cdot/\eps)w'\|_{\Hb^s}<M\frac{\delta}{2M}+\frac{\delta}{2}=\delta.
  \]

  Concerning the second half of the lemma, the case $s=0$ is clear since $x^\alpha\chi(x/\eps)\to 0$ in $L^\infty(\Rhalf)$ as $\eps\to 0$; as above, this implies the statement for $s$ a positive integer, and the case of real $s$ again follows by duality and interpolation.
\end{proof}

\begin{lemma}
\label{LemmaMollifier}
  Let $M$ be a compact manifold with boundary. Then there exists a family of operators $J_\eps\colon\CnI(M)\to\CIc(M^\circ)$, $\eps>0$, such that $J_\eps\in\Psib^{-\infty}(M)$, and for all $s,\alpha\in\R$, $J_\eps$ is a uniformly bounded family of operators on $\Hb^{s,\alpha}(M)$ that converges strongly to the identity map $I$ as $\eps\to 0$.
\end{lemma}
\begin{proof}
  Choosing a product decomposition $\pa M\times[0,\eps_0)_x$ near the boundary of $M$ and $\chi\in\CI_\cl(\R)$, $\chi\equiv 1$ near $0$, $\supp\chi\subset[0,1/2]$, we can define the multiplication operators $\chi(x/\eps)$ globally on $\Hb^{-\infty}(M)$. By the previous lemma, $I-\chi(\cdot/\eps)$ converges strongly to $I$ on $\Hb^s(M)$; moreover, $\supp(u-\chi(\cdot/\eps)u)\subset\{x\geq\eps\}$. Thus, if we let $\wt J_\eps$ be a family of mollifiers, $\wt J_\eps\in\Psib^{-\infty}(M)$, $\wt J_\eps\to I$ in $\Psib^{\delta'}(M)$ for $\delta'>0$, such that on the support of the Schwartz kernel of $\wt J_\eps$, we have $|x_1-x_2|<\eps/2$ near $\pa M\times\pa M$ where $x_1,x_2$ are the lifts of $x$ to the left and right factor of $M\times M$, then we have that $\wt J_\eps(u-\chi(\cdot/\eps)u)$ is an element of $\Hb^\infty(M)$ with support in $\{x\geq\eps/2\}$, thus is smooth. Therefore, the family $J_\eps:=\wt J_\eps\circ(I-\chi(\cdot/\eps))$ satisfies all requirements.
\end{proof}

%%%%%%%%%%%%%%%%%%%%%%%%%%%%%%%%%%%%%%%%%%%%%%%%%
\subsection{Real principal type propagation, complex absorption}
\label{SubsecRealPrType}

We will prove real principal type propagation estimates of b-regularity for operators with non-smooth coefficients by means of a positive commutator argument which is standard in the smooth coefficient case; we recall the argument below.

\begin{thm}
\label{ThmRealPrType}
  Let $m,r,s,\wt s\in\R$, $\alpha>0$. Suppose $\wt P=\wt P_m+\wt P_{m-1}+\wt R$, where $\wt P_m\in\Hb^{s,\alpha}\Psib^m(\Rnhalf;E)$ has a real, scalar, homogeneous principal symbol $\wt p_m$; moreover, let $\wt P_{m-1}\in\Hb^{s-1,\alpha}\Psib^{m-1}(\Rnhalf;E)$ and $\wt R\in\Psib^{m-2}(\Rnhalf;E)+\fpsib^{m-2;0}\Hb^{s-1,\alpha}(\Rnhalf;E)$. Suppose $s$ and $\wt s$ are such that
  \begin{equation}
  \label{EqRealPrTypeCond}
    \wt s\leq s-1, \quad s>n/2+7/2+(2-\wt s)_+.
  \end{equation}
  \begin{enumerate}
    \item \label{EnumRealInt} Let $P\equiv\wt P$ and $p\equiv\wt p_m$, or
	\item \label{EnumRealBdy} let $P=P_0+\wt P$, where $P_0\in\Psib^m(\Rnhalf;E)$ has a real, scalar, homogeneous principal symbol $p_0$. Denote $p=p_0+\wt p_m$.
  \end{enumerate}
  In both cases, if $u\in\Hb^{\wt s+m-3/2,r}(\Rnhalf;E)$ is such that $Pu\in\Hb^{\wt s,r}(\Rnhalf;E)$, then $\WFb^{\wt s+m-1,r}(u)$ is a union of maximally extended null-\-bi\-char\-ac\-te\-ris\-tics of $p$, i.e.\ of integral curves of the Hamilton vector field $H_p$ within the characteristic set $p^{-1}(0)\subset\Tb^*\Rnhalf\setminus o$.
\end{thm}

The proof, given in \S\S\ref{SubsubsecInterior} and \ref{SubsubsecBoundary}, in fact gives an estimate for the $\Hb^{\wt s+m-1,r}$-norm of $u$: Suppose $E,B,G\in\Psib^0$ are such that all forward or backward null-bicharacteristics from $\WFb'(B)$ reach the elliptic set of $E$ while remaining in the elliptic set of $G$, and $\psi\in\CIc(\Rnhalf)$ is identically $1$ on $\pi(\WFb'(B))$, where $\pi\colon\Tb^*\Rnhalf\to\Rnhalf$ is the projection, then
\begin{equation}
\label{EqPropEstimate}
\begin{split}
  \|Bu&\|_{\Hb^{\wt s+m-1,r}} \\
    &\leq C(\|GPu\|_{\Hb^{\wt s,r}}+\|Eu\|_{\Hb^{\wt s+m-1,r}}+\|\psi Pu\|_{\Hb^{\wt s-1,r}}+\|u\|_{\Hb^{\wt s+m-3/2,r}})
\end{split}
\end{equation}
in the sense that if all quantities on the right hand side are finite, then so is the left hand side, and the inequality holds. See Figure~\ref{FigProp} for an illustration. In particular, it suffices to have only microlocal $\Hb^{\wt s,r}$-membership of $Pu$ near the parts of null-bicharacteristics along which we want to propagate $\Hb^{\wt s+m-1,r}$-regularity of $u$. The term involving $\psi Pu$ comes from the local requirements for elliptic regularity, see Remark~\ref{RmkEllipticLocalAssm}. 

\begin{figure}[!ht]
  \includegraphics{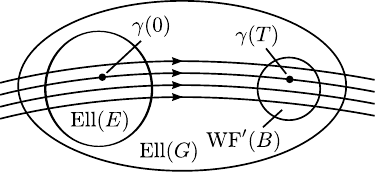}
  \caption{Illustration of the propagation estimate \eqref{EqPropEstimate}. Here, $\gamma$ is a null-bicharacteristic, i.e.\ an integral curve of the Hamilton vector field of the principal symbol of $P$.}
  \label{FigProp}
\end{figure}

In \S\ref{SubsubsecComplexAbsorption}, we will add complex absorption and obtain the following statement.

\begin{thm}
\label{ThmComplexAbs}
  Under the assumptions of Theorem~\ref{ThmRealPrType}, let $Q\in\Psib^m(\Rnhalf;E)$, $Q=Q^*$. Suppose $E,B,G\in\Psib^0$ are such that all forward, resp.\ backward, bicharacteristics from $\WFb'(B)$ reach the elliptic set of $E$ while remaining in the elliptic set of $G$, and suppose moreover that $q\leq 0$, resp.\ $q\geq 0$, on $\WFb'(G)$, further let $\psi\in\CIc(\Rnhalf)$ be identically $1$ on $\pi(\WFb'(B))$, then
  \begin{equation}
  \label{EqComplexEstimate}
    \begin{split}
      \|Bu\|_{\Hb^{\wt s+m-1,r}}\leq C(\|G(P-iQ)&u\|_{\Hb^{\wt s,r}}+\|Eu\|_{\Hb^{\wt s+m-1,r}} \\
	    &+\|\psi(P-iQ)u\|_{\Hb^{\wt s-1,r}}+\|u\|_{\Hb^{\wt s+m-3/2,r}})
	\end{split}
  \end{equation}
  in the sense that if all quantities on the right hand side are finite, then so is the left hand side, and the inequality holds.
\end{thm}

In other words, we can propagate estimates from the elliptic set of $E$ forward along the Hamilton flow to $\WFb'(B)$ if $q\geq 0$, and backward if $q\leq 0$.

Conjugating by $x^r$ (where $x$ is the standard boundary defining function), it suffices to prove Theorems~\ref{ThmRealPrType} and \ref{ThmComplexAbs} for $r=0$. Moreover, as in the smooth setting, we can apply Theorem~\ref{ThmEllipticReg} on the elliptic set of $P$ in both cases and deduce microlocal $\Hb^{\wt s+m}$-regularity of $u$ there, which implies that $\WFb^{\wt s+m-1}(u)$ is a subset of the characteristic set of $P$, and thus we only need to prove the propagation result within the characteristic set. We will begin by proving the first part of Theorem~\ref{ThmRealPrType} in \S\ref{SubsubsecInterior}; the proof is then easily modified in \S\ref{SubsubsecBoundary} to yield the second part of Theorem~\ref{ThmRealPrType}. To keep the notation simple, we will only consider the case of complex-valued symbols (hence, operators acting on functions); in the general, bundle-valued case, all arguments go through with purely notational changes.

Before commencing the proofs of the above theorems, we briefly sketch the proof of Theorem~\ref{ThmRealPrType} (for the weight $r=0$) in the \emph{smooth} setting using a positive commutator estimate in the form given by de Hoop, Uhlmann and Vasy \cite{DeHoopUhlmannVasyDiffraction}, which roughly goes as follows (omitting a number of `irrelevant' terms and glossing over the fact that the argument needs to be regularized in order to make sense of the appearing dual pairings): Suppose $\zeta\notin\WFb^\sigma(u)$, $\sigma=\wt s+m-1$; we want to propagate microlocal $\Hb^\sigma$-regularity of $u$ along a null-bicharacteristic strip $\gamma$ of $P$ from $\zeta$ to a nearby point $\zeta'\in\gamma$. To do so, we choose a symbol $a\in S^{2\sigma-m+1}$ with support localized near $\gamma$, which is decreasing along the Hamilton flow of $p$ except near $\zeta$ (where we have a priori information on $u$), i.e.\ $H_p a=-b^2+e$, where $e\in S^{2\sigma}$ is supported near $\zeta$, and $b\in S^{2\sigma}$ is elliptic near $\zeta'$. Then, denoting by $A,B,E$ formally self-adjoint quantizations of $a,b,e$, respectively, we obtain
\begin{equation}
\label{EqIntroPosComm}
  \|Bu\|_{L^2_\bl(M)}^2=\la B^*Bu,u\ra=\la Eu,u\ra-\la i[P,A]u,u\ra+\la Gu,u\ra,
\end{equation}
where $\la\cdot,\cdot\ra$ denotes the (sesquilinear) dual pairing on $L^2_\bl(M)$, and $G=B^*B-E+i[P,A]\in\Psib^{2\sigma-1}$. For simplicity, let us assume $u\in\Hb^{\sigma-1/2}(M)$ and $P=P^*$; then, expanding the commutator and integrating by parts to write $\la [P,A]u,u\ra=\la Au,Pu\ra-\la Pu,Au\ra$, and using that $Pu\in\Hb^{\sigma-m+1}$, moreover using that $\la Eu,u\ra$ is bounded by the regularity assumption on $u$ at $\zeta$, and using that $\la Gu,u\ra$ is bounded since $u$ is in $\Hb^{\sigma-1/2}$, we obtain $Bu\in L^2_\bl$. Hence by elliptic regularity, $u\in\Hb^\sigma$ microlocally near $\zeta'$, finishing the argument. Notice the loss of one derivative compared to the elliptic setting, which naturally comes about by the use of a commutator: We can only propagate $\Hb^\sigma$-regularity of $u$, not $\Hb^{\sigma+1}$-regularity, even though $Pu\in\Hb^{\sigma+1-m}$.

For the proof in the non-smooth setting, we will be choosing most operators in this argument ($A,B,E,G$ in the above notation) to be smooth ones and thus have to absorb certain non-smooth terms into an additional error term $F$ of symbolic order $2\sigma$, which without further caution would render the above argument invalid; by judiciously choosing $B$ and $E$, we can however ensure that the symbol of $F$ in fact has a sign, thus the additional term $\la Fu,u\ra$ appearing in \eqref{EqIntroPosComm} can be bounded by a version of the sharp G\aa rding inequality which we proved in \S\ref{SubsecSharpGarding}.

%%%%%%%%%%%%%%%%%%%%%%%%%%%%%
\subsubsection{Propagation in the interior}
\label{SubsubsecInterior}

For brevity, denote $M=\Rnhalf$. We start with the first half of Theorem~\ref{ThmRealPrType}, where we can in fact assume $\alpha=0$ since we are working away from the boundary, as explained below. Thus, let $P=P_m+P_{m-1}+R$, where we assume
\[
  m\geq 1
\]
for now,
\begin{align*}
  P_m&\in\Hb^s\Psib^m\tn{ with real homogeneous principal symbol}, \\
  P_{m-1}&\in\Hb^{s-1}\Psib^{m-1}, \\
  R&\in\fpsib^{m-2;0}\Hb^{s-1},
\end{align*}
and let us assume that we are given a solution
\begin{equation}
\label{EqUMembership}
  u\in\Hb^{\sigma-1/2},
\end{equation}
to the equation
\[
  Pu=f\in\Hb^{\sigma-m+1},
\]
where $\sigma=\wt s+m-1$ with $\wt s$ as in the statement of Theorem~\ref{ThmRealPrType}. In fact, since
\[
  R\colon\Hb^{\sigma-1/2}\subset\Hb^{\sigma-1}\to\Hb^{\sigma-m+1}
\]
by Proposition~\ref{PropOpCont},\footnote{\label{FootPCond1}We need $s-1\geq\sigma-m+1$ and $s-1>n/2+(m-\sigma-1)_+$.} we may absorb the term $Ru$ into the right hand side; thus, we can assume $R=0$, hence $P=P_m+P_{m-1}$.

Moreover, let $\gamma$ be a null-bicharacteristic of the principal symbol $p_m$ of $P_m$, and assume that $H_{p_m}$ is never radial on $\gamma$, i.e.\ that $H_{p_m}$ is linearly independent from the generator of dilations in the fibers of $\Tb^*M\setminus o$ -- recall that at radial points, the statement of the propagation of singularities is void. Note that the non-radiality in particular means that $\gamma\cap\Tb^*_{\pa M}M=\emptyset$ since $p_m$ vanishes identically at the boundary, and in fact this setup is the correct one for the discussion of real principal type propagation in the interior of $M$. \textit{All functions we construct in this section are implicitly assumed to have support away from $\pa M$.} Even though we are working away from the boundary, we will still employ the b-notation throughout this section, since the proof of the real principal type propagation result (near and) within the boundary will only require minor changes compared to the proof of the interior result given here.

The objective is to propagate microlocal $\Hb^{\sigma}$-regularity along $\gamma$ to a point $\zeta_0\in\Tb^*M\wozero$, assuming a priori knowledge of microlocal $\Hb^{\sigma}$-regularity of $u$ near a point $\zeta_*$ on the backward bicharacteristic from $\zeta_0$; the location and size of this region will be specified later, see Proposition~\ref{PropCommutantInterior}. We will use a positive commutator argument.

\textbf{Step 0. Outline of the symbolic construction of the commutant.} The idea, following \cite[\S{2}]{DeHoopUhlmannVasyDiffraction}, is to arrange for $\sfH_{p_m}=\rho^{1-m}H_{p_m}$, $\rho=\la\zeta\ra$,
\begin{equation}
\label{EqCommutatorPrimitive}
  \sfH_{p_m}\sfa=-\sfb^2+\sfe-\sff,
\end{equation}
where $\sfa,\sfb,\sfe$ are smooth symbols and $\sff$ is a non-smooth symbol, absorbing non-smooth terms of $\sfH_{p_m}\sfa$ in an appropriate way, which however has a definite sign; by virtue of the sharp G\aa rding inequality, we will be able to bound terms involving $\sff$ using the a priori regularity assumptions on $u$. As in the smooth case, terms involving $\sfe$ will be controlled by the a priori assumptions of $u$ near $\zeta_*$. If $\sfb$ is elliptic at $\zeta_0$, we are thus able to prove the desired $\Hb^{\sigma}$-regularity at $\zeta_0$. The actual commutant to be used, which has the correct symbolic order and is regularized, will be constructed later; see Proposition~\ref{PropCommutantInterior} for its relevant properties.

The general strategy for choosing the non-smooth symbol $\sff$ is as follows: Non-smooth terms $T$, which arise in the computation and are positive, say $T\geq c>0$, are smoothed out using a mollifier $J$, giving a smooth function $JT$, but only as much as to still preserve some positivity $JT-c/4\geq c/4>0$, and in such a way that the error $T-JT+c/4$ is non-negative; then $\sfb^2=JT-c/4$ is a smooth, positive term, and $\sff=T-JT+c/4$ is non-smooth, but has a sign, and $T=\sfb^2+\sff$. The mollifiers we shall use were constructed in Lemma~\ref{LemmaMollifier}.

\textbf{Step 1.1: Symbolic construction of the commutant on $\Sb^*M$.} To start, choose $\wt\eta\in\CI(\Sb^* M)$ with $\wt\eta(\zeta_0)=0$, $\sfH_{p_m}\wt\eta(\zeta_0)>0$, i.e.\ $\wt\eta$ measures, at least locally, propagation along the Hamilton flow. Choose $\sigma_j\in\CI(\Sb^* M)$, $j=1,\ldots,2n-2$, with $\sigma_j(\zeta_0)=0$ and $\sfH_{p_m}\sigma_j(\zeta_0)=0$, and such that $d\wt\eta,d\sigma_j$ span $T^*_{\zeta_0}(\Sb^* M)$. Put $\omega=\sum_{j=1}^{2n-2}\sigma_j^2$, so that $\omega^{1/2}$ approximately measures how far away one is from the bicharacteristic through $\zeta_0$. Thus, $|\wt\eta|+\omega^{1/2}$ is, near $\zeta_0$, equivalent to the distance from $\zeta_0$ with respect to any distance function given by a Riemannian metric on $\Sb^* M$. Then for $\delta\in(0,1),\eps\in(0,1],\beta\in(0,1]$ and $\F>0$ (large) to be chosen later, let
\[
  \phi=\wt\eta+\frac{1}{\eps^2\delta}\omega,
\]
and, taking $\chi_0(t)=e^{-1/t}$ for $t>0$, $\chi_0(t)=0$ for $t\leq 0$, and $\chi_1\in\CI(\R)$, $\chi_1\geq 0$, $\sqrt{\chi_1}\in\CI(\R)$, $\supp\chi_1\subset(0,\infty),\supp\chi_1'\subset(0,1)$, and $\chi_1\equiv 1$ in $[1,\infty)$, consider
\[
  \sfa=\chi_0\left(\F^{-1}\Bigl(2\beta-\frac{\phi}{\delta}\Bigr)\right)\chi_1\left(\frac{\wt\eta+\delta}{\eps\delta}+1\right).
\]
First, we observe that $\sfH_{p_m}\phi(\zeta_0)=\sfH_{p_m}\wt\eta(\zeta_0)>0$; but $\chi_1\left(\frac{\wt\eta+\delta}{\eps\delta}+1\right)\equiv 1$ near $\zeta_0$, so
\[
  \sfH_{p_m}\sfa(\zeta_0)=-\F^{-1}\delta^{-1}\sfH_{p_m}\phi(\zeta_0)\chi_0'(2\F^{-1}\beta)<0
\]
has the right sign at $\zeta_0$.

Next, we analyze the support of $\sfa$: First of all, If $\zeta\in\supp \sfa$, then
\[
  \phi(\zeta)\leq 2\beta\delta,\quad \wt\eta(\zeta)\geq -\delta-\eps\delta\geq -2\delta.
\]
Since $\omega\geq 0$, we get $\wt\eta=\phi-\omega/\eps^2\delta\leq\phi\leq 2\beta\delta\leq 2\delta$, thus $\omega=\eps^2\delta(\phi-\wt\eta)\leq 4\eps^2\delta^2$, i.e.
\begin{equation}
\label{EqCommLocalization1}
  -\delta-\eps\delta\leq\wt\eta\leq 2\beta\delta,\quad \omega^{1/2}\leq 2\eps\delta \quad \tn{on }\supp \sfa.
\end{equation}
In particular, we can make $\supp \sfa$ to be arbitrarily close to $\zeta_0$ by choosing $\delta>0$ small, hence there is $\delta_0>0$ small such that $\sfH_{p_m}\wt\eta\geq c_0>0$ whenever $|\wt\eta|\leq 2\delta_0$ and $\omega^{1/2}\leq 2\delta_0$. The support of $\sfa$ becomes localized near $\omega=0$ by choosing $\eps>0$ small. The parameter $\beta$ then allows one to localize $\supp\sfa$ near the segment $\wt\eta\in[-\delta;0]$. Moreover, we have
\begin{equation}
\label{EqCommLocalization2}
  -\delta-\eps\delta\leq\wt\eta\leq-\delta, \quad \omega^{1/2}\leq 2\eps\delta \quad \tn{on }\supp \sfa\cap\supp\chi_1',
\end{equation}
which is the region where we will assume a priori microlocal control on $u$. Observe that by taking $\eps>0$ small, we can make this region arbitrarily closely localized at $\wt\eta=-\delta$, $\omega=0$.

Choose $\wt\chi_1\in\CI(\R)$, $\wt\chi_1\geq 0$, such that $\wt\chi_1\equiv 1$ on $\supp\chi_1'$, and $\supp\wt\chi_1\subset[0,1]$. Since the coefficients of $\sfH_{p_m}$ are continuous because of $s>n/2+1$, we can choose a mollifier $J$ as in Lemma~\ref{LemmaMollifier}, acting on functions $f$ defined on $\Tb^*\Rnhalf$ by $(Jf)(z,\zeta)=J(f(\cdot,\zeta))(z)$, such that
\begin{align}
\label{EqAprioriControlTerms}
  \sfe&=\chi_0\left(\F^{-1}\left(2\beta-\frac{\phi}{\delta}\right)\right)(J\sfH_{p_m})\left(\chi_1\left(\frac{\wt\eta+\delta}{\eps\delta}+1\right)\right)+\wt\chi_1\left(\frac{\wt\eta+\delta}{\eps\delta}+1\right), \nonumber\\
  \sff' &=\chi_0\left(F^{-1}\left(2\beta-\frac{\phi}{\delta}\right)\right)\biggl[\wt\chi_1\left(\frac{\wt\eta+\delta}{\eps\delta}+1\right) \\
    &\hspace{20ex}+(J\sfH_{p_m}-\sfH_{p_m})\left(\chi_1\left(\frac{\wt\eta+\delta}{\eps\delta}+1\right)\right)\biggr], \nonumber
\end{align}
hence $\sfe-\sff'=\chi_0\sfH_{p_m}\chi_1$, we have $\sff'\geq 0$. Note that $\sfe\in\CI$ has support as indicated in \eqref{EqCommLocalization2}, and $\sff'\in\Hb^{s-1}$ in the base variables.

In order to have \eqref{EqCommutatorPrimitive}, it remains to prove that the remaining term of $\sfH_{p_m}\sfa$, namely $\chi_1 \sfH_{p_m}\chi_0$, is non-positive; for this, it is sufficient to require $\sfH_{p_m}\phi\geq c_0/2$ on $\supp \sfa$ if $\delta<\delta_0$. From the definition of $\phi$, this would follow provided
\begin{equation}
\label{EqStayClose}
  |\sfH_{p_m}\omega|\leq c_0\eps^2\delta/2
\end{equation}
on $\supp\sfa$. Now, since for $s>n/2+2$, $\sfH_{p_m}\sigma_j$ is Lipschitz continuous and vanishes at $\zeta_0$, we have
\begin{equation}
\label{EqHamiltonOmegaLipschitz}
  |\sfH_{p_m}\omega| \leq 2\sum_{j=1}^{2n-2}|\sigma_j||\sfH_{p_m}\sigma_j|\leq C\omega^{1/2}\left(|\wt\eta|+\omega^{1/2}\right),
\end{equation}
hence \eqref{EqStayClose} holds if $2C\eps\delta(2\delta+2\eps\delta)\leq c_0\eps^2\delta/2$, which is satisfied provided $16 C\delta/c_0\leq\eps$. Let us choose $\eps=16C\delta/c_0$, with $\delta$ small enough such that $\eps\leq 1$. For later use, let us note that then near $\wt\eta=-\delta$, the `width' of the support of $\sfa$ is
\begin{equation}
\label{EqAprioriControlSize}
  \omega^{1/2}\leq\frac{c_0\eps^2\delta/2}{C(\omega^{1/2}+|\wt\eta|)}\lesssim\delta^2,
\end{equation}
hence by \eqref{EqCommLocalization2}, the region where we will assume a priori microlocal control on $u$ (i.e.\ $\supp\sfe$) has size $\sim\delta^2$.

Now, let
\begin{align*}
  \sfb&=(\F\delta)^{-1/2}\sqrt{(J\sfH_{p_m})\phi-c_0/4}\sqrt{\chi_0'\left(\F^{-1}\left(2\beta-\frac{\phi}{\delta}\right)\right)}\sqrt{\chi_1\left(\frac{\wt\eta+\delta}{\eps\delta}+1\right)}, \\
  \sff''&=(\F\delta)^{-1}\left((\sfH_{p_m}-J\sfH_{p_m})\phi+c_0/4\right)\chi_0'\left(\F^{-1}\left(2\beta-\frac{\phi}{\delta}\right)\right)\chi_1\left(\frac{\wt\eta+\delta}{\eps\delta}+1\right),
\end{align*}
where $J$ is the same mollifier as used in \eqref{EqAprioriControlTerms}; we assume it is close enough to $I$ so that $|(\sfH_{p_m}-J\sfH_{p_m})\phi|<c_0/8$, which implies $(J\sfH_{p_m})\phi-c_0/4\geq c_0/8>0$ and $\sff''\geq 0$. Putting $\sff=\sff'+\sff''$, which is $\Hb^{s-1}$ in the base variables, we thus have achieved \eqref{EqCommutatorPrimitive}.

\textbf{Step 1.2: Incorporating the correct symbolic order into the commutant.} Next, we have to make the commutant, $a$, a symbol of order $2\sigma-(m-1)$, so that the `principal symbol' of $i[P,A]$, i.e.\ $H_{p_m}a$, is of order $2\sigma$, hence $b$ has order $\sigma$, which is what we need, since we want to prove $\Hb^{\sigma}$-regularity of $u$ at $\zeta_0$. Thus, define
\[
  \check a=\rho^{\sigma-(m-1)/2}\sfa^{1/2},
\]
and let
\begin{equation}
\label{EqRegularizer}
  \varphi_t=(1+t\rho)^{-1}
\end{equation}
be a regularizer, $\varphi_t\in S^{-1}$ for $t>0$, which is uniformly bounded in $S^0$ for $t\in[0,1]$ and satisfies $\varphi_t\to 1$ in $S^\ell$ for $\ell>0$ as $t\to 0$. We define the regularized symbols to be $\check a_t=\varphi_t\check a$ and $a_t=\varphi_t^2\rho^{2\sigma-(m-1)}\sfa=\check a_t^2$.

We compute $\sfH_{p_m}\varphi_t=-t\varphi_t^2\sfH_{p_m}\rho$. Amending \eqref{EqCommutatorPrimitive} by another term which will be used to absorb certain terms later on, we aim to show that we can choose $b_t,e_t$ and $f_t$ such that, in analogy to \eqref{EqCommutatorPrimitive}, for $N>0$ fixed, to be specified later,
\begin{align*}
  H_{p_m}a_t&=\varphi_t^2\rho^{2\sigma}\left(\sfH_{p_m}\sfa+\bigl((2\sigma-m+1)-2t\varphi_t\rho\bigr)(\rho^{-1}\sfH_{p_m}\rho)\sfa\right) \\
    &=-b_t^2-N^2\rho^{m-1}a_t+e_t-f_t,
\end{align*}
that is to say,
\begin{equation}
\label{EqCommutatorSophisticated}
  \varphi_t^2\rho^{2\sigma}\left(\sfH_{p_m}\sfa+\left[\bigl((2\sigma-m+1)-2t\varphi_t\rho\bigr)(\rho^{-1}\sfH_{p_m}\rho)+N^2\right]\sfa\right)=-b_t^2+e_t-f_t.
\end{equation}
Here, note that, using the definition of $\varphi_t$, $t\rho\varphi_t$ is a uniformly bounded family of symbols of order $0$. To achieve \eqref{EqCommutatorSophisticated}, let us take
\begin{equation}
\label{EqAprioriControlTerms2}
  \begin{split}
    e_t&=\varphi_t^2\rho^{2\sigma}\sfe \\
    f_t&=f'_t+f''_t, \quad f'_t=\varphi_t^2\rho^{2\sigma}\sff',
  \end{split}
\end{equation}
where $\sfe,\sff'$ are given by \eqref{EqAprioriControlTerms}; we will define $f''_t$ momentarily. Using $\chi_0(t)=t^2\chi_0'(t)$, we obtain
\begin{align*}
  \sfH_{p_m}\sfa&+\left[\bigl((2\sigma-m+1)-2t\varphi_t\rho\bigr)(\rho^{-1}\sfH_{p_m}\rho)+N^2\right]\sfa \\
    &=\sfe-\sff'-(\F\delta)^{-1}\biggl(\sfH_{p_m}\phi \\
	&\hspace{2ex}-\left[\bigl((2\sigma-m+1)-2t\varphi_t\rho\bigr)(\rho^{-1}\sfH_{p_m}\rho)+N^2\right]\F^{-1}\delta\Bigl(2\beta-\frac{\phi}{\delta}\Bigr)^2\biggr) \\
	&\hspace{15ex}\times\chi_0'\left(\F^{-1}\left(2\beta-\frac{\phi}{\delta}\right)\right)\chi_1\left(\frac{\wt\eta+\delta}{\eps\delta}+1\right)
\end{align*}
Thus, if $\F$ is large enough, the term in the large parentheses is bounded from below by $3c_0/8$ on $\supp\sfa$, since $|2\beta-\phi/\delta|\leq 4$ there. (The last statement follows from $-2\delta\leq\wt\eta\leq\phi\leq 2\beta\delta\leq 2\delta$ and $\beta\leq 1$.) Therefore, we can put
\begin{align}
\label{EqCommSymbolsBF}
  b_t&=(\F\delta)^{-1/2}\varphi_t\rho^{\sigma}\biggl((J\sfH_{p_m})\phi \\
    &\hspace{2ex}-\left[\bigl((2\sigma-m+1)-2t\varphi_t\rho\bigr)(\rho^{-1}(J\sfH_{p_m})\rho)+N^2\right] \nonumber\\
	&\hspace{35ex} \times \F^{-1}\delta\Bigl(2\beta-\frac{\phi}{\delta}\Bigr)^2-\frac{c_0}{8}\biggr)^{1/2} \nonumber\\
	&\hspace{15ex}\times\sqrt{\chi_0'\left(\F^{-1}\left(2\beta-\frac{\phi}{\delta}\right)\right)}\sqrt{\chi_1\left(\frac{\wt\eta+\delta}{\eps\delta}+1\right)}, \nonumber\\
  f''_t&=(\F\delta)^{-1}\varphi_t^2\rho^{2\sigma}\biggl((\sfH_{p_m}-J\sfH_{p_m})\phi \nonumber\\
    &\hspace{2ex}-\left[\bigl((2\sigma-m+1)-2t\varphi_t\rho\bigr)(\rho^{-1}(\sfH_{p_m}-J\sfH_{p_m})\rho)\right] \nonumber\\
	&\hspace{35ex} \times \F^{-1}\delta\Bigl(2\beta-\frac{\phi}{\delta}\Bigr)^2+\frac{c_0}{8}\biggr) \nonumber\\
	&\hspace{15ex}\times\chi_0'\left(\F^{-1}\left(2\beta-\frac{\phi}{\delta}\right)\right) \chi_1\left(\frac{\wt\eta+\delta}{\eps\delta}+1\right), \nonumber
\end{align}
with $f''_t\geq 0$ if the mollifier $J$ is close enough to $I$, and thus obtain \eqref{EqCommutatorSophisticated}.

We now summarize this construction, slightly rephrased, retaining only the important properties of the constructed symbols. Let us fix any Riemannian metric on $\Sb^*M$ near $\zeta_0$ and denote the metric ball around a point $p$ with radius $r$ in this metric by $B(p,r)$.

\begin{prop}
\label{PropCommutantInterior}
  There exist $\delta_0>0$ and $C_0>0$ such that for $0<\delta\leq\delta_0$, the following holds: For any $N>0$, there exist a symbol $\check a\in S^{\sigma-(m-1)/2}$ and uniformly bounded families of symbols $\check a_t=\varphi_t\check a\in S^{\sigma-(m-1)/2}$ (with $\varphi_t$ defined by \eqref{EqRegularizer}), $b_t\in S^{\sigma}$, $e_t\in S^{2\sigma}$ and $f_t\in S^{2\sigma;\infty}\Hb^{s-1}$, $f_t\geq 0$, supported in a coordinate neighborhood (independent of $\delta$) of $\zeta_0$ and supported away from $\pa M$, that satisfy the following properties:
  \begin{enumerate}
    \item $\check a_t H_{p_m}\check a_t=-b_t^2-N^2\rho^{m-1}\check a_t^2+e_t-f_t$.
	\item $b_t\to b_0$ in $S^{\sigma+\ell}$ for $\ell>0$, and $b_0$ is elliptic at $\zeta_0$.
	\item The support of $e_t$ is contained in $B(\zeta_0-\delta\sfH_{p_m}(\zeta_0),C_0\delta^2)$.
	\item For $t>0$, the symbols have lower order: $\check a_t\in S^{\sigma-(m-1)/2-1}$, $b_t\in S^{\sigma-1}$, $e_t\in S^{2\sigma-2}$ and $f_t\in S^{2\sigma-2;\infty}\Hb^{s-1}$.
  \end{enumerate}
\end{prop}

The commutant given by this proposition will now be used to deduce the propagation of regularity in a direction which agrees with the Hamilton flow to first order.

\textbf{Step 2. Expanding the commutator; bounding lower order terms.} Let $\check A\in\Psib^{\sigma-(m-1)/2}$ be a quantization of $\check a$ with $\WFb'(\check A)\subset\supp\check a$, let $\Phi_t$ be a quantization of $\varphi_t$, i.e.\ $\Phi_t\in\Psib^0$ is a uniformly bounded family, $\Phi_t\in\Psib^{-1}$ for $t>0$, and let $\check A_t=\check A\Phi_t$. Moreover, let $B_t\in\Psib^{\sigma}$ be a quantization of $b_t$, with uniform b-microsupport contained in a conic neighborhood of $\gamma$, such that $B_t\in\Psib^{\sigma}$ is uniformly bounded, and $B_t\in\Psib^{\sigma-1}$ for $t>0$. Similarly, let $E_t\in\Psib^{2\sigma}$ be a quantization of $e_t$ with uniform b-microsupport disjoint from $\WFb^\sigma(u)$ in the sense that
\begin{equation}
\label{EqInteriorAPrioriReg}
  \|E_t u\|_{\Hb^\sigma}\tn{ is uniformly bounded for }t>0.
\end{equation}
This is the requirement that $u$ is in $\Hb^{\sigma}$ on a part of the backwards bicharacteristic from $\zeta_0$, more precisely in the ball specified in Proposition~\ref{PropCommutantInterior}.

In a sense that we will make precise below, the principal symbol of the commutator $i\check A_t^*[P_m,\check A_t]$ is given by $\check a_t H_{p_m}\check a_t$, which is what we described in Proposition~\ref{PropCommutantInterior}. We compute for $t>0$, following the proof of \cite[Theorem~3.2]{BealsReedMicroNonsmooth}:
\begin{align}
 \Re\la i\check A_t^*&[P_m,\check A_t]u,u\ra = \Re\bigl(\la i P_m \check A_t u,\check A_t u\ra - \la i\check A_t P_m u,\check A_t u\ra\bigr) \nonumber\\
\label{EqIBP}&=\frac{1}{2}\la i(P_m-P_m^*)\check A_t u,\check A_t u\ra - \Re\la i\check A_t f,\check A_t u\ra+\Re\la i\check A_t P_{m-1} u,\check A_t u\ra,
\end{align}
where $\la\cdot,\cdot\ra$ denotes the sesquilinear pairing between spaces which are dual to each other relative to $L^2_\bl$. The adjoints here are taken with respect to the b-density $\frac{dx}{x}\,dy$, and in the case where $P$ acts on a vector bundle, we use the smooth metric in the fibers of $E$ for the adjoint. This computation needs to be justified, namely we must check that all pairings are well-defined by the a priori assumptions on $u$ so that we can perform the integrations by parts.

First, we observe that
\[
  \check A_t^*\check A_t P_m u\in \check A_t^*\check A_t \Hb^s\cdot \Hb^{\sigma-m-1/2}\subset \Hb^{-\sigma+1/2},
\]
because of $s\geq |\sigma-m-1/2|$ and $\check A_t^*\check A_t\in\Psib^{2\sigma-m-1}$. Since $(\sigma-1/2)+(-\sigma+1/2)=0$ is non-negative, the pairing $\la\check A_t^*\check A_t P_m u,u\ra$ is well-defined. By the same token, the pairing $\la\check A_t P_m u,\check A_t u\ra$ is well-defined, hence we can integrate by parts, justifying half of the first equality in \eqref{EqIBP}. For the second half of the first equality, we use $P_m\in\Hb^s\Psib^m$ and\footnote{\label{FootPCond2}This requires $s\geq m/2$; recall that we are assuming $m\geq 1$.} Corollary~\ref{CorHbModule} to obtain
\begin{gather*}
  P_m\check A_t u \in P_m \Hb^{m/2} \subset \Hb^{-m/2}, \\
  \check A_t^* P_m\check A_t u\in\Hb^{-\sigma+1/2},
\end{gather*}
which by the same reasoning as above proves the first equality in \eqref{EqIBP}. For the second equality, we write $P_m$ as a sum of terms of the form $w Q_m$ with $w\in\Hb^s$, $Q_m\in\Psib^m$, for which we have
\begin{equation}
\label{EqPairing}
  \la \check A_t u, w Q_m \check A_t u \ra=\la \bar w\check A_t u, Q_m\check A_t u\ra=\la Q_m^* \bar w\check A_t u,\check A_t u\ra,
\end{equation}
where the first equality follows from $\check A_t u\in\Hb^{m/2}$ and $Q_m\check A_t u\in\Hb^{-m/2}$,\footnote{\label{FootPCond3}We need $s\geq m/2$ and can then use Corollary~\ref{CorHbModule}.} and for the second equality, one observes that the two pairings on the right hand side in \eqref{EqPairing} are well-defined, and we can integrate by parts, i.e.\ move $Q_m$ to the other side, taking its adjoint.

Now, since the principal symbol of $P_m$ is real, we can apply Theorem~\ref{ThmComp} \itref{ThmCompBRFunny} with $k=1,k'=0$ to obtain $P_m-P_m^*\in\Psib^{m-1}\circ\fpsi^{0;0}\Hb^{s-1}+\fpsi^{m-1;0}\Hb^{s-1}$. Therefore, Proposition~\ref{PropOpCont} implies that $P_m-P_m^*$ defines a continuous map from $\Hb^{(m-1)/2}$ to $\Hb^{-(m-1)/2}$,\footnote{\label{FootPCond4}Provided $s-1\geq(m-1)/2$ and $s-1>n/2+(m-1)/2$.} thus
\[
  |\la(P_m-P_m^*)\check A_t u,\check A_t u\ra|\leq C_1\|\check A_t u\|_{\Hb^{(m-1)/2}}^2
\]
with a constant $C_1$ only depending on $P_m$.

Looking at the next term in \eqref{EqIBP}, we estimate
\[
  |\la\check A_t f,\check A_t u\ra|\leq \frac{1}{4}\|\check A_t f\|_{\Hb^{-(m-1)/2}}^2+\|\check A_t u\|_{\Hb^{(m-1)/2}}^2\leq C_2+\|\check A_t u\|_{\Hb^{(m-1)/2}}^2,
\]
where we use that
\[
  \check A_t f\in\Hb^{\sigma-m+1-\sigma+(m-1)/2}=\Hb^{-(m-1)/2}
\]
uniformly.

For the last term on the right hand side of \eqref{EqIBP}, the well-definedness is easily checked.\footnote{\label{FootPCond5}We need $s-1\geq|\sigma-m+1/2|$ and can then use Corollary~\ref{CorHbModule} to obtain $P_{m-1}u\in\Hb^{\sigma-m+1/2}$.} To bound it, we rewrite it as
\[
  \la\check A_t P_{m-1} u,\check A_t u\ra=\la P_{m-1}\check A_t u,\check A_t u\ra+\la[\check A_t,P_{m-1}]u,\check A_t u\ra.
\]
The first term on the right hand side is bounded by $C_3\|\check A_t u\|_{\Hb^{(m-1)/2}}^2$ for some constant $C_3$ only depending on $P_{m-1}$; indeed, $P_{m-1}\colon\Hb^{(m-1)/2}\to\Hb^{-(m-1)/2}$ is continuous.\footnote{\label{FootPCond6}This requires $s-1\geq(m-1)/2$ and $s-1>n/2$.} For the second term, note that $P_{m-1}\check A_t\in\Hb^{s-1}\Psib^{\sigma+(m-1)/2}$ can be expanded to zeroth order, the first (and only) term being $p_{m-1}\check a_t$ and the remainder being $R_1'\in\Hb^{s-1}\Psib^{\sigma+(m-1)/2-1}$. (For notational convenience, we drop the explicit $t$-dependence here; inclusions are understood to be statements about a $t$-dependent family of operators being uniformly bounded in the respective space.) Next, we can expand $\check A_t P_{m-1}$ to zeroth order by Theorem~\ref{ThmComp} \itref{ThmCompBRFunny} with $k'=0$ -- again obtaining $p_{m-1}\check a_t$ as the first term -- which\footnote{\label{FootPCond7}Assuming $\sigma-(m-1)/2\geq 1$.} yields a remainder term $R_1''+R_2$, where
\begin{align*}
  R_1'' & \in \fpsi^{\sigma+(m-1)/2-1;0}\Hb^{s-2} \\
  R_2 & \in \Psib^{\sigma-(m-1)/2-1}\circ\fpsi^{m-1;0}\Hb^{s-2}.
\end{align*}
We can then use Proposition~\ref{PropOpCont} to conclude that
\[
  R_1:=R_1''-R_1'\in\fpsi^{\sigma+(m-1)/2-1;0}\Hb^{s-2}
\]
is a uniformly bounded family of maps\footnote{\label{FootPCond8}The requirements are $s-2\geq -m/2+1$, $s-2>n/2+(m/2-1)_+$.}
\[
  R_1\colon \Hb^{\sigma-1/2}\to \Hb^{-m/2+1}.
\]
which shows that $\la R_1 u,\check A_t u\ra$ is uniformly bounded. Moreover, we can apply Proposition~\ref{PropOpCont} and use the mapping properties of smooth b-\psdo{}s to prove that $R_2 u\in\Hb^{-(m-1)/2}$ is uniformly bounded.\footnote{\label{FootPCond85}Indeed, we have $u\in\Hb^{\sigma-1/2}\subset\Hb^{\sigma-1}$, and $\fpsi^{m-1;0}\Hb^{s-2}\colon\Hb^{\sigma-1}\to\Hb^{\sigma-m}$ is continuous if $s-2\geq\sigma-m$, $s-2>n/2+(m-\sigma)_+$.} We thus conclude that
\[
  |\la[\check A_t,P_{m-1}]u,\check A_t u\ra|\leq C_4(N)+\|\check A_t u\|_{\Hb^{(m-1)/2}}^2,
\]
where $C_4$, while it depends on $N$ in the sense that it depends on a seminorm of the $N$-dependent operator $\check A$ constructed in Proposition~\ref{PropCommutantInterior}, is independent of $t$.

Plugging all these estimates into \eqref{EqIBP}, we thus obtain
\[
  \Re\la i\check A_t^*[P_m,\check A_t]u,u\ra \geq -(C_2+C_4(N))-(C_1+1+C_3+1)\|\check A_t u\|_{\Hb^{(m-1)/2}}^2,
\]
where all constants are independent of $t>0$, and $C_1,C_2,C_3$ are in addition independent of the real number $N$ in Proposition~\ref{PropCommutantInterior}. Choosing $N^2>C_1+C_3+2$, this implies that there is a constant $C<\infty$ such that for all $t>0$, we have
\begin{equation}
\label{EqAbsorbAtu}
  \Re\left\la\Bigl(i\check A_t^*[P_m,\check A_t]+N^2(\Lambda\check A_t)^*(\Lambda\check A_t)\Bigr)u,u\right\ra\geq -C,
\end{equation}
where $\Lambda:=\Lambda_{(m-1)/2}$. Therefore,
\begin{equation}
\label{EqOperatorCommutator}
  \Re\left\la\Bigl(i\check A_t^*[P_m,\check A_t]+B_t^*B_t+N^2(\Lambda\check A_t)^*(\Lambda\check A_t)-E_t\Bigr)u,u\right\ra \geq -C+\|B_t u\|_{L^2_\bl}^2.
\end{equation}
Here, we use that $\la E_t u,u\ra$ is uniformly bounded by \eqref{EqInteriorAPrioriReg}.

\textbf{Step 3. Using the symbolic commutator calculation.} The next step is to exploit the commutator relation in Proposition~\ref{PropCommutantInterior} in order to find a $t$-independent upper bound for the left hand side of \eqref{EqOperatorCommutator}. The key is that\footnote{\label{FootPCond9}Applicable with $k=2,k'=0$ if $\sigma-(m-1)/2\geq 2$.} Theorem~\ref{ThmComp} \itref{ThmCompBRFunny} gives
\[
  i[P_m,\check A_t]=(H_{p_m}\check a_t)(z,\Db)+\wt R_1+\wt R_2
\]
with uniformly bounded families of operators
\begin{align*}
  \wt R_1 & \in \fpsi^{\sigma+(m-1)/2-1;0}\Hb^{s-2} \\
  \wt R_2 & \in \Psib^{\sigma-(m-1)/2-2}\circ\fpsi^{m;0}\Hb^{s-2}.
\end{align*}
Notice that $H_{p_m}\check a_t\in\Hb^{s-1}S^{\sigma+(m-1)/2}$ uniformly. If we applied Theorem~\ref{ThmComp} \itref{ThmCompBRFunny} directly to the composition $\check A_t^*(H_{p_m}\check a_t)(z,\Db)$, the regularity of the remainder operator, say $R$, obtained by applying Theorem~\ref{ThmComp} \itref{ThmCompBRFunny}, would be too weak in the sense that we could not bound $\la Ru,u\ra$. To get around this difficulty, choose
\[
  J^+ \in \Psib^{\sigma-(m-1)/2-1}, \quad J^- \in \Psib^{-\sigma+(m-1)/2+1}
\]
with real principal symbols $j^+,j^-$ such that
\begin{equation}
\label{EqJpmLocalizer}
  J^+ J^- = I + \wt R,\quad \wt R\in\Psib^{-\infty}.
\end{equation}
Observe that $J^-\check A_t^*$ is uniformly bounded in $\Psib^1$. Then by Theorem~\ref{ThmComp} \itref{ThmCompBRFunny},
\[
  i J^-\check A_t^*[P_m,\check A_t] = (j^-\check a_t H_{p_m}\check a_t)(z,\Db)+R_1+R_2+R_3+R_4,
\]
where
\begin{equation}
\label{EqRem1234}
  \begin{split}
    R_1=J^-\check A_t^*&\wt R_1 \in \Psib^1\circ\fpsi^{\sigma+(m-1)/2-1;0}\Hb^{s-2} \\
	R_2=J^-\check A_t^*&\wt R_2 \in \Psib^{\sigma-(m-1)/2-1}\circ\fpsi^{m;0}\Hb^{s-2} \\
	R_3 & \in \fpsib^{\sigma+(m-1)/2;0}\Hb^{s-2} \\
	R_4 & \in \Psib^0\circ\fpsi^{\sigma+(m-1)/2;0}\Hb^{s-2}.
  \end{split}
\end{equation}
Applying Proposition~\ref{PropOpCont},\footnote{\label{FootPCond10}The conditions $s-2\geq-m/2+1$ and $s-2>n/2+m/2$ are sufficient to treat $R_1,R_3$ and $R_4$. For $R_2$, we need $s-2\geq\sigma-m-1/2$ and $s-2>n/2+(m+1/2-\sigma)_+$.} we conclude that $R_j$ ($1\leq j\leq 4$) is a uniformly bounded family of operators
\[
  \Hb^{\sigma-1/2}\to \Hb^{-m/2},
\]
thus, since $(J^+)^*\in\Hb^{m/2}$, the pairings $\la R_j u,(J^+)^*u\ra$ are uniformly bounded.

Hence, Proposition~\ref{PropCommutantInterior} implies
\begin{align}
  &J^+\Bigl(i J^-\check A_t^*[P_m,\check A_t]+J^-B_t^*B_t+J^-N^2(\Lambda\check A_t)^*(\Lambda\check A_t)-J^-E_t\Bigr) \nonumber\\
   &=J^+\Bigl([j^-(\check a_t H_{p_m}\check a_t+b_t^2+N^2\rho^{m-1}\check a_t^2-e_t)](z,\Db)+R+G\Bigr) \nonumber\\
\label{EqOperatorCommutatorInProof}&=J^+\bigl((-j^-f_t)(z,\Db)+R+G\bigr),
\end{align}
where $R=R_1+R_2+R_3+R_4$ and $G\in\Psib^{\sigma+(m-1)/2}$; $G$ appears because the principal symbols of the smooth operators on both sides are equal. We already proved that $\la J^+ Ru,u\ra$ is uniformly bounded; also, $\la J^+Gu,u\ra$ is uniformly bounded, since $J^+G\in\Psib^{2\sigma-1}$ and $u\in\Hb^{\sigma-1/2}$.

It remains to prove a uniform lower bound on\footnote{To justify the integration by parts here, note that $j^-f_t\in S^{\sigma+(m-1)/2-1;\infty}\Hb^{s-1}$ for $t>0$, thus $(j^-f_t)(z,\Db)u\in\Hb^{-m/2+1}$ provided $s-1\geq-m/2+1$, $s-1>n/2+(m/2-1)_+$, which follows from the conditions in Footnote~\ref{FootPCond10}.}
\[
  \Re\la J^+(j^- f_t)(z,\Db)u,u\ra=\Re\la (j^- f_t)(z,\Db)u,(J^+)^*u\ra.
\]
In order to be able to apply the sharp G\aa rding inequality, Proposition~\ref{PropGardingNonsmooth}, we need to rewrite this. Since $j^+$ is bounded away from $0$, we can write
\[
  (j^- f_t)(z,\Db)=\left[\frac{j^- f_t}{j^+}\right](z,\Db)\circ (J^+)^*+R,\quad R\in\fpsi^{\sigma+(m-1)/2;0}\Hb^{s-1}
\]
by Theorem~\ref{ThmComp} \itref{ThmCompRB}, since $j^-f_t/j^+\in S^{m+1;\infty}\Hb^{s-1}$. Now $\la Ru,(J^+)^*u\ra$ is uniformly bounded, since $(J^+)^*u\in\Hb^{m/2}$ and $Ru\in\Hb^{-m/2}$ are uniformly bounded.\footnote{For $Ru$, we need $s-1>n/2+m/2$, which follows from the conditions in Footnote~\ref{FootPCond10}.} We can now apply the sharp G\aa rding inequality to deduce that
\begin{equation}
\label{EqWithGarding}
  \Re\left\la\left[\frac{j^- f_t}{j^+}\right](z,\Db) (J^+)^*u, (J^+)^* u\right\ra \geq -C\|(J^+)^*u\|_{\Hb^{m/2}}^2\geq -C,
\end{equation}
where the constant $C$ only depends on the uniform $S^{2\sigma;\infty}\Hb^{s-1}$-bounds on $f_t$ and the $\Hb^{\sigma-1/2}$-norm of $u$.\footnote{\label{FootPCond12}This requires $s-1\geq 2-m/2$ and $s-1>n/2+2+m/2$.}

Putting \eqref{EqOperatorCommutator}, \eqref{EqOperatorCommutatorInProof} and \eqref{EqWithGarding} together by inserting $I=J^+J^--\wt R$ in front of the large parenthesis in \eqref{EqOperatorCommutator} and observing that the error term
\[
  \Re\left\la \wt R\Bigl(i\check A_t^*[P_m,\check A_t]+B_t^*B_t+N^2(\Lambda\check A_t)^*(\Lambda\check A_t)-E_t\Bigr)u,u\right\ra
\]
is uniformly bounded,\footnote{Indeed, $\check A_t^*\check A_t P_m u\in\Hb^{-\sigma-3/2}$ is uniformly bounded because of $s\geq|\sigma-m-1/2|$; and $\check A_t u\in\Hb^{m/2-1}$ is uniformly bounded, hence so is $P_m\check A_t u\in\Hb^{-m/2-1}$ in view of $s\geq m/2+1$, which follows from the condition in Footnote~\ref{FootPCond6}, and therefore $\check A_t^*P_m\check A_tu\in\Hb^{-\sigma-3/2}$ is uniformly bounded.} we deduce that $\|B_t u\|_{L^2_\bl}$ is uniformly bounded for $t>0$. Therefore, a subsequence $B_{t_k}u$, $t_k\to 0$, converges weakly to $v\in L^2_\bl$ as $k\to\infty$. On the other hand, $B_{t_k}u\to Bu$ in $\Hb^{-\infty}$; hence $Bu=v\in L^2_\bl$, which implies that $u\in\Hb^{\sigma}$ microlocally on the elliptic set of $B$.

\textbf{Step 4. Removing the restriction on the order $m$.} To eliminate the assumption that $m\geq 1$, notice that the above propagation estimate for a general $m$-th order operator can be deduced from the $m_0$-th order result for any $m_0\geq 1$, simply by considering
\[
  P\Lambda^+(\Lambda^- u)=f+PRu,
\]
where $\Lambda^+\in\Psib^{-(m-m_0)}$ is elliptic with parametrix $\Lambda^-\in\Psib^{m-m_0}$, and $\Lambda^+\Lambda^-=I+R$, $R\in\Psib^{-\infty}$. If we pass from $P$ to $P\Lambda^+$, which means passing from $m$ to $m_0$, we correspondingly have to pass from $\sigma$ to $\sigma_0=\sigma-m+m_0$ in equation~\eqref{EqUMembership}; in other words, the difference $\sigma-m=\sigma_0-m_0$ remains the same.

\textbf{Step 5. Collecting the regularity requirements.} Thus, let us collect the conditions on $s$ and $\wt s=\sigma-m+1$ as given in the footnotes in the course of the argument: All conditions are satisfied provided
\begin{gather}
\label{EqRealPrTypeConditionsProof}
  3/2-s\leq\wt s\leq s-1, \quad \wt s\geq(5-m_0)/2, \\
  s>n/2+2+(3/2-\wt s)_+,\quad s>n/2+3+m_0/2
\end{gather}
for some $m_0\geq 1$. The optimal choice for $m_0$ is thus $m_0=\max(1,5-2\wt s)=1+2(2-\wt s)_+$; plugging this in, we obtain the conditions in the statement of Theorem~\ref{ThmRealPrType}:
\[
  s>n/2+7/2+(2-\wt s)_+, \quad \wt s\leq s-1.
\]
Thus, we have proved a propagation result which propagates estimates in a direction which is `correct to first order'. To obtain the final form of the propagation result, we use an argument by Melrose and Sj\"ostrand \cite{MelroseSjostrandSingBVPI,MelroseSjostrandSingBVPII}, in the form given in \cite[Lemma~8.1]{DeHoopUhlmannVasyDiffraction}. This finishes the proof of the first part of Theorem~\ref{ThmRealPrType}.

\begin{rmk}
  For second order real principal type operators of the form considered above, with the highest order derivative having $\Hb^s$-coefficients, the maximal regularity one can prove for a solution $u$ with right hand side $f\in\Hb^{s-1}$ is $\Hb^{\wt s+1}$ with $\wt s$ being at most $s-1$, i.e.\ one can prove $u\in\Hb^s$, which is exactly what we will need in our quest to solve quasilinear wave equations.
\end{rmk}

%%%%%%%%%%%%%%%%%%%%%%%%%%%%%
\subsubsection{Propagation near the boundary}
\label{SubsubsecBoundary}

We now aim to prove the corresponding propagation result (near and) within the boundary $\pa M$: Thus, let $P=P_0+\wt P$, where $\wt P=\wt P_m+\wt P_{m-1}+\wt R$, with $P_0\in\Psib^m$ and $\wt P_m\in\Hb^{s,\alpha}\Psib^m$ having real homogeneous principal symbols, $\wt P_{m-1}\in\Hb^{s-1,\alpha}\Psib^{m-1}$ and $\wt R\in\fpsib^{m-2;0}\Hb^{s-1,\alpha}$ as before, and let us assume that we are given a solution
\[
  u\in\Hb^{\sigma-1/2}
\]
to the equation
\[
  Pu=f\in\Hb^{\sigma-m+1},
\]
where $\sigma=\wt s+m-1$. In fact, since
\[
  \wt R\colon\Hb^{\sigma-1/2}\subset\Hb^{\sigma-1}\to\Hb^{\sigma-m+1},
\]
we may absorb the term $\wt R u$ into the right hand side; thus, we can assume $\wt R=0$, hence $\wt P=\wt P_m+\wt P_{m-1}$.

Moreover, let $\gamma$ be a null-bicharacteristic of $p=p_0+\wt p_m$; we assume $H_p$ is never radial on $\gamma$. Since $H_{\wt p_m}=0$ at $\Tb^*_{\pa M}M$, this in particular implies that $H_{p_0}$ is not radial on $\gamma\cap\Tb^*_{\pa M}M$, and the positivity of the principal symbol $\check a_t H_p\check a_t$ of the commutator there comes from the positivity of $\check a_t H_{p_0}\check a_t$.

The proof of the interior propagation, with small adaptations, carries over to the new setting. We indicate the changes: First, in the notation of \S\ref{SubsubsecInterior}, $\sfH_p\sigma_j$ now only is H\"older continuous with exponent $\alpha$, thus \eqref{EqHamiltonOmegaLipschitz} becomes
\[
  |\sfH_{p_m}\omega| \leq C\omega^{1/2}\left(|\wt\eta|+\omega^{1/2}\right)^\alpha.
\]
Hence, for \eqref{EqStayClose} to hold, we need
\[
  C\omega^{1/2}(|\wt\eta|+\omega^{1/2})^\alpha\leq c_0\eps^2\delta/2,
\]
which holds if $2^{1+2\alpha}C\delta^{1+\alpha}\leq c_0\eps\delta/2$, suggesting the choice $\eps=4^{1+\alpha}C\delta^\alpha/c_0$; in particular $\eps\leq 1$ for $\delta$ small enough. Thus, the size of the a priori control region near $\wt\eta=-\delta$, cf.\ \eqref{EqAprioriControlSize}, becomes
\[
  \omega^{1/2}\leq\frac{c_0\eps^2\delta}{2C(|\wt\eta|+\omega^{1/2})^\alpha}=C_\alpha\delta^{1+\alpha},
\]
which is small enough for the argument in \cite[Lemma~8.1]{DeHoopUhlmannVasyDiffraction} to work. Further, defining the commutant $\sfa$ as before, we replace the a priori control terms $\sfe,\sff'$ in \eqref{EqAprioriControlTerms} by
\begin{equation}
\label{EqPropBdyAPriori}
  \begin{split}
    \sfe&=\chi_0(\sfH_{p_0}+J\sfH_{\wt p_m})\chi_1+\wt\chi_1, \\
    \sff'&=\chi_0(J\sfH_{\wt p_m}-\sfH_{\wt p_m})\chi_1+\wt\chi_1,
  \end{split}
\end{equation}
where we choose the mollifier $J$ to be so close to $I$ that $\sff'\geq 0$; here, we use that the first summand in the definition of $\sff'$ is an element of $\Hb^{s-1}$ in the base variables, hence for $s>n/2+1$ in particular continuous and vanishing at the boundary $\pa M$, and can therefore be dominated by $\wt\chi_1$. We then let $e_t$ and $f'_t$ be defined as in \eqref{EqAprioriControlTerms2} with the above $\sfe$ and $\sff'$. We change the terms $b_t$ and $f''_t$ in \eqref{EqCommSymbolsBF} in a similar way: We take
\begin{align*}
  b_t&=(\F\delta)^{-1/2}\varphi_t\rho^{\sigma}\biggl((\sfH_{p_0}+J\sfH_{\wt p_m})\phi \\
    &\hspace{4ex}-\left[\bigl((2\sigma-m+1)-2t\varphi_t\rho\bigr)(\rho^{-1}(J\sfH_{\wt p_m}+\sfH_{p_0})\rho)+N^2\right] \\
	&\hspace{10ex}\times \F^{-1}\delta\Bigl(2\beta-\frac{\phi}{\delta}\Bigr)^2-\frac{c_0}{8}\biggr)^{1/2} \\
	&\hspace{15ex}\times \sqrt{\chi_0'\left(\F^{-1}\left(2\beta-\frac{\phi}{\delta}\right)\right)}\sqrt{\chi_1\left(\frac{\wt\eta+\delta}{\eps\delta}+1\right)}, \\
  f''_t&=(\F\delta)^{-1}\varphi_t^2\rho^{2\sigma}\biggl((\sfH_{\wt p_m}-J\sfH_{\wt p_m})\phi \\
    &\hspace{-1ex}-\left[\bigl((2\sigma-m+1)-2t\varphi_t\rho\bigr)(\rho^{-1}(\sfH_{\wt p_m}-J\sfH_{\wt p_m})\rho)\right]\F^{-1}\delta\Bigl(2\beta-\frac{\phi}{\delta}\Bigr)^2+\frac{c_0}{8}\biggr) \\
	&\hspace{15ex}\times\chi_0'\left(\F^{-1}\left(2\beta-\frac{\phi}{\delta}\right)\right) \chi_1\left(\frac{\wt\eta+\delta}{\eps\delta}+1\right).
\end{align*}
As before, we can control the term $\la E_t u,u\ra$ in \eqref{EqOperatorCommutator} by the a priori assumptions on $u$. The new feature here is that $f'_t,f''_t\geq 0$ are not just symbols with coefficients having regularity $\Hb^{s-1}$, but there are additional smooth terms involving $\wt\chi_1$ and $c_0/8$. Thus, we need to appeal to the version of the sharp G\aa rding inequality given in Corollary~\ref{CorGardingSmoothPlusNon} to obtain a uniform lower bound on the term $\la J^+(j^- f_t)(z,\Db)u,u\ra$ in \eqref{EqOperatorCommutatorInProof}.

Since the computation of compositions and commutators in the proof of the previous section for $P_0$ is standard as $P_0$ is a smooth b-ps.d.o, and since $\wt P_m$ and $\wt P_{m-1}$ lie in the same space as the operators called $P_m$ and $P_{m-1}$ there, all arguments now go through after straightforward changes that take care of the smooth b-\psdo{}\ $P_0$.

This finishes the proof of Theorem~\ref{ThmRealPrType}.

%%%%%%%%%%%%%%%%%%%%%%%%%%%%%%%%%%%%%%%%%%%%%%%%%
\subsubsection{Complex absorption}
\label{SubsubsecComplexAbsorption}

We next aim to prove Theorem~\ref{ThmComplexAbs}, namely we add a complex absorbing potential $Q=q(z,\Db)\in\Psib^m$ with $Q=Q^*$ and prove the propagation of $\Hb^{\sigma}$-regularity of solutions $u\in\Hb^{\sigma-1/2}$ to the equation
\[
  (P-iQ)u=f\in\Hb^{\sigma-m+1},
\]
where $\gamma$ is a null-bicharacteristic of $P$, in a direction which depends on the sign of $q$ near $\gamma$. Namely, we can propagate $\Hb^{\sigma}$-regularity \emph{forward} along the flow of the Hamilton vector field $H_p$ if $q\geq 0$ near $\gamma$, and \emph{backward} along the flow if $q\leq 0$ near $\gamma$.

Let $\Gamma$ be an open neighborhood of $\gamma$. It suffices to consider the case that $q\geq 0$ in $\Gamma$. Recall that the proofs of the real principal type propagation results given above only show the propagation in the \emph{forward} direction along the flow; the propagation in the backward direction is proved completely analogously (or simply by considering forward propagation for $-P$), and in the presence of complex absorption leads to the reversal in the condition on the sign of $q$ described above. We thus focus on the forward propagation estimate: Here, the only step that we have to change in the real principal type propagation proofs is the right hand side of equation \eqref{EqIBP}, where we have an additional term in view of $P_m u=f-P_{m-1}u+iQu$, namely
\[
  -\Re\la i\check A_t\,iQu,\check A_t u\ra=\Re\la\check A_t Qu,\check A_t u\ra=\Re\la Q\check A_t u,\check A_t u\ra+\Re\la\check A_t^*[\check A_t,Q]u,u\ra.
\]
The first term on the right is bounded from below by $-C_5\|\check A_t u\|_{\Hb^{(m-1)/2}}$ and will be absorbed as in \eqref{EqAbsorbAtu}, and the second term is bounded by the a priori microlocal $\Hb^{\sigma-1/2}$-regularity of $u$ in $\Gamma$, since
\[
  \Re\la\check A_t^*[\check A_t,Q]u,u\ra=\frac{1}{2}\la \wt Q_t u,u\ra
\]
with
\begin{align*}
  \wt Q_t&=\check A_t^*[\check A_t,Q]+[Q,\check A_t^*]\check A_t \\
    &=(\check A_t^*-\check A_t)[\check A_t,Q]+[\check A_t,[\check A_t,Q]]+[Q,\check A_t^*-\check A_t]\check A_t
\end{align*}
uniformly bounded in $\Psib^{2\sigma-1}$ in view of the principal symbol of $\check A_t$ being real and the presence of double commutators.

This finishes the proof of Theorem~\ref{ThmComplexAbs}.

%%%%%%%%%%%%%%%%%%%%%%%%%%%%%%%%%%%%%%%%%%%%%%%%%
\subsection{Propagation near radial points}
\label{SubsecRadialPoints}

We will only consider the class of radial points which will be relevant in our applications; see \S\ref{SecApp}, where an example of an operator with this radial point structure is presented. We remark that the conditions at the `radial set' $L$ below do \emph{not} entail that $H_p$ is indeed radial there, i.e.\ a multiple of the generator of dilations in the fibers of $\Tb^*M\setminus o$, but this \emph{does} hold in the application on exact static de Sitter space, which is why we kept the terminology. (In more general applications, the setup below thus really gives the propagation at `generalized radial sets.') The point is that the standard propagation of singularities results proved in the previous sections do not give any information at (generalized) radial sets: In the case of radial sets, this is clear since the Hamilton flow at a radial point stays within the fiber of the b-cotangent bundle over that point, and more generally in the case of generalized radial sets $L$ considered here, one cannot propagate any regularity into (or out of) $L$ in finite time using standard propagation.

The setting is very similar to the one in \cite[\S{2}]{HintzVasySemilinear}: There, the authors consider an operator $P\in\Psib^m(M;E)$ with real, scalar, homogeneous principal symbol $p$ on a compact manifold $M$ with boundary $Y=\pa M$ and boundary defining function $x$, where the assumptions on $p$ are as follows:
\begin{enumerate}
  \item\label{EnumRadCond1} At $p=0$, $dp\neq 0$, and at $\Sb^*_Y M\cap p^{-1}(0)$, $dp$ and $dx$ are linearly independent; hence $\Sigma=p^{-1}(0)\subset\Sb^*M$ is a smooth codimension $1$ submanifold transversal to $\Sb^*_Y M$.
  \item\label{EnumRadCond2} $L=L_+\cup L_-$, where $L_\pm$ are smooth disjoint submanifolds of $\Sb^*_Y M$, given by $L_\pm=\cL_\pm\cap\Sb^*_Y M$, where $\cL_\pm$ are smooth disjoint submanifolds of $\Sigma$ transversal to $\Sb^*_Y M$, defined locally near $\Sb^*_Y M$. Moreover, the rescaled Hamilton vector field $\sfH_p=\rho^{1-m}H_p$ (which is homogeneous of degree $0$) is tangent to $\cL_\pm$, where, as before, $\rho=\la\zeta\ra$, and $H_p$ is the Hamilton vector field of $p$.
  \item\label{EnumRadCond3} There are functions $\beta_0,\wt\beta\in\CI(L_\pm)$, $\beta_0,\wt\beta>0$, such that
    \begin{equation}
	\label{EqRadialFiberBdy}
	  \rho\sfH_p\rho^{-1}|_{L_\pm}=\mp\beta_0, \qquad -x^{-1}\sfH_p x|_{L_\pm}=\mp\wt\beta\beta_0.
	\end{equation}
  \item\label{EnumRadCond4} For a homogeneous degree $0$ quadratic defining function $\rho_0$ of $\cL=\cL_+\cup\cL_-$ within $\Sigma$,
    \begin{equation}
	\label{EqRadialQuadr}
	  \mp\sfH_p\rho_0-\beta_1\rho_0\geq 0\tn{ modulo terms that vanish cubically at }L_\pm,
	\end{equation}
	where $\beta_1\in\CI(\Sigma)$, $\beta_1>0$ at $L_\pm$.
  \item\label{EnumRadCond5} The imaginary part of the subprincipal symbol is homogeneous, and equals
    \begin{equation}
	\label{EqRadialSubprSmooth}
	  \sigma_\bl^{m-1}\left(\frac{1}{2i}(P-P^*)\right)=\pm\wh\beta\beta_0\rho^{m-1}\tn{ at }L_\pm,
	\end{equation}
	where $\wh\beta\in\CI(L_\pm;\pi^*\End(E))$, $\pi\colon L_\pm\to M$ being the projection to the base; note that $\wh\beta$ is self-adjoint at every point.
\end{enumerate}
These conditions imply that $L_\pm$ is a sink, resp.\ source, for the bicharacteristic flow within $\Sb^*_Y M$, in the sense that nearby null-bicharacteristics tend to $L_\pm$ in the forward, resp.\ backward, direction; but at $L_\pm$ there is also an unstable, resp.\ stable, manifold, namely $\cL_\pm$. In general applications, $L_+$ (and likewise $L_-$) will be the union of one or several connected components of the (generalized) radial set; on static de Sitter space, $L_+$ and $L_-$ will be the two halves of the conormal bundle of the cosmological horizon within future infinity; see Figure~\ref{FigOmega}.

In the non-smooth setting, we will make the exact same assumptions on the `smooth part' of the operator; the guiding principle is that non-smooth operators with coefficients in $\Hb^{s,\alpha}$, $\alpha>0$, $s>n/2+1$, have symbols and associated Hamilton vector fields that vanish at the boundary, thus would not affect the above conditions anyway, with the exception of condition~\itref{EnumRadCond4}, which however is only used close to, but away from $L_\pm$, and the positivity of $\mp\sfH_p\rho_0$ there is preserved when one adds small non-smooth terms in $\Hb^{s,\alpha}$ to $p$. In order to be able to give a concise expression for the threshold regularity (determining whether one can propagate into or out of the boundary), let us define for a function $b\in\CI(L_\pm,\pi^*\End(E))$ with values in self-adjoint endomorphisms of the fiber,
\begin{gather*}
  \inf_{L_\pm}b:=\sup\{\lambda\in\R\colon b\geq\lambda\,I\tn{ everywhere on }L_\pm\}, \\
  \sup_{L_\pm}b:=\inf\{\lambda\in\R\colon b\leq\lambda\,I\tn{ everywhere on }L_\pm\}.
\end{gather*}
We then have the following theorem:

\begin{thm}
\label{ThmRadialPoints}
  Let $m,r,s,\wt s\in\R$, $\alpha>0$. Let $P=P_0+\wt P$, where $P_0\in\Psib^m(\Rnhalf;E)$ has a real, scalar, homogeneous principal symbol $p_0$, further $\wt P=\wt P_m+\wt P_{m-1}+\wt R$ with $\wt P_m\in\Hb^{s,\alpha}\Psib^m(\Rnhalf;E)$ having a real, scalar, homogeneous principal symbol $\wt p_m$, moreover $\wt P_{m-1}\in\Hb^{s-1,\alpha}\Psib^{m-1}(\Rnhalf;E)$ and $\wt R\in\Psib^{m-2}(\Rnhalf;E)+\fpsib^{m-2;0}\Hb^{s-1,\alpha}(\Rnhalf;E)$. Suppose that the above conditions~\itref{EnumRadCond1}-\itref{EnumRadCond5} hold for $p_0$ and $P_0$ in place of $p$ and $P$. Finally, assume that $s$ and $\wt s$ satisfy
  \begin{equation}
  \label{EqRadialCond}
    \wt s\leq s-1, \quad s>n/2+7/2+(2-\wt s)_+.
  \end{equation}
  Suppose $u\in\Hb^{\wt s+m-3/2,r}(\Rnhalf;E)$ is such that $Pu\in\Hb^{\wt s,r}(\Rnhalf;E)$.
  \begin{enumerate}
    \item \label{EnumThmRadIntoBdy} If $\wt s+(m-1)/2-1+\inf_{L_\pm}(\wh\beta-r\wt\beta)>0$, let us assume that in a neighborhood of $L_\pm$, $\cL_\pm\cap\{x>0\}$ is disjoint from $\WFb^{\wt s+m-1,r}(u)$.
	\item \label{EnumThmRadFromBdy} If $\wt s+(m-1)/2+\sup_{L_\pm}(\wh\beta-r\wt\beta)<0$, let us assume that a punctured neighborhood of $L_\pm$, with $L_\pm$ removed, in $\Sigma\cap\Sb^*_{\pa\Rnhalf}\Rnhalf$ is disjoint from $\WFb^{\wt s+m-1,r}(u)$.
  \end{enumerate}
  Then in both cases, $L_\pm$ is disjoint from $\WFb^{\wt s+m-1,r}(u)$.
\end{thm}

Adjoints are again taken with respect to the b-density $\frac{dx}{x}\,dy$ and the smooth metric on the vector bundle $E$. We remark that condition~\eqref{EqRadialSubprSmooth} is insensitive to changes both of the b-density and the metric on $E$ by the radiality of $H_{p_0}$ at $L_\pm$; see \cite[Footnote~19]{VasyMicroKerrdS} for details.

\begin{rmk}
\label{RmkRelaxedRadial}
  Since $\WFb^{\wt s+m-1,r}(u)$ is closed, we in fact have the conclusion that a neighborhood of $L_\pm$ is disjoint from $\WFb^{\wt s+m-1,r}(u)$. As in the real principal type setting (see equation~\eqref{EqPropEstimate} in particular), one can also rewrite the wavefront set statement as an estimate on the $L^2_\bl$ norm of an operator of order $\wt s+m-1$, elliptic at $L_\pm$, applied to $u$. In particular, we will see that it suffices to have only microlocal $\Hb^{\wt s,r}$-membership of $Pu$ near the part of the radial set that we propagate to/from, and local membership in $\Hb^{\wt s-1}$, which comes from a use of elliptic regularity (Theorem~\ref{ThmEllipticReg}) in our argument.

  Moreover, as the proof will show, the theorem also holds for operators $P$ which are perturbations of those for which it directly applies: Indeed, even though the dynamical assumptions~\itref{EnumRadCond1}-\itref{EnumRadCond4} are (probably) not stable under perturbations, the estimates derived from these are. Here, perturbations are to be understood in the sense that $P_0$ may be perturbed within $\Psib^m$, and $\wt P_m,\wt P_{m-1}$ and $\wt R$ may be changed arbitrarily, with the estimate corresponding to the wavefront set statement of the theorem being locally uniform.
\end{rmk}

The proof is an adaptation of the proof of \cite[Proposition~2.1]{HintzVasySemilinear}, see also \cite[Propositions~2.10 and 2.11]{VasyMicroKerrdS} for a related result, to our non-smooth setting. Before giving the full proof, we briefly recall the commutator proof in the smooth setting, for the weight $r=0$, and under the assumption $P=P^*\in\Psib^m(\Rnhalf;E)$, ignoring issues of regularization: Putting $\sigma=\wt s+m-1$ and $\sfp=\rho^{-m}p$, where $p$ is the principal symbol of $P$, we consider the commutant $a=\rho^{\sigma-(m-1)/2}\psi(\rho_0)\psi_1(x)\psi_0(\sfp)$, where $\psi,\psi_0,\psi_1\in\CIc(\R)$ are identically $1$ and have non-positive derivative on $[0,\infty)$. Since the Hamilton vector field $H_p$ of is radial at $L_\pm$ (or, more generally, tangent to the generalized radial set $L_\pm$), the positivity in the commutator calculation comes from the weight: Concretely, working near $L_+$,
\begin{align*}
  a H_p a = \rho^{2\sigma}\psi\psi_0\psi_1\bigl(\beta_0(\sigma&-(m-1)/2)\psi\psi_0\psi_1 \\
    &+(x^{-1}\sfH_p x)x\psi_1' \psi\psi_0 + (\sfH_p\rho_0)\psi'\psi_0\psi_1 + (\sfH_p\sfp)\psi_0'\psi\psi_1\Bigr).
\end{align*}
The first term is the main term giving `positivity' at $L_\pm$. Let us only treat the case $\sigma-(m-1)/2>0$, then the main term is positive at $L_+$ by \eqref{EqRadialFiberBdy}, as is the third term by \eqref{EqRadialQuadr}, while the second term is negative; it is on its support that we need to make a priori regularity assumption. (The last term is supported on the elliptic set of $P$ and will therefore be controlled by elliptic regularity.) We then let $A$ be a quantization of $a$ and obtain $\la iA^*[P,A]u,u\ra=\la(B_1^*B_1-B_2^*B_2+B_3^*B_3)u,u\ra$ (plus lower order terms and terms controlled by $Pu$ using elliptic regularity), with $B_1\in\Psib^\sigma$ elliptic at $L_+$; therefore, dropping the term involving $B_3$,
\[
  \|B_1 u\|_{L^2_\bl}^2 \leq \|B_2 u\|_{L^2_\bl} + \Re\la iA^*[P,A]u,u\ra.
\]
Upon expanding the commutator, integration by parts (this is where $\wh\beta$ from \eqref{EqRadialSubprSmooth} comes into play if $P$ is not formally self-adjoint) and applying Cauchy-Schwarz, absorbing terms involving $Au$ into the left hand side, we can therefore control $u$ in $\Hb^\sigma$ microlocally at $L_+$ provided we have $\Hb^\sigma$ control on $u$ on the support of $\psi_1'(x)\psi(\rho_0)\psi_0(\sfp)$ and $\Hb^{\sigma-m+1}$ control on $f=Pu$.

\begin{proof}[Proof of Theorem~\ref{ThmRadialPoints}.]
  We again drop the bundle $E$ from the notation. Since $\wt R u\in\Hb^{\wt s}$ by the a priori regularity on $u$, we can absorb $\wt R u$ into
  \[
    f\equiv Pu
  \]
  and thus assume $\wt R=0$.
  
  \textbf{Step 1: Symbolic construction of the commutant.} Let us assume
  \[
    m\geq 1,\quad r=0
  \]
  for now; these conditions will be eliminated at the end of the proof.
  
  Define the regularizer $\varphi_t(\rho)=(1+t\rho)^{-1}$ for $t\geq 0$ as in the proof of Theorem~\ref{ThmRealPrType}, put $\sfp_0=\rho^{-m}p_0$ and $\sigma=\wt s+m-1$, and consider the commutant
  \[
    a_t=\varphi_t(\rho)\psi(\rho_0)\psi_0(\sfp_0)\psi_1(x)\rho^{\sigma-(m-1)/2},
  \]
  where $\psi,\psi_0,\psi_1\in\CI_\cl(\R)$ are equal to $1$ near $0$ and have derivatives which are $\leq 0$ on $[0,\infty)$; we will be more specific about the supports of $\psi,\psi_0,\psi_1$ below. Let us also assume that $\sqrt{-\psi\psi'}$ and $\sqrt{-\psi_1\psi_1'}$ are smooth in a neighborhood of $[0,\infty)$. As usual, we put $\sfH_{\wt p_m}=\rho^{1-m}H_{\wt p_m}$. We then compute, using $\sfH_{\wt p_m}\varphi_t=-t\varphi_t^2\sfH_{\wt p_m}\rho$:
  \begin{align*}
    a_t H_{\wt p_m}a_t &= \varphi_t^2\rho^{2\sigma}\psi\psi_0\psi_1\Bigl((\sigma-(m-1)/2-t\rho\varphi_t)(\rho^{-1}\sfH_{\wt p_m}\rho)\psi\psi_0\psi_1 \\
	&\quad + (x^{-1}\sfH_{\wt p_m}x)x\psi\psi_0\psi_1' + (\sfH_{\wt p_m}\rho_0)\psi'\psi_0\psi_1 + \sfH_{\wt p_m}(\sfp_0)\psi\psi_0'\psi_1\Bigr),
  \end{align*}
  and to compute $a_t H_{p_0}a_t$, we can use \eqref{EqRadialFiberBdy} to simplify the resulting expression.

  To motivate the next step, recall that the objective is to obtain an estimate similar to \eqref{EqAbsorbAtu}; however, since in our situation, the weight $\rho^{\sigma-(m-1)/2}$ can only give a limited amount of positivity at $L_\pm$, we need to absorb error terms, in particular the ones involving $P-P^*$, into the commutator $a_t H_{p_m}a_t$, where $p_m=p_0+\wt p_m$ is the principal symbol of $P$ and also of $P_m:=P_0+\wt P_m$. Thus, consider
  \begin{align*}
    a_t&H_{p_m}a_t \pm \rho^{m-1}a_t^2\beta_0\wh\beta = \pm\varphi_t^2\rho^{2\sigma}\psi\psi_0\psi_1 \\
	  &\times \Bigl(\bigl[\beta_0(\sigma-(m-1)/2-t\rho\varphi_t+\wh\beta) \\
	  &\hspace{10ex}\pm(\sigma-(m-1)/2-t\rho\varphi_t)(\rho^{-1}\sfH_{\wt p_m}\rho)\bigr]\psi\psi_0\psi_1 \\
	  &\qquad + (\wt\beta\beta_0 \pm x^{-1}\sfH_{\wt p_m}x)x\psi\psi_0\psi_1' \pm (\sfH_{p_0}\rho_0+\sfH_{\wt p_m}\rho_0)\psi'\psi_0\psi_1 \\
	  &\qquad + (-m\beta_0\sfp_0 \pm \sfH_{\wt p_m}\sfp_0)\psi\psi_0'\psi_1\Bigr).
  \end{align*}
  Recall that $t\rho\varphi_t$ is a bounded family of symbols in $S^0$, and we in fact have $|t\rho\varphi_t|\leq 1$ for all $t$. We now proceed to prove the first case of the theorem. Let us make the following assumptions:
  \begin{itemize}[leftmargin=\enummargin]
    \item On $\supp(\psi\circ\rho_0)\cap\supp(\psi_0\circ\sfp_0)\cap\supp(\psi_1\circ x)$:
	  \begin{gather}
	    \label{EqThresholdPos} \beta_0(\sigma-(m-1)/2-1+\wh\beta)\geq c_0>0 \\
		|(\sigma-(m-1)/2-t\rho\varphi_t)(\rho^{-1}\sfH_{\wt p_m}\rho)|\leq c_0/4\tn{ for all }t>0. \nonumber
	  \end{gather}
	  The first condition is satisfied at $L_\pm$ by assumption, and the second condition is satisfied close to $Y=\{x=0\}$, since $\rho^{-1}\sfH_{\wt p_m}\rho=o(1)$ as $x\to 0$ by Riemann-Lebesgue.
	\item On $\supp d(\psi_1\circ x)\cap\supp(\psi\circ\rho_0)\cap\supp(\psi_0\circ\sfp_0)$:
	  \[
	    \wt\beta\beta_0\geq c_1>0, \quad |x^{-1}\sfH_{\wt p_m}x|\leq c_1/2.
	  \]
	  The second condition is satisfied close to $Y$, since $x^{-1}\sfH_{\wt p_m}x=o(1)$ as $x\to 0$.
	\item On $\supp d(\psi\circ\rho_0)\cap\supp(\psi_1\circ x)\cap\supp(\psi_0\circ\sfp_0)$:
	  \begin{equation}
	  \label{EqQuadrDefPos}
	    \mp\sfH_{p_0}\rho_0\geq\frac{\beta_1}{2}\rho_0\geq c_2>0,\quad |\sfH_{\wt p_m}\rho_0|\leq c_2/2.
	  \end{equation}
	\item On $\supp d(\psi_0\circ\sfp_0)\cap\supp(\psi\circ\rho_0)\cap\supp(\psi_1\circ x)$:
	  \begin{equation}
	  \label{EqPmElliptic}
	    |\rho^{-m}p_m|\geq c_3>0.
	  \end{equation}
	  This can be arranged as follows: First, note that we can ensure
	  \begin{equation}
	  \label{EqPm''Elliptic}
	    |\sfp_0|\geq 2c_3
	  \end{equation}
	  there; then, since $|\rho^{-m}\wt p_m|=o(1)$ as $x\to 0$, shrinking the support of $\psi_1$ if necessary guarantees \eqref{EqPmElliptic}.
  \end{itemize}
  We can ensure that all these assumptions are satisfied by first choosing $\psi_1$, localizing near $\Sb^*_Y M$, then $\psi$, localizing near $L_\pm$ within the characteristic set $p_0^{-1}(0)$ of $P_0$, such that the inequalities in \eqref{EqThresholdPos} and \eqref{EqQuadrDefPos} are strict on $p_0^{-1}(0)$, then choosing $\psi_0$ (localizing near $p_0^{-1}(0)$) such that strict inequalities hold in \eqref{EqThresholdPos}, \eqref{EqQuadrDefPos} and \eqref{EqPm''Elliptic}, and finally shrinking the support of $\psi_1$, if necessary, such that all inequalities hold.

  We can then write
  \begin{equation}
  \label{EqRadialCommutator} a_t H_{p_m} a_t \pm \rho^{m-1}a_t^2\beta_0\wh\beta = \pm\left(\frac{c_0}{8}\rho^{m-1}a_t^2 + b_{1,t}^2+b_{2,t}^2 - b_{3,t}^2 + f_t + g_t\right),
  \end{equation}
  where, with a mollifier $J$ as in Lemma~\ref{LemmaMollifier},
  \begin{align*}
    b_{1,t} &= \varphi_t\rho^\sigma \psi\psi_0\psi_1 \Bigl[\beta_0(\sigma-(m-1)/2-t\rho\varphi_t+\wh\beta) \\
	  &\hspace{23ex} \pm(\sigma-(m-1)/2-t\rho\varphi_t)(\rho^{-1}J\sfH_{\wt p_m}\rho)-\frac{c_0}{2}\Bigr]^{1/2}, \\
	b_{2,t} &= \varphi_t\rho^\sigma\psi_0\psi_1\sqrt{-\psi\psi'}\left[\mp\left(\sfH_{p_0}\rho_0+J\sfH_{\wt p_m}\rho_0\right)-\frac{c_2}{4}\right]^{1/2}, \\
	b_{3,t} &= \varphi_t\rho^\sigma\psi\psi_0\sqrt{-\psi_1\psi_1'}\left[\left(\wt\beta\beta_0 \pm x^{-1}J\sfH_{\wt p_m}x+\frac{c_1}{4}\right)x\right]^{1/2}, \\
	g_t &= \varphi_t^2\rho^{2\sigma}\psi^2\psi_0\psi_0'\psi_1^2(-m\beta_0\sfp_0 \pm \sfH_{\wt p_m}\sfp_0),
  \end{align*}
  and $f_t=f_{1,t}+f_{2,t}+f_{3,t}$ with
  \begin{align*}
    f_{1,t} &= \varphi_t^2\rho^{2\sigma}\psi^2\psi_0^2\psi_1^2 \\
	  &\qquad\times\left[\pm(\sigma-(m-1)/2-t\rho\varphi_t)\bigl(\rho^{-1}(\sfH_{\wt p_m}-J\sfH_{\wt p_m})\rho\bigr)+\frac{3c_0}{8}\right], \\
	f_{2,t} &= \varphi_t^2\rho^{2\sigma}\psi\psi'\psi_0^2\psi_1^2\left(\pm\bigl(\sfH_{\wt p_m}-J\sfH_{\wt p_m}\bigr)\rho_0-\frac{c_2}{4}\right), \\
	f_{3,t} &= \varphi_t^2\rho^{2\sigma}\psi^2\psi_0^2\psi_1\psi_1'\left(\pm x^{-1}\bigl(\sfH_{\wt p_m}-J\sfH_{\wt p_m}\bigr)x-\frac{c_1}{4}\right)x.
  \end{align*}
  In particular, $b_{1,t},b_{2,t}\in S^\sigma$, $b_{3,t}\in x^{1/2}S^\sigma$, $f_t\in S^{2\sigma;\infty}\Hb^{s-1}+S^{2\sigma}$, $g_t\in\Hb^{s-1}S^{2\sigma}+S^{2\sigma}$ uniformly, with the symbol orders one lower if $t>0$ for $b_{j,t}$, $j=1,2,3$, and two lower for $f_t,g_t$. The term $b_{1,t}^2$ will give rise to an operator which is elliptic at $L_\pm$. The term $b_{2,t}^2$ (which has the same, `advantageous,' sign as $b_{1,t}$) can be discarded, and the term $-b_{3,t}^2$, with a `disadvantageous' sign, will be bounded using the a priori regularity assumptions on $u$. An important point here is that the non-smooth symbol $f_t$ is non-negative if we choose the mollifier $J$ to be close enough to $I$; in fact, we then have $f_{j,t}\geq 0$ for $j=1,2,3$. Lastly, we will be able to estimate the contribution of the term $g_t$ using elliptic regularity, noting that its support is disjoint from the characteristic set $p_m^{-1}(0)$ of $P$.

  \textbf{Step 2. Expanding the commutator; bounding lower order terms.} Let $A_t\in\Psib^{\sigma-(m-1)/2}$ and $B_{1,t},B_{2,t},B_{3,t}\in\Psib^\sigma$ denote quantizations with uniform b-microsupport contained in the support of the respective full symbols $a_t,b_{1,t},b_{2,t}$ and $b_{3,t}$. Then we compute as in the proof of real principal type propagation (see equation~\eqref{EqIBP} there):
  \begin{align*}
    \Re\la iA_t^*[P_m,A_t]u,u\ra &= -\left\la\frac{1}{2i}(P_m-P_m^*)A_t u, A_t u\right\ra \\
	  &\hspace{10ex}-\Re\la i A_t f,A_t u\ra+\Re\la i A_t \wt P_{m-1} u,A_t u\ra.
  \end{align*}
  We split the first term on the right hand side into two pieces corresponding to the decomposition $P_m=P_0+\wt P_m$. The piece involving $P_0$ will be dealt with later. For the other piece, note that $\wt P_m$ is a sum of terms of the form $\tau^\alpha w Q_m$, where $w\in\Hb^s$ is real-valued and $Q_m=q_m(z,\Db)\in\Psib^m$ has a real principal symbol. Now,
  \begin{align*}
    \tau^\alpha &w Q_m-(\tau^\alpha w Q_m)^* \\
	 &=\tau^\alpha w(Q_m-Q_m^*)+\tau^\alpha(w Q_m^*-Q_m^* w)+\tau^\alpha(Q_m^*-\tau^{-\alpha}Q_m^*\tau^\alpha)w,
  \end{align*}
  thus, using Theorem~\ref{ThmComp} \itref{ThmCompBRFunny} with $k=1,k'=0$ (applicable because we are assuming $m\geq 1$) to compute $Q_m^* w$ and with $k=0,k'=0$ to compute the last term, we get
  \[
    i(\wt P_m-\wt P_m^*)=R_1+R_2+R_3,
  \]
  where
  \[
    R_1\in\Hb^{s-1,\alpha}\Psib^{m-1}, R_2\in\Psib^{m-1}\circ\Psi^{0;0}\Hb^{s-1,\alpha}, R_3\in\Psi^{m-1;0}\Hb^{s-1,\alpha}.
  \]
  Let $\chi\in\CIc(\Rhalf)$, $\chi\equiv 1$ near $0$. Writing $R_1$ as the sum of terms of the form $w'Q'$, where $w'\in\Hb^{s-1,\alpha}$ and $Q'\in\Psib^{m-1}$, we have for $\eps'>0$, which we can choose to be as small as we like provided we shrink the support of the Schwartz kernel of $A_t$:
  \[
    \la w'(z)Q' A_t u, A_t u\ra = \la \chi(x/\eps')w'(z) Q' A_t u, A_t u\ra;
  \]
  by Lemma~\ref{LemmaLocNearBdy}, this can be bounded by $c_{\eps'}\|A_t u\|_{\Hb^{(m-1)/2}}^2$, where $c_{\eps'}\to 0$ as $\eps'\to 0$.\footnote{\label{FootRCond1}This argument requires that elements of $\Hb^{s-1}$ are multipliers on $\Hb^{(m-1)/2}$, which is the case if $s-1\geq(m-1)/2$.} In a similar manner, we can treat the terms involving $R_2$ and $R_3$. Hence, under the assumption that the Schwartz kernel of $A_t$ is localized sharply enough near $\pa M\times\pa M$, we have
  \[
    |\la(\wt P_m-\wt P_m^*)A_t u,A_t u\ra|\leq C_\delta + \delta\|A_t u\|_{\Hb^{(m-1)/2}}^2
  \]
  for an arbitrarily small, but fixed $\delta>0$.

  Next, for $\delta>0$, we estimate
  \[
    |\la A_t f,A_t u\ra|\leq C_\delta+\delta\|A_t u\|_{\Hb^{(m-1)/2}}^2,
  \]
  using that $\|A_t f\|_{\Hb^{-(m-1)/2}}$ is uniformly bounded.

  Finally, we can bound the term $\la A_t \wt P_{m-1} u,A_t u\ra$ as in the proof of Theorem~\ref{ThmRealPrType}, thus obtaining
  \[
    |\la A_t \wt P_{m-1} u,A_t u\ra|\leq C_\delta+\delta\|A_t u\|_{\Hb^{(m-1)/2}}^2.
  \]
  Therefore, writing $Q:=\frac{1}{2i}(P_0-P_0^*)\in\Psib^{m-1}$, we get
  \[
    \pm\Re\la (iA_t^*[P_m,A_t]+A_t^*QA_t)u,u\ra \leq C_\delta+\delta\|A_t u\|_{\Hb^{(m-1)/2}}^2
  \]
  Now, using that $|\la B_{t,3}^*B_{t,3}u, u\ra|=\|B_{t,3}u\|_{L^2_\bl}^2$ is uniformly bounded because of the assumed a priori control of $u$ in a neighborhood of $L_\pm$ in $\cL_\pm\cap\{x>0\}$, we deduce, using the operator $\Lambda=\Lambda_{(m-1)/2}$:
  \begin{equation}
  \label{EqRadialCommutatorOp}
    \begin{split}
      \Re&\Bigl\langle\Bigl(\pm i A_t^*[P_m,A_t]\pm A_t^*QA_t-\frac{c_0}{8}(\Lambda A_t)^*(\Lambda A_t) \\
	  &\hspace{20ex}-B_{1,t}^*B_{1,t}-B_{2,t}^*B_{2,t}+B_{3,t}^*B_{3,t}\Bigr)u,u\Bigr\rangle \\
	  &\qquad \leq C_\delta + \left(\delta-\frac{c_0}{8}\right)\|A_t u\|_{\Hb^{(m-1)/2}}^2 - \|B_{1,t}u\|_{L^2_\bl}^2,
	\end{split}
  \end{equation}
  where we discarded the negative term $-\la B_{2,t}^*B_{2,t}u,u\ra$ on the right hand side. If we choose $\delta<c_0/8$, then we can also discard the term on the right hand side involving $A_t u$, hence
  \begin{equation}
  \label{EqRadialEllBound}
    \begin{split}
      \|B_{1,t}u\|_{L^2_\bl}^2&\leq C-\Re\Bigl\langle\Bigl(\pm i A_t^*[P_m,A_t]\pm A_t^*QA_t-\frac{c_0}{8}(\Lambda A_t)^*(\Lambda A_t) \\
	    &\hspace{20ex}-B_{1,t}^*B_{1,t}-B_{2,t}^*B_{2,t}+B_{3,t}^*B_{3,t}\Bigr)u,u\Bigr\rangle.
	\end{split}
  \end{equation}

  \textbf{Step 3. Using the symbolic commutator calculation.} We now exploit the commutator relation~\eqref{EqRadialCommutator} in the same way as in the proof of Theorem~\ref{ThmRealPrType}: If we introduce operators
  \[
    J^+\in\Psib^{\sigma-(m-1)/2-1},\quad J^-\in\Psib^{-\sigma+(m-1)/2+1}
  \]
  with real principal symbols $j^+,j^-$, satisfying $J^+J^-=I+\wt R$, $\wt R\in\Psib^{-\infty}$, we obtain, keeping in mind \eqref{EqRadialSubprSmooth},
  \begin{align*}
    \Re&\Bigl\langle J^-\Bigl(\pm i A_t^*[P_m,A_t]\pm A_t^* QA_t-\frac{c_0}{8}(\Lambda A_t)^*(\Lambda A_t) \\
	  &\hspace{20ex}-B_{1,t}^*B_{1,t}-B_{2,t}^*B_{2,t}+B_{3,t}^*B_{3,t}\Bigr)u,(J^+)^*u\Bigr\rangle \\
	  &\geq \Re\Bigl\langle\Bigl[j^-\Bigl(\pm a_t H_{p_m}a_t+\rho^{m-1}a_t^2\beta_0\wh\beta-\frac{c_0}{8}\rho^{m-1}a_t^2 \\
	  &\hspace{20ex}-b_{1,t}^2-b_{2,t}^2+b_{3,t}^2\Bigr)\Bigr](z,\Db)u,(J^+)^*u\Bigr\rangle-C \\
	  &=\Re\la(j^-f_t)(z,\Db)u,(J^+)^*u\ra+\Re\la(j^-g_t)(z,\Db)u,(J^+)^*u\ra-C,
  \end{align*}
  where we absorbed various error terms in the constant $C$; see the discussion around equation~\eqref{EqOperatorCommutatorInProof} for details. The term involving $f_t$ is uniformly bounded from below as explained in the proof of Theorem~\ref{ThmRealPrType} after equation~\eqref{EqOperatorCommutatorInProof}. It remains to bound the term involving $g_t$. Note that we can write $(j^-g_t)(z,\zeta)$ as a sum of terms of the form $w(z)\varphi_t(\zeta)^2 s(z,\zeta)$, where $w\in\Hb^{s-1}$, or $w\in\CI$, and $s\in S^{\sigma+(m-1)/2+1}$, and we can assume
  \[
    (\Sb^*M\cap\supp s)\cap p_m^{-1}(0)=\emptyset,
  \]
  since this holds for $g_t$ in place of $s$. Thus, on $\Sb^*M\cap\supp s$, we can use elliptic regularity, Theorem~\ref{ThmEllipticReg}, to conclude that $\WFb^{\sigma+1}(u)\cap(\Sb^*M\cap\supp s)=\emptyset$; but this implies that
  \[
    (w\varphi_t^2 s)(z,\Db)u\in\Hb^{-(m-1)/2}
  \]
  is uniformly bounded. Therefore, we finally obtain from \eqref{EqRadialEllBound} a uniform bound on $\|B_{1,t}u\|_{L^2_\bl}$, which implies $B_{1,0}u\in L^2_\bl$ and thus the claimed microlocal regularity of $u$ at $L_\pm$, finishing the proof of the first part of the theorem in the case $m\geq 1$, $r=0$.

  \textbf{Step 3'. Propagation below the threshold regularity.} The proof of the second part is similar, only instead of requiring \eqref{EqThresholdPos}, we require
  \[
    \beta_0(\sigma-(m-1)/2+\wh\beta)\leq -c_0<0
  \]
  on $\supp(\psi\circ\rho_0)\cap\supp(\psi_0\circ\sfp_0)\cap\supp(\psi_1\circ x)$, and we correspondingly define
  \begin{align*}
    b_{1,t} &= \varphi_t\rho^\sigma \psi\psi_0\psi_1 \Bigl[-\beta_0(\sigma-(m-1)/2-t\rho\varphi_t+\wh\beta) \\
	  &\hspace{23ex} \mp(\sigma-(m-1)/2-t\rho\varphi_t)(\rho^{-1}J\sfH_{\wt p_m}\rho)-\frac{c_0}{2}\Bigr]^{1/2}.
  \end{align*}
  We also redefine
  \begin{align*}
	b_{3,t} &= \varphi_t\rho^\sigma\psi\psi_0\sqrt{-\psi_1\psi_1'}\left[\left(\wt\beta\beta_0 \pm x^{-1}J\sfH_{\wt p_m}x-\frac{c_1}{4}\right)x\right]^{1/2}, \\
    f_{1,t} &= \varphi_t^2\rho^{2\sigma}\psi^2\psi_0^2\psi_1^2 \\
	  &\qquad\left[\mp(\sigma-(m-1)/2-t\rho\varphi_t)\bigl(\rho^{-1}(\sfH_{\wt p_m}-J\sfH_{\wt p_m})\rho\bigr)+\frac{3c_0}{8}\right], \\
	f_{3,t} &= \varphi_t^2\rho^{2\sigma}\psi^2\psi_0^2\psi_1\psi_1'\left(\mp x^{-1}\bigl(\sfH_{\wt p_m}-J\sfH_{\wt p_m}\bigr)x-\frac{c_1}{4}\right)x.
  \end{align*}
  Equation~\eqref{EqRadialCommutator} then becomes
  \[
    a_t H_{p_m} a_t \pm \rho^{m-1}a_t^2\beta_0\wh\beta = \mp\left(\frac{c_0}{8}\rho^{m-1}a_t^2 + b_{1,t}^2-b_{2,t}^2 + b_{3,t}^2 + f_t + g_t\right),
  \]
  and the rest of the proof proceeds as before, the most important difference being that now the term $b_{3,t}^2$ has an advantageous sign (namely, the same as $b_{1,t}^2$), whereas $-b_{2,t}^2$ does not, which is the reason for the microlocal regularity assumption on $u$ in a punctured neighborhood of $L_\pm$ within $\Sb^*_{\pa\Rnhalf}\Rnhalf$.

  \textbf{Step 4. Removing the restrictions on the order $m$ and the weight $r$.} The last step in the proof is to remove the restrictions on $m$ (the order of the operator) and $r$ (the growth rate of $u$ and $f$). We accomplish this by rewriting the equation $Pu=f$ (without restrictions on $m$ and $r$) as
  \[
    (x^{-r}P\Lambda^+ x^r)(x^{-r}\Lambda^- u)=x^{-r}f+x^{-r}PRx^r(x^{-r}u),
  \]
  where $\Lambda^\pm\in\Psib^{\mp(m-m_0)}$, $m_0\geq 1$, have principal symbols $\rho^{\mp(m-m_0)}$ and satisfy $\Lambda^+\Lambda^-=I+R$, $R\in\Psib^{-\infty}$. Then $x^{-r}P\Lambda^+ x^r$ has order $m_0$, and, recalling $\wt s=\sigma-m+1$,
  \[
    x^{-r}f\in\Hb^{\wt s}, \quad x^{-r}\Lambda^- u\in\Hb^{\wt s+m_0-3/2}
  \]
  lie in unweighted b-Sobolev spaces. The principal symbol of $P_{0,r}:=x^{-r}P_0\Lambda^+ x^r$ is an elliptic multiple of the principal symbol of $P_0$, hence the Hamilton vector fields of $P_{0,r}$ and $P_0$ agree, up to a positive factor, on the characteristic set of $P_0$; in particular, even though $\beta_0$ in equation~\eqref{EqRadialFiberBdy} may be different for $P_{0,r}$ and $P_0$, $\wt\beta$ does not change on $L_\pm$. However, the imaginary part of the subprincipal symbol, hence $\wh\beta$, does change, resulting in a shift of the threshold values in the statement of the theorem: Concretely, we claim that
  \begin{equation}
  \label{EqGeneralRadialPrSym}
    \sigma_\bl^{m_0-1}\left(\frac{1}{2i}(P_{0,r}-P_{0,r}^*)\right)=\pm\rho^{m_0-1}\beta_0\left(\wh\beta+\frac{m-m_0}{2}-r\wt\beta\right)\tn{ at }L_\pm.
  \end{equation}
  Granted this, the threshold quantity is the $\sup$, resp.\ $\inf$, over $L_\pm$ of
  \[
    \wt s+(m_0-1)/2+\wh\beta+(m-m_0)/2-r\wt\beta=\wt s+(m-1)/2+\wh\beta-r\wt\beta.
  \]
  To prove \eqref{EqGeneralRadialPrSym}, let us write $P_0=P'_m+P'_{m-1}$; we can arrange for $P'_m$ to be (formally) self-adjoint by letting $P'_m=(P_0+P_0^*)/2$ and $P'_{m-1}=(P_0-P_0^*)/2$, and we shall also assume $\Lambda^+=(\Lambda^+)^*$. We then compute
  \begin{align*}
    \sigma_\bl^{m_0-1}&\left(\frac{1}{2i}\bigl(x^{-r}P'_{m-1}\Lambda^+x^r-x^r\Lambda^+(P'_{m-1})^*x^{-r}\bigr)\right) \\
	  &=\rho^{m_0-1}\rho^{1-m}\sigma_\bl^{m-1}\left(\frac{1}{2i}(P'_{m-1}-(P'_{m-1})^*)\right)=\pm\rho^{m_0-1}\beta_0\wh\beta,
  \end{align*}
  and
  \begin{align*}
    &\sigma_\bl^{m_0-1}\left(\frac{1}{2i}\bigl(x^{-r}P'_m\Lambda^+ x^r-x^r\Lambda^+P'_m x^{-r}\bigr)\right) \\
	  &\ =\sigma_\bl^{m_0-1}\left(\frac{1}{2i}[P'_m,\Lambda^+]\right)+\sigma_\bl^{m_0-1}\left(\frac{1}{2i}\bigl(x^{-r}[P'_m\Lambda^+,x^r]-x^r[\Lambda^+P'_m,x^{-r}]\bigr)\right) \\
	  &\ =\pm\frac{m-m_0}{2}\beta_0\rho^{m_0-1}-rx^{-1}H_{p_0\rho^{m_0-m}}x \\
	  &\ =\pm\left(\frac{m-m_0}{2}\beta_0-r\wt\beta\beta_0\right)\rho^{m_0-1}-rp_0x^{-1}H_{\rho^{m_0-m}}x.
  \end{align*}
  The last term on the right hand side involving $p_0$ vanishes at $L_\pm$, proving \eqref{EqGeneralRadialPrSym}.

  \textbf{Step 5. Collecting the regularity requirements.} Lastly, the regularities needed for the proof to go through are that the conditions in \eqref{EqRealPrTypeConditionsProof} hold for some $m_0\geq 1$; thus, choosing $m_0=\max(1,5-2\wt s)=1+2(2-\wt s)_+$, we obtain the conditions given in the statement of Theorem~\ref{ThmRadialPoints}.
\end{proof}

%%%%%%%%%%%%%%%%%%%%%%%%%%%%%%%%%%%%%%%%%%%%%%%%%%%%%%%%%%%%%%%%%%%%%
\section{Global solvability results for second order hyperbolic operators with non-smooth coefficients}
\label{SecGlobalForSecond}

Even though complex absorption is a useful tool to put wave equations on some classes of geometric spaces into a Fredholm framework, as done by Vasy \cite{VasyMicroKerrdS} in various dilation-invariant settings, and microlocally easy to deal with, it is problematic in general non dilation-invariant situations if one wants to prove the existence of forward solutions, as pointed out by Vasy and the author \cite{HintzVasySemilinear}. We shall adopt the strategy of \cite{HintzVasySemilinear} and use standard, non-microlocal, energy estimates for wave operators to show the invertibility of the forward problem on sufficiently weighted spaces; using the microlocal regularity results of \S\S\ref{SecEllipticReg} and \ref{SecPropagation}, we will in fact show higher regularity and the existence of partial expansions of forward solutions.

%%%%%%%%%%%%%%%%%%%%%%%%%%%%%%%%%%%%%%%%%%%%%%%%%
\subsection{Energy estimates}
\label{SubsecEnergy}

Let $(M,g)$ be a compact manifold with boundary e\-quipped with a Lorentzian b-metric $g$ satisfying
\begin{equation}
\label{EqNonsmoothMetric}
  g\in\CI(M;\Sym^2\Tb^*M)+\Hb^s(M;\Sym^2\Tb^*M),
\end{equation}
with $s\in\R$ to be specified later, where the b-Sobolev space here is defined using an arbitrary fixed smooth b-density on $M$. Let $U\subset M$ be open, and suppose $\ft\colon U\to(t_0,t_1)$ is a proper function such that $d\ft$ is timelike on $U$. We consider the operator
\[
  P=\Box_g+L,\quad L\in(\CI+\Hb^{s-1})\Diffb^1+(\CI+\Hb^{s-2}).
\]

\begin{rmk}
\label{RmkBoxOnBundles}
  Although all arguments in this section are presented for $P$ and $\Box_g$ acting on functions, the results are true for $P$ and $\Box_g$ acting on natural vector bundles as well, e.g.\ the bundle of $q$-forms; only minor, mostly notational, changes are needed to verify this.
\end{rmk}

For $s>n/2$, one obtains using Lemma~\ref{LemmaRecHs} and Corollary~\ref{CorHbModule} that in any coordinate system the coefficients $G^{ij}$ of the dual metric $G$ are elements of $\CI+\Hb^s$, and all Christoffel symbols are elements of $\CI+\Hb^{s-1}$. Therefore, by definition of $\Box_g$, one easily obtains that
\[
  \Box_g\in(\CI+\Hb^s)\Diffb^2+(\CI+\Hb^{s-1})\Diffb^1,
\]
thus
\begin{equation}
\label{EqPMembership}
  P\in(\CI+\Hb^s)\Diffb^2+(\CI+\Hb^{s-1})\Diffb^1+(\CI+\Hb^{s-2}).
\end{equation}

\begin{prop}
\label{PropEnergyWave}
  Let $t_0<T_0<T_0'<T_1<t_1$ and $r\in\R$, and suppose $s>n/2+2$. Then there exists a constant $C>0$ such that for all $u\in\Hb^{2,r}(M)$, the following estimate holds:
  \[
    \|u\|_{\Hb^{1,r}(\ft^{-1}([T_0',T_1]))}\leq C(\|Pu\|_{\Hb^{0,r}(\ft^{-1}([T_0,T_1]))}+\|u\|_{\Hb^{1,r}(\ft^{-1}([T_0,T_0']))}).
  \]
  This also holds with $P$ replaced by $P^*$. If one replaces $C$ by any $C'>C$, the estimate also holds for small perturbations of $P$ in the space indicated in \eqref{EqPMembership}.
\end{prop}
\begin{proof}
  Let us work in a coordinate system $z_1=x,z_2=y_1,\ldots,z_n=y_{n-1}$, where $x$ is a boundary defining function in case we are working near the boundary. By piecing together estimates from coordinate patches, one can deduce the full result. Write $\bpa_j=\pa_{z_j}$ for $2\leq j\leq n$, and $\bpa_1=x\pa_x$ if we are working near the boundary, $\bpa_1=\pa_x$ otherwise. Moreover, let us fix the \emph{Riemannian} b-metric
  \[
    \wt g=\frac{dx^2}{x^2}+dy^2
  \]
  near the boundary, $\wt g=dx^2+dy^2$ away from it. \emph{We adopt the summation convention in this proof.}
  
  We will imitate the proof of \cite[Proposition~3.8]{VasyMicroKerrdS}, which proves a similar result in a smooth, semiclassical setting. Thus, consider the commutant $V=-iZ$, where $Z=x^{-2r}\chi(\ft)W$ with $\chi\in\CI(\R)$, chosen later in the proof, and $W=G(-,\bdiff\ft)$, which is timelike in $U$. We will compute the `commutator'
  \begin{equation}
  \label{EqEnergyComm}
    -i(V^*P-P^*V)=-i(V^*\Box_g-\Box_g^*V)-iV^*L+iL^*V,
  \end{equation}
  where the adjoints are taken with respect to the (b-)metric $\wt g$. First, we need to make sense of all appearing operator compositions. Notice that $V\in x^{-2r}(\CI+\Hb^s)\Diffb^1$, and writing $V=-iZ^j\bpa_j$, we get
  \[
    V^*=-i\bpa_j Z^j=V-i(\bpa_j Z^j) \in x^{-2r}(\CI+\Hb^s)\Diffb^1+x^{-2r}(\CI+\Hb^{s-1}),
  \]
  similarly
  \[
    \Box_g,\Box_g^*,P^*\in(\CI+\Hb^s)\Diffb^2+(\CI+\Hb^{s-1})\Diffb^1+(\CI+\Hb^{s-2});
  \]
  now, since 
  \[
    (\CI+\Hb^{s-j})\Diffb^j (\CI+\Hb^{s-k})\Diffb^k \subset \sum_{l\leq j}(\CI+\Hb^{s-j}\Hb^{s-k-l})\Diffb^{j+k-l},
  \]
  it suffices to require $s>n/2+2$, since then $\Hb^{s-j}\Hb^{s-k-j}\subset\Hb^{s-k-j}$ for $0\leq j,k\leq 2$, $0\leq j+k\leq 3$.
  
  Returning to the computation of \eqref{EqEnergyComm}, we conclude that $-i(V^*\Box_g-\Box_g^*V)\in(\CI+\Hb^{s-3,-2r})\Diffb^2$, and thus its principal symbol is defined. Since it is a formally self-adjoint (with respect to $\wt g$) operator with real coefficients that vanishes on constants, it equals $\bdiff^*C\bdiff$ provided the principal symbols are equal. To compute it, we adopt the computation by Vasy \cite[\S3]{VasyWaveOnAdS} to the present context: Let us write
  \[
    -i(V^*\Box_g-\Box_g^*V)=-(\bpa_k Z^k)\Box_g+i[\Box_g,V]-i(\Box_g-\Box_g^*)V.
  \]
  We define $S^i\in\CI+\Hb^{s-1}$ by
  \[
    \sigma_\bl^2(-i(\Box_g-\Box_g^*)V)=2S^iZ^j\zeta_i\zeta_j=(S^iZ^j+S^jZ^i)\zeta_i\zeta_j.
  \]
  Moreover, with $H_G$ denoting the Hamilton vector field of the dual metric of $g$,
  \[
    H_G=G^{ij}\zeta_i\bpa_j+G^{ij}\zeta_j\bpa_i-(\bpa_k G^{ij})\zeta_i\zeta_j\pa_{\zeta_k},
  \]
  we find $\sigma_\bl^2(-i(V^*\Box_g-\Box_g^*V))=B^{ij}\zeta_i\zeta_j$ with
  \begin{align*}
    B^{ij}&=-\bpa_k(Z^k G^{ij})+G^{ik}(\bpa_k Z^j)+G^{jk}(\bpa_k Z^i)+S^iZ^j+S^jZ^i \\
	  &\hspace{50ex} \in x^{-2r}(\CI+\Hb^{s-1}),
  \end{align*}
  thus
  \[
    -i(V^*\Box_g-\Box_g^*V)=\bdiff^*C\bdiff, \quad C_i^j=B^{ij}.
  \]
  Let us now plug $Z=x^{-2r}\chi W$ into the definition of $B^{ij}$ and separate the terms with derivatives falling on $\chi$, the idea being that the remaining terms, considered error terms, can then be dominated by choosing $\chi'$ large compared to $\chi$. We get
  \begin{equation}
  \label{EqBEndomorphism}
  \begin{split}
    B^{ij}&=x^{-2r}(\bpa_k\chi)(G^{ik}W^j+G^{jk}W^i-G^{ij}W^k) \\
	  &\hspace{5ex} +\chi\bigl(G^{ik}(\bpa_k x^{-2r} W^j)+G^{jk}(\bpa_k x^{-2r} W^i) \\
	  &\hspace{22ex} -\bpa_k(x^{-2r} W^k G^{ij})+x^{-2r}(S^iW^j+S^jW^i)\bigr).
  \end{split}
  \end{equation}
  Notice here that for a b-1-form $\omega\in\CI(M;\Tb^*M)$, the quantity
  \begin{align*}
    E_{W,\bdiff\chi}(\omega)&:=\frac{1}{2}(\bpa_k\chi)(G^{ik}W^j+G^{jk}W^i-G^{ij}W^k)\omega_i\overline{\omega_j} \\
	  &=\frac{1}{2}\bigl[(\omega,\bdiff\chi)_G\overline{\omega(W)}+\omega(W)(\bdiff\chi,\omega)_G-\bdiff\chi(W)(\omega,\omega)_G\bigr] \\
	  &=\chi'(\ft)E_{W,\bdiff\ft}(\omega)
  \end{align*}
  is related to the sesquilinear energy-momentum tensor
  \[
    E_{W,\bdiff\ft}(\omega)=\Re\bigl((\omega,\bdiff\ft)_G\overline{\omega(W)}\bigr)-\frac{1}{2}\bdiff\ft(W)(\omega,\omega)_G,
  \]
  where $(\cdot,\cdot)_G$ is the sesquilinear inner product on $\CTb^*M$. This quantity, rewritten in terms of b-vector fields as
  \[
    E_{X,Y}(\omega)=\Re(\omega(X)\overline{\omega(Y)})-\frac{1}{2}\la X,Y\ra(\omega,\omega)_G,
  \]
  is well-known to be positive definite provided $X$ and $Y$ are both future (or both past) timelike, see e.g.\ Alinhac \cite{AlinhacGeoHypPDE}. In our setting, we thus have $E_{W,\bdiff\ft}=E_{W,W}>0$ by our definition of $W$. Correspondingly,
  \begin{equation}
  \label{EqCEndomorphism}
    C=x^{-2r}\chi' A+x^{-2r}\chi R
  \end{equation}
  with $A$ positive definite and $R$ symmetric.

  We obtain\footnote{The integrations by parts here and further below are readily justified using $s>n/2+2$: In fact, since we are assuming $u\in\Hb^{2,r}$, we have $Vu\in\Hb^{1,-r}$ for $s>n/2,s\geq 1$, and then $P^*Vu\in\Hb^{-1,-r}$ provided multiplication with an $\Hb^{s-j}$ function is continuous $\Hb^1\to\Hb^{1-j}$ for $j=0,1,2$, which is true for $s>n/2+1$; similarly, one has $Pu\in\Hb^{0,r}$ provided $s>n/2$, and then $V^*Pu\in\Hb^{-1,-r}$ if multiplication by an $\Hb^{s-j}$ function is continuous $\Hb^{-j}\to\Hb^{-1}$ for $j=0,1$, which holds for $s>n/2,s\geq 1$.}
  \begin{equation}
  \label{EqWaveComm}
    \la-i(V^*P-P^*V)u,u\ra=\la C\bdiff u,\bdiff u\ra-\la iLu,Vu\ra+\la iVu,Lu\ra.
  \end{equation}
  We now finish the proof by making $\chi'$ large compared to $\chi$ on $\ft^{-1}([T_0',T_1])$, as follows: Pick $T_1'\in(T_1,t_1)$ and let
  \[
    \wt\chi(s)=\wt\chi_1\left(\frac{s-T_0}{T_0'-T_0}\right)\chi_0(-\F^{-1}(s-T_1')),\quad \chi(s)=\wt\chi(s)H(T_1-s),
  \]
  where $H$ is the Heaviside step function, $\chi_0(s)=e^{-1/s}H(s)\in\CI(\R)$ (which satisfies $\chi_0'(s)=s^{-2}\chi_0(s)$) and $\wt\chi_1\in\CI(\R)$ equals $0$ on $(-\infty,0]$ and $1$ on $[1,\infty)$; see Figure~\ref{FigCommutant}.

  \begin{figure}[!ht]
    \includegraphics{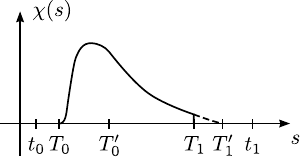}
    \caption{Graph of the commutant $\chi$. The dashed line is the graph of the part of $\wt\chi$ that is cut off using the Heaviside function in the definition of $\chi$.}
    \label{FigCommutant}
  \end{figure}
  
  Then in $(T_0',T_1')$,
  \begin{align*}
    \chi'(s)&=-\F^{-1}\chi_0'(-\F^{-1}(s-T_1'))H(T_1-s)-\chi_0(-\F^{-1}(T_1-T_1'))\delta_{T_1} \\
	  &=-\F(s-T_1')^{-2}\chi(s)-\chi_0(-\F^{-1}(T_1-T_1'))\delta_{T_1},
  \end{align*}
  in particular $\chi(s)=-\F^{-1}(s-T_1')^2\chi'(s)$ on $(T_0',T_1)$; hence for any $\gamma>0$, we can choose $\F>0$ so large that $\chi\leq-\gamma\chi'$ on $(T_0',T_1)$; therefore
  \[
    -(\chi' A+\chi R)\geq -\frac{1}{2}\chi'\wt\chi_1 A\tn{ on }(T_0',T_1).
  \]
  Put $\chi_1(s)=\wt\chi_1(s)H(T_1-s)$, then
  \begin{align*}
    -\la C&\bdiff u,\bdiff u\ra \geq\frac{1}{2}\la x^{-2r}(-\chi'\chi_1)A\bdiff u,\bdiff u\ra \\
	  &\hspace{-3ex}+\chi_0(-\F^{-1}(T_1-T_1'))\la x^{-2r} A\delta_{T_1}\bdiff u,\bdiff u\ra-C'\|\bdiff u\|_{\Hb^{0,r}(\ft^{-1}([T_0,T_0']))}^2,
  \end{align*}
  and the term on the right hand side involving $\delta_{T_1}$ is positive, thus can be dropped. Hence, using equation~\eqref{EqWaveComm} and the positivity of $A$,
  \begin{equation}
  \label{EqWaveControl}
  \begin{split}
    c_0\|&\sqrt{-\chi'\chi_1}\bdiff u\|_{\Hb^{0,r}}^2 \leq \frac{1}{2}\la x^{-2r}(-\chi'\chi_1)A\bdiff u,\bdiff u\ra \\
	  &\leq C'\|\bdiff u\|_{\Hb^{0,r}(\ft^{-1}([T_0,T_0']))}^2+C'\|\chi^{1/2}Pu\|_{\Hb^{0,r}}\|\chi^{1/2}\bdiff u\|_{\Hb^{0,r}} \\
	  & \hspace{10ex} + C'\|\chi^{1/2}\bdiff u\|_{\Hb^{0,r}}^2+C'\|\chi^{1/2}\bdiff u\|_{\Hb^{0,r}}\|\chi^{1/2}u\|_{\Hb^{0,r}} \\
	  & \leq C''\|u\|_{\Hb^{1,r}(\ft^{-1}([T_0,T_0']))}^2+C'\|\chi^{1/2}Pu\|_{\Hb^{0,r}}^2+C'\gamma\|\sqrt{-\chi'\chi_1}\bdiff u\|_{\Hb^{0,r}}^2 \\
	  & \hspace{10ex} + C'\gamma\|\sqrt{-\chi'\chi_1}u\|^2_{\Hb^{0,r}},
  \end{split}
  \end{equation}
  where the norms are on $\ft^{-1}([T_0,T_1])$ unless otherwise specified. Choosing $\F$ large and thus $\gamma$ small allows us to absorb the second to last term on the right into the left hand side. To finish the proof, we need to treat the last term, as follows: We compute, using $W\chi=\chi' G(\bdiff\ft,\bdiff\ft)\equiv m\chi'$ with $m\in\CI+\Hb^s$ positive,
  \begin{align*}
    \la&(W^*x^{-2r}\chi+x^{-2r}\chi W)u,u\ra \\
	  &= -\la(W x^{-2r}\chi-x^{-2r}\chi W)u,u\ra-\la(\dv_{\wt g}W) x^{-2r}\chi u,u\ra \\
	  &\geq -\la x^{-2r} m\chi'u,u\ra_{\Hb^0(\ft^{-1}([T_0,T_1]))}-\la x^{-2r} w\chi u,u\ra-\la(\dv_{\wt g}W) x^{-2r}\chi u,u\ra \\
	  &\geq \|\sqrt{-\chi'\chi_1}m^{1/2} u\|_{\Hb^{0,r}(\ft^{-1}([T_0',T_1]))}^2-\|\sqrt{|\chi'|}m^{1/2}u\|_{\Hb^{0,r}(\ft^{-1}([T_0,T_0']))}^2 \\
	  &\hspace{10ex} -C\|\sqrt{\chi}u\|_{\Hb^{0,r}(\ft^{-1}([T_0,T_1]))}^2,
  \end{align*}
  where $w=x^{2r}Wx^{-2r}\in\CI+\Hb^s$. Similarly as above, we now choose $\F$ large to obtain
  \begin{align*}
    \|\sqrt{-\chi'\chi_1}&u\|_{\Hb^{0,r}(\ft^{-1}([T_0',T_1]))}^2 \\
	  &\leq C\|\sqrt{\chi}\bdiff u\|_{\Hb^{0,r}}^2+C\|\sqrt{\chi+|\chi'|}u\|_{\Hb^{0,r}(\ft^{-1}([T_0,T_0']))}^2.
  \end{align*}
  Using this estimate in \eqref{EqWaveControl} and absorbing one of the resulting terms, namely $\gamma\|\sqrt{\chi}\bdiff u\|_{\Hb^{0,r}(\ft^{-1}([T_0',T_1]))}^2$, into the left hand side finishes the proof of the estimate, since $\sqrt{-\chi'\chi_1}$ has a positive lower bound on $\ft^{-1}([T_0',T_1])$.

  That the estimate holds for perturbations of $P$ follows simply from the observation that all constants in this proof depend on finitely many seminorms of the coefficients of $P$, hence the constants only change by small amounts if one makes a small perturbation of $P$.
\end{proof}

%%%%%%%%%%%%%%%%%%%%%%%%%%%%%%%%%%%%%%%%%%%%%%%%%
\subsection{Analytic, geometric and dynamical assumptions on non-smooth linear problems}
\label{SubsecHypotheses}

The arguments of the first half of \cite[\S{2.1.3}]{HintzVasySemilinear} leading to a Fredholm framework for the forward problem for certain $P$, e.g.\ wave operators on non-smooth perturbations of the static model of de Sitter space, now go through with only minor technical modifications. Because there are large di\-men\-sion-de\-pen\-dent losses in estimates for the adjoint of $P$ relative to the regularity of the coefficients of $P$, say $\CI+\Hb^s$ for the highest order ones, the spaces that $P$ acts on as a Fredholm operator are roughly of the order $s-n/2$.

This can be vastly improved with a calculus for right quantizations of non-smooth symbols just like the one developed in this paper for left quantizations. Right quantizations have `good' mapping properties on \emph{negative} order (but lossy ones on positive order) b-Sobolev spaces. Correspondingly, all microlocal results (elliptic regularity, propagation of singularities, including at radial points) hold by the same proofs \emph{mutatis mutandis}. Then, viewing $P^*$ as the right quantization of a non-smooth symbol gives estimates which allow one to put $P$ into a Fredholm framework on spaces with regularity $s-\eps$, $\eps>0$.
 
Our focus here however is to prove the \emph{invertibility} of the forward problem, whose discussion in the second half of \cite[\S{2.1.3}]{HintzVasySemilinear} (in the smooth setting) we follow.

We now describe the general setup; a concrete example to keep in mind for the remainder of the section is the wave operator on a (perturbed) static asymptotically de Sitter space. Let us from now assume that the operator
\[
  P=\Box_g+L,\quad L\in(\CI+\Hb^{s-1,\alpha})\Diffb^1+(\CI+\Hb^{s-1,\alpha}),
\]
with $\alpha>0$, and $g$ now satisfying
\[
  g\in\CI(M;\Sym^2\Tb^*M)+\Hb^{s,\alpha}(M;\Sym^2\Tb^*M),
\]
is such that:
\begin{enumerate}
  \item \label{EnumHypPDyn} $P$ satisfies the dynamical assumptions of Theorem~\ref{ThmRadialPoints}, i.e.\ has the indicated radial point structure. Let $L_\pm$, $\wt\beta$, $\wh\beta$ be defined as in the statement of Theorem~\ref{ThmRadialPoints},
  \item \label{EnumHypPCoeffs} $P\in(\CI+\Hb^{s,\alpha})\Diffb^2+(\CI+\Hb^{s-1,\alpha})\Diffb^1+(\CI+\Hb^{s-1,\alpha})$; note that the regularity of the lowest order term is higher than what we assumed before,
  \item \label{EnumHypCharpm} the characteristic set $\Sigma$ of $P$ has the form $\Sigma=\Sigma_+\cup\Sigma_-$ with $\Sigma_\pm$ a union of connected components of $\Sigma$, and $L_\pm\subset\Sigma_\pm$.
\setcounter{CounterEnumi}{\value{enumi}}
\end{enumerate}
We denote by $\ft_1$ and $\ft_2$ two smooth functions on $M$ and put for $\delta_1,\delta_2$ small
\begin{gather*}
  \Omega_{\delta_1,\delta_2}:=\ft_1^{-1}([\delta_1,\infty))\cap\ft_2^{-1}([\delta_2,\infty)),\quad \Omega\equiv\Omega_{0,0}, \\
  \Omega_{\delta_1,\delta_2}^\circ:=\ft_1^{-1}((\delta_1,\infty))\cap\ft_2^{-1}((\delta_2,\infty)),
\end{gather*}
where we assume that:
\begin{enumerate}
\setcounter{enumi}{\value{CounterEnumi}}
  \item \label{EnumHypFt} The differentials of $\ft_1$ and $\ft_2$ have the opposite timelike character near their respective zero sets within $\Omega$, more specifically, $d\ft_1$ is future timelike, $d\ft_2$ past timelike,
  \item \label{EnumHypArtBdy} putting $H_j:=\ft_j^{-1}(0)$, the $H_j$ intersect the boundary $\pa M$ transversally, and $H_1$ and $H_2$ intersect only in the interior of $M$, and they do so transversally,
  \item \label{EnumHypCompact} $\Omega_{\delta_1,\delta_2}$ is compact.
\setcounter{CounterEnumi}{\value{enumi}}
\end{enumerate}

Let us make two additional assumptions:
\begin{enumerate}
\setcounter{enumi}{\value{CounterEnumi}}
  \item \label{EnumHypTimelike} Assume that there is a boundary defining function $x$ of $M$ such that $dx/x$ is timelike everywhere on $\Omega$ with timelike character opposite to the one of $\ft_1$, i.e.\ $dx/x$ is past oriented.
  \item \label{EnumHypNontrapping} The metric $g$ is \emph{non-trapping} in the following sense: All bicharacteristics in $\Sigma_\Omega:=\Sigma\cap\Sb^*_\Omega M$ from any point in $\Sigma_\Omega\cap(\Sigma_+\setminus L_+)$ flow (within $\Sigma_\Omega$) to $\Sb^*_{H_1}M\cup L_+$ in the forward direction (i.e.\ either enter $\Sb^*_{H_1}M$ in finite time or tend to $L_+$) and to $\Sb^*_{H_2}M\cup L_+$ in the backward direction, and from any point in $\Sigma_\Omega\cap(\Sigma_-\setminus L_-)$ to $\Sb^*_{H_2}M\cup L_-$ in the forward direction and to $\Sb^*_{H_1}M\cup L_-$ in the backward direction.
\end{enumerate}

See Figure~\ref{FigOmega} for the setup.
\begin{figure}[!ht]
\includegraphics{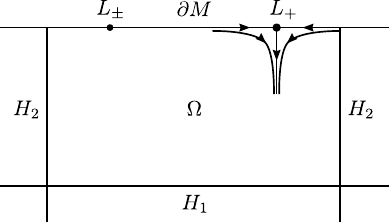}
\caption{The domain $\Omega$ on which we have a global energy estimate as well as solvability and uniqueness on appropriate weighted b-Sobolev spaces. The `artificial' spacelike boundary hypersurfaces $H_1$ and $H_2$ are also indicated.}
\label{FigOmega}
\end{figure}

Conditions~\itref{EnumHypPDyn} and \itref{EnumHypNontrapping} are (probably) not stable under perturbations of $P$, and it will in fact be crucial later that they can be relaxed. Namely, we do not need to require that null-bicharacteristics of a small perturbation $\wt P$ of $P$ tend to $L_\pm$, but only that they reach a fixed neighborhood of $L_\pm$, since then Theorem~\ref{ThmRadialPoints} is still applicable to $\wt P$, see Remark~\ref{RmkRelaxedRadial}; and this condition \emph{is} stable under perturbations.

Denote by $\Hb^{s,r}(\Omega_{\delta_1,\delta_2})^{\bullet,-}$ distributions which are supported ($\bullet$) at the `artificial' boundary hypersurface $\ft_1^{-1}(\delta_1)$ and extendible ($-$) at $\ft_2^{-1}(\delta_2)$, and the other way around for $\Hb^{s,r}(\Omega_{\delta_1,\delta_2})^{-,\bullet}$. Then we have the following \emph{global} energy estimate:
\begin{lemma}
\label{LemmaWaveGlobalEnergy}
  (Cf.\ \cite[Lemma~2.15]{HintzVasySemilinear}.) Suppose $s>n/2+2$. There exists $r_0<0$ such that for $r\leq r_0$, $-\wt r\leq r_0$, there is $C>0$ such that for $u\in\Hb^{2,r}(\Omega_{\delta_1,\delta_2})^{\bullet,-}$, $v\in\Hb^{2,\wt r}(\Omega_{\delta_1,\delta_2})^{-,\bullet}$, one has
  \begin{align*}
    \|u\|_{\Hb^{1,r}(\Omega_{\delta_1,\delta_2})^{\bullet,-}} & \leq C \|Pu\|_{\Hb^{0,r}(\Omega_{\delta_1,\delta_2})^{\bullet,-}}, \\
	\|v\|_{\Hb^{1,\wt r}(\Omega_{\delta_1,\delta_2})^{-,\bullet}} & \leq C \|P^*v\|_{\Hb^{0,\wt r}(\Omega_{\delta_1,\delta_2})^{-,\bullet}}.
  \end{align*}
  If one replaces $C$ by any $C'>C$, the estimates also hold for small perturbations of $P$ in the space indicated in assumption~\itref{EnumHypPCoeffs}.
\end{lemma}
\begin{proof}
  The proof follows \cite[Lemma~2.15]{HintzVasySemilinear}, adapted to the non-smooth setting as in Proposition~\ref{PropEnergyWave}, the point being that the terms in \eqref{EqBEndomorphism} with $x^{-2r}$ differentiated and thus possessing a factor of $r$ can be used to dominate the other, `error,' terms in \eqref{EqCEndomorphism}. We require $s>n/2+2$ here in order for the arguments in Proposition~\ref{PropEnergyWave} to apply.
\end{proof}

\begin{rmk}
  For this lemma we in fact only need to assume conditions~\itref{EnumHypFt}-\itref{EnumHypTimelike}.
\end{rmk}

By a duality argument, we thus obtain solvability; furthermore, the results on microlocal elliptic regularity and the propagation of singularities proved in \S\S\ref{SecEllipticReg} and \ref{SecPropagation} give us a robust way to show higher regularity:
\begin{lemma}
\label{LemmaWaveGlobalSolve}
  (Cf.\ \cite[Corollaries~2.10 and 2.16]{HintzVasySemilinear}.) Let $0\leq s'\leq s$ and assume $s>n/2+6$. There exists $r_0<0$ such that for $r\leq r_0$, there is $C>0$ with the following property: If $f\in\Hb^{s'-1,r}(\Omega)^{\bullet,-}$, then there exists a unique $u\in\Hb^{s',r}(\Omega)^{\bullet,-}$ such that $Pu=f$, and $u$ moreover satisfies
  \[
    \|u\|_{\Hb^{s',r}(\Omega)^{\bullet,-}} \leq C\|f\|_{\Hb^{s'-1,r}(\Omega)^{\bullet,-}}.
  \]
  If one replaces $C$ by any $C'>C$, this result also holds for small perturbations of $P$ in the space indicated in assumption~\itref{EnumHypPCoeffs}.
\end{lemma}

\begin{rmk}
  One can prove this Lemma without appealing to microlocal regularity results by commuting b-vector fields through the equation $Pu=f$. However, we use this opportunity to show in the simplest setting how the microlocal results are used before applying them in a more substantial way in the proof of Theorem~\ref{ThmExistenceAndExpansion} below. Moreover, the proof given here immediately shows that one can choose $r_0$ independently of $s',s$.
\end{rmk}

\begin{proof}[Proof of Lemma~\ref{LemmaWaveGlobalSolve}.]
  We follow the proofs of \cite[Lemmas~2.7 and 2.9]{HintzVasySemilinear}. Choose $\delta_1<0$ and $\delta_2<0$ small, and choose an extension
  \[
    \wt f\in\Hb^{s'-1,r}(\Omega_{0,\delta_2})^{\bullet,-}\subset\Hb^{-1,r}(\Omega_{0,\delta_2})^{\bullet,-}
  \]
  satisfying
  \begin{equation}
  \label{EqFExtension}
    \|\wt f\|_{\Hb^{s'-1,r}(\Omega_{0,\delta_2})^{\bullet,-}}\leq 2\|f\|_{\Hb^{s'-1,r}(\Omega)^{\bullet,-}}.
  \end{equation}
  By Lemma~\ref{LemmaWaveGlobalEnergy}, applied with $\wt r=-r$, we have
  \[
    \|\phi\|_{\Hb^{1,\wt r}(\Omega_{0,\delta_2})^{-,\bullet}}\leq C\|P^*\phi\|_{\Hb^{0,\wt r}(\Omega_{0,\delta_2})^{-,\bullet}}
  \]
  for $\phi\in\Hb^{2,\wt r}(\Omega_{0,\delta_2})^{-,\bullet}$. Therefore, by the Hahn-Banach theorem, there exists $\wt u\in\Hb^{0,\wt r}(\Omega_{0,\delta_2})^{\bullet,-}$ such that
  \[
    \la P\wt u,\phi\ra=\la\wt u,P^*\phi\ra=\la\wt f,\phi\ra,\quad \phi\in\Hb^{2,\wt r}(\Omega_{0,\delta_2})^{-,\bullet},
  \]
  and
  \begin{equation}
  \label{EqSolOnBiggerSetBound}
    \|\wt u\|_{\Hb^{0,\wt r}(\Omega_{0,\delta_2})^{\bullet,-}}\leq C\|\wt f\|_{\Hb^{-1,\wt r}(\Omega_{0,\delta_2})^{\bullet,-}}.
  \end{equation}
  We can view $\wt u$ as an element of $\Hb^{0,\wt r}(\Omega_{\delta_1,\delta_2})^{\bullet,-}$ with support in $\Omega_{0,\delta_2}$, similarly for $\wt f$; then $\la P\wt u,\phi\ra=\la\wt f,\phi\ra$ for all $\phi\in\CIdotc(\Omega_{\delta_1,\delta_2}^\circ)$ (with the dot referring to infinite order of vanishing at $\pa M$), i.e.\ $P\wt u=\wt f$ as distributions on $\Omega_{\delta_1,\delta_2}^\circ$.

  Now, $\wt u$ vanishes on $\Omega_{\delta_1,\delta_2}^\circ\setminus\Omega_{0,\delta_2}$, in particular is in $\Hbloc^{s',r}$ there. Elliptic regularity and the propagation of singularities, Theorems~\ref{ThmEllipticReg}, \ref{ThmRealPrType} and \ref{ThmRadialPoints}, imply that $\wt u\in\Hbloc^{s',r}(\Omega_{\delta_1,\delta_2}^\circ)$. Indeed, by Theorem~\ref{ThmEllipticReg} with $\wt s=-1$, $\wt u$ is in $\Hb^{1/2,r}$ on the elliptic set of $P$ within $\Omega_{\delta_1,\delta_2}^\circ$; Theorem~\ref{ThmRealPrType} with $\wt s=-1/2$ gives $\Hb^{1/2,r}$-control of $\wt u$ on the characteristic set away from radial points, and then an application of Theorem~\ref{ThmRadialPoints} gives $\Hb^{1/2,r}$-control of $\wt u$ on all of $\Omega_{\delta_1,\delta_2}^\circ$.\footnote{The conditions of all theorems used here are satisfied because of $s>n/2+6$; if necessary, we need to make $r_0$ smaller, i.e.\ assume that $r\leq r_0$ is more negative, in order for the assumptions of Theorem~\ref{ThmRadialPoints} to be fulfilled. Strictly speaking, we in fact need to use localized estimates in the following sense: If $\wt u\in\Hb^{\wt s,r}$ and $P\wt u\in\Hb^{\wt s-1/2,r}$, and if $\chi\in\CIc(\Omega_{\delta_1,\delta_2}^\circ)$ is identically $1$ near a point $x_0$, then $P\chi\wt u=\chi P\wt u+[P,\chi]\wt u$ is in $\Hb^{\wt s-1/2,r}$ in a neighborhood of $x_0$ and globally in $\Hb^{\wt s-1,r}$, since $[P,\chi]$ is a first order operator. By inspection of the relevant theorems, in particular \eqref{EqPropEstimate}, this regularity suffices to apply the relevant microlocal regularity results and deduce microlocal $\Hb^{\wt s+1/2,r}$-regularity of $\wt u$.} Iterating this argument gives $\Hbloc^{s',r}(\Omega_{\delta_1,\delta_2})^\circ$, and we in fact get an estimate
  \[
    \|\chi\wt u\|_{\Hb^{s',r}(\Omega_{\delta_1,\delta_2})}\leq C\bigl(\|\wt\chi P\wt u\|_{\Hb^{s'-1,r}(\Omega_{\delta_1,\delta_2})}+\|\wt\chi\wt u\|_{\Hb^{0,r}(\Omega_{\delta_1,\delta_2})}\bigr)
  \]
  for appropriate $\chi,\wt\chi\in\CI_\cl(\Omega_{\delta_1,\delta_2}^\circ)$, $\wt\chi\equiv 1$ on $\supp\chi$. In view of the support properties of $\wt u$, an appropriate choice of $\chi$ and $\wt\chi$ gives that the restriction of $\wt u$ to $\Omega$ is an element of $\Hb^{s',r}(\Omega)^{\bullet,-}$, with norm bounded by the $\Hb^{s'-1,r}(\Omega)^{\bullet,-}$-norm of $f$ in view of \eqref{EqSolOnBiggerSetBound} and \eqref{EqFExtension}.

  To prove uniqueness, suppose $u\in\Hb^{s',r}(\Omega)^{\bullet,-}$ satisfies $Pu=0$, then, viewing $u$ as a distribution on $\Omega_{\delta_1,0}^\circ$ with support in $\Omega$, elliptic regularity and the propagation of singularities, applied as above, give $u\in\Hbloc^{s,r}(\Omega_{\delta_1,0}^\circ)\subset\Hbloc^{2,r}(\Omega_{\delta_1,0}^\circ)$; hence, for any $\wt\delta>0$, Lemma~\ref{LemmaWaveGlobalEnergy} applied to $u'=u|_{\Omega_{0,\wt\delta}}\in\Hb^{2,r}(\Omega_{0,\wt\delta})^{\bullet,-}$ gives $u'=0$, thus, since $\wt\delta>0$ is arbitrary, $u=0$.
\end{proof}

\begin{cor}
\label{CorWaveGlobalCont}
  (Cf.\ \cite[Corollary~2.17]{HintzVasySemilinear}.) Let $0\leq s'\leq s$ and assume $s>n/2+6$. There exists $r_0<0$ such that for $r\leq r_0$, there is $C>0$ with the following property: If $u\in\Hb^{s',r}(\Omega)^{\bullet,-}$ is such that $Pu\in\Hb^{s'-1,r}(\Omega)^{\bullet,-}$, then
  \[
    \|u\|_{\Hb^{s',r}(\Omega)^{\bullet,-}} \leq C\|Pu\|_{\Hb^{s'-1,r}(\Omega)^{\bullet,-}}.
  \]
  If one replaces $C$ by any $C'>C$, this result also holds for small perturbations of $P$ in the space indicated in assumption~\itref{EnumHypPCoeffs}.
\end{cor}
\begin{proof}
  Let $u'\in\Hb^{s',r}(\Omega)^{\bullet,-}$ be the solution of $Pu'=Pu$ given by the existence part Lemma~\ref{LemmaWaveGlobalSolve}, then $P(u-u')=0$, and the uniqueness part implies $u=u'$.
\end{proof}

We also obtain the following propagation of singularities type result:
\begin{cor}
\label{CorWavePropSing}
  Let $0\leq s''\leq s'\leq s$ and assume $s>n/2+6$; moreover, let $r\in\R$ be such that $s''-1+\inf_{L_\pm}(\wh\beta-r\wt\beta)>0$. Then there is $C>0$ such that the following holds: Any $u\in\Hb^{s'',r}(\Omega)^{\bullet,-}$ with $Pu\in\Hb^{s'-1,r}(\Omega)^{\bullet,-}$ in fact satisfies $u\in\Hb^{s',r}(\Omega)^{\bullet,-}$, and obeys the estimate
  \[
    \|u\|_{\Hb^{s',r}(\Omega)^{\bullet,-}}\leq C(\|Pu\|_{\Hb^{s'-1,r}(\Omega)^{\bullet,-}}+\|u\|_{\Hb^{s'',r}(\Omega)^{\bullet,-}}).
  \]
  If one replaces $C$ by any $C'>C$, this result also holds for small perturbations of $P$ in the space indicated in assumption~\itref{EnumHypPCoeffs}.
\end{cor}
\begin{proof}
  As in the proof of Lemma~\ref{LemmaWaveGlobalSolve}, working on $\Omega_{\delta_1,0}$ for $\delta_1<0$ small, we obtain $u\in\Hbloc^{s',r}$ by iteratively using elliptic regularity, real principal type propagation and the propagation near radial points; the latter, applied in the first step with $\wt s=s''-1/2$, is the reason for the condition on $s''$. Thus, $u\in\Hb^{s',r}(\Omega_{0,\wt\delta})^{\bullet,-}$ for $\wt\delta>0$. From here, arguing as in the proof of Proposition~2.13 in \cite{HintzVasySemilinear}, we obtain the desired conclusion.
\end{proof}

Let us rephrase Lemma~\ref{LemmaWaveGlobalSolve} and Corollary~\ref{CorWaveGlobalCont} as an invertibility statement:
\begin{thm}
\label{ThmWaveGlobalSolveCont}
  (Cf.\ \cite[Theorem~2.18]{HintzVasySemilinear}.) Let $0\leq s'\leq s$ and assume $s>n/2+6$. There exists $r_0<0$ with the following property: Let $r\leq r_0$ and define the spaces
  \[
    \cX^{s,r}=\{u\in\Hb^{s,r}(\Omega)^{\bullet,-}\colon Pu\in\Hb^{s-1,r}(\Omega)^{\bullet,-}\},\quad \cY^{s,r}=\Hb^{s,r}(\Omega)^{\bullet,-}.
  \]
  Then $P\colon\cX^{s,r}\to\cY^{s-1,r}$ is a continuous, invertible map with continuous inverse.
  
  Moreover, the operator norm of the inverse, as a map from $\Hb^{s-1,r}(\Omega)^{\bullet,-}$ to $\Hb^{s,r}(\Omega)^{\bullet,-}$, of small perturbations of $P$ in the space indicated in assumption~\itref{EnumHypPCoeffs} is uniformly bounded.
\end{thm}

We can now apply the arguments of \cite{HintzVasySemilinear}, see also \cite{VasyMicroKerrdS} for the dilation-invariant case, to obtain more precise asymptotics of solutions $u$ to $Pu=f$ using the knowledge of poles of the inverse of the Mellin transformed normal operator family $\widehat P(\sigma)$, where the normal operator $N(P)$ of $P$ is defined just as in the smooth setting by `freezing' the coefficients of $P$ at the boundary $\pa M$. This makes sense in our setting since the coefficients of $P$ are continuous; also, the coefficients of $N(P)$ are then smooth, since all non-smooth contributions to $P$ vanish at the boundary.

\begin{thm}
\label{ThmExistenceAndExpansion}
  (Cf.\ \cite[Theorem~2.20]{HintzVasySemilinear}.) Let $s>n/2+6$, $0<\alpha<1$, and assume $g\in\CI(M;\Sym^2\Tb^*M)+\Hb^{s,\alpha}(M;\Sym^2\Tb^*M)$. Let
  \[
    P=\Box_g+L, \quad L\in(\CI+\Hb^{s-1,\alpha})\Diffb^1+(\CI+\Hb^{s-1,\alpha}).
  \]
  Further, let $\ft_1$ and $\Omega\subset M$ be as above, and suppose $P$, $\Omega$ and $g$ satisfy the assumptions~\itref{EnumHypPDyn}-\itref{EnumHypNontrapping} above. Let $\sigma_j$ be the poles of $\wh P(\sigma)^{-1}$, of which there are only finitely many in any half space $\Im\sigma\geq-C$ by assumption~\itref{EnumHypTimelike}.\footnote{See the arguments leading to \cite[Theorem~7.5]{VasyMicroKerrdS} for an explanation.} Let $r\in\R$ be such that $r\neq\Im\sigma_j$ and $r\leq-\Im\sigma_j+\alpha$ for all $j$, and let $r_0\in\R$. Moreover, let $1\leq s_0\leq s'\leq s$, and suppose that
  \[
    s'-2+\inf_{L_\pm}(\wh\beta-r\wt\beta)>0.
  \]
  Finally, let $\phi\in\CI(\R)$ be such that $\supp\phi\subset(0,\infty)$ and $\phi\circ\ft_1\equiv 1$ near $\pa M\cap\Omega$.

  If we \emph{assume} that the poles $\sigma_j$ or $\wh P(\sigma)^{-1}$ are simple, then any solution $u\in\Hb^{s_0,r_0}(\Omega)^{\bullet,-}$ of $Pu=f$ with $f\in\Hb^{s'-1,r}(\Omega)^{\bullet,-}$ satisfies
  \begin{equation}
  \label{EqExpansion}
    u-\sum_j x^{i\sigma_j}(\phi\circ\ft_1)a_j=u'\in\Hb^{s',r}(\Omega)^{\bullet,-}
  \end{equation}
  for some $a_j\in\CI(\pa M\cap\Omega)$, where the sum is understood over the finite set of $j$ such that $-\Im\sigma_j<r<-\Im\sigma_j+\alpha$.

  The result is stable under small perturbations of $P$ in the space indicated in assumption~\itref{EnumHypPCoeffs} in the sense that, even though the $\sigma_j$ might change, all $\CI$-seminorms of the expansion terms $a_j$ and the $\Hb^{s',r}(\Omega)^{\bullet,-}$-norm of the remainder term $u'$ are bounded by $C(\|u\|_{\Hb^{s_0,r_0}(\Omega)^{\bullet,-}}+\|f\|_{\Hb^{s'-1,r}(\Omega)^{\bullet,-}})$ for some uniform constant $C$ (depending on which norm we are bounding).

  In general, \emph{without the simplicity assumption}, the expansion \eqref{EqExpansion} includes log terms $x^{i\sigma_j}|\log x|^k$, $0\leq k\leq m_j<\infty$, with $(m_j+1)$ equal to the order of the pole of $\wh P(\sigma)^{-1}$ at $\sigma=\sigma_j$.
\end{thm}

\begin{proof}
  By making $r_0$ smaller (i.e.\ more negative) if necessary, we may assume that $r_0\leq r$ and
  \[
    s_0-1+\inf_{L_\pm}(\wh\beta-r_0\wt\beta)>0.
  \]
  First, assume $\sigma_*:=\min_j\{-\Im\sigma_j\}>r$. Then $u\in\Hb^{s_0,r_0}(\Omega)^{\bullet,-}$ and $Pu=f\in\Hb^{s'-1,r}(\Omega)^{\bullet,-}$ imply $u\in\Hb^{s',r_0}(\Omega)^{\bullet,-}$ by Corollary~\ref{CorWavePropSing}. Since
  \[
    P-N(P)\in (x\CI+\Hb^{s,\alpha})\Diffb^2+(x\CI+\Hb^{s-1,\alpha})\Diffb^1+(x\CI+\Hb^{s-1,\alpha}),
  \]
  we thus obtain $\wt f:=(P-N(P))u\in\Hb^{s'-2,r_0+\alpha}(\Omega)^{\bullet,-}$, where we use $s\geq s'-2$ and $s-1\geq s'-1$; hence
  \[
    N(P)u=f-\wt f\in\Hb^{s'-2,r'}(\Omega)^{\bullet,-}
  \]
  with $r'=\min(r,r_0+\alpha)$. Applying\footnote{This requires $s'\geq 1$ in view of the supported/extendible spaces that we are using here; see also \cite[Footnote~28]{HintzVasySemilinear}.} \cite[Lemma~3.1]{VasyMicroKerrdS} gives $u\in\Hb^{s'-1,r'}(\Omega)^{\bullet,-}$ in view of the absence of poles of $\widehat P(\sigma)$ in $\Im\sigma\geq -r$; but then $Pu\in\Hb^{s'-1,r}(\Omega)^{\bullet,-}$ implies $u\in\Hb^{s',r'}(\Omega)^{\bullet,-}$, again by Corollary~\ref{CorWavePropSing}, where we use
  \[
    (s'-1)-1+\inf(\wh\beta-r'\wt\beta)\geq s'-2+\inf(\wh\beta-r\wt\beta)>0.
  \]
  If $r'=r$, we are done; otherwise, we iterate, replacing $r_0$ by $r_0+\alpha$, and obtain $u\in\Hb^{s',r}(\Omega)^{\bullet,-}$ after finitely many steps.

  If there are $\sigma_j$ with $-\Im\sigma_j<r$, then, assuming that $\sigma_*-\alpha<r_0<\sigma_*$, in fact that $r_0$ is arbitrarily close to $\sigma_*$, as we may by the first part of the proof, the application of \cite[Lemma~3.1]{VasyMicroKerrdS} gives a partial expansion $u_1$ of $u$ with remainder $u'\in\Hb^{s'-1,r'}(\Omega)^{\bullet,-}$, where $r'=\min(r,r_0+\alpha)$. Now $N(P)u_1=0$, and $u_1$ is a sum of terms of the form $a_j x^{i\sigma_j}$ with $\Im\sigma_j\leq-\sigma_*$ and $a_j\in\CI(\pa M\cap\Omega)$, in particular $u_1\in\Hb^{\infty,r_0}(\Omega)^{\bullet,-}$; thus
  \begin{equation}
  \label{EqPNPu1}
    (P-N(P))u_1\in\Hb^{\infty,r_0+1}(\Omega)^{\bullet,-}+\Hb^{s-1,\sigma_*+\alpha}(\Omega)^{\bullet,-}\subset\Hb^{s-1,\sigma_*+\alpha}(\Omega)^{\bullet,-},
  \end{equation}
  where the two terms correspond to the coefficients of $P-N(P)$ being sums of $x\CI$- and $\Hb^{s-1,\alpha}$-functions. Therefore,
  \begin{equation}
  \label{EqPuprime}
    Pu'=Pu-N(P)u_1-(P-N(P))u_1\in\Hb^{s'-1,r}(\Omega)^{\bullet,-},
  \end{equation}
  which by Corollary~\ref{CorWavePropSing} implies $u'\in\Hb^{s',r'}(\Omega)^{\bullet,-}$, finishing the proof in the case that $r'=r$, i.e.\ $r<\sigma_*+\alpha$. If $r=\sigma_*+\alpha$, we need one more iterative step to establish the improvement in the weight of $u'$: We use $u'\in\Hb^{s',r'}$ to deduce
  \[
    N(P)u=f-(P-N(P))u \in \Hb^{s'-1,r} + \Hb^{s'-2,r'+\alpha} + \Hb^{s-1,\sigma_*+\alpha} \subset \Hb^{s'-2,\sigma_*+\alpha},
  \]
  where we use $(P-N(P))u'\in\Hb^{s'-2,r'+\alpha}$ and \eqref{EqPNPu1}. Hence \cite[Lemma~3.1]{VasyMicroKerrdS} implies that the partial expansion $u=u_1+u'$ in fact holds with $u'\in\Hb^{s'-1,r}$, and then Corollary~\ref{CorWavePropSing} and \eqref{EqPuprime} imply $u'\in\Hb^{s',r}$, finishing the proof in the case $r=\sigma_*+\alpha$.
\end{proof}

\begin{rmk}
\label{RmkNonsmoothExpansionLimits}
  In the smooth setting, one can use the partial expansion $u_1$ to obtain better information on $\wt f$ for a next step in the iteration. This however relies on the fact that $P-N(P)\in x\Diffb^2$ there (see the proof of \cite[Theorem~2.20]{HintzVasySemilinear}); here, however, we also have terms in the space $\Hb^{s-1,\alpha}\Diffb^2$ in $P-N(P)$, and $\Hb^{s-1}$-functions do not have a Taylor expansion at $x=0$, hence the above iteration scheme does not yield additional information after the first step in which one gets a non-trivial part $u_1$ of the expansion of $u$.
  
  If however we encode more precise asymptotics in the function space in which $g$ lies, then $P-N(P)$ has a partial polyhomogeneous expansion which can be used to obtain more precise asymptotics for $u$. See also Remark~\ref{RmkPreciseExpansion}.
\end{rmk}

Combining Theorem~\ref{ThmExistenceAndExpansion} with Theorem~\ref{ThmWaveGlobalSolveCont} gives us a forward solution operator for $P$ which, provided we understand the poles of $\widehat P(\sigma)^{-1}$, will be the key tool in our discussion of quasilinear wave equations in the next section.

%%%%%%%%%%%%%%%%%%%%%%%%%%%%%%%%%%%%%%%%%%%%%%%%%%%%%%%%%%%%%%%%%%%%%
\section{Quasilinear wave and Klein-Gordon equations on the static model of de Sitter space}
\label{SecApp}

%%%%%%%%%%%%%%%%%%%%%%%%%%%%%%%%%%%%%%%%%%%%%%%%%
\subsection{The static model of de Sitter space}
\label{SubsecStaticdS}

In order to introduce the static model of de Sitter space, we start by considering $(n+1)$-dimensional Minkowski space $\R^{n+1}_z$ with metric $g_0:=dz_{n+1}^2-dz_1^2-\cdots-dz_n^2$. Then $n$-dimensional de Sitter space is the one-sheeted hyperboloid
\[
  \wt M^\circ=\{ z_{n+1}^2-z_1^2-\cdots-z_n^2 = -1 \}
\]
with metric $g$ induced by $g_0$; thus, $g$ has signature $(+,-,\ldots,-)$. Moreover, $\wt M^\circ$ inherits the usual time orientation from the ambient Minkowski space, in which $\pa_{z_{n+1}}$ is future timelike. We can introduce global coordinates using the map $\R_{z_{n+1}}\times\Sph^{n-1}_\theta, \quad (z_{n+1},\theta)\mapsto ((1+z_{n+1}^2)^{1/2}\theta,z_{n+1})\in\R^{n+1}$, and the metric becomes
\[
  g = \frac{dz_{n+1}^2}{1+z_{n+1}^2} - (1+z_{n+1}^2)d\theta^2
\]
We compactify $\wt M^\circ$, first at future infinity by introducing $x=z_{n+1}^{-1}$ in $z_{n+1}\geq 1$, say, so the metric becomes
\begin{equation}
\label{EqGeometryDeSitterGlobal0Metric}
  g = x^{-2}\left(\frac{dx^2}{1+x^2} - (1+x^2)d\theta^2\right) =: x^{-2}\bar g,
\end{equation}
where $\bar g$ is a smooth Lorentzian metric down to $x=0$, and likewise at past infinity; thus, we have compactified $\wt M^\circ$ to a cylinder
\[
  \wt M\cong[-1,1]_T\times\Sph^{n-1},
\]
say with $T=1-x$ near $x=0$, and $T=0$ at $z_{n+1}=0$. The metric $g$ is a so-called \emph{0-metric}, see \cite{MazzeoMelroseHyp}. Null-geodesics of $g$ are merely reparametrizations of null-geodesics of the conformally related metric $\bar g$. From the point of view of causality, one can localize the study of (nonlinear) wave equations on de Sitter space $\wt M$ by picking a point $q$ at future infinity, say $T=1$, $\theta=e_1\in\Sph^{n-1}\subset\R^n$, considering only the interior $M_S^\circ$ of the backward light cone from $q$, intersected with $\{T\geq 0\}$ for convenience; we call $M_S^\circ$ the \emph{static model of de Sitter space}.\footnote{Strictly speaking, this is only the future half of the static model; the full static model is the intersection of the interior of the backward light cone from $(T=1,\theta=e_1)$ with the forward light cone from $(T=-1,\theta=e_1)$, see \cite[\S4]{VasyMicroKerrdS} for details.} For an illustration, see Figure~\ref{FigGeometryStaticdS}.

\begin{figure}[!ht]
  \centering
  \includegraphics{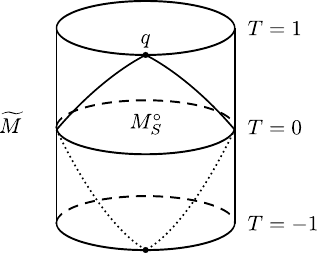}
  \caption[Static model of de Sitter space]{The `future half' of the static model $M_S^\circ$ of de Sitter space, a submanifold of (compactified) de Sitter space $\wt M$, is the backward light cone from the point $q$ at future infinity, intersected with $T\geq 0$. The full static model is the intersection of the interiors of the backward light cone from $q$ and the corresponding point at past infinity.}
  \label{FigGeometryStaticdS}
\end{figure}

We make this explicit in the coordinates $z_1,\ldots,z_{n+1}$ of the ambient Minkowski space: Namely, for each fixed $\omega\in\Sph^{n-2}\subset\R^{n-1}_{z_2,\ldots,z_n}$, the affine curve
\[
  \gamma_\omega(z_{n+1})=(z_{n+1},\omega;z_{n+1}) \in \wt M^\circ\subset\R^{1+(n-1)+1}
\]
is a geodesic on de Sitter space $\wt M^\circ$, and written in the coordinates $x=z_{n+1}^{-1}$, $\theta=(z_1^2+\cdots+z_n^2)^{-1/2}(z_1,\ldots,z_n)\in\Sph^{n-1}$ introduced above, it is equal to
\[
  \gamma_\omega(z_{n+1}) = (x,\theta(x)),\quad \theta(x)=(1+x^2)^{-1/2}(e_1+x\omega).
\]
Thus, we see that the family $\{\gamma_\omega\colon\omega\in\Sph^{n-2}\}$ sweeps out the backward light cone from the point $x=0$, $\theta=e_1$, thus is the boundary of the static model $M_S^\circ$. In other words, in the Minkowskian coordinates,
\[
  M_S^\circ=\{(z_1,\ldots,z_{n+1})\in \wt M^\circ\colon z_{n+1}\geq 0, z_2^2+\cdots+z_n^2<1\}.
\]
The backward light cone is a \emph{cosmological horizon} for $M_S^\circ$: Any causal (timelike or null) future-oriented curve in $\wt M^\circ$, starting at a point in $M_S^\circ$, which crosses the cosmological horizon, can never return to $M_S^\circ$.

On $M_S^\circ$, one can choose \emph{static coordinates} $t\in\R,Z\in\R^{n-1}$, $|Z|<1$, with respect to which the metric $g$ is $t$-independent, and writing $Z=r\omega$, $r\in(0,1)$, $\omega\in\Sph^{n-2}$, away from $Z=0$, one has
\[
  g = (1-r^2)\,dt^2 - (1-r^2)^{-1}\,dr^2 - r^2\,d\omega^2.
\]
One can compactify this at future infinity using $\wt\tau=e^{-t}$ as a boundary defining function; the coordinate singularity of the metric at $r=1$ can then be resolved by means of a suitable blow up of the corner $\wt\tau=0,r=1$, with $g$ extending smoothly and non-degenerately past the front face of the blow-up near the side face $\wt\tau=0$; see \cite[\S4]{VasyMicroKerrdS} for details.

We describe a different way of arriving at such a smooth extension of the static metric past the cosmological horizon. First, we recall the relation of hyperbolic space $\HH^n=\{z_{n+1}^2-z_1^2-\cdots-z_n^2=1,z_{n+1}>0\}$ as a subset of Minkowski space to the upper half plane model: Define the global coordinate chart
\begin{gather*}
  \Phi\colon\HH^n \ni (z_1,z_2,\ldots,z_n,z_{n+1}) \mapsto (x,y), \\
  x=\frac{2}{z_1+z_{n+1}}\in(0,\infty), \quad y=\frac{2(z_2,\ldots,z_n)}{z_1+z_{n+1}}\in\R^{n-2},
\end{gather*}
then the induced metric on $\HH^n$ takes the simple form $x^{-2}(dx^2+dy^2)$. The above map $\Phi$ is in fact well-defined on $\{z_1+z_{n+1}>0\}$, and restricting $\Phi$ to $\wt M^\circ\cap\{z_1+z_{n+1}>0\}$, the metric on $\wt M^\circ$ has the form
\[
  g=x^{-2}(dx^2-dy^2).
\]
Moreover, the point $q$ at future infinity singled out above has coordinates $x=0,y=0$, where we extended $\Phi$ by continuity to $\wt M\cap\{z_1+z_{n+1}>0\}$, and the backward light cone from $q$ is simply the set $\{|y|=x\}$; the static model, compactified at future infinity, therefore is
\[
  M_S=\{|y|<x,x\geq 0\}.
\]
Blowing up $(0,0)$ spherically, we can introduce coordinates $\tau=x\in[0,1),Z=y/x\in\R^{n-1}$, $|Z|<1$ near the interior of the front face, with respect to which
\begin{equation}
\label{EqGeometryStaticdSMetric}
  g = (1-|Z|^2)\,\frac{d\tau^2}{\tau^2} - 2Z\frac{d\tau}{\tau}\,dZ - dZ^2 = \frac{d\tau^2}{\tau^2} - \Bigl(Z\frac{d\tau}{\tau}+dZ\Bigr)^2
\end{equation}
which extends non-degenerately as a Lorentzian b-metric past the cosmological horizon $|Z|=1$. Moreover, $\tau$ and $1/z_{n+1}$ are comparable (i.e.\ bounded by constant multiples of each other) near $q$, and in terms of the static time coordinate $t$, we have $\tau\sim e^{-t}$ over compact subsets of $M_S$, i.e.\ away from the cosmological horizon. In fact, we can define a new time coordinate $t_*$ by
\[
  \tau = e^{-t_*},
\]
which is thus smooth on $M_S^\circ$ \emph{up to} (and beyond) the cosmological horizons.

The dual metric of $g$ is
\begin{equation}
\label{EqGeometryStaticdSDualMetricGlob}
  G = (Z\pa_Z-\tau\pa_\tau)^2 - \pa_Z^2.
\end{equation}
Concretely, with $r=|Z|$ and $\omega=r^{-1}Z$, we introduce $\mu=1-r^2$ as a defining function of $r=1$, and compute
\begin{equation}
\label{EqGeometryStaticdSDualMetricPolar}
  G = -4\mu r^2\pa_\mu^2 + 4r^2 \tau\pa_\tau\,\pa_\mu + (\tau\pa_\tau)^2 - r^{-2}\pa_\omega^2,
\end{equation}
valid away from $r=0$, which extends non-degenerately to $\mu<0$; the Lorentzian geometry at $\mu=0$ is crucial for our analysis, since $\Box_g$ has radial points at the conormal bundle to $\mu=0$ within $\tau=0$, as we will verify below. Let us therefore use the coordinates $\mu\in(-\delta,1),\omega\in\Sphere^{n-2},\tau\in[0,\infty)$, for small $\delta>0$, on physical space near the cosmological horizon $\mu=0$ and the natural coordinates in the fiber of the b-cotangent bundle, which come from writing b-covectors as
\[
  \xi\,d\mu+\eta\,d\omega+\sigma\,\frac{d\tau}{\tau}.
\]
Moreover, we write $K$ for the dual metric on the round sphere; in a coordinate system on the sphere, its components are denoted $K^{ij}$. We shall also have occasion to use the coordinates $Z=r\omega\in\R^{n-1}$ and $\tau$, valid near $r=0$, with b-covectors written
\[
  \zeta\,dZ+\sigma\frac{d\tau}{\tau}.
\]
Then the quadratic form associated with the dual metric $G$ of the static de Sitter metric $g$, which is the same as the b-principal symbol of $P:=\Box_g$, is given by
\begin{equation}
\label{EqDSDualMetric}
\begin{split}
  p=\sigma_\bl^2(P)&=-4r^2\mu\xi^2+4r^2\sigma\xi+\sigma^2-r^{-2}|\eta|_K^2 \\
    &=(Z\cdot\zeta-\sigma)^2-|\zeta|^2.
\end{split}
\end{equation}
Correspondingly, the Hamilton vector field is
\begin{equation}
\label{EqDSHamilton}
  \begin{split}
  H_p&=(\pa_\xi p)\pa_\mu-(\pa_\mu p)\pa_\xi+(\pa_\sigma p)\tau\pa_\tau-(\tau\pa_\tau p)\pa_\sigma-r^{-2}H_{|\eta|_K^2} \\
    &=4r^2(-2\mu\xi+\sigma)\pa_\mu-(4\xi^2(1-2r^2)-4\sigma\xi-r^{-4}|\eta|^2_K)\pa_\xi \\
	&\qquad\qquad +(4r^2\xi+2\sigma)\tau\pa_\tau-r^{-2}H_{|\eta|_K^2} \\
	&=2(Z\cdot\zeta-\sigma)(Z\pa_Z-\zeta\pa_\zeta-\tau\pa_\tau)-2\zeta\cdot\pa_Z.
  \end{split}
\end{equation}

We aim to show that $P$ fits into the framework of \S\ref{SecGlobalForSecond}, so that Theorems~\ref{ThmWaveGlobalSolveCont} and \ref{ThmExistenceAndExpansion} apply to $P$ and non-smooth perturbations of it. Since the metric $g$ is $\tau$-independent, the computations are very similar to those performed by Vasy \cite{VasyMicroKerrdS} in the Mellin transformed picture; also, in \cite[\S{2}]{HintzVasySemilinear}, it is used, even if not explicitly stated, that $P$ does fit into the smooth framework there, but we will provide all details here for the sake of completeness.

Denote the characteristic set of $p$ by $\Sigma=p^{-1}(0)\subset\Tb^*M\wozero$.
\begin{lemma}
  $\Sigma$ is a smooth conic codimension 1 submanifold of $\Tb^*M\wozero$ transversal to $\Tb^*_Y M$.
\end{lemma}
\begin{proof}
  We have to show that $dp\neq 0$ whenever $p=0$. We compute
  \begin{align*}
    dp&=(4\xi^2(1-2r^2)-4\sigma\xi-r^{-4}|\eta|^2_K)d\mu+4r^2(-2\mu\xi+\sigma)d\xi \\
	  &\quad +(4(1-\mu)\xi+2\sigma)d\sigma-r^{-2}d|\eta|_K^2.
  \end{align*}
  Thus if $dp=0$, all coefficients have to vanish, thus $\sigma=2\mu\xi$ and $\sigma=2(\mu-1)\xi$, giving $\xi=0$ and thus $\sigma=0$, hence also $\eta=0$. Thus $dp$ vanishes only at the zero section of $\Tb^*M$ in this coordinate system. In the coordinates valid near $r=0$, we compute
  \[
    dp=2(Z\cdot\zeta-\sigma)\zeta\cdot dZ+2\bigl((Z\cdot\zeta-\sigma)Z-2\zeta\bigr)\cdot d\zeta-2(Z\cdot\zeta-\sigma)\,d\sigma,
  \]
  thus $dp=0$ implies $Z\cdot\zeta=\sigma$, hence $\zeta=0$ and then $\sigma=0$. Thus, the first half of the statement is proved. The transversality is clear since $dp$ and $d\tau$ are linearly independent at $\Sigma$ by inspection.
\end{proof}

We will occasionally also use $\Sigma$ to denote the characteristic set viewed as a subset of the radially compactified b-cotangent bundle $\rcTb^*M$ or as a subset of the boundary $\Sb^* M$ of $\rcTb^*M$ at fiber infinity.

%%%%%%%%%%%%%%%%%%%%%%%%%%%%%
\subsubsection{Radial points}
\label{SubsubsecDSRadialPoints}

Since $g$ is a Lorentzian b-metric, the Hamilton vector field $H_p$ cannot be radial except at the boundary $Y=\pa M$ at future infinity, where $\tau=0$. In the coordinate system near $r=0$, one easily checks using \eqref{EqDSHamilton} that there are no radial points over $Z=0$. At radial points, we then moreover have $H_p\mu=4r^2(-2\mu\xi+\sigma)=0$, thus $\sigma=2\mu\xi$. Further, we compute
\[
  H_{|\eta|_K^2}=H_{\eta_i K^{ij}(\omega)\eta_j}=2\eta_j K^{ij}(\omega)\pa_{\omega_i}-2\eta_i(\pa_{\omega_k}K^{ij})\eta_j\pa_{\eta_k}.
\]
The coefficient of $\pa_{\omega_i}$ must vanish for all $i$, which implies $\eta=0$, since $K$ is non-degenerate. Now, if $\xi=0$, then $\sigma=0$, i.e. all fiber variables vanish and we are outside the characteristic set $\Sigma$; thus $\xi\neq 0$. At points where $\sigma=2\mu\xi,\eta=0,\tau=0$, the expression for $p$ simplifies to $4r^2\mu\xi^2+4\mu^2\xi^2=4\mu\xi^2$, which does not vanish unless $\mu=0$. Hence, $\mu=0,\tau=0,\eta=0,\sigma=0$, and we easily check that these conditions are also sufficient for a point in this coordinate patch to be a radial point. Thus:
\begin{lemma}
\label{LemmaDSRadial}
  The set of radial points of $\Box_g$ is a disjoint union $\cR=\cR_+\cup\cR_-$, where
  \begin{align*}
    \cR_\pm&=\{\mu=0,\tau=0,\eta=0,\sigma=0,\pm\xi>0\} \\
	  &=\{\tau=0,\sigma=0,Z=\mp\zeta/|\zeta|\}\subset\Sigma.
  \end{align*}
\end{lemma}

To analyze the flow near $L_\pm:=\pa\cR_\pm\subset\Sb^*M$, we introduce normalized coordinates
\[
  \wh\rho=\frac{1}{\xi},\wh\eta=\frac{\eta}{\xi},\wh\sigma=\frac{\sigma}{\xi}
\]
and consider the homogeneous degree 0 vector field $\sfH_p:=|\wh\rho|H_p$. Following \cite[\S3]{BaskinVasyWunschRadMink}, we get a good qualitative understanding of the dynamics near $L_\pm$ by looking at the linearization $W$ of $\pm\sfH_p=\wh\rho H_p$; note that $\la\xi\ra^{-1}$ is a defining function of the boundary of $\rcTb^*M$ at fiber infinity near $L_\pm$. The coordinate vector fields in the new coordinate system are
\[
  \pa_\eta=\wh\rho\pa_{\wh\eta},\quad \xi\pa_\xi=-\wh\rho\pa_{\wh\rho}-\wh\eta\pa_{\wh\eta}-\wh\sigma\pa_{\wh\sigma}.
\]
Hence
\begin{align*}
  \wh\rho H_p&=4r^2(-2\mu+\wh\sigma)\pa_\mu+(4(1-2r^2)-4\wh\sigma-r^{-4}|\wh\eta|_K^2)(\wh\rho\pa_{\wh\rho}+\wh\eta\pa_{\wh\eta}+\wh\sigma\pa_{\wh\sigma}) \\
    &\qquad +(4r^2+2\wh\sigma)\tau\pa_\tau-r^{-2}\wh\rho H_{|\eta|_K^2}.
\end{align*}
We have $\wh\rho H_p\in\Vb(\rcTb^*M)$, i.e. it is tangent to the boundary $\wh\rho=0$ at fiber infinity, and to the boundary of $M$, given by $\tau=0$. Since $\wh\rho H_p$ vanishes at a radial point $q\in\rcTb^*M$, it maps the ideal $\cI$ of functions in $C^\infty(\rcTb^*M)$ vanishing at $q$ into itself. The linearization of $\wh\rho H_p$ at $q$ then is the vector field $\wh\rho H_p$ acting on $\cI/\cI^2\cong T^*_q\rcTb^*M$, where the isomorphism is given by $f+\cI^2\mapsto df|_q$. Computing the linearization $W$ of $\wh\rho H_p$ at $q$ now amounts to ignoring terms of $\wh\rho H_p$ that vanish to at least second order at $q$, which gives
\[
  W=4(-2\mu+\wh\sigma)\pa_\mu-4(\wh\rho\pa_{\wh\rho}+\wh\eta\pa_{\wh\eta}+\wh\sigma\pa_{\wh\sigma})+4\tau\pa_\tau-2K^{ij}(\omega)\wh\eta_j\pa_{\omega_i}.
\]
We read off the eigenvectors and corresponding eigenvalues:
\begin{align*}
  d\wh\rho,d\wh\eta,d\wh\sigma & \tn{ with eigenvalue }-4, \\
  d\mu-d\wh\sigma & \tn{ with eigenvalue }-8, \\
  d\tau & \tn{ with eigenvalue }+4, \\
  d\omega_i-\tfrac{1}{2}K^{ij}d\wh\eta_j & \tn{ with eigenvalue } 0.
\end{align*}
Thus, $L_+$ ($L_-$) is a sink (source) of the Hamilton flow within $\Sb^*_Y M$, with an unstable (stable) direction normal to the boundary. More precisely, the $\tau$-independence of the metric suggests the definition
\[
  \cL_\pm=\pa\{\mu=0,\sigma=0,\eta=0,\pm\xi>0\}\subset\Sb^* M,
\]
so that $L_\pm=\Sb^*_Y M\cap\cL_\pm$; moreover $\cL_\pm\subset\Sigma$, and $\sfH_p$ is tangent to $\cL_\pm$; indeed,
\begin{equation}
\label{EqHamOnUnstableMf}
  H_p=4\xi^2\pa_\xi+4\xi\tau\pa_\tau\tn{ at }\cL_\pm.
\end{equation}
Lastly, $\cL_+$ ($\cL_-$) is indeed the unstable (stable) manifold at $L_\pm$. Now, going back to the full rescaled Hamilton vector field $\sfH_p$, we have at $L_\pm$ (in fact, at $\cL_\pm$):
\[
  |\wh\rho|^{-1}\sfH_p|\wh\rho|=\mp\beta_0, \quad -\tau^{-1}\sfH_p\tau=\mp\wt\beta\beta_0
\]
with $\beta_0=4$ and $\wt\beta=1$; furthermore, near $\cL_\pm$,
\[
  \mp\sfH_p\wh\eta=4\wh\eta,\quad \mp\sfH_p\wh\sigma=4\wh\sigma,\quad \mp\sfH_p(\mu-\wh\sigma)=8(\mu-\wh\sigma)
\]
modulo terms that vanish quadratically at $\cL_\pm$, hence, putting $\beta_1=8$, the quadratic defining function $\rho_0:=\wh\eta^2+\wh\sigma^2+(\mu-\wh\sigma)^2$ of $\cL_\pm$ within $\Sigma$ satisfies
\[
  \mp\sfH_p\rho_0-\beta_1\rho_0\geq 0
\]
modulo terms that vanish cubically at $\cL_\pm$.

We have thus verified the geometric and dynamical assumptions~\itref{EnumRadCond1}-\itref{EnumRadCond5} in \S\ref{SubsecRadialPoints} regarding the characteristic set and the Hamilton flow of $p$ near the radial set. Note that assumption~\itref{EnumRadCond5} is automatic here with $\wh\beta=0$, since $P$ is formally self-adjoint with respect to the metric b-density. In other words, we have verified assumption~\itref{EnumHypPDyn} of \S\ref{SubsecHypotheses}.

%%%%%%%%%%%%%%%%%%%%%%%%%%%%%
\subsubsection{Global behavior of the characteristic set}
\label{SubsubsecDSCharGlobal}

The next assumption to be checked is \itref{EnumHypCharpm} in \S\ref{SubsecHypotheses}. This is easily accomplished: Indeed, from \eqref{EqDSDualMetric}, we have
\begin{equation}
\label{EqDSDualRewrite}
  p=(\sigma+2r^2\xi)^2-4r^2\xi^2-r^{-2}|\eta|_K^2,
\end{equation}
and thus $\Sigma=\Sigma_+\cup\Sigma_-$, where
\begin{equation}
\label{EqDSSigmapm}
  \Sigma_\pm=\{\pm(\sigma+2r^2\xi)>0\}\cap\Sigma=\{\pm(\sigma-Z\cdot\zeta)>0\}
\end{equation}
since $p=0$, $\sigma+2r^2\xi=0$ implies $\xi=\eta=0$, thus $\sigma=0$, thus $\{\sigma+2r^2\xi=0\}$ does not intersect the characteristic set $p^{-1}(0)$, and similarly in the $(\tau,Z;\sigma,\zeta)$ coordinates. Moreover, we have $\cL_\pm\subset\Sigma_\pm$ by definition of $\cL_\pm$.

We proceed with a description of the domain $\Omega\subset M$ with artificial boundaries $H_1$ and $H_2$, which have defining functions $\ft_1,\ft_2$, and check the assumptions~\itref{EnumHypFt}-\itref{EnumHypNontrapping} in \S\ref{SubsecHypotheses}. We first observe that $G\bigl(\frac{d\tau}{\tau},\frac{d\tau}{\tau}\bigr)=1>0$. Now, pick any $\delta>0$ and $\tau_0>0$ and define
\[
  \ft_1=\tau_0-\tau,\quad \ft_2=\mu+\delta.
\]
Then
\begin{gather*}
  G(\bdiff\ft_1,\bdiff\ft_1)|_{\ft_1=0}=G\Bigl(-\tau\frac{d\tau}{\tau},-\tau\frac{d\tau}{\tau}\Bigr)|_{\tau=\tau_0}=\tau_0^2>0, \\
  G(\bdiff\ft_2,\bdiff\ft_2)|_{\ft_2=0}=G(d\mu,d\mu)|_{\mu=-\delta}=4\delta(1+\delta)>0, \\
  G(\bdiff\ft_1,\bdiff\ft_2)|_{\ft_1=\ft_2=0}=-4(1+\delta)\tau_0<0,
\end{gather*}
thus $\bdiff\ft_1$ and $\bdiff\ft_2$ are timelike with opposite timelike character; indeed, with the usual time orientation on de Sitter space (namely where $-d\tau/\tau$ is future oriented), $\ft_1$ is future oriented and $\ft_2$ is past oriented, as is $d\tau/\tau$. Moreover, $d\ft_2$ and $d\tau$ are clearly linearly independent at $Y\cap H_2$, as are $d\ft_1$ and $d\ft_2$ at $H_1\cap H_2$. Thus, the assumptions~\itref{EnumHypFt}-\itref{EnumHypTimelike} are verified.

It remains to check the non-trapping assumption~\itref{EnumHypNontrapping}. Let us first analyze the flow in $\Tb^*_\Omega M\setminus\Tb^*_Y M$; notice that $H_p\tau=0$ in $\Tb^*_Y M$, thus bicharacteristics that intersect $\Tb^*_Y M$ are in fact contained in $\Tb^*_Y M$, and correspondingly bicharacteristics containing points in $\Tb^*_\Omega M\setminus\Tb^*_Y M$ stay in $\Tb^*_\Omega M\setminus\Tb^*_Y M$. There,
\begin{equation}
\label{EqDSHtau}
  \pm H_p\tau=\pm 2(\sigma+2r^2\xi)\tau>0\tn{ on }\Sigma_\pm.
\end{equation}
In particular, in $\Sigma_\pm\setminus\Tb^*_Y M$, bicharacteristics reach $\Tb^*_{H_1}M$ (i.e.\ $\tau=\tau_0$) in finite time in the forward ($+$), resp.~backward ($-$), direction. We show that they stay within $\Tb^*_\Omega M$: For this, observe that $p=0$ and $\mu<0$, thus $r>1$, imply
\[
  2|\xi|\leq 2r|\xi|\leq|\sigma+2r^2\xi|
\]
by equation~\eqref{EqDSDualRewrite}. In fact, if $\xi\neq 0$, the first inequality is strict, and if $\xi=0$, the second inequality is strict, and we conclude the strict inequality
\[
  2|\xi|<|\sigma+2r^2\xi|\tn{ if }p=0, \mu<0.
\]
Hence, on $(\Sigma_\pm\setminus\Tb^*_Y M)\cap\Sigma_\Omega$, if $\mu<0$, then
\begin{equation}
\label{EqDSHmu}
  \pm H_p\mu=\pm 4r^2(\sigma+2r^2\xi-2\xi)>0,
\end{equation}
thus in the forward (on $\Sigma_+$), resp.~backward (on $\Sigma_-$), direction, bicharacteristics cannot cross $\Tb^*_{H_2}M=\{\mu=-\delta\}$.

Next, backward, resp.~forward, bicharacteristics in $\cL_\pm\setminus L_\pm$ tend to $L_\pm$ by equation~\eqref{EqDSHtau}, since $H_p$ is tangent to $\cL_\pm$, and $L_\pm=\cL_\pm\cap\{\tau=0\}$; in fact, by equation~\eqref{EqHamOnUnstableMf}, more is true, namely these bicharacteristics, as curves in $\rcTb^*M\wozero$, tend to $L_\pm$ if the latter is considered a subset of the boundary $\Sb^*M$ of $\rcTb^*M$ at fiber infinity. Now, consider a backward, resp.~forward, bicharacteristic $\gamma$ in $(\Sigma_\pm\setminus\cL_\pm)\cap\Tb^*_\Omega M$, including those within $\Tb^*_Y M$. By \eqref{EqDSHtau}, $\tau$ is non-increasing along $\gamma$, and by \eqref{EqDSHmu}, $\mu$ is strictly decreasing along $\gamma$ once $\gamma$ enters $\mu<0$, hence it then reaches $\Tb^*_{H_2}M$ in finite time, staying within $\Tb^*_\Omega M$. We have to show that $\gamma$ necessarily enters $\mu<0$ in finite time. Assume this is not the case. Then observe that
\begin{equation}
\label{EqDSHEscape}
  \mp H_p(\sigma-Z\cdot\zeta)=\mp 2|\zeta|^2=\mp 2(\sigma-Z\cdot\zeta)^2\tn{ on }\Sigma_\pm,
\end{equation}
thus $\sigma-Z\cdot\zeta$ converges to $0$ along $\gamma$. Now on $\Sigma$, $|\zeta|=|\sigma-Z\cdot\zeta|$, thus, also $\zeta$ converges to $0$, and moreover, on $\Sigma$, we have
\[
  |\sigma|\leq |Z\cdot\zeta|+|Z\cdot\zeta-\sigma|\leq(1+|Z|)|\zeta|
\]
since we are assuming $|Z|\leq 1$ on $\gamma$, hence $\sigma$ converges to $0$ along $\gamma$. But $H_p\sigma=0$, i.e.\ $\sigma$ is constant. Thus necessarily $\sigma=0$, hence $p=0$ gives $|Z\cdot\zeta|=|\zeta|$, and thus we must in fact have $|Z|=1$ on $\gamma$, more precisely $Z=\mp\zeta/|\zeta|$, and thus $\gamma$ lies in $\cL_\pm$, which contradicts our assumption $\gamma\not\subset\cL_\pm$. Hence, $\gamma$ enters $|Z|>1$ in finite time, and thus, as we have already seen, reaches $\Tb^*_{H_2}M$ in finite time.

Finally, we show that forward, resp.~backward, bicharacteristics $\gamma$ in $(\Sigma_\pm\cap\Tb^*_Y M\setminus\cR_\pm)\cap\Sigma_\Omega$ tend to $L_\pm$. By equation~\eqref{EqDSHEscape}, $\pm(\sigma-Z\cdot\zeta)\to\infty$ (in finite time) along $\gamma$, hence $|\zeta|=|\sigma-Z\cdot\zeta|$ on $\gamma\subset\Sigma$ tends to $\infty$, and therefore
\[
  |Z|\geq\frac{|Z\cdot\zeta|}{|\zeta|}\geq\frac{|\sigma-Z\cdot\zeta|}{|\zeta|}-\frac{|\sigma|}{|\zeta|}\to 1
\]
since $\sigma$ is constant along $\gamma$. On the other hand, at points on $\gamma$ where $|Z|>1$, i.e.\ $\mu<0$, we have $\pm H_p\mu>0$ by \eqref{EqDSHmu}. We conclude that $\gamma$ tends to $|Z|=1$, i.e.\ $\mu=0$. Moreover,
\[
  \left(Z\cdot\frac{\zeta}{|\zeta|}-\frac{\sigma}{|\zeta|}\right)^2=1\tn{ on }\Sigma,
\]
thus $\bigl|Z\cdot\zeta/|\zeta|\bigr|\to 1$ along $\gamma$; together with $|Z|\to 1$, this implies $Z\to\mp\zeta/|\zeta|$, and since $\sigma$ is constant and $|\zeta|\to\infty$, we conclude that $\gamma$ tends to $L_\pm$. Thus, assumption~\itref{EnumHypNontrapping} in \S\ref{SubsecHypotheses} is verified.

%%%%%%%%%%%%%%%%%%%%%%%%%%%%%
\subsubsection{The normal operator}
\label{SubsubsecDSNormal}

The Mellin transformed normal operator $\widehat P(\sigma)$ of $P=\Box_g$, with principal symbol (in the high energy sense, $\sigma$ being the large parameter) given by the right hand side of \eqref{EqDSDualMetric}, fits into the framework of Vasy \cite{VasyMicroKerrdS}. In the current setting, the poles of $\widehat P(\sigma)^{-1}$ for $P$ acting on functions have been computed explicitly by Vasy \cite{VasyWaveOndS}. In fact, if more generally $P_\lambda=\Box_g-\lambda$, the only possible poles of $\widehat P_\lambda(\sigma)^{-1}$ are in
\begin{equation}
\label{EqDSResonances}
  i\wh s_\pm(\lambda)-i\N, \quad \wh s_\pm(\lambda)=-\frac{n-1}{2}\pm\sqrt{\frac{(n-1)^2}{4}-\lambda},
\end{equation}
and the pole with largest imaginary part is simple unless $\lambda=(n-1)^2/4$, in which case it is a double pole. Notice that for $\lambda=0$, all non-zero resonances have imaginary part $\leq -1$.

%%%%%%%%%%%%%%%%%%%%%%%%%%%%%%%%%%%%%%%%%%%%%%%%%
\subsection{Quasilinear wave equations}
\label{SubsecQuasilinear}

We are now prepared to discuss existence, uniqueness and asymptotics of solutions to quasilinear wave and Klein-Gordon equations for complex- and/or real-valued functions on the static model of de Sitter space, in fact on the domain $\Omega$ described in the previous section, with small data, i.e.\ small forcing. Keep in mind though that the methods work in greater generality, as explained in the introduction. In particular, we will prove Theorems~\ref{ThmIntroBaby} and \ref{ThmIntroFull}.

\begin{rmk}
\label{RmkBoxOnBundles2}
  The only reason for us to stick to the scalar case here as opposed to considering wave equations on natural vector bundles is the knowledge of the location of resonances in this case, see \S\ref{SubsubsecDSNormal}; the author is not aware of corresponding statements for bundle-valued equations. The general statement is that as long as there is no resonance or only a simple resonance at $0$ in the closed upper half plane, the arguments presented in this section go through. Likewise, we can work on the more general class of static asymptotically de Sitter spaces, since the normal operator, hence the resonances are the same as on exact static de Sitter space, and in fact on much more general spacetimes, provided the above resonance condition as well as all assumptions in \S\ref{SubsecHypotheses} are satisfied; examples of the latter kind include perturbations (even of the asymptotic model) of asymptotically de Sitter spaces, see also Remark~\ref{RmkPoleAt0Fixed}.
\end{rmk}

Let us from now on denote by $g_\dS$ the static de Sitter metric. We start with a discussion of quasilinear wave equations.

\begin{definition}
  For $s,\alpha\in\R$, define the Hilbert space
  \[
    \cX^{s,\alpha}:=\C\oplus\Hb^{s,\alpha}(\Omega)^{\bullet,-}
  \]
  with norm $\|(c,v)\|_{\cX^{s,\alpha}}^2=|c|^2+\|v\|_{\Hb^{s,\alpha}(\Omega)^{\bullet,-}}^2$. We will identify an element $(c,v)\in\cX^{s,\alpha}$ with the distribution $(\phi\circ\ft_1)c+v$, where $\phi$ and $\ft_1$ are as in the statement of Theorem~\ref{ThmExistenceAndExpansion}.
\end{definition}

We are now in a position to prove Theorem~\ref{ThmIntroFull} whose statement we recall:
\begin{thm}
\label{ThmAppWaveExistence}
  Let $s>n/2+7$ and $0<\alpha<1$. Assume that for $j=0,1$,
  \begin{gather*}
    g\colon\cX^{s-j,\alpha}\to(\CI+\Hb^{s-j,\alpha})(M;\Sym^2\Tb^*M), \\
	q\colon\cX^{s-j,\alpha}\times\Hb^{s-1-j,\alpha}(\Omega;\Tb^*_\Omega M)^{\bullet,-}\to\Hb^{s-1-j,\alpha}(\Omega)^{\bullet,-}
  \end{gather*}
  are continuous, $g$ is Lipschitz near $0$, and
  \begin{equation}
  \label{EqQContQuant1}
    \|q(u,\bdiff u)-q(v,\bdiff v)\|_{\Hb^{s-1-j,\alpha}(\Omega)^{\bullet,-}}\leq L_q(R)\|u-v\|_{\cX^{s-j,\alpha}}
  \end{equation}
  for $u,v\in\cX^{s-j,\alpha}$ with norm $\leq R$, then there is a constant $C_L>0$ so that the following holds: If $L_q(0)<C_L$, then for small $R>0$, there is $C_f>0$ such that for all $f\in\Hb^{s-1,\alpha}(\Omega)^{\bullet,-}$ with norm $\leq C_f$, there exists a unique solution $u\in\cX^{s,\alpha}$ of the equation
  \begin{equation}
  \label{EqQuasilinearWavePDE}
    \Box_{g(u)}u=f+q(u,\bdiff u)
  \end{equation}
  with norm $\leq R$, and in the topology of $\cX^{s-1,\alpha}$, $u$ depends continuously on $f$.
\end{thm}

\begin{rmk}
\label{RmkPoleAt0Fixed}
  Note that the poles of the meromorphic family $\widehat{\Box_{g(u)}}(\sigma)^{-1}$ depend continuously on $u$, see \cite[\S2.7]{VasyMicroKerrdS}, and the simple pole at $0$, corresponding to constant functions being annihilated by $N(\Box_{g(u)})$, is preserved under perturbations. This will be crucial in the proof, and it also shows that we may allow the metric $g(0)$ to be a perturbation (in the b-sense) of $g_\dS$, rather than exact $g_\dS$, without any additional work.
\end{rmk}

\begin{rmk}
  The permitted range of $\alpha$ is directly linked to the width of the resonance free strip below the real axis for the basic linear operator. In particular, $R>0$ depends on $\alpha$ in the sense that one must ensure that the linear operator $\Box_{g(u)}$ for $u\in\cX^{s,\alpha}$ with norm $\leq R$ has no resonances other than $0$ in $\Im\sigma>-\alpha$. Since the only resonances of $\Box_{g_\dS}$ in $\Im\sigma\geq -1$ are $0$ and $-i$, the continuous dependence of resonances on the b-metric ensures that $R>0$ can indeed be chosen in this way.
\end{rmk}

\begin{rmk}
  Of course, we require all sections $g(u)$ of $\Sym^2\Tb^*M$ to take values in symmetric 2-tensors with real coefficients. If we assume that $q$ and $f$ are real-valued, we may therefore work in the \emph{real} Hilbert space
  \begin{equation}
  \label{EqcXR}
    \cX^{s,\alpha}_\R:=\R\oplus\Hb^{s,\alpha}(\Omega;\R)^{\bullet,-}
  \end{equation}
  and find the solution $u$ there. This remark also applies to all theorems later in this section.
\end{rmk}

\begin{proof}[Proof of Theorem~\ref{ThmAppWaveExistence}.]
  To simplify the notation, we will occasionally write $\Hb^{\sigma,\rho}$ in place of $\Hb^{\sigma,\rho}(\Omega)^{\bullet,-}$ if the context is clear.

  By assumption on $g$, there exists $R_S$ such that for $u\in\cX^{s,\alpha}$ with $\|u\|_{\cX^{s,\alpha}}\leq R_S$, the operator $\Box_{g(u)}$ satisfies the relaxed versions of the assumptions~\itref{EnumHypPDyn}-\itref{EnumHypNontrapping} in \S\ref{SubsecHypotheses} (see the discussion after assumption~\itref{EnumHypNontrapping}), thus Theorem~\ref{ThmWaveGlobalSolveCont} is applicable, giving a continuous forward solution operator $S_{g(u)}$ on sufficiently weighted b-Sobolev spaces. For such $u$, the normal operator $N(\Box_{g(u)})$ is a small perturbation of $N(\Box_{g_\dS})$ in $\Diff^2(Y\cap\Omega)$, and since further $s-2-\alpha>0$, we can apply Theorem~\ref{ThmExistenceAndExpansion} to conclude that the solution operator in fact maps
  \[
    S_{g(u)}\colon\Hb^{s-1,\alpha}(\Omega)^{\bullet,-}\to\cX^{s,\alpha}
  \]
  continuously, with uniformly bounded operator norm
  \begin{equation}
  \label{EqSolOpContQuant}
    \|S_{g(u)}\|\leq C_S,\quad \|u\|_{\cX^{s,\alpha}}\leq R_S.
  \end{equation}
  Let $C_L:=C_S^{-1}$, and assume that $L_q(0)<C_L$, then $L_q(R_q)<C_L$ for $R_q>0$ small. Put $R:=\min(R_S,R_q)$ and $C_f=R(C_S^{-1}-L_q(R))$; let $f\in\Hb^{s-1,\alpha}(\Omega)^{\bullet,-}$ have norm $\leq C_f$. Define $u_0:=0$ and iteratively $u_{k+1}\in\cX^{s,\alpha}$ by solving
  \begin{equation}
  \label{EqWaveIteration}
    \Box_{g(u_k)}u_{k+1}=f+q(u_k,\bdiff u_k),
  \end{equation}
  i.e.\ $u_{k+1}=S_{g(u_k)}\bigl(f+q(u_k,\bdiff u_k)\bigr)$. For $u_{k+1}$ to be well-defined, we need to check that $\|u_k\|_{\cX^{s,\alpha}}\leq R$ for all $k$. For $k=0$, this is clear; for $k>0$, we deduce from \eqref{EqSolOpContQuant} and \eqref{EqQContQuant1} that
  \begin{align*}
    \|u_{k+1}\|_{\cX^{s,\alpha}} & \leq C_S\bigl(\|f\|_{\Hb^{s-1,\alpha}}+L_q(R)\|u_k\|_{\cX^{s,\alpha}}\bigr) \\
	  &\leq C_S\bigl(R(C_S^{-1}-L_q(R))+L_q(R)R\bigr)=R.
  \end{align*}
  We aim to show that the sequence $(u_k)_k$ is in fact Cauchy in $\cX^{s-1,\alpha}$. First, we observe that for $u\in\cX^{s-1,\alpha}$, we have
  \[
    \Box_{g(u)}=g^{ij}(u)\Db_i\Db_j+\wt g^j(u,\bdiff u)\Db_j
  \]
  with $g^{ij}(u)\in\CI+\Hb^{s-1,\alpha}$, $\wt g^j(u,\bdiff u)\in\CI+\Hb^{s-2,\alpha}$; using the explicit formula for the inverse of a metric, Corollary~\ref{CorHbModule} and Lemma~\ref{LemmaRecHs}, we deduce from the Lipschitz assumption on $g$ that
  \[
    g^{ij}\colon\cX^{s-1,\alpha}\to\CI+\Hb^{s-1,\alpha},\quad \wt g^j\colon\cX^{s-1,\alpha}\to\CI+\Hb^{s-2,\alpha}
  \]
  are Lipschitz as well; hence, for some constant $C_g(R)$, we obtain
  \[
    \|\Box_{g(u)}-\Box_{g(v)}\|_{\cL(\cX^{s,\alpha},\Hb^{s-2,\alpha})}\leq C_g(R)\|u-v\|_{\cX^{s-1,\alpha}}
  \]
  for $u,v\in\cX^{s-1,\alpha}$ with norms $\leq R$. Therefore, we get the following estimate for the difference of two solution operators $S_{g(u)}$ and $S_{g(v)}$, $u,v\in\cX^{s,\alpha}$, with a loss of $2$ derivatives relative to the elliptic setting, using a `resolvent identity:'
  \begin{align}
  \label{EqResolventId} \|S_{g(u)}-&S_{g(v)}\|_{\cL(\Hb^{s-1,\alpha},\cX^{s-1,\alpha})}=\|S_{g(u)}(\Box_{g(v)}-\Box_{g(u)})S_{g(v)}\|_{\cL(\Hb^{s-1,\alpha},\cX^{s-1,\alpha})} \\
    &\leq C_S^2\|\Box_{g(u)}-\Box_{g(v)}\|_{\cL(\cX^{s,\alpha},\Hb^{s-2,\alpha})}\leq C_S^2 C_g(R) \|u-v\|_{\cX^{s-1,\alpha}}. \nonumber
  \end{align}
  Here, we assumed $C_S$ is such that $\|S_{g(u)}\|_{\cL(\Hb^{s-2,\alpha},\cX^{s-1,\alpha})}\leq C_S$ for small $u\in\cX^{s,\alpha}$, which is where we use that $s-1>n/2+6$. Returning to the goal of proving that $(u_k)_k$ is Cauchy in $\cX^{s-1,\alpha}$, we estimate
  \begin{align*}
    \|u_{k+1}&-u_k\|_{\cX^{s-1,\alpha}} \leq \bigl\|(S_{g(u_k)}-S_{g(u_{k-1})})\bigl(f+q(u_{k-1},\bdiff u_{k-1})\bigr)\bigr\|_{\cX^{s-1,\alpha}} \\
      &\hspace{17ex} + \|S_{g(u_k)}(q(u_k,\bdiff u_k)-q(u_{k-1},\bdiff u_{k-1}))\|_{\cX^{s-1,\alpha}} \\
      &\leq C_S\bigl(L_q(R)+C_S C_g(R)(C_f+L_q(R)R)\bigr)\|u_k-u_{k-1}\|_{\cX^{s-1,\alpha}}.
  \end{align*}
  Since $C_S L_q(0)<1$, the constant on the right hand side is less than $1$ for small $R>0$, recalling that $C_f=C_f(R)\to 0$ as $R\to 0$. Therefore, $(u_k)_k$ converges exponentially fast to a limit $u\in\cX^{s-1,\alpha}$ as $k\to\infty$. Since $\{u_k\}$ is bounded in the Hilbert space $\cX^{s,\alpha}$, it in fact has a weakly convergent subsequence in $\cX^{s,\alpha}$, and the limit is necessarily equal to $u$, so $u\in\cX^{s,\alpha}$. This easily implies the weak convergence of the full sequence $u_k\weakto u$ in $\cX^{s,\alpha}$.
  
  We can prove uniqueness and stability in one stroke: Suppose that $u_1,u_2\in\cX^{s,\alpha}$ have norm $\leq R$ and satisfy
  \[
    \Box_{g(u_j)}u_j=f_j+q(u_j,\bdiff u_j),\quad j=1,2,
  \]
  where the $f_j\in\Hb^{s-1,\alpha}$, $j=1,2$, have norm $\leq C_f$. Then the estimate~\eqref{EqResolventId} yields
  \begin{align*}
    \|u_1-u_2\|_{\cX^{s-1,\alpha}}&\leq C_S\bigl(\|f_1-f_2\|_{\Hb^{s-2,\alpha}}+ \\
	  &\hspace{5ex}(L_q(R)+C_S C_g(R)(C_f+L_q(R)R))\|u_1-u_2\|_{\cX^{s-1,\alpha}}\bigr).
  \end{align*}
  Arguing as before, the second term on the right can be absorbed into the left hand side for small $R>0$. Hence
  \[
    \|u_1-u_2\|_{\cX^{s-1,\alpha}}\leq C'\|f_1-f_2\|_{\Hb^{s-2,\alpha}}
  \]
  as desired.
\end{proof}

\begin{rmk}
\label{RmkConstantMetric}
  In the case that $g(u)\equiv g$ is constant, see \cite[Theorem~2.24]{HintzVasySemilinear} for a discussion of the corresponding semilinear equations. There, one in particular obtains more precise asymptotics in the case of polynomial non-linearities, see \cite[Theorem~2.35]{HintzVasySemilinear}; see also Remark~\ref{RmkNonsmoothExpansionLimits}.
\end{rmk}

We next turn to a special case of Theorem~\ref{ThmAppWaveExistence} which is very natural and allows for a stronger conclusion.

\begin{thm}
\label{ThmAppWavePoly}
  Let $s>n/2+7$ and $0<\alpha<1$. Let $N,N'\in\N$, and suppose $c_k\in\CI(\R;\R)$, $g_k\in(\CI+\Hb^s)(M;\Sym^2\Tb^*M)$ for $1\leq k\leq N$; define the map $g\colon\cX_\R^{s,\alpha}\to(\CI+\Hb^{s,\alpha})(M;\Sym^2\Tb^*M)$ by
  \[
    g(u)=\sum_{k=1}^N c_k(u)g_k,
  \]
  and assume $g(0)=g_\dS$. Moreover, define
  \[
    q(u,\bdiff u)=\sum_{j=1}^{N'} u^{e_j}\prod_{k=1}^{N_j} X_{jk}u,\quad e_j+N_j\geq 2, N_j\geq 1, X_{jk}\in(\CI+\Hb^\infty)\Vb.
  \]
  Then for small $R>0$, there exists $C_f>0$ such that for all forcing terms $f\in\Hb^{s-1,\alpha}(\Omega;\R)^{\bullet,-}$ with norm $\leq C_f$, the equation
  \begin{equation}
  \label{EqQuasilinearWavePDESmooth}
    \Box_{g(u)}u=f+q(u,\bdiff u)
  \end{equation}
  has a unique solution $u\in\cX_\R^{s,\alpha}$, with norm $\leq R$, and in the topology of $\cX_\R^{s-1,\alpha}$, $u$ depends continuously on $f$. If one in fact has $f\in\Hb^{s'-1,\alpha}(\Omega;\R)^{\bullet,-}$ for some $s'\in(s,\infty]$, then $u\in\cX_\R^{s',\alpha}$.
\end{thm}

The initial metric $g(0)$ can be more general, see Remark~\ref{RmkPoleAt0Fixed}.

\begin{rmk}
  As alluded to in \S\ref{SecIntroduction}, the class $(\CI+\Hb^\infty)\Vb$ of b-vector fields is well-defined irrespective of the exponential compactification of the static de Sitter space: Using the function $t_*$ from the beginning of \S\ref{SubsecStaticdS}, one could use $\wt\tau:=e^{-a(x)t_*}$, with $a>0$ a smooth function of the `spatial' variables, to compactify the spacetime, i.e.\ use $\wt\tau$ as a boundary defining function of future infinity, which in view of $\wt\tau=\tau^{a(x)}$, $\tau=e^{-t_*}$, leads to a different smooth structure than if one compactifies using $\tau$, however one easily checks that the space $\Hb^\infty$ on the compactification does not depend on the choice of boundary defining function, and neither does the space $\CI+\Hb^\infty$.
\end{rmk}

\begin{rmk}
  One could, for instance, choose the metrics $g_k$ such that at every point $p\in M$, the linear space $\Sym^2\Tb_p M$ is spanned by the $g_k(p)$, and in a similar manner the b-vector fields $X_{jk}$.
\end{rmk}

\begin{rmk}
  The point of the last part of the theorem is that even though a priori the radius of the ball which is the set of $f\in\Hb^{s'-1,\alpha}(\Omega)^{\bullet,-}$ for which one has solvability in $\cX^{s',\alpha}$ according to Theorem~\ref{ThmAppWaveExistence} could shrink to $0$ as $s'\to\infty$, this does not happen in the setting of Theorem~\ref{ThmAppWavePoly}. We use a straightforward approach to proving this by differentiating the PDE; a somewhat more robust way could be to use Nash-Moser iteration, see e.g.\ \cite{SaintRaymondNashMoser}, which however would require a more careful analysis of all estimates in \S\S\ref{SecCalculus}--\ref{SecGlobalForSecond}, as indicated, for instance, in Remark~\ref{RmkNoMoser}.\footnote{Recently, this analysis has been done by Vasy and the author \cite{HintzVasyQuasilinearKdS} and was used there to study nonlinear waves on asymptotically Kerr-de Sitter spaces.}
\end{rmk}

\begin{rmk}
\label{RmkPreciseExpansion}
  If $f$ has more decay, say $f\in\Hb^{\infty,\infty}$, it is relatively straightforward to show that the solution $u$ in fact has an asymptotic expansion to any fixed order, assuming $f$ is small in an appropriate space. Indeed, for such a statement, one only needs to replace the spaces $\cX^{s,\alpha}$ by similar spaces which now encode more precise partial asymptotic expansions, as in \cite[\S{2.3}]{HintzVasySemilinear}, and prove the persistence of such spaces under taking reciprocals, compositions with smooth functions etc.
\end{rmk}

For the proof, we need one more definition:
\begin{definition}
\label{DefMlSpaces}
  (Cf.\ \cite[Definition~1.1]{BealsReedMicroNonsmooth}.) For $s'>s$, $\alpha\in\R$ and $\Gamma\subset\Sb^*M$, let
  \[
    \Hb^{s,\alpha;s',\Gamma}:=\{u\in\Hb^{s,\alpha}\colon \WFb^{s',\alpha}(u)\cap\Gamma=\emptyset\}.
  \]
\end{definition}

\begin{proof}[Proof of Theorem~\ref{ThmAppWavePoly}.]
  The map $g$ satisfies the requirements of Theorem~\ref{ThmAppWaveExistence} by Proposition~\ref{PropCompWithSmooth2}, and $q$ satisfies \eqref{EqQContQuant1} with $L_q(0)=0$, thus Theorem~\ref{ThmAppWaveExistence} implies the existence and uniqueness of solutions in $\cX^{s,\alpha}$ with small norm as well as their stability in the topology of $\cX^{s-1,\alpha}$. The uniqueness of $u$ in all of $\cX_\R^{s,\alpha}$, in fact in $\Hbloc^s(\Omega^\circ)$, follows from local uniqueness for quasilinear symmetric hyperbolic systems, see e.g.\ Taylor \cite[\S{16.3}]{TaylorPDE}.
  
  It remains to establish the higher regularity statement; by an iterative argument, it suffices to prove the following: If $s'>s$, $u\in\cX_\R^{s'-1/2,\alpha}$, $\|u\|_{\cX^{s,\alpha}}\leq R$, and $u$ solves~\eqref{EqQuasilinearWavePDESmooth} with $f\in\Hb^{s'-1,\alpha}$, then $u\in\cX_\R^{s',\alpha}$; note that we only make the $s'$-independent assumption of the $\cX^{s,\alpha}$-norm of $u$ being small -- the reason for this assumption is that it ensures that $\Box_{g(u)}$ fits into our framework.
  
  \emph{We will use the summation convention for the remainder of the proof.} Equation~\eqref{EqQuasilinearWavePDESmooth} in local coordinates reads
  \begin{equation}
  \label{EqWaveInCoords}
    \bigl(g^{ij}(u)\bpa_{ij}^2+h^j(u,\bpa u)\bpa_j\bigr)u=f+q(u,\bpa u),
  \end{equation}
  where $g^{ij}(v)$, $h^j(v;z)$ and $q(v;z)$ are $\CI$-functions of $v$ and $z$. As is standard in ODEs to obtain higher regularity (and exploited in a similar setting by Beals and Reed \cite[\S{4}]{BealsReedMicroNonsmooth}), we will differentiate this equation with respect to certain b-vector field $V$: After differentiating and collecting/rewriting terms, one obtains an equation like~\eqref{EqWaveInCoords} for $Vu$, where only the coefficients of first order terms are changed, and without $q$ and with a different forcing term; one can then appeal to the regularity theory for the equation for $Vu$, which is thus again a wave equation with lower order terms. Concretely, suppose $\wt\Sigma\subset\Sigma$ is a closed subset of the characteristic set of $\Box_{g(u)}$, consisting of bicharacteristic strips and contained in the coordinate patch we are working in; we want to propagate $\cX^{s',\alpha}$-regularity of $u$ into $\wt\Sigma$, assuming we have this regularity on backward/forward bicharacteristics from $\wt\Sigma$ or in a punctured neighborhood of $\wt\Sigma$. With $\pi\colon\Sb^*M\to M$ denoting the projection to the base, choose $\chi,\chi_0\in\CI_\cl(\Rnhalf)$ so that $\chi$ is identically $1$ near $\pi(\wt\Sigma)$ and $\chi_0$ is identically $1$ on $\supp\chi$. Let $V_0\in\Vb(\Rnhalf)$ be a constant coefficient b-vector field which is non-characteristic (in the b-sense) on $\wt\Sigma$, which is possible if $\wt\Sigma$ is sufficiently small, and put $V=\chi_0 V_0$. Applying $V$ to~\eqref{EqWaveInCoords}, we obtain, suppressing the arguments $u,\bpa u$,
  \begin{align*}
    &\bigl(g^{ij}\bpa_{ij}^2+[h^j+(\pa_{z_j}h^k)\bpa_k u-\pa_{z_j}q]\bpa_j\bigr)Vu+(g^{ij})'Vu\,\bpa_{ij}^2 u+g^{ij}[V,\bpa_{ij}^2]u \\
	  &\hspace{2ex} = Vf+(\pa_v q)Vu+(\pa_{z_j}q)[V,\bpa_j]u \\
	  &\hspace{12ex} - (\pa_v h^j)Vu\,\bpa_j u - h^j[V,\bpa_j]u - (\pa_{z_j}h^k)[V,\bpa_k]u =: f_1.
  \end{align*}
  Since $V_0$ annihilates constants, $Vu\in\Hb^{s'-3/2,\alpha}$ locally near $\pi(\wt\Sigma)$. Similarly, $[V,\bpa_j]u\in\Hb^{s'-3/2,\alpha}$ locally near $\pi(\wt\Sigma)$, and $h^j(u,\bpa u)\in\CI+\Hb^{s'-3/2,\alpha}$, $q(u,\bpa u)\in\Hb^{s'-3/2,\alpha}$, similarly for derivatives of $h^j$ and $q$; lastly, $Vf\in\Hb^{s'-2,\alpha}$, thus $f_1\in\Hb^{s'-2,\alpha}$ locally near $\pi(\wt\Sigma)$. We need to analyze the last two terms on the left hand side: Since $V$ is non-characteristic on $\supp\chi\supset\pi(\wt\Sigma)$, we can write
  \[
    \bpa_j=(1-\chi)\bpa_j+Q_jV+\wt R_j,\quad Q_j\in\Psib^0, \wt R_j\in\Psib^1, \WFb'(\wt R_j)\cap\wt\Sigma=\emptyset;
  \]
  put $R_j:=(1-\chi)\bpa_j+\wt R_j=\bpa_j-Q_jV$. Note that $R_j$ annihilates constants. We can then write
  \[
    \bpa_{ij}^2 u=\bpa_i Q_j Vu+\bpa_i R_j u,
  \]
  and the second term is in $\Hb^{\infty,\alpha}$ microlocally near $\wt\Sigma$. Thus, we have
  \[
    (g^{ij})' Vu\,\bpa_{ij}^2 u=\bigl((g^{ij})' Vu\,\bpa_i Q_j\bigr)Vu+(g^{ij})' Vu\,\bpa_i R_j u;
  \]
  the second term on the right is a product of a function in $\Hb^{s'-3/2,\alpha}$ with $\bpa_i R_j u$, the latter a priori being an element of $\Hb^{s'-5/2,\alpha;\infty,\wt\Sigma}$; we will prove below in Lemma~\ref{LemmaMicrolocalProduct} that this product is an element of $\Hb^{s'-5/2,\alpha;s'-3/2,\wt\Sigma}$. Moreover, $[V,\bpa_{ij}^2]$ is a second order b-differential operator, vanishing on constants, with coefficients vanishing near $\pi(\wt\Sigma)$; this implies $g^{ij}[V,\bpa_{ij}^2]u\in\Hb^{s'-5/2,\alpha;\infty,\wt\Sigma}$. We conclude that
  \begin{equation}
  \label{EqPVu}
    P_1(Vu)=f_2\in\Hb^{s'-5/2,\alpha;s'-2,\wt\Sigma},
  \end{equation}
  where
  \[
    P_1=\Box_{g(u)}+\wt P,\quad \wt P=[(\pa_{z_j}h^k)\bpa_k u-\pa_{z_j}q]\bpa_j+(g^{ij})' Vu\,\bpa_i Q_j.
  \]
  Since we are assuming $u\in\cX^{s'-1/2,\alpha}$, and moreover $\wt P$ is an element of $\Hb^{s'-3/2,\alpha}\Psib^1$ near $\pi(\wt\Sigma)$, we see that, a forteriori,
  \[
    P_1\in(\CI+\Hb^{s'-1,\alpha})\Diffb^2+(\CI+\Hb^{s'-2,\alpha})\Psib^1.
  \]
  Hence, we can propagate $\Hb^{s'-1,\alpha}$-regularity of $Vu$ into $\wt\Sigma$ by Theorems~\ref{ThmRealPrType} and \ref{ThmRadialPoints}; recall that these two theorems only deal with the propagation of regularity which is $1/2$ more than than the a priori regularity of $Vu$, which is $\Hb^{s'-3/2,\alpha}$. The point here is that real principal type propagation only depends on the principal symbol of $P_1$, which is the same as the principal symbol of $\Box_{g(u)}$, and the propagation of $\Hb^{s'-1,\alpha}$-regularity near radial points works for arbitrary $\Hb^{s'-2,\alpha}\Psib^1$-perturbations of $\Box_{g(u)}$; see Remark~\ref{RmkRelaxedRadial}. Therefore, writing $u=c+u'$ with $u'\in\Hb^{s'-1/2,\alpha}$ a priori, we obtain $u'\in\Hb^{s',\alpha}$ microlocally near $\wt\Sigma$ by standard elliptic regularity, since $V$ is non-characteristic on $\wt\Sigma$. Away from the characteristic set of $\Box_{g(u)}$,\footnote{Notice that $P_1$ and $\Box_{g(u)}$ have the same characteristic set.} we simply use $P_1 Vu\in\Hb^{s'-5/2,\alpha}$ and elliptic regularity for $P_1 V$ to deduce that $u'\in\Hb^{s'+1/2,\alpha}$ there;\footnote{Let us stress the importance of only using local rather than microlocal regularity information of $P_1Vu$, since the proof of Theorem~\ref{ThmEllipticReg}, giving elliptic regularity for $Vu$ solving $P_1(Vu)=f$, only works with local assumptions on $f$, see Remark~\ref{RmkEllipticLocalAssm}.} here, we would choose $V$ such that it is non-characteristic on a set disjoint from $\Sigma$. Putting all such pieces of regularity information together by choosing finitely many such sets $\wt\Sigma$, we obtain $u'\in\Hbloc^{s',\alpha}(\Omega)^{\bullet,-}$.
  
  We can make this is a global rather than local statement by extending $\Omega$ to the slightly larger domain $\Omega_{0,\delta_2}$, $\delta_2<0$, solving the quasilinear PDE there, and restricting back to $\Omega$; thus $u'\in\Hb^{s',\alpha}(\Omega)^{\bullet,-}$.
\end{proof}

To finish the proof, we need the following lemma, which we prove using ideas from \cite[Theorem~1.3]{BealsReedMicroNonsmooth}.
\begin{lemma}
\label{LemmaMicrolocalProduct}
  Let $\alpha\in\R$ and $s>n/2+1$. Then, in the notation of Definition~\ref{DefMlSpaces}, for $u\in\Hb^s$ and $v\in\Hb^{s-1,\alpha;s,\Gamma}$, we have $uv\in\Hb^{s-1,\alpha;s,\Gamma}$.
\end{lemma}
\begin{proof}
  Without loss, we may assume $\alpha=0$. By Corollary~\ref{CorHbModule}, $uv\in\Hb^{s-1}$, and we must prove the microlocal regularity of $uv$. Using a partition of unity, it suffices to assume $\Gamma=(\Rnhalf)_z\times K$ for a conic set $K\subset\R^n_\zeta\wozero$; moreover, since the complement of the wave front set is open, we can assume that $K$ is open. By assumption, we can then write
  \[
    |\wh u(\zeta)|=\frac{u_0(\zeta)}{\la\zeta\ra^s}, u_0\in L^2, \quad |\wh v(\zeta)|=\biggl(\frac{\chi_K(\zeta)}{\la\zeta\ra^s}+\frac{\chi_{K^c}(\zeta)}{\la\zeta\ra^{s-1}}\biggr)v_0(\zeta), v_0\in L^2,
  \]
  where $\chi_K$ denotes the characteristic function of $K$, and $K^c$ the complement of $K$. Now, let $K_0\subset K$ be closed and conic. Then
  \[
    \chi_{K_0}(\zeta)|\widehat{uv}(\zeta)|\la\zeta\ra^s \leq \int \frac{\chi_{K_0}(\zeta)\la\zeta\ra^s}{\la\zeta-\xi\ra^s}\biggl(\frac{\chi_K(\xi)}{\la\xi\ra^s}+\frac{\chi_{K^c}(\xi)}{\la\xi\ra^{s-1}}\biggr)u_0(\zeta-\xi)v_0(\xi)\,d\xi
  \]
  We want to use Lemma~\ref{LemmaIntegralEstimate} to show that this is an element of $L^2$, thus finishing the proof. But we have
  \[
    \frac{\la\zeta\ra^s}{\la\zeta-\xi\ra^s\la\xi\ra^s}\in L^\infty_\zeta L^2_\xi,
  \]
  and on the support of $\chi_{K_0}(\zeta)\chi_{K^c}(\xi)$, we have $|\zeta-\xi|\geq c|\zeta|$, $c>0$, thus
  \[
    \frac{\chi_{K_0}(\zeta)\chi_{K^c}(\xi)\la\zeta\ra^s}{\la\zeta-\xi\ra^s\la\xi\ra^{s-1}}\lesssim\frac{1}{\la\xi\ra^{s-1}}\in L^\infty_\zeta L^2_\xi,
  \]
  since $s>n/2+1$.
\end{proof}

%%%%%%%%%%%%%%%%%%%%%%%%%%%%%%%%%%%%%%%%%%%%%%%%%
\subsection{Conformal changes of the metric}
\label{SubsecAppConformal}

Reconsidering the proof of Theorem~\ref{ThmAppWaveExistence}, one \emph{cannot} bound
\[
  \|(S_{g(u)}-S_{g(v)})\|_{\cL(\Hb^{s-1,\alpha},\cX^{s,\alpha})}\lesssim\|u-v\|_{\cX^{s,\alpha}}
\]
in general,\footnote{Indeed, consider a similar situation for scalar first order operators $P_a:=\pa_t-a\pa_x$, $a\in\R$, on $[0,1]_t\times\R_x$. The forward solution operator $S_a$ is constructed by integrating the forcing term along the bicharacteristics $s\mapsto (s,x_0-as)$ of $P_a$, and it is easy to see that $S_a\in\cL(L^2,L^2)$. However, $S_a-S_b$ is constructed using the difference of integrals of the forcing $f$ along two different bicharacteristics, which one can naturally only bound using $df$, i.e.\ one only obtains the estimate $\|(S_a-S_b)f\|_{L^2}\lesssim|a-b|\|f\|_{H^1}$, which is an estimate with a loss of $2$ derivatives, similar to \eqref{EqResolventId}. The core of the problem is that there is no estimate of the form $\|f(\cdot+a)-f\|_{L^2}\lesssim|a|\|f\|_{L^2}$, although such an estimate holds if the norm on the right is replaced by the $H^1$-norm.} which however would immediately give uniqueness and stability of solutions to~\eqref{EqQuasilinearWavePDE} in the space $\cX^{s,\alpha}$. But there is a situation where we do have good control on $S_{g(u)}-S_{g(v)}$ as an operator from $\Hb^{s-1,\alpha}$ to $\cX^{s,\alpha}$, namely when $\Box_{g(u)}$ and $\Box_{g(v)}$ have the same characteristic set, since in this case, in \eqref{EqResolventId} the composition of $\Box_{g(v)}-\Box_{g(u)}$ with $S_{g(v)}$ loses no derivative (ignoring issues coming from the limited regularity of $g(u),g(v)$ for the moment -- they will turn out to be irrelevant). This situation arises if $g(u)=\mu(u)g(0)$ for $\mu(u)\in\CI(M)+\Hb^s(M)$; that this is in fact the only possibility is shown by a pointwise application of the following lemma.

\begin{lemma}
\label{LemmaSameLightCones}
  Let $d\geq 1$, and assume $g,g'$ are bilinear forms on $\R^{1+d}$ with signature $(1,d)$ such that the zero sets of the associated quadratic forms $q,q'$ coincide. Then $g=\mu g'$ for some $\mu\in\R^\times$.
\end{lemma}
\begin{proof}
  By a linear change of coordinates, we may assume that $g'$ is the Minkowski bilinear form on $\R^{1+d}$. Let $g_{ij}$, $0\leq i,j\leq d$, be the components of $g$, and let us write vectors in $\R^{1+d}$ as $(x_1,x')\in\R\times\R^d$. Since $g'(1,0)\neq 0$, we have $g(1,0)=g_{00}\neq 0$. Dividing $g$ by $\mu:=g_{00}$, we may assume $g_{00}=1$; we now show that $g=g'$. For all $x'\in\R^d$, $|x'|=1$ (Euclidean norm!), we have $q(1,x')=0$ and $q(1,-x')=0$, hence $q(1,x')-q(1,-x')=0$, in coordinates
  \[
    4\sum_{i\geq 1}g_{0i}x'_i=0, \quad |x'|=1,
  \]
  and thus $g_{0i}=0$ for all $i\geq 1$. Now let $\wt q(x'):=q(0,x')$ and $\wt q'(x'):=q'(0,x')$, then
  \[
    \wt q(x')=-1 \iff q(1,x')=0 \iff q'(1,x')=0 \iff \wt q'(x')=-1,
  \]
  thus by scaling $\wt q\equiv\wt q'$ on $\R^d$, hence by polarization $g_{ij}=g'_{ij}$ for $1\leq i,j\leq d$, and the proof is complete.
\end{proof}

In this restricted setting, we have the following well-posedness result; notice that the topology in which we have stability is stronger than in Theorem~\ref{ThmAppWaveExistence}, and we also allow more general non-linearities $q$.

\begin{thm}
\label{ThmAppWaveMultMetric}
  Let $s>n/2+6$, $0<\alpha<1$. Let
  \[
    g_0\in(\CI+\Hb^{s,\alpha})(M;\Sym^2\Tb^*M)
  \]
  be a metric satisfying the assumptions~\itref{EnumHypPDyn}-\itref{EnumHypNontrapping} in \S\ref{SubsecHypotheses} on $\Omega$, for example $g_0=g_\dS$ (see also Remark~\ref{RmkPoleAt0Fixed}) and let $\mu\colon\cX^{s,\alpha}\to\cX^{s,0}_\R$ be\footnote{$\cX^{s,\alpha}_\R$ was defined in~\eqref{EqcXR}.} a continuous map with $\mu(0)=1$ and
  \begin{equation}
  \label{EqMuCont}
    \|\mu(u)-\mu(v)\|_{\cX^{s,0}}\leq L_\mu(R)\|u-v\|_{\cX^{s,\alpha}}
  \end{equation}
  for all $u,v\in\cX^{s,\alpha}$ with norms $\leq R$, where $L_\mu\colon\R_{\geq 0}\to\R$ is continuous and non-decreasing. Put $g(u):=\mu(u)g_0$.
  \begin{enumerate}[leftmargin=\enummargin]
    \item Let
  \begin{equation}
  \label{EqQCont2}
    q\colon\cX^{s,\alpha}\times\Hb^{s-1,\alpha}(\Omega;\Tb^*_\Omega M)^{\bullet,-} \to \Hb^{s-1,\alpha}(\Omega)^{\bullet,-}
  \end{equation}
  be continuous with $q(0)=0$, satisfying
  \begin{equation}
  \label{EqQContQuant2}
    \|q(u,\bdiff u)-q(v,\bdiff v)\|_{\Hb^{s-1,\alpha}(\Omega)^{\bullet,-}}\leq L_q(R)\|u-v\|_{\cX^{s,\alpha}}
  \end{equation}
  for all $u,v\in\cX^{s,\alpha}$ with norms $\leq R$, where $L_q\colon\R_{\geq 0}\to\R$ is continuous and non-decreasing. Then there is a constant $C_L>0$ so that the following holds: If $L_q(0)<C_L$, then for small $R>0$, there is $C_f>0$ such that for all $f\in\Hb^{s-1,\alpha}(\Omega)^{\bullet,-}$ with norm $\leq C_f$, there exists a unique solution $u\in\cX^{s,\alpha}$ of the equation
  \begin{equation}
  \label{EqQuasilinearWavePDECopy}
    \Box_{g(u)}u=f+q(u,\bdiff u)
  \end{equation}
  with norm $\leq R$, which depends continuously on $f$.

  \item More generally, if
  \begin{equation}
  \label{EqQCont3}
    q\colon\cX^{s,\alpha}\times\Hb^{s-1,\alpha}(\Omega;\Tb^*_\Omega M)^{\bullet,-}\times\Hb^{s-1,\alpha}(\Omega)^{\bullet,-}\to\Hb^{s-1,\alpha}(\Omega)^{\bullet,-}
  \end{equation}
  is continuous with $q(0)=0$ and satisfies
  \begin{equation}
  \label{EqQContQuant3}
    \begin{split}
      \|q(u_1,\bdiff u_1,w_1)-&q(u_2,\bdiff u_2,w_2)\|_{\Hb^{s-1,\alpha}(\Omega)^{\bullet,-}} \\
	    &\leq L_q(R)\bigl(\|u_1-u_2\|_{\cX^{s,\alpha}}+\|w_1-w_2\|_{\Hb^{s-1,\alpha}(\Omega)^{\bullet,-}}\bigr)
    \end{split}
  \end{equation}
  for all $u_j\in\cX^{s,\alpha}$, $w_j\in\Hb^{s-1,\alpha}(\Omega)^{\bullet,-}$ with $\|u_j\|+\|w_j\|\leq R$, then there is a constant $C_L>0$ such that the following holds: If $L_q(0)<C_L$, then for small $R>0$, there is $C_f>0$ such that for all $f\in\Hb^{s-1,\alpha}(\Omega)^{\bullet,-}$ with norm $\leq C_f$, there exists a unique solution $u\in\cX^{s,\alpha}$ of the equation
  \begin{equation}
  \label{EqQuasilinearWavePDEQ}
    \Box_{g(u)}u=f+q(u,\bdiff u,\Box_{g(u)}u)
  \end{equation}
  with $\|u\|_{\cX^{s,\alpha}}+\|\Box_{g_0}u\|_{\Hb^{s-1,\alpha}}\leq R$, which depends continuously on $f$.
  \end{enumerate}
\end{thm}
\begin{proof}
  First, note that $N(\Box_{g(u)})=\mu(u)|_Y N(\Box_{g_0})$, which is a constant multiple of $N(\Box_{g_0})$ by the definition of the space $\cX^{s,\alpha}$. Thus, as in the proof of Theorem~\ref{ThmAppWaveExistence}, there exists $R_S>0$ such that
  \[
    S_{g(u)}\colon\Hb^{s-1,\alpha}(\Omega)^{\bullet,-}\to\cX^{s,\alpha}
  \]
  is continuous with uniformly bounded operator norm
  \[
    \|S_{g(u)}\|\leq C_S;
  \]
  for $\|u\|_{\cX^{s,\alpha}}\leq R_S$; let us also assume that
  \begin{equation}
  \label{EqMuBound}
    |\mu(u)|\geq c_0>0,\quad \|u\|_{\cX^{s,\alpha}}\leq R_S.
  \end{equation}

  We now prove the first half of the theorem. Let $C_L:=C_S^{-1}$, and assume that $L_q(0)<C_L$, then $L_q(R_q)<C_L$ for $R_q>0$ small. Put $\wt R:=\min(R_S,R_q)$; let $0<R\leq\wt R$, to be specified later, and put and $C_f(R)=R(C_S^{-1}-L_q(R))$; let $f\in\Hb^{s-1,\alpha}(\Omega)^{\bullet,-}$ have norm $\leq C_f(R)$. Let $B(R)$ denote the metric ball of radius $R$ in $\cX^{s,\alpha}$, and define $T\colon B(R)\to B(R)$,
  \[
    Tu:=S_{g(u)}\bigl(f+q(u,\bdiff u)\bigr).
  \]
  By the choice of $R,C_L$ and $C_f$, $T$ is well-defined by the same estimate as in the proof of Theorem~\ref{ThmAppWaveExistence}. The crucial new feature here is that for $R$ sufficiently small, $T$ is in fact a contraction. This follows once we prove the existence of a constant $C_i>0$ such that for $u,v\in\cX^{s,\alpha}$ with norms $\leq R$, we have
  \begin{equation}
  \label{EqSolOpLipschitz}
    \|S_{g(u)}-S_{g(v)}\|_{\cL(\Hb^{s-1,\alpha},\cX^{s,\alpha})}\leq C_S C_i L_\mu(R) \|u-v\|_{\cX^{s,\alpha}}.
  \end{equation}
  Indeed, assuming this, we obtain
  \begin{align*}
    \|&Tu-Tv\|_{\cX^{s,\alpha}} \\
	  &\leq \left\|S_{g(u)}\bigl(q(u,\bdiff u)-q(v,\bdiff v)\bigr)\right\|_{\cX^{s,\alpha}}+\|(S_{g(u)}-S_{g(v)})(f+q(v,\bdiff v))\|_{\cX^{s,\alpha}} \\
	  & \leq \bigl(C_S L_q(R) + C_S C_i L_\mu(R) (C_f(R)+L_q(R)R)\bigr) \|u-v\|_{\cX^{s,\alpha}};
  \end{align*}
  and since $C_S L_q(R)\leq C_S L_q(\wt R)<\theta<1$ for $R\leq\wt R$, we can choose $R$ so small that
  \begin{equation}
  \label{EqSmallR}
    C_S C_i L_\mu(R) (C_f(R)+L_q(R)R) \leq \theta-C_S L_q(R),
  \end{equation}
  where we use that $C_f(R)\to 0$ as $R\to 0$. With this choice of $R$, $T$ is a contraction, thus has a unique fixed point $u\in\cX^{s,\alpha}$ which solves the PDE~\eqref{EqQuasilinearWavePDECopy}.

  Continuing to assume \eqref{EqSolOpLipschitz}, let us prove the continuous dependence of the solution $u$ on $f$. For this, let us assume that $u_j\in\cX^{s,\alpha}$, $j=1,2$, solves
  \[
    \Box_{g(u_j)}u_j=f_j+q(u_j,\bdiff u_j),
  \]
  where $f_j\in\Hb^{s-1,\alpha}$ has norm $\leq C_f$. Then, as in the proof of Theorem~\ref{ThmAppWaveExistence},
  \begin{align*}
    \|u_1-u_2\|_{\cX^{s,\alpha}}&\leq C_S\bigl(\|f_1-f_2\|_{\Hb^{s-1,\alpha}} \\
	  &\hspace{5ex}+ (L_q(R) + C_i L_\mu(R) (C_f+L_q(R)R))\|u_1-u_2\|_{\cX^{s,\alpha}}\bigr).
  \end{align*}
  Because of \eqref{EqSmallR}, the prefactor of $\|u_1-u_2\|$ on the right hand side is $\leq\theta<1$, hence we conclude
  \[
    \|u_1-u_2\|_{\cX^{s,\alpha}}\leq\frac{C_S}{1-\theta}\|f_1-f_2\|_{\Hb^{s-1,\alpha}},
  \]
  as desired.

  We now prove the crucial estimate~\eqref{EqSolOpLipschitz} by using the identity in \eqref{EqResolventId}, as follows: By definition of $\Box$, we have
  \begin{equation}
  \label{EqDiffOfBox}
    \Box_{g(u)}=\Box_{\mu(v)g_0\frac{\mu(u)}{\mu(v)}}=\frac{\mu(v)}{\mu(u)}\Box_{g(v)}+E_{u,v},
  \end{equation}
  where $E_{u,v}\in\Hb^{s-1,\alpha}\Vb$ satisfies the estimate\footnote{To define a norm of an element $E\in\Hb^{\sigma,\rho}\Vb(M)$, use a partition of unity on $M$ to reduce this task to a local one, and as the norm of $E\in\Hb^{\sigma,\rho}\Vb(\Rnhalf)$, take the sum of the $\Hb^{\sigma,\rho}$-norms of the coefficients of $E$.}
  \[
    \|E_{u,v}\|_{\Hb^{s-1,\alpha}\Vb}\leq C\left\|\bdiff\left(\frac{\mu(v)}{\mu(u)}\right)\right\|_{\Hb^{s-1}},
  \]
  where the constant $C$ is uniform for $\|u\|_{\cX^{s,\alpha}},\|v\|_{\cX^{s,\alpha}}\leq R$. Thus,
  \begin{align*}
    \|(\Box_{g(v)}&-\Box_{g(u)})S_{g(v)}\|_{\cL(\Hb^{s-1,\alpha})} \\
	  &\leq \left\|1-\frac{\mu(v)}{\mu(u)}\right\|_{\cL(\Hb^{s-1,\alpha})} + \|E_{u,v}\|_{\cL(\cX^{s,\alpha},\Hb^{s-1,\alpha})}\|S_{g(v)}\|_{\cL(\Hb^{s-1,\alpha},\cX^{s,\alpha})} \\
	  &\leq \left\|1-\frac{\mu(v)}{\mu(u)}\right\|_{\cX^{s-1,0}}+C C_S\left\|\bdiff\left(\frac{\mu(v)}{\mu(u)}\right)\right\|_{\Hb^{s-1}}.
  \end{align*}
  Now,
  \begin{equation}
  \label{EqMuFractionEstimate}
  \begin{split}
    \left\|1-\frac{\mu(v)}{\mu(u)}\right\|_{\cX^{s-1,0}} &\leq C'\left\|\frac{1}{\mu(u)}\right\|_{\cX^{s-1,0}}\|\mu(u)-\mu(v)\|_{\cX^{s-1,0}} \\
	  &\leq C'_i L_\mu(R) \|u-v\|_{\cX^{s,\alpha}},
  \end{split}
  \end{equation}
  where
  \[
    C'_i:=C'\sup_{\|w\|_{\cX^{s,\alpha}}\leq R}\left\|\frac{1}{\mu(w)}\right\|_{\cX^{s,0}}<\infty
  \]
  by assumption~\eqref{EqMuBound} and Lemma~\ref{LemmaRecHs}. Likewise, since $\bdiff(\mu(v)/\mu(u))=\bdiff\bigl(\mu(v)/\mu(u)-1\bigr)$,
  \[
    \left\|\bdiff\left(\frac{\mu(v)}{\mu(u)}\right)\right\|_{\Hb^{s-1}} \leq \left\|1-\frac{\mu(v)}{\mu(u)}\right\|_{\cX^{s,0}} \leq C'_i L_\mu(R) \|u-v\|_{\cX^{s,\alpha}};
  \]
  therefore,
  \begin{equation}
  \label{EqDiffOfBoxEstimate}
    \|(\Box_{g(v)}-\Box_{g(u)})S_{g(v)}\|_{\cL(\Hb^{s-1,\alpha})}\leq C_i L_\mu(R) \|u-v\|_{\cX^{s,\alpha}}
  \end{equation}
  for $C_i=C_i'(1+CC_S)$, and with $\|S_{g(u)}\|_{\cL(\Hb^{s-1,\alpha},\cX^{s,\alpha})}\leq C_S$ and the identity in \eqref{EqResolventId}, we finally obtain the estimate~\eqref{EqSolOpLipschitz}.

  We proceed to prove the second half of the theorem along the lines of the proof of \cite[Theorem~2.24]{HintzVasySemilinear}. We work on the Banach space
  \begin{equation}
  \label{EqSpacecY}
    \cY^{s,\alpha}:=\{u\in\cX^{s,\alpha}\colon \Box_{g_0}u\in\Hb^{s-1,\alpha}\}, \quad \|u\|_{\cY^{s,\alpha}}=\|u\|_{\cX^{s,\alpha}}+\|\Box_{g_0}u\|_{\Hb^{s-1,\alpha}}.
  \end{equation}
  The idea is that all operators $\Box_{g(u)}$ are (pointwise) multiples of each other modulo first order operators, thus $\Box_{g_0}$ is as good as any other such operator, and therefore $\Box_{g_0}$ in the third argument of the non-linearity $q$ acts as a first order operator on the successive approximations $T^k(0)$ in the iteration scheme implicit in the application of the Banach fixed point theorem used above to solve equation~\eqref{EqQuasilinearWavePDECopy}. Thus, let $B(R)$ denote the metric ball of radius $R\leq R_S$ in $\cY^{s,\alpha}$, and define $T\colon B(R)\to\cY^{s,\alpha}$,
  \[
    Tu:=S_{g(u)}\bigl(f+q(u)\bigr)
  \]
  where we write $q(u):=q(u,\bdiff u,\Box_{g(u)}u)$ to simplify the notation. We will prove that for $R>0$ small enough, the image of $T$ is contained in $B(R)$. We first estimate for $u\in B(R)$ and $w\in\cY^{s,\alpha}$, using \eqref{EqDiffOfBox} and an estimate similar to \eqref{EqMuFractionEstimate} (with $v=0$):
  \begin{align*}
    \|\Box_{g(u)}&w\|_{\Hb^{s-1,\alpha}}\leq\|\Box_{g(0)}w\|_{\Hb^{s-1,\alpha}}+\|(\Box_{g(u)}-\Box_{g(0)})w\|_{\Hb^{s-1,\alpha}} \\
	  &\leq\|w\|_{\cY^{s,\alpha}} + \wt C_i\|u\|_{\cX^{s,\alpha}}\|w\|_{\cY^{s,\alpha}} \leq (1+\wt C_i R)\|w\|_{\cY^{s,\alpha}}
  \end{align*}
  for some constant $\wt C_i>0$. For convenience, we choose $R\leq\wt C_i^{-1}$, thus
  \[
    \|\Box_{g(u)}w\|_{\Hb^{s-1,\alpha}}\leq 2\|w\|_{\cY^{s,\alpha}}, \quad w\in\cY^{s,\alpha}.
  \]
  Using this, we obtain for $u,v\in B(R)$:
  \begin{align*}
    \|\Box_{g(u)}u&-\Box_{g(v)}v\|_{\Hb^{s-1,\alpha}}\leq \|\Box_{g(u)}(u-v)\|_{\Hb^{s-1,\alpha}}+\|(\Box_{g(u)}-\Box_{g(v)})v\|_{\Hb^{s-1,\alpha}} \\
	  &\leq 2\|u-v\|_{\cY^{s,\alpha}}+\left\|\left(\Bigl(1-\frac{\mu(u)}{\mu(v)}\Bigr)\Box_{g(u)}-E_{v,u}\right)v\right\|_{\Hb^{s-1,\alpha}} \\
	  &\leq 2\|u-v\|_{\cY^{s,\alpha}}+C'L_\mu(R)\|u-v\|_{\cX^{s,\alpha}}\left(\|\Box_{g(u)}v\|_{\Hb^{s-1,\alpha}}+\|v\|_{\cX^{s,\alpha}}\right) \\
	  &\leq (2+3C'L_\mu(R)R)\|u-v\|_{\cY^{s,\alpha}} \leq 3\|u-v\|_{\cY^{s,\alpha}}
  \end{align*}
  for sufficiently small $R$, where $C'=C'_i(1+C)$. Thus, with $L'_q(R):=3L_q(R)$, we have
  \[
    \|q(u)-q(v)\|_{\Hb^{s-1,\alpha}}\leq L'_q(R)\|u-v\|_{\cY^{s,\alpha}}
  \]
  for $u,v\in\cY^{s,\alpha}$ with norm $\leq R$.
  
  We can now analyze the map $T$: First, for $u\in B(R)$ and $f\in\Hb^{s-1,\alpha}$, $\|f\|\leq C_f$, we have, recalling \eqref{EqDiffOfBoxEstimate}, here applied with $v=0$,
  \[
    \|Tu\|_{\cX^{s,\alpha}} \leq C_S(C_f+L'_q(R)R)
  \]
  and
  \begin{align*}
	\|\Box_{g(0)}Tu\|_{\Hb^{s-1,\alpha}}&\leq \|(\Box_{g(0)}-\Box_{g(u)})S_{g(u)}(f+q(u))\|_{\Hb^{s-1,\alpha}} \\
	  &\hspace{3ex}+\|f+q(u)\|_{\Hb^{s-1,\alpha}} \\
	  &\leq (1+C_iL_\mu(R)R)(C_f+L'_q(R)R).
  \end{align*}
  Thus, if $L'_q(0)<(1+C_S)^{-1}$, then
  \[
    C_f(R) := R\bigl((1+C_S+C_iL_\mu(R)R)^{-1}-L'_q(R)\bigr)
  \]
  is positive for small enough $R>0$. We conclude that for $f\in\Hb^{s-1,\alpha}$ with norm $\leq C_f(R)$, the map $T$ indeed maps $B(R)$ into itself. We next have to check that $T$ is in fact a contraction on $B(R)$, where we choose $R$ even smaller if necessary. As in the proof of the first half of the theorem, we can arrange
  \begin{equation}
  \label{EqTuTv}
    \|Tu-Tv\|_{\cX^{s,\alpha}}\leq\theta\|u-v\|_{\cY^{s,\alpha}},\quad u,v\in B(R)
  \end{equation}
  for some fixed $\theta<1$. Moreover, for $u,v\in B(R)$,
  \begin{equation}
  \label{EqBoxTuTv}
    \begin{split}
	  \|\Box_{g(0)}(Tu-Tv)\|_{\Hb^{s-1,\alpha}}&\leq \|\Box_{g(0)}S_{g(u)}(q(u)-q(v))\|_{\Hb^{s-1,\alpha}} \\
	    &\hspace{3ex}+\|\Box_{g(0)}(S_{g(u)}-S_{g(v)})(f+q(v))\|_{\Hb^{s-1,\alpha}}.
	\end{split}
  \end{equation}
  The first term on the right can be estimated by
  \begin{align*}
    \|q(u)-&q(v)\|_{\Hb^{s-1,\alpha}}+\|(\Box_{g(u)}-\Box_{g(0)})S_{g(u)}(q(u)-q(v))\|_{\Hb^{s-1,\alpha}} \\
	  &\leq L'_q(R)(1+C_iL_\mu(R)R) \|u-v\|_{\cY^{s,\alpha}}.
  \end{align*}
  For the second term on the right hand side of \eqref{EqBoxTuTv}, we use the algebraic identity
  \[
    \Box_{g(0)}(S_{g(u)}-S_{g(v)})=(I+(\Box_{g(0)}-\Box_{g(u)})S_{g(u)})(\Box_{g(v)}-\Box_{g(u)})S_{g(v)},
  \]
  which gives
  \[
    \|\Box_{g(0)}(S_{g(u)}-S_{g(v)})\|_{\cL(\cX^{s-1,\alpha})} \leq (1+C_iL_\mu(R)R) C_i L_\mu(R) \|u-v\|_{\cY^{s,\alpha}}.
  \]
  Plugging this into equation~\eqref{EqBoxTuTv}, we obtain
  \[
    \|\Box_{g(0)}(Tu-Tv)\|_{\Hb^{s-1,\alpha}}\leq C'(R)\|u-v\|_{\cY^{s,\alpha}}
  \]
  with
  \[
    C'(R)=(1+C_iL_\mu(R)R)\bigl(L'_q(R)+C_i L_\mu(R)(C_f(R)+L'_q(R)R)\bigr).
  \]
  Now if $L'_q(0)$ is sufficiently small, then since the second summand of the second factor of $C'(R)$ tends to $0$ as $R\to 0$, we can choose $R$ so small that $C'(R)<1-\theta$, and we finally get with \eqref{EqTuTv}:
  \[
    \|Tu-Tv\|_{\cY^{s,\alpha}}\leq\theta'\|u-v\|_{\cY^{s,\alpha}},\quad u,v\in B(R),
  \]
  for some $\theta'<1$, which proves that $T$ is a contraction on $B(R)$, thus has a unique fixed point, which solves the PDE~\eqref{EqQuasilinearWavePDEQ}. The continuous dependence on $f$ is shown as in the proof of the first half of the theorem.
\end{proof}

\begin{rmk}
  The space $\cY^{s,\alpha}$ introduced in the proof of the second part, see equation~\eqref{EqSpacecY}, which the solution $u$ of equation~\eqref{EqQuasilinearWavePDEQ} belongs to, is a coisotropic space similar to the ones used in \cite{VasyMicroKerrdS,HintzVasySemilinear}, with the difference being that here $\Box_{g_0}$ is allowed to have non-smooth coefficients. It still is a natural space in the sense that the space of elements of the form $c(\phi\circ\ft_1)+w$, $c\in\C,w\in\CIdotc$, is dense. Indeed, since $\Box_{g_0}$ annihilates constants, it suffices to check that $\CIdotc$ is dense in $\cY_0^{s,\alpha}:=\{u\in\Hb^{s,\alpha}\colon\Box_{g_0}u\in\Hb^{s-1,\alpha}\}$. Let $J_\eps$ be a mollifier as in Lemma~\ref{LemmaMollifier}. Given $u\in\cY_0^{s,\alpha}$, put $u_\eps:=J_\eps u$. Then $u_\eps\to u$ in $\Hb^{s,\alpha}$, and
  \[
    \Box_{g_0}u_\eps=J_\eps\Box_{g_0}u+[\Box_{g_0},J_\eps]u;
  \]
  the first term converges to $\Box_{g_0}u$ in $\Hb^{s-1,\alpha}$. To analyze the second term, observe that we have
  \[
    \Box_{g_0}J_\eps-J_\eps\Box_{g_0}=\Box_{g_0}(J_\eps-I)+(I-J_\eps)\Box_{g_0}\to 0\tn{ strongly in }\cL(\Hb^{s+1,\alpha},\Hb^{s-1,\alpha}),
  \]
  and since $\Hb^{s+1,\alpha}\subset\Hb^{s,\alpha}$ is dense, it suffices to show that $[\Box_{g_0},J_\eps]$ is a bounded family in $\cL(\Hb^{s,\alpha},\Hb^{s-1,\alpha})$. Write
  \[
    \Box_{g_0}=Q_1+Q_2+E,\quad Q_1\in\Diffb^2,Q_2\in\Hb^{s,\alpha}\Diffb^2, E\in(\CI+\Hb^{s-1,\alpha})\Diffb^1.
  \]
  Then $[Q_1,J_\eps]$ and $[E,J_\eps]$ are bounded in $\cL(\Hb^{s,\alpha},\Hb^{s-1,\alpha})$. Now $Q_2J_\eps$ can be expanded into a leading order term $Q'_\eps$ and a remainder $R_{1,\eps}$ which is uniformly bounded in $\Hb^s\Psib^1$; but also $J_\eps Q_2$ has an expansion by Theorem~\ref{ThmComp} \itref{ThmCompBR} (with $k=k'=1$) into the same leading order term $Q'_\eps$ and a remainder $R_{2,\eps}$ which is uniformly bounded in $\fpsib^{1;0}\Hb^{s-1}$. Hence $[Q_2,J_\eps]=R_{1,\eps}-R_{2,\eps}$ is bounded in $\cL(\Hb^{s,\alpha},\Hb^{s-1,\alpha})$ by Proposition~\ref{PropOpCont}, finishing the argument.
\end{rmk}

%%%%%%%%%%%%%%%%%%%%%%%%%%%%%%%%%%%%%%%%%%%%%%%%%
\subsection{Quasilinear Klein-Gordon equations}
\label{SubsecQlKleinGordon}

One has corresponding results to the theorems in the previous two sections for quasilinear Klein-Gordon equations, i.e.\ for Theorems~\ref{ThmAppWaveExistence}, \ref{ThmAppWavePoly} and \ref{ThmAppWaveMultMetric} with $\Box$ replaced by $\Box-m^2$; only the function spaces need to be adapted to the situation at hand, as follows: Denote $P:=\Box_{g_\dS}-m^2$ and let $(\sigma_j)_{j\in\N}$ be the sequence of poles of $\widehat P(\sigma)^{-1}$, with multiplicity, sorted by increasing $-\Im\sigma_j$.\footnote{See equation~\eqref{EqDSResonances} for the explicit formula. Also keep in mind that everything we do works in greater generality; we stick to the case of exact de Sitter space here for clarity.} Let us assume that the `mass' $m\in\C$ is such that $\Im\sigma_1<0$. A major new feature of Klein-Gordon equations as compared to wave equations is that non-linearities like $q(u)=u^p$ can be dealt with, more generally
\[
  q(u,\bdiff u)=\sum_j u^{e_j}\prod_{l=1}^{N_j} X_{jk}u,\quad e_j+N_r\geq 2, X_{jk}\in(\CI+\Hb^\infty)\Vb.
\]
See \cite[Theorem~2.24]{HintzVasySemilinear} for the related discussion of semilinear equations. We give an (incomplete) short list of possible scenarios and the relevant function spaces; for concreteness, we work on exact de Sitter space, but our methods work in much greater generality.

\begin{enumerate}[leftmargin=\enummargin]
  \item If $\Im\sigma_1\neq\Im\sigma_2$, as is e.g.\ the case for small mass $m^2<(n-1)^2/4$, let $\alpha_0=\min(1,\Im\sigma_1-\Im\sigma_2)$, and for $-\Im\sigma_1<\alpha<-\Im\sigma_1+\alpha_0$, put
    \[
	  \cX^{s,\alpha}:=\C(\tau^{i\sigma_1})\oplus\Hb^{s,\alpha}.
	\]
	We can then solve quasilinear equations of the form explained above with forcing in $\Hb^{s-1,\alpha}$ and get one term, $c\tau^{i\sigma_1}$, in the expansion of the solution. Notice that if the mass is real and small, then all $\sigma_j$ are purely imaginary, hence the term in the expansion is real as well if all data are, which is necessary for an analogue of Theorem~\ref{ThmAppWavePoly} to hold.
  \item If $\Im\sigma_1-\Im\sigma_2<1$, e.g.\ if $m^2\geq n(n-2)/4$, let $\alpha_0:=\min(1,\Im\sigma_1-\Im\sigma_3)$, and for $-\Im\sigma_2<\alpha<-\Im\sigma_1+\alpha_0$, put
  \begin{gather*}
    \cX^{s,\alpha}:=\C(\tau^{i\sigma_1})\oplus\C(\tau^{i\sigma_2})\oplus\Hb^{s,\alpha},\quad \sigma_2\neq\sigma_1, \\
	\cX^{s,\alpha}:=\C(\tau^{i\sigma_1})\oplus\C(\tau^{i\sigma_1}\log\tau)\oplus\Hb^{s,\alpha},\quad\sigma_2=\sigma_1,
  \end{gather*}
  then we can solve equations as above with forcing in $\Hb^{s-1,\alpha}$ and obtain two terms in the expansion. For masses $m^2>(n-1)^2/4$, we have $\Im\sigma_1=\Im\sigma_2=:-\sigma$ and $\Re\sigma_1=-\Re\sigma_2=:\rho$, hence the terms in the expansion for real data are a linear combination of $\tau^\sigma\cos(\rho\tau)$ and $\tau^\sigma\sin(\rho\tau)$.
  \item If the forcing decays more slowly than $\tau^{i\sigma_1}$, then with $0<\alpha<-\Im\sigma_1$, we can work on the space
    \[
	  \cX^{s,\alpha}:=\Hb^{s,\alpha},
	\]
	with forcing in $\Hb^{s-1,\alpha}$.
\end{enumerate}

To prove the higher regularity statement in Theorem~\ref{ThmAppWavePoly} for quasilinear Klein-Gordon equations, one first obtains higher regularity $\Hb^{s',\alpha}$ with $0\leq\alpha<-\Im\sigma_1$ and then, if the amount of decay of the forcing is high enough to allow for it, applies Theorem~\ref{ThmExistenceAndExpansion} to obtain a partial expansion of $u$.

In the third setting, the assumption that the mass $m$ is independent of the solution $u$ can easily be relaxed: Namely, assuming that $m=m(u)$ or $m=m(u,\bdiff u)$ with continuous (or Lipschitz) dependence on $u\in\cX^{s,\alpha}$, the poles of the inverse of the normal operator family of $\Box_{g(u)}-m(u)^2$ depend continuously on $u$, hence for small $u$, there is still no pole with imaginary part $\geq-\alpha$, therefore the solution operator produces an element of $\Hb^{s,\alpha}$ for small $u$; thus, well-posedness results analogous to Theorems~\ref{ThmAppWaveExistence} and \ref{ThmAppWaveMultMetric} continue to hold in this setting. If the forcing in fact does decay faster than $\tau^{i\sigma_1}$, these results can be improved in many cases: Once one has the solution $u\in\Hb^{s,\alpha}$, in particular the mass $m(u)$ is now fixed, one can apply Theorem~\ref{ThmExistenceAndExpansion} to obtain a partial expansion of $u$.

%%%%%%%%%%%%%%%%%%%%%%%%%%%%%%%%%%%%%%%%%%%%%%%%%
\subsection{Backward problems}
\label{SubsecBackward}

We briefly indicate how our methods also apply to backward problems on static patches of (asymptotically) de Sitter spaces; see Figure~\ref{FigBackward} for an exemplary setup.

\begin{figure}[!ht]
\centering
  \includegraphics{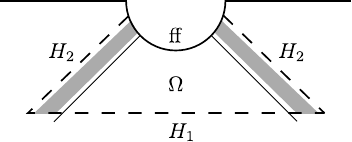}
  \caption{Setup for a backward problem on static de Sitter space: We work on spaces with high decay, consisting of functions supported at $H_2$ and extendible at $H_1$ (notice the switch compared to the forward problem). In the situation shown, we prescribe initial data at $H_2$ or, put differently, forcing in the shaded region.}
\label{FigBackward}
\end{figure}

We only state an analogue of Theorem~\ref{ThmAppWavePoly}, but remark that analogues of Theorems~\ref{ThmAppWaveExistence} and \ref{ThmAppWaveMultMetric} also hold. For simplicity, we again only work on static de Sitter spaces. We use the notation from \S\ref{SubsecHypotheses}.

\begin{thm}
\label{ThmAppWaveBack}
  Let $s>n/2+6$, $N,N'\in\N$, and suppose $c_k\in\CI(\R;\R)$, $g_k\in(\CI+\Hb^s)(M;\Sym^2\Tb^*M)$ for $1\leq k\leq N$; for $r\in\R$, define the map
  \[
    g\colon\Hb^{s,r}(\Omega)^{-,\bullet}\to(\CI+\Hb^{s,r})(M;\Sym^2\Tb^*M),\quad g(u)=\sum_{k=1}^N c_k(u)g_k,
  \]
  and assume $g(0)=g_\dS$. Moreover, define
  \[
    q(u,\bdiff u)=\sum_{j=0}^{N'} u^{e_j}\prod_{k=1}^{N_j} X_{jk}u,\quad e_j+N_j\geq 2, X_{jk}\in(\CI+\Hb^\infty)\Vb(M),
  \]
  and let further $L\in\Diffb^1$ with real coefficients. Then there is $r_*\in\R$ such that for all $r>r_*$, the following holds: For small $R>0$, there exists $C_f>0$ such that for all $f\in\Hb^{s-1,r}(\Omega;\R)^{-,\bullet}$ with norm $\leq C_f$, the equation
  \[
    (\Box_{g(u)}+L)u=f+q(u,\bdiff u)
  \]
  has a unique solution $u\in\Hb^{s,r}(\Omega;\R)^{-,\bullet}$ with norm $\leq R$, and in the topology of $\Hb^{s-1,r}(\Omega)^{-,\bullet}$, $u$ depends continuously on $f$. If one in fact has $f\in\Hb^{s'-1,r}(\Omega;\R)^{-,\bullet}$ for some $s'\in(s,\infty]$, then $u\in\Hb^{s',r}(\Omega;\R)^{-,\bullet}$.
\end{thm}

\begin{rmk}
\label{RmkBackwardEasy}
  Notice that the structure of lower order terms is completely irrelevant here! One could in fact let $L$ depend on $u$ in a Lipschitz fashion and still have well-posedness.
\end{rmk}

\begin{proof}[Proof of Theorem~\ref{ThmAppWaveBack}]
  Let $r_0<0$ as given by Lemma~\ref{LemmaWaveGlobalEnergy}, and suppose $r>-r_0$. As in the proof of Lemma~\ref{LemmaWaveGlobalSolve}, we obtain for $u\in\Hb^{s,r}(\Omega)^{-,\bullet}$ with $\|u\|\leq R$, $R>0$ sufficiently small, a \emph{backward} solution operator
  \[
    S_{g(u)}\colon\Hb^{-1,r}(\Omega)^{-,\bullet}\to\Hb^{0,r}(\Omega)^{-,\bullet}
  \]
  for $\Box_{g(u)}+L$, with uniformly bounded operator norm. Now, if we take $r>r_*$ with $r_*\geq -r_0$ sufficiently large, $S_{g(u)}$ restricts to an operator
  \[
    S_{g(u)}\colon\Hb^{s-1,r}(\Omega)^{-,\bullet}\to\Hb^{s,r}(\Omega)^{-,\bullet}.
  \]
  Indeed, given $v\in\Hb^{0,r}(\Omega)^{-,\bullet}$ solving $\Box_{g(u)}v\in\Hb^{s-1,r}(\Omega)^{-,\bullet}$, we apply the propagation near radial points, Theorem~\ref{ThmRadialPoints}, this time propagating regularity \emph{away from} the boundary, and the real principal type propagation and elliptic regularity iteratively to prove $v\in\Hb^{s,r}(\Omega)^{-,\bullet}$; the last application of the radial points result requires that $r$ be larger than an $s$-dependent quantity, hence the condition on $r_*$ in the statement of the theorem. From here, a Picard iteration argument, namely considering
  \[
    u\mapsto S_{g(u)}(f+q(u,\bdiff u)),
  \]
  gives existence and well-posedness. The higher regularity statement is proved as in the proof of Theorem~\ref{ThmAppWavePoly}.
\end{proof}

A slightly more elaborate version of this theorem, applied to the Einstein vacuum equations, should enable us to construct vacuum asymptotically de Sitter spacetimes as done in the Kerr setting in \cite{DafermosHolzegelRodnianskiKerrBw}. In fact, apart from constructing appropriate initial data, this should work in the Kerr-de Sitter setting as well, yielding the existence of dynamical vacuum black holes in de Sitter spacetimes; the point here is that for the backward problem, one works in decaying spaces, where one has non-trapping estimates in the smooth setting, as proved in \cite{HintzVasyNormHyp} by a positive commutator argument, which, along the lines of the proofs in \S\ref{SecPropagation}, holds in the non-smooth setting as well.\footnote{Indeed, the latter has recently been shown in \cite{HintzVasyQuasilinearKdS}.} We point out however that the authors of \cite{DafermosHolzegelRodnianskiKerrBw} consider a \emph{characteristic} problem, whereas our analysis, without further modifications, would require initial data on \emph{spacelike} hypersurfaces placed beyond the horizons, which makes the construction of initial data much more difficult.

%%%%%%%%%%%%%%%%%%%%%%%%%%%%%%%%%%%%%%%%%%%%%%%%%%%%%%%%%%%%%%%%%%%%%
\bibliographystyle{plain}
\bibliography{../bib/math,../bib/mathcheck}

\end{document}